\documentclass{article}

\usepackage{amsthm, amsmath, epsfig}

\usepackage{amsfonts, amssymb, latexsym}
\usepackage{pgf}
\usepackage{tikz}
\usepackage{pgfplots}
\usetikzlibrary{arrows.meta}

\def\Uhatn{\widehat{U}^{(n)}}
\def\Vhatn{\widehat{V}^{(n)}}
\def\Whatn{\widehat{W}^{(n)}}
\def\Xhatn{\widehat{X}^{(n)}}
\def\Yhatn{\widehat{Y}^{(n)}}
\def\Zhatn{\widehat{Z}^{(n)}}
\def\ve{\varepsilon}
\def\R{\mathbb{R}} 
\def\ind{\mathbb{I}}
\def\sE{{\cal E}}
\def\sB{{\cal B}}
\def\P{{\mathbb P}}
\def\E{{\mathbb E}}
\def\sD{{\cal D}}
\def\sF{{\cal F}}
\def\W{\mathbb{W}}
\def\sWn{{\cal W}^n}
\def\sXn{{\cal X}^n}
\def\sV{{\cal V}}
\def\sW{{\cal W}}
\def\sX{{\cal X}}
\def\sY{{\cal Y}}
\def\sU{{\cal U}}
\def\sZ{{\cal Z}}
\def\sA{{\cal A}}
\def\sAhatn{\widehat{\cal A}^n}
\def\sUhatn{\widehat{\sU}^n}
\def\sZhatn{\widehat{\sZ}^n}
\def\Ghatn{\widehat{G}^n}
\def\Hhatn{\widehat{H}^n}
\def\sVhatn{\widehat{\sV}^n}
\def\sVhat{\widehat{\sV}}
\def\sWhatn{\widehat{\sW}^n}
\def\sXhatn{\widehat{\sX}^n}
\def\sYhatn{\widehat{\sY}^n}
\def\Mhatn{\widehat{M}^n}
\def\Pbarn{\overline{P}^n}
\def\Pbar{\overline{P}}
\def\Phatn{\widehat{P}^n}
\def\F{{\mathbb F}}
\def\sF{{\cal F}}
\def\Psihatn{\widehat{\Psi}^n}
\def\Phihatn{\widehat{\Phi}^n}
\def\Thetahatn{\widehat{\Theta}^n}
\def\sR{{\cal R}}
\def\ArrowJ1{\stackrel{J_1}{\Longrightarrow}}
\def\ArrowM_1{\stackrel{M_1}{\Longrightarrow}}
\def\Pbar{\overline{P}}
\def\sS{{\cal S}}
\def\ve{\varepsilon}

\def\Ctilde{\widetilde{C}}
\def\Ehat{\widehat{E}}
\def\Dhat{\widehat{D}}

\def\sL{{\cal L}}
\def\Shatn{\widehat{S}^n}

\newtheorem{theorem}{Theorem}[section]
\newtheorem{definition}[theorem]{Definition}
\newtheorem{corollary}[theorem]{Corollary}
\newtheorem{proposition}[theorem]{Proposition}
\newtheorem{lemma}[theorem]{Lemma}
\newtheorem{assumption}[theorem]{Assumption}
\newtheorem{remark}[theorem]{Remark}

\newcommand{\ba}{\begin{array}}
\newcommand{\ea}{\end{array}}
\newcommand{\be}{\begin{equation}}
\newcommand{\ee}{\end{equation}}
\newcommand{\bi}{\begin{itemize}}
\newcommand{\ei}{\end{itemize}}

\newcommand{\noi}{\noindent}
\newcommand{\bt}{\begin{theorem}}
\newcommand{\et}{\end{theorem}}
\newcommand{\bc}{\begin{center}}
\newcommand{\ec}{\end{center}}
\newcommand{\bass}{\begin{assumption}}
\newcommand{\eass}{\end{assumption}}
\newcommand{\ben}{\begin{enumerate}}
\newcommand{\een}{\end{enumerate}}

\newcommand{\vsp}{\vspace{\baselineskip}}

\DeclareMathOperator{\id}{id}

\begin{document}

\begin{center}
{\bf\Huge Diffusion Limit of Poisson}\\

\vspace{8pt}
{\bf\Huge Limit-Order Book Models}\\
\mbox{}\\
\today
\end{center}
\begin{center}
Christopher Almost\footnote{Partially supported by
the Natural Sciences and Engineering Research Council of Canada
PGS D3 under Grant No.\ 358495.  
This paper includes personal views
and opinions of Dr.\ Almost, and they have not been
reviewed or endorsed by Waterfront International Ltd.}\\
Waterfront International, Ltd.\\
Toronto, ON\\
Canada\\
\texttt{cdalmost@gmail.com}\\
\mbox{}\\
John Lehoczky\\
Department of Statistics and Data Science\\
Carnegie Mellon University\\
Pittsburgh, PA 15213\\
USA\\
\texttt{jpl@stat.cmu.edu}\\
\mbox{}\\
Steven Shreve\footnote{Partially supported 
by the National Science Foundation under 
Grant No. DMS-0903475.}
\\
Department of Mathematical Sciences\\
Carnegie Mellon University\\
Pittsburgh, PA  15213\\
USA\\
\texttt{shreve@andrew.cmu.edu}\\
\mbox{}\\
Xiaofeng Yu\footnote{The views represented here are those of the 
author and not those of Millennium International Management LP.}\\
Millenium International Management LP\\
New York, NY\\
USA\\
\texttt{yuxiaofeng1987@gmail.com}
\end{center}

\begin{flushleft}
{\bf Short title:} Diffusion Limit
\end{flushleft}

\newpage
\begin{abstract}
This ia a companion paper to Almost, Lehoczky,
Shreve \& Yu \cite{ALSY}, where the rationale
for studying the diffusion limit
of Poisson limit-order book models
is explained and the results of a particular
``representative'' model are detailed. This 
paper contains the proofs and
technical details cited in that work.
\end{abstract}

\section{Introduction and Main Results}
\label{SecIntro}

\setcounter{equation}{0}
\setcounter{theorem}{0}

We consider a limit-order book in which
buy and sell orders arrive at various
price ticks according to Poisson processes.
Orders cancel when an exponentially
distributed amount of time
is accumulated too far from the bid
or ask prices.  We scale the rate of
order arrival by a parameter $n$,
the number of orders queued at
the price ticks by $\sqrt{n}$,
and pass to the limit.  The limit
is shown to exist and its statistics
are identified.    
An informal description of
the nature of the limit appears in
\cite{ALSY}, which also surveys
the relevant literature and provides
the practical context for this work.  
Precise statements
of the results with proofs are
contained in this paper.
This paper is based on the PhD
dissertations of Almost \cite{Almost}
and Yu \cite{Yu}.

Consider six adjacent
price ticks in a doubly-infinite
limit-order book as shown in Figure 1.1.
We show queued limit buy orders as
positive histograms and queued limit
sell orders as negative histograms.
The bid price is the largest price
at which at least one buy order is queued,
and the ask price is the smallest
price at which at least one sell
order is queued.
We assume that limit buy orders,
all of size one, arrive according
to Poisson processes with rates
$\lambda_1>0$ and $\lambda_2>0$
one and two ticks below the ask,
respectively.  Similarly,
limit sell orders, all of
size one, arrive according to
Poisson processes one and two ticks
above the bid at rate $\mu_1$
and $\mu_2$, respectively.
Market buy orders, all of size one, arrive at
the ask according to a Poisson
process with rate $\lambda_0$;
market sell orders, all of size one, arrive at the
bid according to a Poisson process
with rate $\mu_0$.
Finally, we assume that buy orders two
or more ticks below the bid price
are subject to cancellation at rate
$\theta_b/\sqrt{n}$ (per order).
Analogously, sell orders two or more ticks above
the ask price are subject to cancellation at rate
$\theta_s/\sqrt{n}$ (per order).
We denote the order queue lengths at
the six price ticks by $U^n$, $V^n$,
$W^n$, $X^n$, $Y^n$ and $Z^n$
as indicated in Figure 1.1.

\begin{figure}[h]
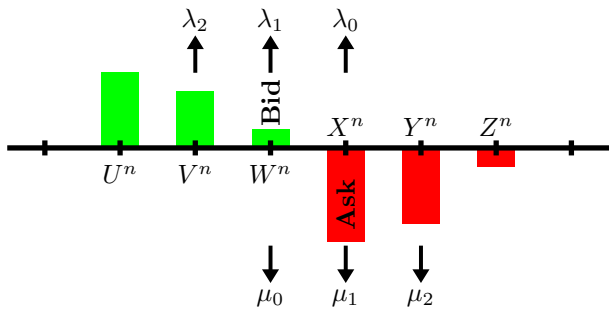
\label{F1.2}
\begin{pgfpicture}{-0cm}{-1.0cm}{14cm}{3.0cm}
%
%
\begin{pgfscope}
\pgfsetlinewidth{2pt}
\color{green}
\pgfmoveto{\pgfxy(4.25,1.00)}
\pgflineto{\pgfxy(4.25,2.00)}
\pgflineto{\pgfxy(4.75,2.00)}
\pgflineto{\pgfxy(4.75,1.00)}
\pgffill
\pgfclosestroke
\pgfmoveto{\pgfxy(5.25,1.00)}
\pgflineto{\pgfxy(5.25,1.75)}
\pgflineto{\pgfxy(5.75,1.75)}
\pgflineto{\pgfxy(5.75,1.00)}
\pgffill
\pgfclosestroke
\pgfmoveto{\pgfxy(6.25,1.00)}
\pgflineto{\pgfxy(6.25,1.25)}
\pgflineto{\pgfxy(6.75,1.25)}
\pgflineto{\pgfxy(6.75,1.00)}
\pgffill
\pgfclosestroke
\end{pgfscope}
\begin{pgfscope}
\pgfsetlinewidth{2pt}
\color{red}
\pgfmoveto{\pgfxy(7.25,1.00)}
\pgflineto{\pgfxy(7.25,-0.25)}
\pgflineto{\pgfxy(7.75,-0.25)}
\pgflineto{\pgfxy(7.75,1.00)}
\pgffill
\pgfclosestroke
\pgfmoveto{\pgfxy(8.25,1.00)}
\pgflineto{\pgfxy(8.25,0.00)}
\pgflineto{\pgfxy(8.75,0.00)}
\pgflineto{\pgfxy(8.75,1.00)}
\pgffill
\pgfclosestroke
\pgfmoveto{\pgfxy(9.25,1.00)}
\pgflineto{\pgfxy(9.25,0.75)}
\pgflineto{\pgfxy(9.75,0.75)}
\pgflineto{\pgfxy(9.75,1.00)}
\pgffill
\pgfclosestroke
\end{pgfscope}
\begin{pgfscope}
\pgfsetlinewidth{2pt}
\pgfxyline(3,1)(11,1)
\pgfxyline(3.5,0.9)(3.5,1.1)
\pgfxyline(4.5,0.9)(4.5,1.1)
\pgfxyline(5.5,0.9)(5.5,1.1)
\pgfxyline(6.5,0.9)(6.5,1.1)
\pgfxyline(7.5,0.9)(7.5,1.1)
\pgfxyline(8.5,0.9)(8.5,1.1)
\pgfxyline(9.5,0.9)(9.5,1.1)
\pgfxyline(10.5,0.9)(10.5,1.1)
\end{pgfscope}
\pgfputlabelrotated{0.40}{\pgfxy(6.50,1.2)}%
{\pgfxy(6.50,2.2)}{0pt}{\pgfbox[center,center]{{\bf Bid}}}
\pgfputlabelrotated{0.40}{\pgfxy(7.50,-0.19)}%
{\pgfxy(7.50,0.90)}{0pt}{\pgfbox[center,center]{{\bf Ask}}}
\pgfputat{\pgfxy(4.5,0.65)}%
{\pgfbox[center,center]{$U^n$}}
\pgfputat{\pgfxy(5.5,0.65)}%
{\pgfbox[center,center]{$V^n$}}
\pgfputat{\pgfxy(6.5,0.65)}%
{\pgfbox[center,center]{$W^n$}}
\pgfputat{\pgfxy(7.5,1.28)}%
{\pgfbox[center,center]{$X^n$}}
\pgfputat{\pgfxy(8.5,1.28)}%
{\pgfbox[center,center]{$Y^n$}}
\pgfputat{\pgfxy(9.5,1.28)}%
{\pgfbox[center,center]{$Z^n$}}
\begin{pgfscope}
\pgfsetlinewidth{1.5pt}
\pgfsetarrowsend{Triangle[scale=0.75pt]}
\pgfxyline(6.5,-0.3)(6.5,-0.8)
\pgfxyline(7.5,-0.3)(7.5,-0.8)
\pgfxyline(8.5,-0.3)(8.5,-0.8)
\pgfxyline(7.5,2.0)(7.5,2.5)
\pgfxyline(6.5,2.0)(6.5,2.5)
\pgfxyline(5.5,2.0)(5.5,2.5)
\pgfputat{\pgfxy(6.5,-1.0)}%
{\pgfbox[center,center]{$\mu_0$}}
\pgfputat{\pgfxy(7.5,-1.0)}%
{\pgfbox[center,center]{$\mu_1$}}
\pgfputat{\pgfxy(8.5,-1.0)}%
{\pgfbox[center,center]{$\mu_2$}}
\pgfputat{\pgfxy(5.5,2.7)}%
{\pgfbox[center,center]{$\lambda_2$}}
\pgfputat{\pgfxy(6.5,2.7)}%
{\pgfbox[center,center]{$\lambda_1$}}
\pgfputat{\pgfxy(7.5,2.7)}%
{\pgfbox[center,center]{$\lambda_0$}}
\end{pgfscope}
\end{pgfpicture}
\caption{Typical limit-order book}
\end{figure}

We define the diffusion scaled queue
lengths at time $t$ to be
$\widehat{Q}^{(n)}(t)=\frac{1}{\sqrt{n}}Q^n(t)$,
where $Q^n$ stands in for $U^n$,
$V^n$, $W^n$, $X^n$, $Y^n$ and $Z^n$.
We study the limit
$(U^*,V^*,W^*,X^*,Y^*,Z^*)$ 
of $(\Uhatn,\Vhatn,\Whatn,\Xhatn,\Yhatn,\Zhatn)$.  
In Figure 1.2,
we call the highest price tick at which
buy orders are queued the {\em essential bid};
it might fail to be the limit of the bid prices
in Figure 1.1.  Similarly, we call the lowest
price tick at which sell orders are queued
the {\em essential ask}.
Of course, as the essential bid and essential
ask prices move,
we will need to shift the window of six price
ticks whose queue lengths we study.

\begin{figure}[h]
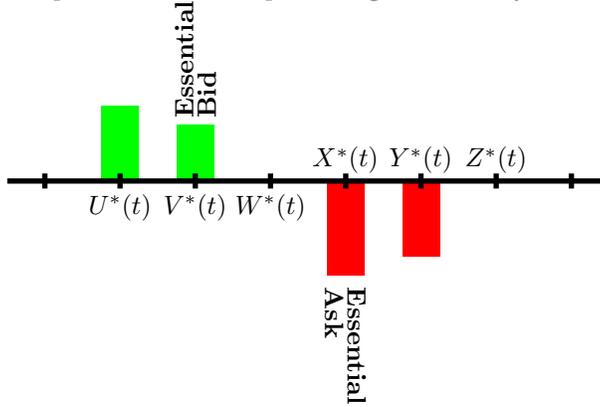
\label{F1.3}
\begin{pgfpicture}{0cm}{-1.7cm}{14cm}{3cm}
%
%
\begin{pgfscope}
\pgfsetlinewidth{2pt}
\color{green}
\pgfmoveto{\pgfxy(4.25,1.00)}
\pgflineto{\pgfxy(4.25,2.00)}
\pgflineto{\pgfxy(4.75,2.00)}
\pgflineto{\pgfxy(4.75,1.00)}
\pgffill
\pgfclosestroke
\pgfmoveto{\pgfxy(5.25,1.00)}
\pgflineto{\pgfxy(5.25,1.75)}
\pgflineto{\pgfxy(5.75,1.75)}
\pgflineto{\pgfxy(5.75,1.00)}
\pgffill
\pgfclosestroke
\end{pgfscope}
\begin{pgfscope}
\pgfsetlinewidth{2pt}
\color{red}
\pgfmoveto{\pgfxy(7.25,1.00)}
\pgflineto{\pgfxy(7.25,-0.25)}
\pgflineto{\pgfxy(7.75,-0.25)}
\pgflineto{\pgfxy(7.75,1.00)}
\pgffill
\pgfclosestroke
\pgfmoveto{\pgfxy(8.25,1.00)}
\pgflineto{\pgfxy(8.25,0.00)}
\pgflineto{\pgfxy(8.75,0.00)}
\pgflineto{\pgfxy(8.75,1.00)}
\pgffill
\pgfclosestroke
\end{pgfscope}
\begin{pgfscope}
\pgfsetlinewidth{2pt}
\pgfxyline(3,1)(11,1)
\pgfxyline(3.5,0.9)(3.5,1.1)
\pgfxyline(4.5,0.9)(4.5,1.1)
\pgfxyline(5.5,0.9)(5.5,1.1)
\pgfxyline(6.5,0.9)(6.5,1.1)
\pgfxyline(7.5,0.9)(7.5,1.1)
\pgfxyline(8.5,0.9)(8.5,1.1)
\pgfxyline(9.5,0.9)(9.5,1.1)
\pgfxyline(10.5,0.9)(10.5,1.1)
\end{pgfscope}
\pgfputlabelrotated{0.40}{\pgfxy(5.35,1.4)}%
{\pgfxy(5.35,2.5)}{0pt}{\pgfbox[left,center]{{\bf Essential}}}
\pgfputlabelrotated{0.40}{\pgfxy(5.65,1.4)}%
{\pgfxy(5.65,2.5)}{0pt}{\pgfbox[left,center]{{\bf Bid}}}
\pgfputlabelrotated{0.40}{\pgfxy(7.35,0.45)}%
{\pgfxy(7.35,-1.7)}{0pt}{\pgfbox[left,center]{{\bf Ask}}}
\pgfputlabelrotated{0.40}{\pgfxy(7.65,0.45)}%
{\pgfxy(7.65,-1.7)}{0pt}{\pgfbox[left,center]{{\bf Essential}}}
\pgfputat{\pgfxy(4.5,0.65)}%
{\pgfbox[center,center]{$U^*(t)$}}
\pgfputat{\pgfxy(5.5,0.65)}%
{\pgfbox[center,center]{$V^*(t)$}}
\pgfputat{\pgfxy(6.5,0.65)}%
{\pgfbox[center,center]{$W^*(t)$}}
\pgfputat{\pgfxy(7.5,1.28)}%
{\pgfbox[center,center]{$X^*(t)$}}
\pgfputat{\pgfxy(8.5,1.28)}%
{\pgfbox[center,center]{$Y^*(t)$}}
\pgfputat{\pgfxy(9.5,1.28)}%
{\pgfbox[center,center]{$Z^*(t)$}}
\end{pgfpicture}
\caption{Limiting system}
\end{figure}

Here is a brief summary of the results
of this paper.
For the moment, let us consider
the limit processes during a time interval
in which, as in Figure 1.2, $V^*>0$
and $Y^*<0$, so that we do not need to
shift the window.  We call $V^*$
and $Y^*$ the {\em bracketing queues}
and $W^*$ and $X^*$ the {\em interior queues}.
During such a time interval,
the pair $(W^*,X^*)$ is given by
$$
(W^*,X^*)=\big(\max\{B_{w,x},0\},
\min\{B_{w,x},0\}\big),
$$
where $B_{w,x}$ is
a {\em two-speed Brownian motion},
a concept defined and analyzed in
Appendix \ref{TwoSpeedBrMot}.
When $B_{w,x}<0$. so that $W^*=0$, then $V^*>0$
and $X^*<0$
behave as correlated Brownian motions with
standard deviations
\begin{align}
\sigma_+
&=
\sqrt{2(\lambda_0+b\lambda_1)},
\label{sigmaplus}\\
\sigma_-
&=
\sqrt{2(\mu
_0+a\mu_1)},\label{sigmaminus}
\end{align}
respectively, and correlation
\be\label{rho}
-\rho=\frac{\lambda_1+\mu_1}{\sqrt{(\lambda_0+b\lambda_1)
(\mu_0+a\mu_1)}}=\frac{2(\lambda_1+\mu_1)}{\sigma_+
\sigma_-}
\ee
On the other hand, when $B_{w,x}>0$,
so that $X^*=0$, then $W^*>0$
and $Y^*<0$ behave this way.
When $B_{w,x}=0$, then there is a three-tick
spread between the essential bid
$V^*$ and the essential ask $Y^*$.
We call this a renewal state.
  
During a time interal in which $V^*$
is positive and $B_{w,x}$ is negative,
so that $W^*=0$ and $X^*<0$, $V^*$ 
can reach zero.  When this happens, the system
has entered a different renewal state,
one tick to the left of the renewal
state in which $W^*=X^*=B_{w,x}=0$.
At such a moment, $U^*$ will
be positive, $X^*$ will be negative,
and we can shift the window left, declaring
$U^*$ and $X^*$ to be the bracketing
queues and $V^*$ and $W^*$ to be
the interior queues.
From a given
state with $W^*=0$
and negative $X^*$, we compute the probability
that this occurs.  From a renewal
state, we also compute the probability that
the next renewal state is to the left
or to the right, and we determine the characteristic
function of the time to the next renewal state,
conditioned on the direction of the price change.   

\section{Section-by-section summary}\label{Summary}

We begin with a short section, Section \ref{Notation},
on notation.
In Section \ref{Interior queues}
we identify the limit of the scaled interior
queues $W^n$ and $X^n$.  Section \ref{BracketingQueues} 
does the same for the bracketing queues $V^n$ and $Y^n$.
In both Sections \ref{Interior queues}
and \ref{BracketingQueues}, we simplify the analysis
by considering processes that are governed  
by the same dynamics before and after
one of the bracketing queues vanishes.  These
processes agree with $V^n$, $W^n$, $X^n$
and $Y^n$ until one of the bracketing queues
vanishes.  In Section \ref{Vanish},
we elaborate on the relationship among
these processes.  Section \ref{Statistics}
computes the
distribution of the time until price
change and the direction of price change.
Appenidx
\ref{TwoSpeedBrMot} defines and develops the
properties of two-speed Brownian motion.  

\section{Notation and Assumption}\label{Notation}

\setcounter{theorem}{0}
\setcounter{equation}{0}
\setcounter{figure}{0}

This work uses weak convergence of probability
measures on the space $D[0,\infty)$
of c\`adl\`ag functions from $[0,\infty)$
to $\R$
with the $J_1$ topology;
see \cite{EthierKurtz} and \cite{Whitt}.
We denote this convergence, which we call
{\em weak-$J_1$ convergence}, by
$\stackrel{J_1}{\Longrightarrow}$.
When we write $X^n\ArrowJ1 X$ for a sequence
of c\`adl\`ag processes $\{X^n\}_{n=1}^{\infty}$
and a limiting c\`adl\`ag process $X$, we mean that
the probability measures induced by
$X^n$ on $D[0,\infty)$ converge weakly-$J_1$
to the probability measure induced on
$D[0,\infty)$ by $X$.
We will need to sometimes
assign a ``left-limit at zero.''
To do this, we consider 
the space $D_0[-1,\infty)$ of c\`adl\`ag
functions from $[-1,\infty)$ to $\R$ that are
constant on $[-1,0)$.  These may be identified
with $D[0-,\infty):=\R\times D[0,\infty)$.  
An element
$x\in D[0-,\infty)$
is a c\`adl\`ag function from
$[0,\infty)$ to $\R$ together with a real
number denoted $x(0-)$.  The $M_1$ topology on 
$D_0[-1,\infty)$ provides a topology
on $D[0-,\infty)$.  Weak convergence
of probability measures on $D[0-,\infty)$
with the $M_1$ topology is called
{\em weak-$M_1$ convergence}.
The convergence we establish in Theorem \ref{T.VWXY}
below is joint weak convergence of four processes,
two in the $M_1$-topology and two in the $J_1$-topology.
It is characterized by the condition (\ref{5.88}).

\begin{definition}\label{D3.1a}
{\rm We say that a sequence of c\`adl\`ag processes
$\{X^n\}_{n=1}^\infty$ on $[0,\infty)$
is {\em bounded above in probability
on compact time intervals}
if, for every $T>0$ and $\ve>0$, there exists
$K>0$ and $N\geq 1$ such that
$\P\{\sup_{0\leq t\leq T}X^n(t)>K\}<\ve$
for all $n\geq N$.  We say that
$\{X^n\}_{n=1}^{\infty}$ is {\em bounded
below in probability on compact time
intervals} if $\{-X^n\}_{n=1}^{\infty}$
is bounded above in probability on compact
time intervals.
We say that $\{X^n\}_{n=1}^{\infty}$
is {\em bounded in probability on compact time intervals} if
it is both bounded above and bounded below
in probability
on compact time intervals, in which
case we write $X^n=O(1)$.}
\end{definition}

\begin{definition}\label{D3.2a}
{\rm A sequence of c\`adl\`ag processes
$\{X^n\}_{n=1}^{\infty}$ is said to
be $o(1)$ (written $X^n=o(1)$) if
$X_n\stackrel{J_1}{\Longrightarrow}0$.}
\end{definition}

\begin{remark}\label{R3.2a}
{\rm
The sequence $\{X^n\}_{n=1}^{\infty}$ is $o(1)$ if and only
if $\sup_{0\leq t\leq T}|X^n(t)|$ converges
in probability to zero, which
we write as $\sup_{0\leq t\leq T}|X^n(t)|
\stackrel{\P}{\rightarrow}0$, for every $T>0$.}
\end{remark}

We denote by $\id$ the identity
mapping on $[0,\infty]$.  We use
the term {\em standard Brownian motion}
to refer to a Brownian motion with initial
value zero that accumulates quadratic
variation at rate one per unit time.

We make the following assumption throughout.

\begin{assumption}\label{Assumption1}
There are two numbers $a>1$ and $b>1$ satisfying $a+b>ab$
such that
\begin{align*}
a\lambda_0&=
b\mu_0,\\
\lambda_1&=
(a-1)\lambda_0,\\
\lambda_2&=
(a+b-ab)\lambda_0,\\
\mu_1&=
(b-1)\mu_0,\\
\mu_2&=(a+b-ab)\mu_0.
\end{align*}
\end{assumption}

\begin{remark}\label{R2.1}
{\rm An immediate consequence of Assumption 
\ref{Assumption1} is that
\be\label{a6}
c:=\mu_0-\lambda_1
=(a+b-ab)\frac{\lambda_0}{b}
=(a+b-ab)\frac{\mu_0}{a}
=\lambda_0-\mu_1>0.
\ee
}
\end{remark}

For future reference, we define
\be\label{kappa}
\kappa_L:=
\frac{\lambda_2\mu_1}{\theta_b\lambda_1},\quad
\kappa_R:=
-\frac{\mu_2\lambda_1}{\theta_s\mu_1}.
\ee

\section{Interior queues}\label{Interior queues}

We begin the analysis of
the limit-order book with a study
of the interior queues.

\setcounter{theorem}{0}
\setcounter{equation}{0}
\setcounter{figure}{0}

\subsection{Initial condition}\label{InitialCondition}

To avoid a lengthy analysis
of an initial transient period, 
we make the following assumption
about the initial conditions in the sequence of 
limit-order book models.
\begin{assumption}\label{Assumption2}
There exist six adjacent price
ticks $p_u<p_v<p_w<p_x<p_y<p_z$ such that the initial
values of the corresponding queues satisfy
\begin{align*}
\lefteqn{
\left(\frac{1}{\sqrt{n}}U^n(0),\frac{1}{\sqrt{n}}V^n(0),
\frac{1}{\sqrt{n}}W^n(0),
\frac{1}{\sqrt{n}}X^n(0),
\frac{1}{\sqrt{n}}Y^n(0),\frac{1}{\sqrt{n}}Z^n(0)\right)}
\hspace{2in}
\\
&
\Longrightarrow
\big(U^*(0),V^*(0),0,0,Y^*(0),Z^*(0)\big),
\end{align*}
where $U^n(0)\geq 0$ for every $n$
so that $U^*(0)$ is a nonnegative
constant, $V^*(0)$ is a positive constant,
$Y^*(0)$ is a negative constant, and
$Z^n(0)\leq0$ for every $n$ so that
$Z^*(0)$ is a nonpositive constant.
The convergence is weak convergence
of probability measures
on the space $\R^6$.
\end{assumption}

Given the initial condition of
Assumption \ref{Assumption2},
we have $V^n(0)>0$ and $Y^n(0)<0$
for all sufficiently large $n$.
We designate $V^n$ and $Y^n$
the {\em bracketing queues}
and call $W^n$ and $X^n$ the {\em interior queues},
at least until either $V^n$ or $Y^n$ vanishes.
Until this occurs, at a stopping time we denote
by
\be\label{3.1} 
S^n:=\inf\{t\geq 0: V^n(t)=0\mbox{ or }
Y^n(t)=0\},
\ee
we consider the dynamics
of the interior queues.

\subsection{Interior queue dynamics}\label{Interior}

Given that $V^n$ is positive and $Y^n$
is negative, there are eight possible configurations
of the interior queues, depending
on whether $W^n$ is positive, zero or negative
and whether $X^n$ is positive, zero or negative.
(It cannot happen that $W^n$ is negative
and $X^n$ is positive; this would correspond
to a sell order queued at the price below
the price of a queued buy order.)
These eight configurations are shown in Figure 4.1.
This figure indicates the locations and
rates of arriving orders using the conventions
adopted for Figure 1.1.  In four of these
configurations, one of the bracketing queues
is subject to cancellations, and this queue
is indicated by the letter $C$ above
or below its histogram.

\begin{figure}[h]
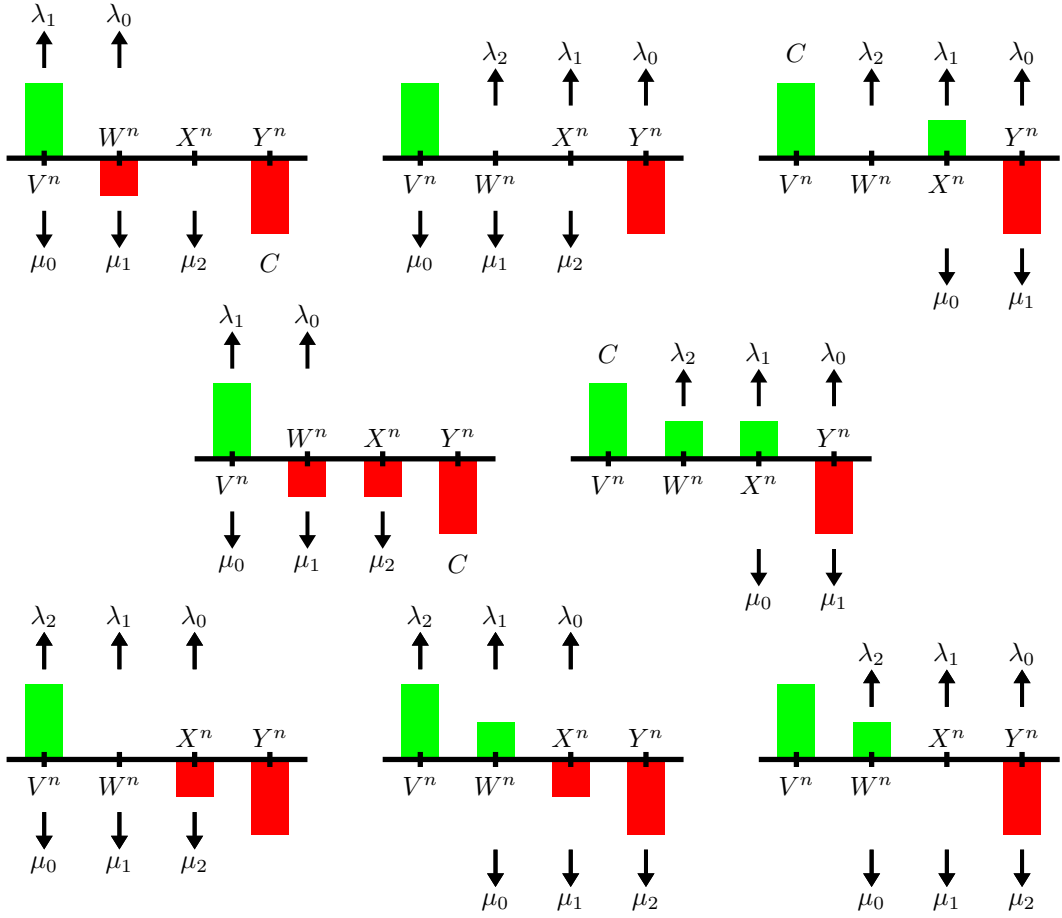
\label{F3.1}
\begin{pgfpicture}{0cm}{-1cm}{14cm}{11cm}
%
%
\begin{pgfscope}
\color{green}
\pgfmoveto{\pgfxy(0.25,1.00)}
\pgflineto{\pgfxy(0.25,2.00)}
\pgflineto{\pgfxy(0.75,2.00)}
\pgflineto{\pgfxy(0.75,1.00)}
\pgffill
\pgfclosestroke
\pgfmoveto{\pgfxy(5.25,1.00)}
\pgflineto{\pgfxy(5.25,2.00)}
\pgflineto{\pgfxy(5.75,2.00)}
\pgflineto{\pgfxy(5.75,1.00)}
\pgffill
\pgfclosestroke
\pgfmoveto{\pgfxy(10.25,1.00)}
\pgflineto{\pgfxy(10.25,2.00)}
\pgflineto{\pgfxy(10.75,2.00)}
\pgflineto{\pgfxy(10.75,1.00)}
\pgffill
\pgfclosestroke
\end{pgfscope}
\begin{pgfscope}
\color{green}
\pgfmoveto{\pgfxy(6.25,1.00)}
\pgflineto{\pgfxy(6.25,1.50)}
\pgflineto{\pgfxy(6.75,1.50)}
\pgflineto{\pgfxy(6.75,1.00)}
\pgffill
\pgfclosestroke
\pgfmoveto{\pgfxy(11.25,1.00)}
\pgflineto{\pgfxy(11.25,1.50)}
\pgflineto{\pgfxy(11.75,1.50)}
\pgflineto{\pgfxy(11.75,1.00)}
\pgffill
\pgfclosestroke
\end{pgfscope}
\begin{pgfscope}
\color{red}
\pgfmoveto{\pgfxy(3.25,1.00)}
\pgflineto{\pgfxy(3.25,0.00)}
\pgflineto{\pgfxy(3.75,0.00)}
\pgflineto{\pgfxy(3.75,1.00)}
\pgffill
\pgfclosestroke
\pgfmoveto{\pgfxy(8.25,1.00)}
\pgflineto{\pgfxy(8.25,0.00)}
\pgflineto{\pgfxy(8.75,0.00)}
\pgflineto{\pgfxy(8.75,1.00)}
\pgffill
\pgfclosestroke
\pgfmoveto{\pgfxy(13.25,1.00)}
\pgflineto{\pgfxy(13.25,0.00)}
\pgflineto{\pgfxy(13.75,0.00)}
\pgflineto{\pgfxy(13.75,1.00)}
\pgffill
\pgfclosestroke
\end{pgfscope}
\begin{pgfscope}
\color{red}
\pgfmoveto{\pgfxy(2.25,1.00)}
\pgflineto{\pgfxy(2.25,0.50)}
\pgflineto{\pgfxy(2.75,0.50)}
\pgflineto{\pgfxy(2.75,1.00)}
\pgffill
\pgfclosestroke
\pgfmoveto{\pgfxy(7.25,1.00)}
\pgflineto{\pgfxy(7.25,0.50)}
\pgflineto{\pgfxy(7.75,0.50)}
\pgflineto{\pgfxy(7.75,1.00)}
\pgffill
\pgfclosestroke
\end{pgfscope}
\begin{pgfscope}
\pgfsetlinewidth{1.5pt}
\pgfsetarrowsend{Triangle[scale=0.75pt]}
\pgfxyline(0.5,2.2)(0.5,2.7)
\pgfxyline(1.5,2.2)(1.5,2.7)
\pgfxyline(2.5,2.2)(2.5,2.7)
\pgfxyline(0.5,0.3)(0.5,-0.2)
\pgfxyline(1.5,0.3)(1.5,-0.2)
\pgfxyline(2.5,0.3)(2.5,-0.2)
\pgfxyline(5.5,2.2)(5.5,2.7)
\pgfxyline(6.5,2.2)(6.5,2.7)
\pgfxyline(7.5,2.2)(7.5,2.7)
\pgfxyline(6.5,-0.2)(6.5,-0.7)
\pgfxyline(7.5,-0.2)(7.5,-0.7)
\pgfxyline(8.5,-0.2)(8.5,-0.7)
\pgfxyline(11.5,1.7)(11.5,2.2)
\pgfxyline(12.5,1.7)(12.5,2.2)
\pgfxyline(13.5,1.7)(13.5,2.2)
\pgfxyline(11.5,-0.2)(11.5,-0.7)
\pgfxyline(12.5,-0.2)(12.5,-0.7)
\pgfxyline(13.5,-0.2)(13.5,-0.7)
\end{pgfscope}
\begin{pgfscope}
\pgfsetlinewidth{2pt}
\pgfxyline(0,1)(4,1)
\pgfxyline(5,1)(9,1)
\pgfxyline(10,1)(14,1)
\pgfxyline(0.5,0.9)(0.5,1.1)
\pgfxyline(1.5,0.9)(1.5,1.1)
\pgfxyline(2.5,0.9)(2.5,1.1)
\pgfxyline(3.5,0.9)(3.5,1.1)
\pgfxyline(5.5,0.9)(5.5,1.1)
\pgfxyline(6.5,0.9)(6.5,1.1)
\pgfxyline(7.5,0.9)(7.5,1.1)
\pgfxyline(8.5,0.9)(8.5,1.1)
\pgfxyline(10.5,0.9)(10.5,1.1)
\pgfxyline(11.5,0.9)(11.5,1.1)
\pgfxyline(12.5,0.9)(12.5,1.1)
\pgfxyline(13.5,0.9)(13.5,1.1)
\end{pgfscope}
\pgfputat{\pgfxy(0.5,2.9)}%
{\pgfbox[center,center]{$\lambda_2$}}
\pgfputat{\pgfxy(1.5,2.9)}%
{\pgfbox[center,center]{$\lambda_1$}}
\pgfputat{\pgfxy(2.5,2.9)}%
{\pgfbox[center,center]{$\lambda_0$}}
\pgfputat{\pgfxy(5.5,2.9)}%
{\pgfbox[center,center]{$\lambda_2$}}
\pgfputat{\pgfxy(6.5,2.9)}%
{\pgfbox[center,center]{$\lambda_1$}}
\pgfputat{\pgfxy(7.5,2.9)}%
{\pgfbox[center,center]{$\lambda_0$}}
\pgfputat{\pgfxy(11.5,2.4)}%
{\pgfbox[center,center]{$\lambda_2$}}
\pgfputat{\pgfxy(12.5,2.4)}%
{\pgfbox[center,center]{$\lambda_1$}}
\pgfputat{\pgfxy(13.5,2.4)}%
{\pgfbox[center,center]{$\lambda_0$}}
\pgfputat{\pgfxy(0.5,-0.4)}%
{\pgfbox[center,center]{$\mu_0$}}
\pgfputat{\pgfxy(1.5,-0.4)}%
{\pgfbox[center,center]{$\mu_1$}}
\pgfputat{\pgfxy(2.5,-0.4)}%
{\pgfbox[center,center]{$\mu_2$}}
\pgfputat{\pgfxy(6.5,-0.9)}%
{\pgfbox[center,center]{$\mu_0$}}
\pgfputat{\pgfxy(7.5,-0.9)}%
{\pgfbox[center,center]{$\mu_1$}}
\pgfputat{\pgfxy(8.5,-0.9)}%
{\pgfbox[center,center]{$\mu_2$}}
\pgfputat{\pgfxy(11.5,-0.9)}%
{\pgfbox[center,center]{$\mu_0$}}
\pgfputat{\pgfxy(12.5,-0.9)}%
{\pgfbox[center,center]{$\mu_1$}}
\pgfputat{\pgfxy(13.5,-0.9)}%
{\pgfbox[center,center]{$\mu_2$}}
\pgfputat{\pgfxy(0.5,0.65)}%
{\pgfbox[center,center]{$V^n$}}
\pgfputat{\pgfxy(1.5,0.65)}%
{\pgfbox[center,center]{$W^n$}}
\pgfputat{\pgfxy(2.5,1.28)}%
{\pgfbox[center,center]{$X^n$}}
\pgfputat{\pgfxy(3.5,1.28)}%
{\pgfbox[center,center]{$Y^n$}}
\pgfputat{\pgfxy(5.5,0.65)}%
{\pgfbox[center,center]{$V^n$}}
\pgfputat{\pgfxy(6.5,0.65)}%
{\pgfbox[center,center]{$W^n$}}
\pgfputat{\pgfxy(7.5,1.28)}%
{\pgfbox[center,center]{$X^n$}}
\pgfputat{\pgfxy(8.5,1.28)}%
{\pgfbox[center,center]{$Y^n$}}
\pgfputat{\pgfxy(10.5,0.65)}%
{\pgfbox[center,center]{$V^n$}}
\pgfputat{\pgfxy(11.5,0.65)}%
{\pgfbox[center,center]{$W^n$}}
\pgfputat{\pgfxy(12.5,1.28)}%
{\pgfbox[center,center]{$X^n$}}
\pgfputat{\pgfxy(13.5,1.28)}%
{\pgfbox[center,center]{$Y^n$}}
%
\begin{pgfscope}
\color{green}
\pgfmoveto{\pgfxy(2.75,5.00)}
\pgflineto{\pgfxy(2.75,6.00)}
\pgflineto{\pgfxy(3.25,6.00)}
\pgflineto{\pgfxy(3.25,5.00)}
\pgffill
\pgfclosestroke
\pgfmoveto{\pgfxy(7.75,5.00)}
\pgflineto{\pgfxy(7.75,6.00)}
\pgflineto{\pgfxy(8.25,6.00)}
\pgflineto{\pgfxy(8.25,5.00)}
\pgffill
\pgfclosestroke
\end{pgfscope}
\begin{pgfscope}
\color{green}
\pgfmoveto{\pgfxy(8.75,5.00)}
\pgflineto{\pgfxy(8.75,5.50)}
\pgflineto{\pgfxy(9.25,5.50)}
\pgflineto{\pgfxy(9.25,5.00)}
\pgffill
\pgfclosestroke
\pgfmoveto{\pgfxy(9.75,5.00)}
\pgflineto{\pgfxy(9.75,5.50)}
\pgflineto{\pgfxy(10.25,5.50)}
\pgflineto{\pgfxy(10.25,5.00)}
\pgffill
\pgfclosestroke
\end{pgfscope}
\begin{pgfscope}
\color{red}
\pgfmoveto{\pgfxy(5.75,5.00)}
\pgflineto{\pgfxy(5.75,4.00)}
\pgflineto{\pgfxy(6.25,4.00)}
\pgflineto{\pgfxy(6.25,5.00)}
\pgffill
\pgfclosestroke
\pgfmoveto{\pgfxy(10.75,5.00)}
\pgflineto{\pgfxy(10.75,4.00)}
\pgflineto{\pgfxy(11.25,4.00)}
\pgflineto{\pgfxy(11.25,5.00)}
\pgffill
\pgfclosestroke
\end{pgfscope}
\begin{pgfscope}
\color{red}
\pgfmoveto{\pgfxy(3.75,5.00)}
\pgflineto{\pgfxy(3.75,4.50)}
\pgflineto{\pgfxy(4.25,4.50)}
\pgflineto{\pgfxy(4.25,5.00)}
\pgffill
\pgfclosestroke
\pgfmoveto{\pgfxy(4.75,5.00)}
\pgflineto{\pgfxy(4.75,4.50)}
\pgflineto{\pgfxy(5.25,4.50)}
\pgflineto{\pgfxy(5.25,5.00)}
\pgffill
\pgfclosestroke
\end{pgfscope}
\begin{pgfscope}
\pgfsetlinewidth{2pt}
\pgfxyline(2.5,5)(6.5,5)
\pgfxyline(7.5,5)(11.5,5)
\pgfxyline(3.0,4.9)(3.0,5.1)
\pgfxyline(4.0,4.9)(4.0,5.1)
\pgfxyline(5.0,4.9)(5.0,5.1)
\pgfxyline(6.0,4.9)(6.0,5.1)
\pgfxyline(8.0,4.9)(8.0,5.1)
\pgfxyline(9.0,4.9)(9.0,5.1)
\pgfxyline(10.0,4.9)(10.0,5.1)
\pgfxyline(11.0,4.9)(11.0,5.1)
\end{pgfscope}
\begin{pgfscope}
\pgfsetlinewidth{1.5pt}
\pgfsetarrowsend{Triangle[scale=0.75pt]}
\pgfxyline(0.5,2.2)(0.5,2.7)
\pgfxyline(1.5,2.2)(1.5,2.7)
\pgfxyline(2.5,2.2)(2.5,2.7)
\pgfxyline(0.5,0.3)(0.5,-0.2)
\pgfxyline(1.5,0.3)(1.5,-0.2)
\pgfxyline(2.5,0.3)(2.5,-0.2)
\pgfxyline(5.5,2.2)(5.5,2.7)
\pgfxyline(6.5,2.2)(6.5,2.7)
\pgfxyline(7.5,2.2)(7.5,2.7)
\pgfxyline(6.5,-0.2)(6.5,-0.7)
\pgfxyline(7.5,-0.2)(7.5,-0.7)
\pgfxyline(8.5,-0.2)(8.5,-0.7)
\pgfxyline(11.5,1.7)(11.5,2.2)
\pgfxyline(12.5,1.7)(12.5,2.2)
\pgfxyline(13.5,1.7)(13.5,2.2)
\pgfxyline(11.5,-0.2)(11.5,-0.7)
\pgfxyline(12.5,-0.2)(12.5,-0.7)
\pgfxyline(13.5,-0.2)(13.5,-0.7)
\end{pgfscope}
\begin{pgfscope}
\pgfsetlinewidth{1.5pt}
\pgfsetarrowsend{Triangle[scale=0.75pt]}
\pgfxyline(3.0,6.2)(3.0,6.7)
\pgfxyline(4.0,6.2)(4.0,6.7)
\pgfxyline(3.0,4.3)(3.0,3.8)
\pgfxyline(4.0,4.3)(4.0,3.8)
\pgfxyline(5.0,4.3)(5.0,3.8)
\pgfxyline(9.0,5.7)(9.0,6.2)
\pgfxyline(10.0,5.7)(10.0,6.2)
\pgfxyline(11.0,5.7)(11.0,6.2)
\pgfxyline(10.0,3.8)(10.0,3.3)
\pgfxyline(11.0,3.8)(11.0,3.3)
\end{pgfscope}
\pgfputat{\pgfxy(3.0,6.9)}%
{\pgfbox[center,center]{$\lambda_1$}}
\pgfputat{\pgfxy(4.0,6.9)}%
{\pgfbox[center,center]{$\lambda_0$}}
\pgfputat{\pgfxy(9.0,6.4)}%
{\pgfbox[center,center]{$\lambda_2$}}
\pgfputat{\pgfxy(10.0,6.4)}%
{\pgfbox[center,center]{$\lambda_1$}}
\pgfputat{\pgfxy(11.0,6.4)}%
{\pgfbox[center,center]{$\lambda_0$}}
\pgfputat{\pgfxy(3.0,3.6)}%
{\pgfbox[center,center]{$\mu_0$}}
\pgfputat{\pgfxy(4.0,3.6)}%
{\pgfbox[center,center]{$\mu_1$}}
\pgfputat{\pgfxy(5.0,3.6)}%
{\pgfbox[center,center]{$\mu_2$}}
\pgfputat{\pgfxy(10.0,3.1)}%
{\pgfbox[center,center]{$\mu_0$}}
\pgfputat{\pgfxy(11.0,3.1)}%
{\pgfbox[center,center]{$\mu_1$}}
\pgfputat{\pgfxy(3.0,4.65)}%
{\pgfbox[center,center]{$V^n$}}
\pgfputat{\pgfxy(4.0,5.28)}%
{\pgfbox[center,center]{$W^n$}}
\pgfputat{\pgfxy(5.0,5.28)}%
{\pgfbox[center,center]{$X^n$}}
\pgfputat{\pgfxy(6.0,5.28)}%
{\pgfbox[center,center]{$Y^n$}}
\pgfputat{\pgfxy(8.0,4.65)}%
{\pgfbox[center,center]{$V^n$}}
\pgfputat{\pgfxy(9.0,4.65)}%
{\pgfbox[center,center]{$W^n$}}
\pgfputat{\pgfxy(10.0,4.65)}%
{\pgfbox[center,center]{$X^n$}}
\pgfputat{\pgfxy(11.0,5.28)}%
{\pgfbox[center,center]{$Y^n$}}
\begin{pgfscope}
\color{green}
\pgfmoveto{\pgfxy(0.25,9.00)}
\pgflineto{\pgfxy(0.25,10.00)}
\pgflineto{\pgfxy(0.75,10.00)}
\pgflineto{\pgfxy(0.75,9.00)}
\pgffill
\pgfclosestroke
\pgfmoveto{\pgfxy(5.25,9.00)}
\pgflineto{\pgfxy(5.25,10.00)}
\pgflineto{\pgfxy(5.75,10.00)}
\pgflineto{\pgfxy(5.75,9.00)}
\pgffill
\pgfclosestroke
\pgfmoveto{\pgfxy(10.25,9.00)}
\pgflineto{\pgfxy(10.25,10.00)}
\pgflineto{\pgfxy(10.75,10.00)}
\pgflineto{\pgfxy(10.75,9.00)}
\pgffill
\pgfclosestroke
\end{pgfscope}
\begin{pgfscope}
\color{green}
\pgfmoveto{\pgfxy(12.25,9.00)}
\pgflineto{\pgfxy(12.25,9.50)}
\pgflineto{\pgfxy(12.75,9.50)}
\pgflineto{\pgfxy(12.75,9.00)}
\pgffill
\pgfclosestroke
\end{pgfscope}
\begin{pgfscope}
\color{red}
\pgfmoveto{\pgfxy(3.25,9.00)}
\pgflineto{\pgfxy(3.25,8.00)}
\pgflineto{\pgfxy(3.75,8.00)}
\pgflineto{\pgfxy(3.75,9.00)}
\pgffill
\pgfclosestroke
\pgfmoveto{\pgfxy(8.25,9.00)}
\pgflineto{\pgfxy(8.25,8.00)}
\pgflineto{\pgfxy(8.75,8.00)}
\pgflineto{\pgfxy(8.75,9.00)}
\pgffill
\pgfclosestroke
\pgfmoveto{\pgfxy(13.25,9.00)}
\pgflineto{\pgfxy(13.25,8.00)}
\pgflineto{\pgfxy(13.75,8.00)}
\pgflineto{\pgfxy(13.75,9.00)}
\pgffill
\pgfclosestroke
\end{pgfscope}
\begin{pgfscope}
\color{red}
\pgfmoveto{\pgfxy(1.25,9.00)}
\pgflineto{\pgfxy(1.25,8.50)}
\pgflineto{\pgfxy(1.75,8.50)}
\pgflineto{\pgfxy(1.75,9.00)}
\pgffill
\pgfclosestroke
\end{pgfscope}
\begin{pgfscope}
\pgfsetlinewidth{1.5pt}
\pgfsetarrowsend{Triangle[scale=0.75pt]}
\pgfxyline(0.5,10.2)(0.5,10.7)
\pgfxyline(1.5,10.2)(1.5,10.7)
\pgfxyline(0.5,8.3)(0.5,7.8)
\pgfxyline(1.5,8.3)(1.5,7.8)
\pgfxyline(2.5,8.3)(2.5,7.8)
\pgfxyline(6.5,9.7)(6.5,10.2)
\pgfxyline(7.5,9.7)(7.5,10.2)
\pgfxyline(8.5,9.7)(8.5,10.2)
\pgfxyline(5.5,8.3)(5.5,7.8)
\pgfxyline(6.5,8.3)(6.5,7.8)
\pgfxyline(7.5,8.3)(7.5,7.8)
\pgfxyline(11.5,9.7)(11.5,10.2)
\pgfxyline(12.5,9.7)(12.5,10.2)
\pgfxyline(13.5,9.7)(13.5,10.2)
\pgfxyline(12.5,7.8)(12.5,7.3)
\pgfxyline(13.5,7.8)(13.5,7.3)
\end{pgfscope}
\begin{pgfscope}
\pgfsetlinewidth{2pt}
\pgfxyline(0,9)(4,9)
\pgfxyline(5,9)(9,9)
\pgfxyline(10,9)(14,9)
\pgfxyline(0.5,8.9)(0.5,9.1)
\pgfxyline(1.5,8.9)(1.5,9.1)
\pgfxyline(2.5,8.9)(2.5,9.1)
\pgfxyline(3.5,8.9)(3.5,9.1)
\pgfxyline(5.5,8.9)(5.5,9.1)
\pgfxyline(6.5,8.9)(6.5,9.1)
\pgfxyline(7.5,8.9)(7.5,9.1)
\pgfxyline(8.5,8.9)(8.5,9.1)
\pgfxyline(10.5,8.9)(10.5,9.1)
\pgfxyline(11.5,8.9)(11.5,9.1)
\pgfxyline(12.5,8.9)(12.5,9.1)
\pgfxyline(13.5,8.9)(13.5,9.1)
\end{pgfscope}
\pgfputat{\pgfxy(0.5,10.9)}%
{\pgfbox[center,center]{$\lambda_1$}}
\pgfputat{\pgfxy(1.5,10.9)}%
{\pgfbox[center,center]{$\lambda_0$}}
\pgfputat{\pgfxy(6.5,10.4)}%
{\pgfbox[center,center]{$\lambda_2$}}
\pgfputat{\pgfxy(7.5,10.4)}%
{\pgfbox[center,center]{$\lambda_1$}}
\pgfputat{\pgfxy(8.5,10.4)}%
{\pgfbox[center,center]{$\lambda_0$}}
\pgfputat{\pgfxy(11.5,10.4)}%
{\pgfbox[center,center]{$\lambda_2$}}
\pgfputat{\pgfxy(12.5,10.4)}%
{\pgfbox[center,center]{$\lambda_1$}}
\pgfputat{\pgfxy(13.5,10.4)}%
{\pgfbox[center,center]{$\lambda_0$}}
\pgfputat{\pgfxy(0.5,7.6)}%
{\pgfbox[center,center]{$\mu_0$}}
\pgfputat{\pgfxy(1.5,7.6)}%
{\pgfbox[center,center]{$\mu_1$}}
\pgfputat{\pgfxy(2.5,7.6)}%
{\pgfbox[center,center]{$\mu_2$}}
\pgfputat{\pgfxy(5.5,7.6)}%
{\pgfbox[center,center]{$\mu_0$}}
\pgfputat{\pgfxy(6.5,7.6)}%
{\pgfbox[center,center]{$\mu_1$}}
\pgfputat{\pgfxy(7.5,7.6)}%
{\pgfbox[center,center]{$\mu_2$}}
\pgfputat{\pgfxy(12.5,7.1)}%
{\pgfbox[center,center]{$\mu_0$}}
\pgfputat{\pgfxy(13.5,7.1)}%
{\pgfbox[center,center]{$\mu_1$}}
\pgfputat{\pgfxy(0.5,8.65)}%
{\pgfbox[center,center]{$V^n$}}
\pgfputat{\pgfxy(1.5,9.28)}%
{\pgfbox[center,center]{$W^n$}}
\pgfputat{\pgfxy(2.5,9.28)}%
{\pgfbox[center,center]{$X^n$}}
\pgfputat{\pgfxy(3.5,9.28)}%
{\pgfbox[center,center]{$Y^n$}}
\pgfputat{\pgfxy(5.5,8.65)}%
{\pgfbox[center,center]{$V^n$}}
\pgfputat{\pgfxy(6.5,8.65)}%
{\pgfbox[center,center]{$W^n$}}
\pgfputat{\pgfxy(7.5,9.28)}%
{\pgfbox[center,center]{$X^n$}}
\pgfputat{\pgfxy(8.5,9.28)}%
{\pgfbox[center,center]{$Y^n$}}
\pgfputat{\pgfxy(10.5,8.65)}%
{\pgfbox[center,center]{$V^n$}}
\pgfputat{\pgfxy(11.5,8.65)}%
{\pgfbox[center,center]{$W^n$}}
\pgfputat{\pgfxy(12.5,8.65)}%
{\pgfbox[center,center]{$X^n$}}
\pgfputat{\pgfxy(13.5,9.28)}%
{\pgfbox[center,center]{$Y^n$}}
\pgfputat{\pgfxy(3.5,7.6)}%
{\pgfbox[center,center]{$C$}}
\pgfputat{\pgfxy(10.5,10.4)}%
{\pgfbox[center,center]{$C$}}
\pgfputat{\pgfxy(6,3.6)}%
{\pgfbox[center,center]{$C$}}
\pgfputat{\pgfxy(8,6.4)}%
{\pgfbox[center,center]{$C$}}
\end{pgfpicture}
\caption{Interior queue configurations}
\end{figure}


We may represent the eight configurations
in Figure 4.1 as eight distinct regions in
the $(W^n,X^n)$ plane.  A representative
point in each of these  regions
is shown in Figure 4.2.  The arrows attached to
the point indicate the forces acting on
the pair $(W^n,X^n)$ in each of the eight
regions.  These eight regions correspond
to the signs of $W^n$ and $X^n$, but it is
convenient for us to group the two regions
corresponding to $X^n>0$, noting
that when $X^n>0$,
the forces acting on
the pair $(W^n,X^n)$ are the same
regardless of whether
$W^n=0$ or $W^n>0$.  We thus define the
{\em northeast region}
\be\label{3.2}
NE:=\big\{(w,x):w\geq 0,x>0\}.
\ee
Similarly, we group two regions to create
the {\em southwest region}
\be\label{3.3}
SW:=\big\{(w,x):w<0,x\leq 0\}.
\ee
We next define three regions
corresponding to three configurations
in Figure 4.1, namely, the {\em origin},
the {\em eastern region}, and the
{\em southern region},
\be\label{3.4}
O:=
\big\{(0,0)\big\},\quad
E:=
\big\{(w,0):w>0\big\},\quad
S:=
\big\{(0,x):x<0\big\}.
\ee
There remains the configuration in which 
$W^n$ is positive and $X^n$ is negative.
We partition this configuration into
three {\em southeastern regions},
\begin{align}
SE_+
&:=
\big\{(w,x):x<0,w>0,x+w>0\big\},\label{3.7}\\
SE
&:=
\big\{(w,x):x<0,w<0,x+w=0\big\},\label{3.8}\\
SE_-
&:=
\big\{(w,x):x<0,w>0,x+w<0\big\}.\label{3.9}
\end{align}
The eight regions (\ref{3.2})--(\ref{3.9})
are shown in Figure 4.2.  We denote
this collection of regions by
$\sR:=\{NE,E,SE_+,SE,SE_-,S,SW,O\}$.

\begin{figure}[h]
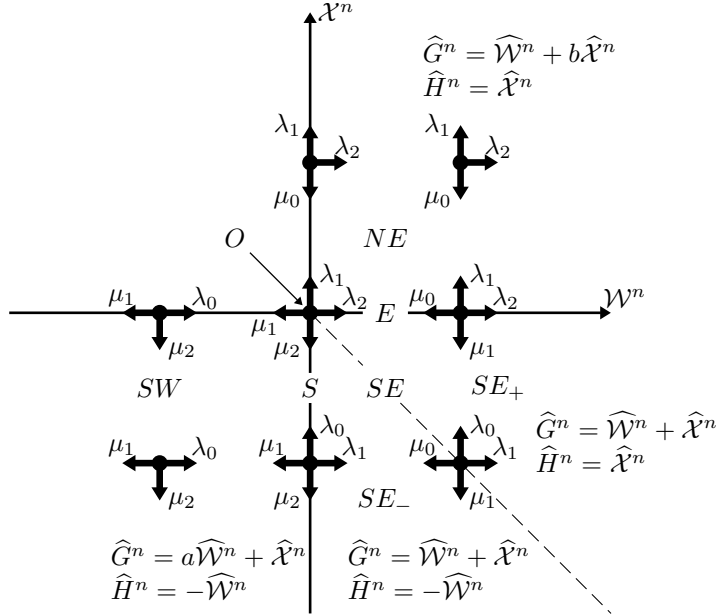
\label{F3.2}
\begin{pgfpicture}{0cm}{0cm}{14cm}{8cm}
%
%
\begin{pgfscope}
\pgfsetlinewidth{1.0pt}
\pgfsetarrowsend{Triangle[scale=0.75pt]}
\pgfxyline(8.3,4)(11,4)
\pgfxyline(7,3.2)(7,8)
\end{pgfscope}
\begin{pgfscope}
\pgfsetlinewidth{1.0pt}
\pgfxyline(3,4)(7.7,4)
\pgfxyline(7,0)(7,2.8)
\end{pgfscope}
\pgfcircle[fill]{\pgfxy(5,4)}{3pt}
\pgfcircle[fill]{\pgfxy(7,4)}{3pt}
\pgfcircle[fill]{\pgfxy(9,4)}{3pt}
\pgfcircle[fill]{\pgfxy(7,6)}{3pt}
\pgfcircle[fill]{\pgfxy(7,2)}{3pt}
\pgfcircle[fill]{\pgfxy(9,6)}{3pt}
\pgfcircle[fill]{\pgfxy(9,2)}{3pt}
\pgfcircle[fill]{\pgfxy(5,2)}{3pt}
\begin{pgfscope}
\pgfsetlinewidth{2.5pt}
\pgfsetarrowsend{Triangle[scale=0.5pt]}
\pgfxyline(5,4)(5.5,4)
\pgfxyline(5,4)(4.5,4)
\pgfxyline(5,4)(5,3.5)
\pgfxyline(7,4)(7,4.5)
\pgfxyline(7,4)(7,3.5)
\pgfxyline(7,4)(6.5,4)
\pgfxyline(7,4)(7.5,4)
\pgfxyline(9,4)(9,4.5)
\pgfxyline(9,4)(9,3.5)
\pgfxyline(9,4)(8.5,4)
\pgfxyline(9,4)(9.5,4)
\pgfxyline(7,6)(7,6.5)
\pgfxyline(7,6)(7.5,6)
\pgfxyline(7,6)(7,5.5)
\pgfxyline(9,6)(9,6.5)
\pgfxyline(9,6)(9.5,6)
\pgfxyline(9,6)(9,5.5)
\pgfxyline(5,2)(5.5,2)
\pgfxyline(5,2)(5,1.5)
\pgfxyline(5,2)(4.5,2)
\pgfxyline(7,2)(7.5,2)
\pgfxyline(7,2)(7,1.5)
\pgfxyline(7,2)(6.5,2)
\pgfxyline(7,2)(7,2.5)
\pgfxyline(9,2)(9.5,2)
\pgfxyline(9,2)(9,1.5)
\pgfxyline(9,2)(8.5,2)
\pgfxyline(9,2)(9,2.5)
\end{pgfscope}
\begin{pgfscope}
\pgfsetdash{{0.2cm}{0.1cm}}{0cm}
\pgfxyline(7,4)(7.7,3.3)
\pgfxyline(8.2,2.8)(11,0)
\end{pgfscope}
\pgfputat{\pgfxy(8,3)}{\pgfbox[center,center]%
{$SE$}}
\pgfputat{\pgfxy(8,1.5)}{\pgfbox[center,center]%
{$SE_-$}}
\pgfputat{\pgfxy(9.5,3)}{\pgfbox[center,center]%
{$SE_+$}}
\pgfputat{\pgfxy(5,3)}{\pgfbox[center,center]%
{$SW$}}
\pgfputat{\pgfxy(8,5)}{\pgfbox[center,center]%
{$NE$}}
\pgfputat{\pgfxy(8,4)}{\pgfbox[center,center]%
{$E$}}
\pgfputat{\pgfxy(7,3)}{\pgfbox[center,center]%
{$S$}}
\pgfputat{\pgfxy(6,5)}{\pgfbox[center,center]%
{$O$}}
\begin{pgfscope}
\pgfsetlinewidth{0.5pt}
\pgfsetarrowsend{Triangle[scale=0.75pt]}
\pgfxyline(6.2,4.8)(6.9,4.1)
\end{pgfscope}
\pgfputat{\pgfxy(4.5,4.2)}{\pgfbox[center,center]%
{$\mu_1$}}
\pgfputat{\pgfxy(5.6,4.2)}{\pgfbox[center,center]%
{$\lambda_0$}}
\pgfputat{\pgfxy(5.3,3.5)}{\pgfbox[center,center]%
{$\mu_2$}}
\pgfputat{\pgfxy(4.5,2.2)}{\pgfbox[center,center]%
{$\mu_1$}}
\pgfputat{\pgfxy(5.6,2.2)}{\pgfbox[center,center]%
{$\lambda_0$}}
\pgfputat{\pgfxy(5.3,1.5)}{\pgfbox[center,center]%
{$\mu_2$}}
\pgfputat{\pgfxy(7.3,4.5)}{\pgfbox[center,center]%
{$\lambda_1$}}
\pgfputat{\pgfxy(6.4,3.8)}{\pgfbox[center,center]%
{$\mu_1$}}
\pgfputat{\pgfxy(7.6,4.2)}{\pgfbox[center,center]%
{$\lambda_2$}}
\pgfputat{\pgfxy(6.7,3.5)}{\pgfbox[center,center]%
{$\mu_2$}}
\pgfputat{\pgfxy(6.7,6.5)}{\pgfbox[center,center]%
{$\lambda_1$}}
\pgfputat{\pgfxy(6.7,5.5)}{\pgfbox[center,center]%
{$\mu_0$}}
\pgfputat{\pgfxy(7.5,6.2)}{\pgfbox[center,center]%
{$\lambda_2$}}
\pgfputat{\pgfxy(8.7,6.5)}{\pgfbox[center,center]%
{$\lambda_1$}}
\pgfputat{\pgfxy(8.7,5.5)}{\pgfbox[center,center]%
{$\mu_0$}}
\pgfputat{\pgfxy(9.5,6.2)}{\pgfbox[center,center]%
{$\lambda_2$}}
\pgfputat{\pgfxy(7.3,2.5)}{\pgfbox[center,center]%
{$\lambda_0$}}
\pgfputat{\pgfxy(6.5,2.2)}{\pgfbox[center,center]%
{$\mu_1$}}
\pgfputat{\pgfxy(7.6,2.2)}{\pgfbox[center,center]%
{$\lambda_1$}}
\pgfputat{\pgfxy(6.7,1.5)}{\pgfbox[center,center]%
{$\mu_2$}}
\pgfputat{\pgfxy(9.3,2.5)}{\pgfbox[center,center]%
{$\lambda_0$}}
\pgfputat{\pgfxy(8.5,2.2)}{\pgfbox[center,center]%
{$\mu_0$}}
\pgfputat{\pgfxy(9.6,2.2)}{\pgfbox[center,center]%
{$\lambda_1$}}
\pgfputat{\pgfxy(9.3,1.5)}{\pgfbox[center,center]%
{$\mu_1$}}
\pgfputat{\pgfxy(9.3,4.5)}{\pgfbox[center,center]%
{$\lambda_1$}}
\pgfputat{\pgfxy(8.5,4.2)}{\pgfbox[center,center]%
{$\mu_0$}}
\pgfputat{\pgfxy(9.6,4.2)}{\pgfbox[center,center]%
{$\lambda_2$}}
\pgfputat{\pgfxy(9.3,3.5)}{\pgfbox[center,center]%
{$\mu_1$}}
\pgfputat{\pgfxy(8.5,7.5)}{\pgfbox[left,center]%
{$\Ghatn=\sWhatn+b\sXhatn$}}
\pgfputat{\pgfxy(8.5,7.05)}{\pgfbox[left,center]%
{$\Hhatn=\sXhatn$}}
\pgfputat{\pgfxy(10,2.5)}{\pgfbox[left,center]%
{$\Ghatn=\sWhatn+\sXhatn$}}
\pgfputat{\pgfxy(10,2.05)}{\pgfbox[left,center]%
{$\Hhatn=\sXhatn$}}
\pgfputat{\pgfxy(7.5,0.8)}{\pgfbox[left,center]%
{$\Ghatn=\sWhatn+\sXhatn$}}
\pgfputat{\pgfxy(7.5,0.35)}{\pgfbox[left,center]%
{$\Hhatn=-\sWhatn$}}
\pgfputat{\pgfxy(4.35,0.8)}{\pgfbox[left,center]%
{$\Ghatn=a\sWhatn+\sXhatn$}}
\pgfputat{\pgfxy(4.35,0.35)}{\pgfbox[left,center]%
{$\Hhatn=-\sWhatn$}}
\pgfputat{\pgfxy(10.9,4.2)}{\pgfbox[left,center]%
{$\sWn$}}
\pgfputat{\pgfxy(7.1,8)}{\pgfbox[left,center]%
{$\sXn$}}
\end{pgfpicture}
\caption{Dynamics of the interior queues}
\end{figure}

To postpone consideration of
the time $S^n$ when one of the bracketing
queues vanishes, we temporarily replace $(W^n,X^n)$
in the analysis by the pair of processes
$(\sWn,\sXn)$ whose dynamics is described by
Figure 4.2 without regard to the bracketing
queues.  The process pair of interest
$(W^n,X^n)$ agrees with $(\sW^n,\sX^n)$
up to time $S^n$.

One way to describe $(\sW^n,\sX^n)$
is to begin at an initial state,
say $(\sW^n(0),\sX^n(0))=(1,1)$, and
consider three independent exponential
random variables (times) with means $1/\lambda_1$,
$1/\lambda_2$ and $1/\mu_0$.  The first transition
of $(\sW^n,\sX^n)$ occurs at the minimum
of these three random times, and the minimizing
random time determines the direction
of the transition.  If the random time
with mean $1/\mu_0$ is the minimum,
then the transition is to $(1,0)$.
The time of the next transition is
the minimum of four exponential
random variables (times), independent
of one another and of the three previous random
variables, with means
$1/\lambda_1$, $1/\lambda_2$, $1/\mu_1$
and $1/\mu_0$.  Continuing in this way,
we construct the Markov process
$(\sW^n,\sX^n)$.

An equivalent construction that does not make
the Markov property so obvious but is
more convenient for our purposes uses Poisson
processes rather than exponentially distributed
random variables.
For each of the eight regions we create
independent Poisson process 
$N_{\times,*,\pm}$, where the symbol $\times$
indicates the region in which the Poisson process
acts, the symbol
$*$ is either $W$ or $X$,
depending on which component of $(\sW^n,\sX^n)$
is affected by the Poisson process,
and the sign $+$ or $-$ indicates whether
the Poisson process increases or decreases
this component. There are thirty
of these Poisson processes.
Given the thirty independent Poisson
processes, we can construct
by forward recursion the unique pair
of processes $(\sWn,\sXn)$ satisfying the equations
\begin{align}
P_\times^n(t)
&=
\int_0^t\ind_{\{(\sWn,\sXn)
\in\times\}}ds,\quad \times\in\sR,
\label{3.10}\\
\sWn(t)
&=
W^n(0)
+N_{NE,W,+}\circ\lambda_2P^n_{NE}(t)
+N_{E,W,+}\circ\lambda_2P^n_{E}(t)
-N_{E,W,-}\circ\mu_0P^n_{E}(t)\nonumber\\
&\quad
+N_{SE_+,W,+}\circ\lambda_1P^n_{SE_+}(t)
-N_{SE_+,W,-}\circ\mu_0P^n_{SE_+}(t)
+N_{SE,W,+}\circ\lambda_1P^n_{SE}(t)\nonumber\\
&\quad
-N_{SE,W,-}\circ\mu_0P^n_{SE}(t)
+N_{SE_-,W,+}\circ\lambda_1P^n_{SE_-}(t)
-N_{SE_-,W,-}\circ\mu_0P^n_{SE_-}(t)\nonumber\\
&\quad
+N_{S,W,+}\circ\lambda_1P^n_{S}(t)
-N_{S,W,-}\circ\mu_1P^n_{S}(t)
+N_{SW,W,+}\circ\lambda_0P^n_{SW}(t)\nonumber\\
&\quad
-N_{SW,W,-}\circ\mu_1P^n_{SW}(t)
+N_{O,W,+}\circ\lambda_2P^n_O(t)
-N_{O,W,-}\circ\mu_1P^n_O(t),\label{3.11}\\
\sXn(t)
&=
X^n(0)
-N_{SW,X,-}\circ\mu_2P_{SW}^n(t)
-N_{S,X,-}\circ\mu_2P_S^n(t)
+N_{S,X,+}\circ\lambda_0P_S^n(t)\nonumber\\
&\quad
-N_{SE_-,X,-}\circ\mu_1P_{SE_-}^n(t)
+N_{SE_-,X,+}\circ\lambda_0P_{SE_-}^n(t)
-N_{SE,X,-}\circ\mu_1P_{SE}^n(t)\nonumber\\
&\quad
+N_{SE,X,+}\circ\lambda_0P_{SE}^n(t)
-N_{SE_+,X,-}\circ\mu_1P_{SE_+}^n(t)
+N_{SE_+,X,+}\circ\lambda_0P_{SE_+}^n(t)\nonumber\\
&\quad
-N_{E,X,-}\circ\mu_1P_E^n(t)
+N_{E,X,+}\circ\lambda_1P_E^n(t)
-N_{NE,X,-}\circ\mu_0P_{NE}^n(t)\nonumber\\
&\quad
+N_{NE,X,+}\circ\lambda_1P_{NE}^n(t)
-N_{O,X,-}\circ\mu_2P_O(t)
+N_{O,X,+}\circ\lambda_1P_O(t).\label{3.12}
\end{align}

\subsection{Transformation of variables}\label{TV}

Recalling the positive constants $a$ and $b$
from Assumption \ref{Assumption1},
we define $(G^n,H^n)$ via a continuous piecewise
linear transformation of $(\sW^n,\sX^n)$ by
\begin{align}
G^n(t)
&:=
\left\{\begin{array}{ll}
\sW^n(t)+b\sX^n(t)&\mbox{if }
\big(\sW^n(t),\sX^n(t)\big)\in NE\cup E,\\
\sW^n(t)+\sX^n(t)&\mbox{if }
\big(\sW^n(t),\sX^n(t)\big)\in 
SE_+\cup SE\cup SE_-\cup O,\\
a\sW^n(t)+\sX^n(t)&\mbox{if }
\big(\sW^n(t),\sX^n(t)\big)\in SW\cup S,
\end{array}\right.\label{4.14c}
\end{align}
\begin{align}
H^n(t)
&:=
\left\{\begin{array}{ll}
\sX^n(t)&\mbox{if }\big(\sW^n(t),\sX^n(t)\big)
\in NE\cup E\cup SE_+,\\
\sX^n(t)=-\sW^n(t)&\mbox{if }
\big(\sW^n(t),\sX^n(t)\big)\in SE\cup O,\\
-\sW^n(t)&\mbox{if }
\big(\sW^n(t),\sX^n(t)\big)\in SW\cup S\cup SE_-.
\end{array}\right.\label{4.15c}
\end{align}
This transformation is invertible with
continuous piecewise linear inverse
\begin{align}
\sW^n(t)
&=
\left\{\begin{array}{ll}
G^n(t)-bH^n(t)&\mbox{if }
G^n(t)\geq 0, 0\leq H^n(t)\leq G^n(t)/b,\\
G^n(t)-H^n(t)&\mbox{if }
G^n(t)\geq 0,H^n(t)<0,\\
-H^n(t)&\mbox{if }
G^n(t)< 0, H^n(t)\leq-G^n(t)/a,
\end{array}\right.\label{4.16d}\\
\sX^n(t)
&=
\left\{\begin{array}{ll}
H^n(t)&\mbox{if }
G^n(t)\geq 0, H^n(t)\leq G^n(t)/b,\\
G^n(t)+H^n(t)&\mbox{if }
G^n(t)<0,H^n(t)<0,\\
G^n(t)+aH^n(t)&\mbox{if }
G^n(t)<0,0\leq H^n(t)\leq - G^n(t)/a.
\end{array}\right.\label{4.17d}
\end{align}
Therefore, like $(\sW^n,\sX^n)$,
the pair $(G^n,H^n)$ is Markov.

We show in Proposition \ref{P4.3} below that
the diffusion-scaled version of $G^n$
is a martingale.
One can see intuitively in Figure 4.2 
that $G^n$ is a martingale by computing
the net drift at each of eight points with arrows.
For example, in the region $E$,
$\Ghatn$ moves up $b$ 
at rate $\lambda_1$ (due to an increase
in $\sX^n$),
moves down $1$ at rate
$\mu_1$ (due to a decrease in
$\sX^n$), moves up $1$
at rate $\lambda_2$ (due to an increase
in $\sW^n$), and moves down
$1$ at rate $\mu_0$
(due to a decrease in $\sWn$).
The net drift rate is
$b\lambda_1-\mu_1+\lambda_2-\mu_0$,
which is zero because of Assumption \ref{Assumption1}.

To reduce the amount of notation, we write
\begin{align}
G^n
&=
G^n(0)
+b\Phi_1^n+\Phi_2^n+\Phi_3^n+\Phi_4^n
+\Phi_5^n+\Phi_6^n+\Phi_7^n+a\Phi_8^n,
\label{4.16c}\\
\big|G^n\big|
&=
\big|G^n(0)\big|
+b\Phi_1^n+\Phi_2^n+\Phi_3^n+\Phi_4^n-\Phi^n_5-\Phi_6^n
-\Phi^n_7-a\Phi_8^n,
\label{4.17c}\\
H^n
&=
H^n(0)+\Phi_1^n+\Phi_2^n-\Phi_7^n-\Phi_8^n
+\Phi_9^n+\Phi_{10}^n,
\label{4.18c}\\
\big|H^n\big|
&=
\big|H^n(0)\big|+
\Phi_1^n-\Phi_2^n+\Phi_7^n-\Phi_8^n
-\Phi_9^n+\Phi_{10}^n,
\label{4.19c}
\end{align}
where
\begin{align*}
\Phi^n_1
&:=
N_{NE,X,+}\circ\lambda_1P^n_{NE}
-N_{NE,X,-}\circ\mu_0P^n_{NE}
+N_{E,X,+}\circ\lambda_1P^n_E,\\
\Phi^n_2
&:=
-N_{E,X,-}\circ\mu_1P^n_E
+N_{SE_+,X,+}\circ\lambda_0P^n_{SE_+}
-N_{SE_+,X,-}\circ\mu_1P^n_{SE_+},\\
\Phi^n_3
&:=
N_{NE,W,+}\circ\lambda_2P^n_{NE}
+N_{E,W,+}\circ\lambda_2P^n_E
-N_{E,W,-}\circ\mu_0P^n_E\nonumber\\
&\qquad
+N_{SE_+,W,+}\circ\lambda_1P^n_{SE_+}
-N_{SE_+,W,-}\circ\mu_0P^n_{SE_+},\\
\Phi^n_4
&:=
N_{SE,X,+}\circ\lambda_0P^n_{SE}
+N_{SE,W,+}\circ\lambda_1P^n_{SE}
+N_{O,W,+}\circ\lambda_2 P^n_O
+bN_{O,X,+}\circ\lambda_1P^n_O,\\
\Phi^n_5
&:=
-N_{SE,W,-}\circ\mu_0P^n_{SE}
-N_{SE,X,-}\circ\mu_1P^n_{SE}
-N_{O,X,-}\circ\mu_2P^n_O
-aN_{O,W,-}\circ\mu_1P^n_O,\\
\Phi^n_6
&:=
-N_{SW,X,-}\circ\mu_2P^n_{SW}
-N_{S,X,-}\circ\mu_2P^n_S
+N_{S,X,+}\circ\lambda_0P^n_S\nonumber\\
&\qquad
-N_{SE_-,X,-}\circ\mu_1P^n_{SE_-}
+N_{SE_-,X,+}\circ\lambda_0P^n_{SE_-},\\
\Phi^n_7
&:=
N_{S,W,+}\circ\lambda_1P^n_S
-N_{SE_-,W,-}\circ\mu_0P^n_{SE_-}
+N_{SE_-,W,+}\circ\lambda_1P^n_{SE_-},\\
\Phi^n_8
&:=
-N_{SW,W,-}\circ\mu_1P^n_{SW}
+N_{SW,W,+}\circ\lambda_0P^n_{SW}
-N_{S,W,-}\circ\mu_1P^n_S,\\
\Phi^n_9
&:=
N_{SE,X,+}\circ\lambda_0P^n_{SE}
+N_{SE,W,-}\circ\mu_0P^n_{SE},\\
\Phi^n_{10}
&:=
N_{O,X,+}\circ\lambda_1P^n_O
+N_{O,W,-}\circ\mu_1P^n_O.\\
\end{align*}

\subsection{Diffusion scaling}\label{DiffusionScaling}

We introduce the diffusion scaled processes,
defined for $t\geq 0$,
\be\label{DS}
\sXhatn(t):=\frac{1}{\sqrt{n}}\sX^n(nt),\,\,
\sWhatn(t):=\frac{1}{\sqrt{n}}\sW^n(nt),\,\,
\Ghatn(t):=\frac{1}{\sqrt{n}}G^n(nt),\,\,
\Hhatn(t):=\frac{1}{\sqrt{n}}H^n(nt).
\ee
The transformation
(\ref{4.14c}) and (\ref{4.15c}) is preserved
by diffusion scaling, i.e.,
\begin{align}
\Ghatn(t)
&:=
\left\{\begin{array}{ll}
\sWhatn(t)+b\sXhatn(t)&\mbox{if }
\big(\sWhatn(t),\sXhatn(t)\big)
\in NE\cup E,\\
\sWhatn(t)+\sXhatn(t)&\mbox{if }
\big(\sWhatn(t),\sXhatn(t)\big)
\in SE_+\cup SE \cup SE_-\cup O,\\
a\sWhatn(t)+\sXhatn(t)&\mbox{if }
\big(\sWhatn(t),\sXhatn(t)\big)
\in SW\cup S,
\end{array}\right.\label{3.17a}\\
\Hhatn(t)
&:=
\left\{\begin{array}{ll}
\sXhatn(t)&\mbox{if }\big(\sWhatn(t),\sXhatn(t)\big)
\in NE\cup E\cup SE_+,\\
\sXhatn(t)=-\sWhatn(t)&\mbox{if }
\big(\sWhatn(t),\sXhatn(t)\big)
\in SE \cup O,\\
-\sWhatn(t)&\mbox{if }
\big(\sWhatn(t),\sXhatn(t)\big)
\in SW\cup S\cup SE_-.
\end{array}\right.\label{3.18a}
\end{align}
This transformation has continuous piecewise
linear inverse (cf.\ (\ref{4.16d}), (\ref{4.17d}))
\begin{align}
\sWhatn(t)
&=
\left\{\begin{array}{ll}
\Ghatn(t)-b\Hhatn(t)&\mbox{ if }
\Ghatn(t)\geq 0, 0\leq \Hhatn(t)\leq \Ghatn(t)/b,\\
\Ghatn(t)-\Hhatn(t)&\mbox{ if }
\Ghatn(t)\geq 0,\Hhatn(t)<0,\\
-\Hhatn(t)&\mbox{ if }
\Ghatn(t)< 0, \Hhatn(t)\leq-\Ghatn(t)/a,
\end{array}\right.\label{4.34d}\\
\sXhatn(t)
&=
\left\{\begin{array}{ll}
\Hhatn(t)&\mbox{ if }
\Ghatn(t)\geq 0, \Hhatn(t)\leq\Ghatn(t)/b,\\
\Ghatn(t)+\Hhatn(t)&\mbox{ if }
\Ghatn(t)<0,\Hhatn(t)<0,\\
\Ghatn(t)+a\Hhatn(t)&\mbox{ if }
\Ghatn(t)<0,0\leq\Hhatn(t)\leq -\Ghatn(t)/a.
\end{array}\right.\label{4.35d}
\end{align}

The Continuous Mapping Theorem
implies we can determine the $J_1$-weak limit
of $(\sWhatn,\sXhatn)$
by determining the limit of $(\Ghatn,\Hhatn)$.
We show in Section \ref{Crushing}
that $\Hhatn\ArrowJ1 0$ and in
Section \ref{Ghatn} that $\Ghatn$ converges
to a two-speed Brownian motion.

We also define $\Phihatn_i(t)=\Phi_i(nt)/\sqrt{n}$,
so that (\ref{4.16c})--(\ref{4.19c}) becomes
\begin{align}
\Ghatn
&=
\Ghatn(0)+b\Phihatn_1+\Phihatn_2+\Phihatn_3
+\Phihatn_4+\Phihatn_5+\Phihatn_6+\Phihatn_7
+a\Phihatn_8,\label{4.32c}\\
\big|\Ghatn\big|
&=
\big|\Ghatn(0)\big|+b\Phihatn_1+\Phihatn_2+\Phihatn_3
+\Phihatn_4-\Phihatn_5-\Phihatn_6-\Phihatn_7
-a\Phihatn_8,\label{4.33c}\\
\Hhatn
&=
\Hhatn(0)
+\Phihatn_1+\Phihatn_2-\Phihatn_7-\Phihatn_8
+\Phihatn_9+\Phihatn_{10},\label{4.34c}\\
\big|\Hhatn(0)\big|
&=
\big|\Hhatn(0)\big|
+\Phihatn_1-\Phihatn_2+\Phihatn_7-\Phihatn_8
-\Phihatn_9+\Phihatn_{10}.\label{4.35c}
\end{align}
It will be useful to replace the independent 
diffusion-scaled Poisson processes appearing
in $\Phihatn_i$ by the {\em centered}
diffusion-scaled Poisson processes
\be
\Mhatn_{\times,*,\pm}(t):=
\frac{1}{\sqrt{n}}
\big(N_{\times,*,\pm}(nt)-nt\big),\quad
\times\in\sR, *\in\{W,X\},\label{3.13}\
\ee
which converge $J_1$-weakly 
to independent standard Brownian motions
$B_{\times,*,\pm}$, i.e.,
\be\label{4.41e}
\big(\Mhatn_{\times,*,\pm}\big)_{\times\in\sR, *\in\{W,X\}}
\ArrowJ1
\big(B_{\times,*,\pm}\big)_{\times\in\sR,*\in\{W,X\}}.
\ee\
We also define the {\em fluid-scaled
occupation times}
\be\label{3.14}
\Pbarn_{\times}(t):=
\frac{1}{n}P_{\times}(nt)
=\int_0^t\ind_{\{(\sWhatn(s),\sXhatn(s))
\in\times\}}ds,
\quad
\times\in\sR.
\ee
We denote by $\Thetahatn_i$ the
centered versions of $\Phihatn_i$, i.e.,
\begin{align}
\Thetahatn_1
&:=
\Mhatn_{NE,X,+}\circ\lambda_1\Pbarn_{NE}
-\Mhatn_{NE,X,-}\circ\mu_0\Pbarn_{NE}
+\Mhatn_{E,X,+}\circ\lambda_1\Pbarn_E\nonumber\\
&=
\Phihatn_1
-\sqrt{n}(\lambda_1-\mu_0)\Pbarn_{NE}
-\sqrt{n}\lambda_1\Pbarn_E,\label{4.39c}\\
\Thetahatn_2
&:=
-\Mhatn_{E,X,-}\circ\mu_1\Pbarn_E
+\Mhatn_{SE_+,X,+}\circ\lambda_0\Pbarn_{SE_+}
-\Mhatn_{SE_+,X,-}\circ\mu_1\Pbarn_{SE_+}\nonumber\\
&=
\Phihatn_2
+\sqrt{n}\mu_1\Pbarn_E
-\sqrt{n}(\lambda_0-\mu_1)\Pbarn_{SE_+},\label{4.40c}\\
\Thetahatn_3
&:=
\Mhatn_{NE,W,+}\circ\lambda_2\Pbarn_{NE}
+\Mhatn_{E,W,+}\circ\lambda_2\Pbarn_E
-\Mhatn_{E,W,-}\circ\mu_0\Pbarn_E\nonumber\\
&\qquad
+\Mhatn_{SE_+,W,+}\circ\lambda_1\Pbarn_{SE_+}
-\Mhatn_{SE_+,W,-}\circ\mu_0\Pbarn_{SE_+}\nonumber\\
&=
\Phihatn_3-\sqrt{n}\lambda_2\Pbarn_{NE}
-\sqrt{n}(\lambda_2-\mu_0)\Pbarn_E
-\sqrt{n}(\lambda_1-\mu_0)\Pbarn_{SE_+},\label{4.41c}\\
\Thetahatn_4
&:=
\Mhatn_{SE,X,+}\circ\lambda_0\Pbarn_{SE}
+\Mhatn_{SE,W,+}\circ\lambda_1\Pbarn_{SE}
+\Mhatn_{O,W,+}\circ\lambda_2\Pbarn_O
+b\Mhatn_{O,X,+}\circ\lambda_1\Pbarn_O\nonumber\\
&=
\Phihatn_4
-\sqrt{n}(\lambda_0+\lambda_1)\Pbarn_{SE}
-\sqrt{n}(\lambda_2+b\lambda_1)\Pbarn_O,\label{4.42c}\\
\Thetahatn_5
&:=
-\Mhatn_{SE,W,-}\circ\mu_0\Pbarn_{SE}
-\Mhatn_{SE,X,-}\circ\mu_1\Pbarn_{SE}
-\Mhatn_{O,X,-}\circ\mu_2\Pbarn_O
-a\Mhatn_{O,W,-}\circ\mu_1\Pbarn_O\nonumber\\
&=
\Phihatn_5
+\sqrt{n}(\mu_0+\mu_1)\Pbarn_{SE}
+\sqrt{n}(\mu_2+a\mu_1)\Pbarn_O,\label{4.43c}\\
\Thetahatn_6
&:=
-\Mhatn_{SW,X,-}\circ\mu_2\Pbarn_{SW}
-\Mhatn_{S,X,-}\circ\mu_2\Pbarn_S
+\Mhatn_{S,X,+}\circ\lambda_0\Pbarn_S\nonumber\\
&\qquad
-\Mhatn_{SE_-,X,-}\circ\mu_1\Pbarn_{SE_-}
+\Mhatn_{SE_-,X,+}\circ\lambda_0\Pbarn_{SE_-}\nonumber\\
&=
\Phihatn_6+\sqrt{n}\mu_2\Pbarn_{SW}
+\sqrt{n}(\mu_2-\lambda_0)\Pbarn_S
+\sqrt{n}(\mu_1-\lambda_0)\Pbarn_{SE_-},\label{4.44c}\\
\Thetahatn_7
&:=
\Mhatn_{S,W,+}\circ\lambda_1\Pbarn_S
-\Mhatn_{SE_-,W,-}\circ\mu_0\Pbarn_{SE_-}
+\Mhatn_{SE_-,W,+}\circ\lambda_1\Pbarn_{SE_-}\nonumber\\
&=
\Phihatn_7-\sqrt{n}\lambda_1\Pbarn_S
+\sqrt{n}(\mu_0-\lambda_1)\Pbarn_{SE_-},\label{4.45c}\\
\Thetahatn_8
&:=
-\Mhatn_{SW,W,-}\circ\mu_1\Pbarn_{SW}
+\Mhatn_{SW,W,+}\circ\lambda_0\Pbarn_{SW}
-\Mhatn_{S,W,-}\circ\mu_1\Pbarn_S\nonumber\\
&=
\Phihatn_8+\sqrt{n}(\mu_1-\lambda_0)\Pbarn_{SW}
+\sqrt{n}\mu_1\Pbarn_S,\label{4.46c}\\
\Thetahatn_9
&:=
\Mhatn_{SE,X,+}\circ\lambda_0\Pbarn_{SE}
+\Mhatn_{SE,W,-}\circ\mu_0\Pbarn_{SE}=
\Phihatn_9
-\sqrt{n}(\lambda_0+\mu_0)\Pbarn_{SE},\label{4.47c}\\
\Thetahatn_{10}
&:=
\Mhatn_{O,X,+}\circ\lambda_1\Pbarn_O
+\Mhatn_{O,W,-}\circ\mu_1\Pbarn_O=
\Phihatn_{10}
-\sqrt{n}(\lambda_1+\mu_1)\Pbarn_O.\label{4.48c}
\end{align}
Using Assumption \ref{Assumption1} to simplify, we
rewrite (\ref{4.32c})--(\ref{4.35c}) as
\begin{align}
\Ghatn
&=
\Ghatn(0)+b\Thetahatn_1+\Thetahatn_2+\Thetahatn_3
+\Thetahatn_4+\Thetahatn_5+\Thetahatn_6
+\Thetahatn_7+a\Thetahatn_8,\label{4.49c}\\
\big|\Ghatn\big|
&=\big|\Ghatn(0)\big|
+b\Thetahatn_1+\Thetahatn_2+\Thetahatn_3
+\Thetahatn_4-\Thetahatn_5-\Thetahatn_6
-\Thetahatn_7-a\Thetahatn_8\nonumber\\
&\qquad
+\sqrt{n}(a\lambda_0+b\mu_0)(\Pbarn_{SE}+\Pbarn_O),
\label{4.50c}\\
\Hhatn
&=
\Hhatn(0)+\Thetahatn_1+\Thetahatn_2-\Thetahatn_7
-\Thetahatn_8+\Thetahatn_9+\Thetahatn_{10}
+\sqrt{n}c\big(-\Pbarn_{NE}+\Pbarn_{SE_+}\!+\Pbarn_{SE_-}
\!-\Pbarn_{SW}\big)
\nonumber\\
&\qquad
+\sqrt{n}(\mu_1-\lambda_1)\big(\Pbarn_S-\Pbarn_E)
+\sqrt{n}(\lambda_0+\mu_0)\Pbarn_{SE}
+\sqrt{n}(\lambda_1+\mu_1)\Pbarn_O,
\label{4.51c}\end{align}
\begin{align}
\big|\Hhatn\big|
&=
\big|\Hhatn(0)\big|
+\Thetahatn_1-\Thetahatn_2+\Thetahatn_7
-\Thetahatn_8-\Thetahatn_9+\Thetahatn_{10}
-\sqrt{n}\,c\big(\Pbarn_{NE}+\Pbarn_{SE_+}\!
+\Pbarn_{SE_-}\!+\Pbarn_{SW}\big)\nonumber\\
&\qquad
+\sqrt{n}(\lambda_1+\mu_1)
\big(\Pbarn_E+\Pbarn_S+\Pbarn_O\big)
-\sqrt{n}(\lambda_0+\mu_0)\Pbarn_{SE}.
\label{4.52c}
\end{align}

\begin{remark}\label{R4.2}
{\rm Let $\F^n=\{\sF^n(t)\}_{t\geq 0}$ denote
the filtration generated by $(\sWhatn,\sXhatn)$,
or equivalently, by $(\Ghatn,\Hhatn)$.
The second equality in (\ref{3.14}) shows
that the time changes $\Pbarn_{\times}$
are adapted to $\F^n$.
}
\end{remark}

\begin{proposition}\label{P4.3}
The process $\Ghatn$ is martingale
relative $\F^n$.
\end{proposition}

\begin{proof}
The pair $(G^n,H^n)$ is a time-homogeneous
Markov process, and consequently
so is $(\Ghatn,\Hhatn)$.  Thus the martingale
property relative to $\F^n$, the filtration
generated by $(\Ghatn,\Hhatn)$, reduces to showing
that for every initial condition
$(\Ghatn(0),\Hhatn(0))$ and for every $t\geq 0$,
\be\label{4.57d}
\E^{(\Ghatn(0),\Hhatn(0))}[\Ghatn(t)]
=\Ghatn(0).
\ee

Let us begin the evolution of
$(\Ghatn,\Hhatn)$ at an initial state 
$(\Ghatn(0),\Hhatn(0))=(g,h)$
in one of the eight regions in $\sR$.  
Set $\tau_0=0$.  Denote
by $\tau_1$ the time of exit of $(\Ghatn,\Hhatn)$
from the initial region, and by $\tau_k$,
$k=2,3,\dots$, the ordered sequence
of times of successive exits
from regions in $\sR$.
In each of the regions in $\sR$, $\Ghatn$
evolves as a sum of diffusion-scaled centered
Poisson processes.  For example, when $(\Ghatn,\Hhatn)$
is in $NE$,
$$
d\Ghatn=d\big(b\Mhatn_{NE,X,+}\circ \lambda_1\id\big)
-d\big(b\Mhatn_{NE,X,-}\circ \mu_0\id\big)
+d\big(\Mhatn_{NE,W,+}\circ\lambda_2\id\big).
$$
This implies, in particular, that 
$\Ghatn(t\wedge\tau_1)$, $t\geq 0$, is
a martingale, and hence
\be\label{4.58d}
g=\E^{(g,h)}\big[\Ghatn(t\wedge\tau_0)\big]
=\E^{(g,h)}\big[\Ghatn(t\wedge\tau_1)\big].
\ee

We generalize this to show that
for all $k=0,1,\dots$,
\be\label{4.59d}
\E^{(g,h)}\big[\Ghatn(t\wedge\tau_k)\big]
=\E^{(g,h)}\big[\Ghatn(t\wedge\tau_{k+1})\big].
\ee
We first observe from the Markov property
and (\ref{4.58d}) that
\begin{align*}
\lefteqn{\E^{(g,h)}\big[\Ghatn(t\wedge\tau_{k+1})
\ind_{\{t>\tau_k\}}\big]}\\
&=
\sum_{(g',h')}\int_0^t\E^{(g',h')}
\big[\Ghatn\big((t-s)\wedge\tau_1\big)\big]
\P^{(g,h)}\big\{\big(\Ghatn(\tau_k),\Hhatn(\tau_k)\big)
=(g',h'),\tau_k\in ds\big\}\\
&=
\sum_{(g',h')}\int_0^tg'\,
\P^{(g,h)}\big\{\big(\Ghatn(\tau_k),\Hhatn(\tau_k)\big)
=(g',h'),\tau_k\in ds\big\}\\
&=
\E^{(g,h)}\big[\Ghatn(\tau_k)\ind_{\{t>\tau_k\}}\big].
\end{align*}
Consequently,
\begin{align*}
\E^{(g,h)}\big[\Ghatn(t\wedge\tau_{k+1})\big]
&=
\E^{(g,h)}\big[\Ghatn(t)\ind_{\{t\leq \tau_k\}}\big]
+\E^{(g,h)}\big[\Ghatn(t\wedge\tau_{k+1})
\ind_{\{t>\tau_k\}}\big]\\
&=
\E^{(g,h)}\big[\Ghatn(t)
\ind_{\{t\leq \tau_k\}}\big]
+\E^{(g,h)}\big[\Ghatn(\tau_k)
\ind_{\{t>\tau_k\}}\big]\\
&=
\E^{(g,h)}\big[\Ghatn(t\wedge\tau_k)\big],
\end{align*}
and (\ref{4.59d}) is established.
From (\ref{4.58d}) and (\ref{4.59d}) we have
$g=\E^{(g,h)}[\Ghatn(t\wedge\tau_k)]$ for
$k=1,2,\dots$
Observing the
$\E^{(g,h)}[\max_{0\leq s\leq t}|\Ghatn(s)|]<\infty$,
we let $k\rightarrow\infty$ to obtain
(\ref{4.57d}). 
\end{proof}

\subsection{Crushing $\Hhatn$}\label{Crushing}

\begin{theorem}\label{T3.2}
We have $\Hhatn
\stackrel{J_1}{\Longrightarrow} 0$.
\end{theorem}
\begin{proof}
First observe that, having continuous weak-$J_1$ limits,
the processes $\Mhatn_{\times,*,\pm}$
are $O(1)$.  The processes $\Pbarn_{\times}$
are dominated by the identity, so 
$\Thetahatn_i$ are $O(1)$ for i=1,\dots,10.

We now adapt a proof from \cite{Petersen}.
For $t\geq 0$, define
$$
\tau^n(t):=
\left\{\begin{array}{ll}
\sup\big\{s\in[0,t]:\Hhatn(s)=0\big\}&
\mbox{if }\big\{s\in[0,t]:\Hhatn(s)=0\big\}
\neq \emptyset,\\
0&\mbox{otherwise}.
\end{array}\right.
$$
Because $\Hhatn\neq 0$ on $(\tau^n(t),t]$,
\begin{align*}
\lefteqn{\Pbarn_E(t)+\Pbarn_S(t)+\Pbarn_O(t)}
\qquad\\
&=
\Pbarn_E\big(\tau^n(t)\big)
+\Pbarn_S\big(\tau^n(t)\big)
+\Pbarn_O\big(\tau^n(t)\big),\\
\lefteqn{\Pbarn_{NE}(t)+\Pbarn_{SE_+}(t)
+\Pbarn_{SE}(t)+\Pbarn_{SE_-}(t)
+\Pbarn_{SW}(t)}\qquad\\
&=
\Pbarn_{NE}\big(\tau^n(t)\big)
+\Pbarn_{SE_+}\big(\tau^n(t)\big)
+\Pbarn_{SE}\big(\tau^n(t)\big)
+\Pbarn_{SE_-}\big(\tau^n(t)\big)
+\Pbarn_{SW}\big(\tau^n(t)\big)
+t-\tau^n(t).
\end{align*}
Substitution into (\ref{4.52c}) yields
\begin{align}
0
&\leq
\big|\Hhatn(t)\big|\nonumber\\
&=
\big|\Hhatn\big(\tau^n(t)\big)\big| +O(1)\nonumber\\
&\quad
-\sqrt{n}\,c\big[\Pbarn_{NE}(t)
+\Pbarn_{SE_+}(t)+\Pbarn_{SE_-}(t)+\Pbarn_{SW}(t)
\nonumber\\
&\qquad\qquad\quad
-\Pbarn_{NE}\big(\tau^n(t)\big)
-\Pbarn_{SE_+}\big(\tau^n(t)\big)
-\Pbarn_{SE_-}\big(\tau^n(t)\big)
-\Pbarn_{SW}\big(\tau^n(t)\big)\big]\nonumber\\
&\quad
-\sqrt{n}\,(\lambda_0+\mu_0)
\big[\Pbarn_{SE}(t)
-\Pbarn_{SE}\big(\tau^n(t)\big)\big]\nonumber\\
&\leq
\Hhatn\big(\tau^n(t)\big)
+O(1)-\min(c,\lambda_0+\mu_0)\sqrt{n}
\big(t-\tau^n(t)\big).\label{4.34}
\end{align}
Assumption \ref{Assumption2} implies
that $\Hhatn(0)\ArrowJ1 0$,
so $\Hhatn(\tau^n(t))\ArrowJ1 0$
if $\tau^n(t)=0$.  Otherwise, 
$\widehat{H}(\tau^n(t))=0$, and because jumps
in $\Hhatn$ are of size $1/\sqrt{n}$,
we have $|\Hhatn(\tau^n(t))|=1/\sqrt{n}
\ArrowJ1 0$.
Thus, (\ref{4.34}) implies
$\tau^n\ArrowJ1\id$.
But this implies 
$\Pbarn_{\times}-\Pbarn_{\times}\circ\tau^n
=o(1)$ and hence
$\Mhatn_{\times,*,\pm}\circ\alpha\Pbarn_{\times}
-\Mhatn_{\times,*,\pm}\circ\alpha\Pbarn_{\times}
\circ\tau^n=o(1)$ for all positive $\alpha$.
Using this we can upgrade (\ref{4.34}) to
$$
0\leq \big|\Hhatn(t)\big|
=o(1)-\min(c,\lambda_0+\mu_0)\sqrt{n}
\big(t-\tau^n(t)\big),
$$
which implies 
$\sqrt{n}(\id-\tau^n)=o(1)$.
In particular, $|\Hhatn|=o(1)$.
\end{proof}

\begin{remark}\label{R4.5}
{\rm Dividing (\ref{4.51c}) and (\ref{4.52c})
by $\sqrt{n}$ and passing to the limit,
we obtain
\begin{align}
c\big(\Pbarn_{SE_+}+\Pbarn_{SE_-}-\Pbarn_{NE}
-\Pbarn_{SW}\big)
+(\mu_1-\lambda_1)(\Pbarn_S-\Pbarn_E)
+(\lambda_0+\mu_0)\Pbarn_{SE}\nonumber\\
+(\lambda_1+\mu_1)\Pbarn_O
&\ArrowJ1
0,\label{4.36}\\
-c\big(\Pbarn_{NE}
+\Pbarn_{SE_+}+\Pbarn_{SE_-}+\Pbarn_{SW}\big)
+(\lambda_1+\mu_1)\big(\Pbarn_E+\Pbarn_S+\Pbarn_O\big)
-(\lambda_0+\mu_0)\Pbarn_{SE}
&\ArrowJ1
0.\label{4.37}
\end{align}
}
\end{remark}

\subsection{Convergence of $\Ghatn$}\label{Ghatn}

In this section we show that $\Ghatn$
converges to a two-speed Brownian motion.
The proof proceeds through several steps.
Along the way we identify the limits
of the processes $\Pbarn_{\times}$.

\begin{lemma}\label{L4.6}
For k=1,2,\dots, let 
$\varphi_k:\R\rightarrow[0,\infty)$ be defined by
$\varphi_k(\xi)=\max(-k|\xi|+1,0)$,
$\xi\in\R$.  For $1<\alpha<\infty$, define
$F_k:D[0,\infty)\rightarrow\R$ by
$$
F_0(x)=\int_0^{\infty}e^{-\alpha s}\ind_{\{0\}}
\big(x(s)\big)ds,\quad
F_k(x)=\int_0^{\infty}e^{-\alpha s}
\varphi_k\big(x(s)\big)ds,
\quad k=1,2,\dots.
$$
Then in the $J_1$ topology,
$F_k$ is continuous for $k=1,2,\dots$,
and $F_0$ is upper semicontinuous.
\end{lemma}

\begin{proof}
Since $F_0=\inf_{k\geq 1}F_k$, it suffices
to show that each $F_k$ is continuous.

Let $k\geq 1$ be given.  Recall from 
\cite{EthierKurtz}, Section 3.5, that a metric
for the $J_1$ topology is
$$
d(x,y):=\inf_{\lambda\in\Lambda}
\left[\gamma(\lambda)\vee\int_0^{\infty}
e^{-u}d(x,y,\lambda,u)du\right],
$$
where
\begin{align*}
d(x,y,\lambda,u)
&:=
1\wedge\sup_{t\geq 0}\big|x(t\wedge u)
-y\big(\lambda(t)\wedge u\big)\big|,\\
\gamma(\lambda)
&:=
\mbox{ess sup}_{t\geq 0}\big|\log\lambda'(t)\big|
=\sup_{s>t\geq 0}
\left[\log\frac{\lambda(s)-\lambda(t)}{s-t}\right],
\end{align*}
and $\Lambda$ is the set of all strictly
increasing Lipschitz continuous functions
$\lambda$ mapping $[0,\infty)$ onto $[0,\infty)$
with $\gamma(\lambda)<\infty$.
Let $x_n\rightarrow x$ in the $J_1$ topology
on $D[0,\infty)$.  Then there exists a sequence
$\{\lambda_n\}_{n=1}^{\infty}$ in $\Lambda$
such that
$\lim_{n\rightarrow\infty}(\gamma(\lambda_n)
\vee\int_0^{\infty}e^{-u}d(x,x_n,\lambda_n,u)du)
=0$.
We compute
\begin{align}
\big|F_k(x_n)-F_k(x)\big|
&=
\left|\int_0^{\infty}e^{-\alpha s}
\big[\varphi_k\big(x_n(s)\big)
-\varphi_k\big(x(s)\big)\big]ds\right|\nonumber\\
&\leq
\left|\int_0^{\infty}e^{-\alpha s}
\big[\varphi_k\big(x_n(s)\big)
-\varphi_k\big(x\big(\lambda_n(s)\big)\big)\big]ds
\right|\nonumber\\
&\qquad
+\left|\int_0^{\infty}e^{-\alpha s}
\big[\varphi_k\big(x\big(\lambda_n(s)\big)\big)
-\varphi_k\big(x(s)\big)\big]ds\right|.\label{4.38}
\end{align}
We consider each of the last two terms.

Because $\varphi_k$ is bounded by $1$
and is Lipschitz with constant $k$,
for $n$ large enough that $\lambda_n(s)\leq \alpha s$,
$s\geq 0$, we have
\begin{align*}
\lefteqn{\left|\int_0^{\infty}e^{-\alpha s}
\big[\varphi_k\big(x_n(s)\big)
-\varphi_k\big(x\big(\lambda_n(s)\big)\big)\big]ds
\right|}\qquad\nonumber\\
&\leq
\int_0^{\infty}e^{-\alpha s}
\Big(2\wedge k\big|x_n(s)-x\big(\lambda_n(s)\big)\big|
\Big)ds\nonumber\\
&\leq
2k\int_0^{\infty}e^{-\alpha s}
\Big(1\wedge\sup_{t\geq 0}
\big|x_n\big(t\wedge(\alpha s)\big)
-x\big(\lambda_n(t)\wedge(\alpha s)\big)\big|\Big)ds
\nonumber\\
&\leq
\frac{2k}{\alpha}
\int_0^{\infty}e^{-u}d(x_n,x,\lambda_n,u)du,
\end{align*}
which has limit zero as $n\rightarrow\infty$.
For the last term in (\ref{4.38}), we observe
first that each $\lambda_n$ is absolutely
continuous, $\lambda_n(0)=0$,
$\lambda_n'$ is defined almost everywhere,
and $|\lambda_n'-1|$ is uniformly bounded
by a constant that goes to zero as 
$n\rightarrow\infty$.  Therefore,
\begin{align*}
\lefteqn{\left|\int_0^{\infty}
e^{-\alpha s}
\big[\varphi_k\big(x\big(\lambda_n(s)\big)\big)
-\varphi_k\big(x(s)\big)\big]ds\right|}\qquad
\nonumber\\
&=
\left|\int_0^{\infty}e^{-\alpha s}
\varphi_k\big(x\big(\lambda_n(s)\big)\big)ds
-\int_0^{\infty}e^{-\alpha \lambda_n(t)}
\varphi_k\big(x\big(\lambda_n(t)\big)\big)
\lambda_n'(t)dt\right|\nonumber\\
&\leq
\int_0^{\infty}e^{-\alpha s}
\big|1-e^{\alpha (s-\lambda_n(s))}\big|
\big|\lambda_n'(s)\big|ds,
\end{align*}
which also has limit zero as $n\rightarrow\infty$
by the Dominated Convergence Theorem.
\end{proof}

We reorganize the terms in (\ref{4.39c})--(\ref{4.46c})
appearing in (\ref{4.49c}),
grouping them by region.  More precisely,
we write $\Ghatn$ as
$$
\Ghatn=\Ghatn(0)+\sum_{\times\in\sR }
\Psihatn_{\times}\circ\Pbarn_{\times},
$$
where
\begin{align*}
\Psihatn_{NE}
&:=
b\Mhatn_{NE,X,+}\circ\lambda_1\id
-b\Mhatn_{NE,X,-}\circ\mu_0\id
+\Mhatn_{NE,W,+}\circ\lambda_2\id,\\
\Psihatn_{E}
&:=
b\Mhatn_{E,X,+}\circ\lambda_1\id
-\Mhatn_{E,X,-}\circ\mu_1\id
+\Mhatn_{E,W,+}\circ\lambda_2\id
-\Mhatn_{E,W,-}\circ\mu_0\id,\\
\Psihatn_{SE_+}
&:=
\Mhatn_{SE_+,X,+}\circ\lambda_0\id
-\Mhatn_{SE_+,X,-}\circ\mu_1\id
+\Mhatn_{SE_+,W,+}\circ\lambda_1\id
-\Mhatn_{SE_+,W,-}\circ\mu_0\id,\\
\Psihatn_{SE}
&:=
\Mhatn_{SE,X,+}\circ\lambda_0\id
+\Mhatn_{SE,W,+}\circ\lambda_1\id
-\Mhatn_{SE,W,-}\circ\mu_0\id
-\Mhatn_{SE,X,-}\circ\mu_1\id,\\
\Psihatn_{SE_-}
&:=
-\Mhatn_{SE_-,W,-}\circ\mu_0\id
+\Mhatn_{SE_-,W,+}\circ\lambda_1\id
-\Mhatn_{SE_-,X,-}\circ\mu_1\id
+\Mhatn_{SE_-,X,+}\circ\lambda_0\id,\\
\Psihatn_{S}
&:=
-a\Mhatn_{S,W,-}\circ\mu_1\id
+\Mhatn_{S,W,+}\circ\lambda_1\id
-\Mhatn_{S,X,-}\circ\mu_2\id
+\Mhatn_{S,X,+}\circ\lambda_0\id,\\
\Psihatn_{SW}
&:=
-a\Mhatn_{SW,W,-}\circ\mu_1\id
+a\Mhatn_{SW,W,+}\circ\lambda_0\id
-\Mhatn_{SW,X,-}\circ\mu_2\id,\\
\Psihatn_{O}
&:=
\Mhatn_{O,W,+}\circ\lambda_2\id
+b\Mhatn_{O,X,+}\circ\lambda_1\id
-\Mhatn_{O,X,-}\circ\mu_2\id
-a\Mhatn_{O,W,-}\circ\mu_1\id.\\
\end{align*}

\begin{proposition}\label{P4.8}
The sequence $\{\Ghatn\}_{n=1}^{\infty}$
is tight in the $J_1$ topology,
the limit of
every convergent subsequence is continuous, 
and the limit spends zero Lebesgue
time at the origin.
\end{proposition}

\begin{proof}
Because $[\Mhatn_{\times,*,\pm},\Mhatn_{\times,*,\pm}]
\stackrel{J_1}\Longrightarrow\id$, using
Assumption \ref{Assumption1} we have
\begin{align}
[\Psihatn_{NE},\Psihatn_{NE}]
&\stackrel{J_1}{\Longrightarrow}
(b^2\lambda_1+b^2\mu_0+\lambda_2)\id
=:A_{NE},\label{4.39}\\
[\Psihatn_{E},\Psihatn_{E}]
&\stackrel{J_1}{\Longrightarrow}
(b^2\lambda_1+b\mu_0+\lambda_2)\id
=:A_{E},\label{4.40}\\
[\Psihatn_{SE_+},\Psihatn_{SE_+}]
&\stackrel{J_1}{\Longrightarrow}
(a\lambda_0+b\mu_0)\id
=:A_{SE_+},\label{4.41}\\
[\Psihatn_{SE},\Psihatn_{SE}]
&\stackrel{J_1}{\Longrightarrow}
(a\lambda_0+b\mu_0)\id
=:A_{SE},\label{4.42}\\
[\Psihatn_{SE_-},\Psihatn_{SE_-}]
&\stackrel{J_1}{\Longrightarrow}
(b\mu_0+a\lambda_0)\id
=:A_{SE_-},\label{4.43}\\
[\Psihatn_{S},\Psihatn_{S}]
&\stackrel{J_1}{\Longrightarrow}
(a^2\mu_1+a\lambda_0+\mu_2)\id
=:A_{S},\label{4.44}
\end{align}
\begin{align}
[\Psihatn_{SW},\Psihatn_{SW}]
&\stackrel{J_1}{\Longrightarrow}
(a^2\mu_1+a^2\lambda_0+\mu_2)\id
=:A_{SW},\label{4.45}\\
[\Psihatn_{O},\Psihatn_{O}]
&\stackrel{J_1}{\Longrightarrow}
(\lambda_2+b^2\lambda_1+\mu_2+a^2\mu_1)\id
=:A_{O}.\label{4.46}
\end{align}
Then
\be\label{4.48a}
A^n:=\sum_{\times\in\sR}A_{x}\circ\Pbarn_{\times}
\ee
is strictly increasing and piecewise linear
with slope bounded 
between $m:=\min\{A'_{\times}:\mathnormal{\times}\in\sR\}$
and $M:=\max\{A'_{\times}:\mathnormal{\times}\in\sR\}$.
Let $I^n$ be the inverse of $A^n$, which
is also strictly increasing and piecewise
linear with slope bounded between $1/M$
and $1/m$. For each $s$, $I^n(s)$ is a stopping
time for $\{\sF(t)\}_{t\geq 0}$ (see
Remark \ref{R4.2}).  We have  
\begin{align*}
[\Ghatn\circ I^n,\Ghatn\circ I^n]
&=
\sum_{\times\in\sR}[\Psihatn_{\times},
\Psihatn_{\times}]\circ\Pbarn_{\times}
\circ I^n\\
&=
\sum_{\times\in\sR}
\big([\Psihatn_{\times},\Psihatn_{\times}]
-A_{\times}\big)\circ\Pbarn_{\times}
\circ I^n
+\sum_{\times\in\R}A_{\times}\circ\Pbarn_{\times}
\circ I^n\\
&=
\sum_{\times\in\sR}
\big([\Psihatn_{\times},\Psihatn_{\times}]
-A_{\times}\big)\circ\Pbarn_{\times}
\circ I^n+\id
\stackrel{J_1}{\Longrightarrow}
\id,
\end{align*}
by (\ref{4.39})--(\ref{4.46}).
We now apply 
\cite[Theorem~1.4,~Section~7.1]{EthierKurtz}
to conclude that
\be\label{4.50a}
\Ghatn\circ I^n\ArrowJ1 B,
\ee
where $B$ is a standard Brownian
motion.  Since the time changes
$A^n$ have bounded slopes uniformly
in $n$, $\{\Ghatn\}_{n=1}^{\infty}
=\{\Ghatn\circ I^n\circ A^n\}_{n=1}^{\infty}$
is tight.

We next argue that the limit of every
convergent subsequence of $\{\Ghatn\}_{n=1}^{\infty}$
is continuous.  
Let $\{\widehat{G}^{n_j}\}_{j=1}^{\infty}$
be a convergent subsequence with limit
$G^*$.  Recall from \cite[Theorem~10.2,
Section~3.10]{EthierKurtz}
that $G^*$ is continuous if and only if 
$J(\widehat{G}^{n_j})\ArrowJ1 0$, where
$$
J(x)(t):=\sum_{0\leq u\leq t}
\big|x(u)-x(u-)\big|,\quad x\in D[0,\infty).
$$
But $\widehat{G}^{n_j}\circ I^{n_j}$
converges weakly to a Brownian motion, which
is continuous, and 
$$
J\big(\widehat{G}^{n_j}\big)(t)
=J\big(\widehat{G}^{n_j}\circ I^{n_j}
\circ A^{n_j}\big)(t)
= J\big(\widehat{G}^{n_j}\circ I^{n_j}\big)
\big(A^{n_j}(t)\big).
$$
Because $A^{n_j}(t)\leq Mt$ and 
$ J\big(\widehat{G}^{n_j}\circ I^{n_j}\big)
\ArrowJ1 0$, we have $J(\widehat{G}^{n_j})\ArrowJ1 0$
and $G^*$ is continuous.

Finally, we must show that $G^*$
spends zero Lebesgue time at the origin.
Given $\ve>0$, there exists
a finite random integer $k_0$ such that $F_k(B)<\ve$
for all $k\geq k_0$, where we use the notation
of Lemma \ref{L4.6} with $\alpha:=2(M\vee 1)/M$.
Consequently, there exists $j_0$
such that $F_k(\widehat{G}^{n_j}\circ I^{n_j})
<\ve$ for all $k\geq k_0$ and $j\geq j_0$.
Making the change of variables $s=A^{n_j}(u)$,
we see that
\be\label{4.49}
m\int_0^{\infty}e^{-2(M\vee 1)u}\varphi_{k}
\big(\widehat{G}^{n_j}(u)\big)du
\leq\int_0^{\infty}e^{-\alpha s}
\varphi_k\big(\widehat{G}^{n_j}\circ I^{n_j}(s)\big)
ds=F_k(\widehat{G}^{n_j}\circ I^n)<\ve.
\ee
According to Lemma \ref{L4.6},
\be\label{4.50}
\int_0^{\infty}e^{-2(M\vee 1)u}\varphi_k
\big(\widehat{G}^{n_j}(u)\big)du
\ArrowJ1 \int_0^{\infty}e^{-2(M\vee 1)u}
\varphi_k\big(G^*(u)\big)du.
\ee
Combining (\ref{4.49}) and (\ref{4.50}), we obtain
$$
\E\int_0^{\infty}e^{-2(M\vee 1)u}
\ind_{\{G^*(u)=0\}}du
\leq
\E\int_0^{\infty}e^{-2(M\vee 1)u}
\varphi_k\big(G^*(u)\big)du
\leq \frac{\ve}{m}.
$$
But $\ve>0$ is arbitrary,
so $\int_0^{\infty}e^{-2(M\vee 1)u}
\ind_{\{G^*(u)=0\}}du=0$ almost surely.
\end{proof}

\begin{corollary}\label{C4.9}
We have $\Pbarn_{SE}\ArrowJ1 0$
and $\Pbarn_{O}\ArrowJ1 0$.
\end{corollary}
\begin{proof}
Again we use the notation of Lemma \ref{L4.6},
this time with $\alpha=2$.
Suppose $\{\P^n\}_{n=1}^{\infty}$
is a sequence of probability measures
on $D[0,\infty)$ converging weakly to
$\P$.   For $k=1,2,\dots$,
$$
\limsup_{n\rightarrow\infty}\int_{D[0,\infty)}
F_0d\P^n\leq\lim_{n\rightarrow\infty}
\int_{D[0,\infty)}F_kd\P^n
=\int_{D[0,\infty)}F_kd\P.
$$ 
Letting $k\rightarrow\infty$, we conclude that
\be\label{4.51}
\limsup_{n\rightarrow\infty}
\int_{D[0,\infty)}F_0d\P^n
\leq \int_{D[0,\infty)}F_0d\P.
\ee
Let $\P^n$ be the probability measure
induced on $D[0,\infty)$ by $\Ghatn\circ I^n$.
In the proof of Proposition \ref{P4.8},
we show that this sequence converges
weakly to Wiener measure.  We have
\begin{align*}
m\E\int_0^{\infty}e^{-2Mu}
\big(d\Pbarn_{SE}(u)+\Pbarn_O(u)\big)
&\leq
\E\int_0^{\infty}e^{-2A^n(u)}
\big(A_{SE}'(u)d\Pbarn_{SE}(u)
+A_O'(u)d\Pbarn_O(u)\big)\\
&=
\E\int_0^{\infty}e^{-2A^n(u)}
\ind_{\{\Ghatn(u)=0\}}dA^n(u)\\
&=
\E\int_0^{\infty}e^{-2s}
\ind_{\{G^n\circ I^n(s)=0\}}ds\\
&=
\int_{D[0,\infty)}F_0d\P^n.
\end{align*}
By (\ref{4.51}), the limit as $n\rightarrow\infty$ of
this last expression
is zero because Brownian motion spends
zero Lebesgue time at the origin.
We conclude that $\Pbarn_{SE}+\Pbarn_{O}$
converges in probability to zero uniformly on compact
time intervals, which is equivalent to
the convergences stated in the corollary.
\end{proof}

We have shown in (\ref{4.50a}) that
$\Ghatn\circ I^n$ converges to a standard
Brownian motion.  We want to identify
the limit of $\Ghatn=\Ghatn\circ I^n\circ A^n$.
Thus we need to determine the limit of $A^n$.
To do this, we determine the limits
of the processes
$\Pbarn_{\times}$, $\times\in\sR$,
appearing in (\ref{4.48a}).
We have just done that for
$\Pbarn_{SE}$ and $\Pbarn_O$.
For the other processes, we have the following
result.

\begin{proposition}\label{P4.10}
Consider a convergent subsequence 
$\{\widehat{G}^{n_j}\}_{j=1}^{\infty}$ with
limit $G^*$.  Define
$$
P_+^{G^*}(t)
:=\int_0^t\ind_{\{G^*(s)>0\}}ds,\quad
P_-^{G^*}(t)
:=\int_0^t\ind_{\{G^*(s)<0\}}ds,\quad t\geq 0.
$$
Then
\be\label{4.55}
\big(\widehat{G}^{n_j},\widehat{H}^{n_j},
(\Pbar^{n_j}_{\times})_{\times\in\R}\big)
\ArrowJ1
\big(G^*,0,(\Pbar_{\times})_{\times\in\R}\big),
\ee
where $\Pbar_{SE}=\Pbar_O=0$ and
\begin{eqnarray}
\Pbar_{NE}=\frac{\lambda_1}{\lambda_0+\lambda_1}
P^{G^*}_+,
&
\Pbar_E=\displaystyle
\frac{c}{\lambda_0+\lambda_1}
P^{G^*}_+,
&
\Pbar_{SE_+}=\frac{\mu_1}{\lambda_0+\lambda_1}
P^{G^*}_+,\label{4.56}\\
\Pbar_{SE_-}=\frac{\lambda_1}{\mu_0+\mu_1}
P^{G^*}_-,
&
\Pbar_{S}=\displaystyle
\frac{c}{\mu_0+\mu_1}
P^{G^*}_-,
&
\Pbar_{SW}=\frac{\mu_1}{\mu_0+\mu_1}
P^{G^*}_-.\label{4.57}
\end{eqnarray}
\end{proposition}

\begin{remark}\label{R4.11}
{\rm Equations (\ref{4.56}) and (\ref{4.57})
have been written to emphasize symmetry
in the formulas.
However, (\ref{a6})
implies that 
$\lambda_0+\lambda_1=a\lambda_0=b\mu_0=\mu_0+\mu_1$,
so all the denominators are the same.
}
\end{remark}

\begin{remark}\label{R4.12a}
{\rm The pre-limit model has a one-tick spread
if and only if $(G^n,H^n)$ is in
$NE\cup SE_+\cup SE\cup SE_-\cup SW$.
Summing the corresponding terms in
(\ref{4.56}) and (\ref{4.57}), we see
that this is the case for the
fraction $2-(a+b)/(ab)$ of the time,
as reported in \cite[equation~(2.4)]{ALSY}.
Summing $\Pbar_E$ and $\Pbar_S$, we see
that the two-tick spread prevails for
the remaining fraction $(a+b)/(ab)-1$
of the time.  The convergences
in Proposition \ref{P4.10} are along
a subsequence of $\{(\Ghatn,\Hhatn)\}_{n=1}^{\infty}$,
but we will see in Corollary \ref{C4.15}
below that in fact the convergence
takes place along the full sequence.
}
\end{remark}

\noi
{\em Proof of Proposition \ref{P4.10}}.
Each $\Pbarn_{\times}$ is nondecreasing
and Lipschitz continuous with Lipschitz constant
$1$.  Moreover, $\Pbarn_{\times}(0)=0$.
This together with Theorem \ref{T3.2}
and Proposition \ref{P4.8} implies that
the sequence 
$\{(\widehat{G}^{n},\widehat{H}^{n},
(\Pbarn_{\times})_{\times\in\sR})\}_{n=1}^{\infty}$
is tight.  Given any subsequence 
$\{(\widehat{G}^{n_j},\widehat{H}^{n_j},
(\Pbar^{n_j}_{\times})_{\times\in\sR})\}_{j=1}^{\infty}$
of this
sequence, there exists a sub-subsequence
that converges weakly-$J_1$ to some limit
$(G^*,0,(\Pbar_{\times})_{\times\in\sR}))$.
We know from
Proposition \ref{P4.8} and Corollary
\ref{C4.9} that
$G^*$ is continuous and $\Pbar_{SE}=\Pbar_O=0$.
Furthermore, each $\Pbar_{\times}$
is Lipschitz continuous with Lipschitz constant
$1$.  We show that
$\Pbar_{NW}$, $\Pbar_E$,
$\Pbar_{SE_+}$, $\Pbar_{SE_-}$,
$\Pbar_S$ and $\Pbar_{SW}$ satisfy 
(\ref{4.56})
and $(\ref{4.57})$.

Substituting $\Pbar_{SE}=\Pbar_O=0$
into the full-limit convergences
(\ref{4.36}) and (\ref{4.37}), we obtain
\begin{align}
c(\Pbar_{NE}+\Pbar_{SW}-\Pbar_{SE_+}-\Pbar_{SE_-})
+(\mu_1-\lambda_1)(\Pbar_E-\Pbar_S)
&=
0,\label{4.58}\\
-c(\Pbar_{NE}+\Pbar_{SE_+}+\Pbar_{SE_-}
+\Pbar_{SW})
+(\lambda_1+\mu_1)(\Pbar_E+\Pbar_S)
&=
0.\label{4.59}
\end{align}
We also have $\sum_{\times\in\sR}\Pbar_{\times}=\id$,
which implies
\be\label{4.60}
\Pbar_{NE}+\Pbar_E+\Pbar_{SE_+}+
\Pbar_{SE_-}+\Pbar_S+\Pbar_{SW}=\id.
\ee
From (\ref{a6}), (\ref{4.59})
and (\ref{4.60}), we obtain
\begin{align}
\Pbar_{NE}+\Pbar_{SE_+}+\Pbar_{SE_-}+\Pbar_{SW}
&=
\frac{\lambda_1+\mu_1}{\lambda_0+\lambda_1}\id,
\nonumber\\
\Pbar_E+\Pbar_S
&=
\frac{c}{\lambda_0+\lambda_1}.
\label{4.62}
\end{align}
Subtracting (\ref{4.59}) from (\ref{4.58}),
we see that
\be\label{4.63}
\Pbar_{NE}+\Pbar_{SW}
=\frac{\lambda_1}{c}\Pbar_E
+\frac{\mu_1}{c}\Pbar_S.
\ee
Adding (\ref{4.59}) to (\ref{4.58}) yields
$$
\Pbar_{SE_+}+\Pbar_{SE_-}
=\frac{\mu_1}{c}\Pbar_E+\frac{\lambda_1}{c}\Pbar_S.
$$

We use the Skorohod Representation Theorem
to put the pre-limit and limit processes
on a common probability space so that the convergence
of $\{\widehat{G}^{n_j},\widehat{H}^{n_j},
(\Pbar^{n_j}_{\times})_{\times\in\sR})\}_{j=1}^{\infty}$
to the continuous process
$(G^*,0,(\Pbar_{\times})_{\times\in\sR})$
is uniform on compact time intervals almost surely.
Because each $\Pbar_{\times}$ is Lipschitz,
to identify $\Pbar_{\times}$ it suffices to
identify $\Pbar_{\times}'$ for Lebesgue almost
every $t\geq 0$.  We identify $\Pbar_{\times}'(t)$
for all $t$ such that $G^*(t)\neq 0$, a set
of full Lebesgue measure by Proposition
\ref{P4.8}.

Assume first that $G^*(t)>0$.  Then for sufficiently
large $j$, $\widehat{G}^{n_j}$ is strictly
positive in a neighborhood of $t$.  From
(\ref{3.17a}), we see that
$(\widehat{W}^{n_j},\widehat{X}^{n_j})$
must be in $NE\cup E\cup SE_+$, which
implies that $\Pbar_{SE_-}^{n_j}$,
$\Pbar_{S}^{n_j}$ and $\Pbar_{SW}^{n_j}$
are constant in this neighborhood
of $t$ for sufficiently large $j$.
Consequently, $\Pbar_{SE_-}$,
$\Pbar_S$ and $\Pbar_{SW}$ are constant in this
neighborhood. We conclude that
\be\label{4.65}
\Pbar_{SE_-}'(t)=\Pbar_S'(t)=\Pbar_{SW}'(t)=0
\mbox{ if }G^*(t)>0.
\ee  
Equation (\ref{4.62}) now implies
\be\label{4.66}
\Pbar_E'(t)=\frac{c}{\lambda_0+\lambda_1}
\mbox{ if }G^*(t)>0.
\ee
Substitution of this into (\ref{4.63}) yields
\be\label{4.67}
\Pbar_{NE}'(t)=\frac{\lambda_1}{\lambda_0+\lambda_1}
\mbox{ if }G^*(t)>0.
\ee 
Equation (\ref{4.60}) allows us to now conclude that
\be\label{4.68}
\Pbar_{SE_+}'(t)=\frac{\mu_1}{\lambda_0+\lambda_1}
\mbox{ if }G^*(t)>0.
\ee
An analogous argument for $t$ such that $G^*(t)<0$
yields
\be\label{4.69}
\Pbar_{NE}'(t)=\Pbar_E'(t)=\Pbar_{SE_+}'(t)=0
\mbox{ if }G^*(t)<0,
\ee
and
\be\label{4.70}
\Pbar_{SE_-}'(t)=\frac{\lambda_1}{\mu_0+\mu_1},\quad
\Pbar_S'(t)=\frac{c}{\mu_0+\mu_1},\quad
\Pbar_{SW}'(t)=\frac{\mu_1}{\mu_0+\mu_1}
\mbox{ if }G^*(t)<0.
\ee
Integrating (\ref{4.65})--(\ref{4.70}),
we obtain (\ref{4.56}) and (\ref{4.57}).
\hfill$\Box$

\begin{corollary}\label{C4.11}
Under the assumptions of Proposition \ref{P4.10},
\be\label{4.71}
A^{n_j}\ArrowJ1 \sigma_+^2P^{G^*}_+
+\sigma_-^2P^{G^*}_-,
\ee
where $A^{n_j}$
is defined by (\ref{4.48a}) and 
$\sigma_{\pm}$ are defined by
(\ref{sigmaplus}) and (\ref{sigmaminus}).  Furthermore,
\be\label{4.72}
G^*=B\circ\big(\sigma_+^2P^{G^*}_+
+\sigma_-^2P^{G^*}_-\big),
\ee
where $B$ is the Brownian motion in (\ref{4.50a}).
\end{corollary}

\begin{proof}
For (\ref{4.71}), it suffices to verify that
$A_{NE}\circ\Pbar_{NE}+A_{E}\circ \Pbar_E
+A_{SE_+}\circ\Pbar_{SE_+}=\sigma_+^2P^{G^*}_+$
and $A_{SE_-}\circ\Pbar_{SE_-}
+A_S\circ\Pbar_S+A_{SW}\circ\Pbar_{SW}=\sigma_-^2P^{G^*}_-$.
The first of these equations
can be verified by a lengthy computation using
(\ref{4.39})--(\ref{4.42}), (\ref{4.56}) and
(\ref{a6}).  It helps in this computation to note
from (\ref{a6}) that $\lambda_0+\lambda_1=a\lambda_0$
and $\lambda_2=a\lambda_0-b\lambda_1$.
The second equation is obtained by a similar computation.

Because $(\widehat{G}^{n_j}\circ I^{n_j},
A^{n_j})\ArrowJ1
(B,\sigma_+^2P^{G^*}_++\sigma_-^2P^{G^*}_-)$
and $B$ is continuous,
we can use the time-change lemma in 
\cite[Section~14]{Billingsley} to obtain
(\ref{4.72}).  That lemma is stated for $D[0,1]$,
but the modification of the proof to extend the
result to $D[0,\infty)$ is straightforward.
\end{proof}

\begin{corollary}\label{C4.13}
Every weakly convergent subsequence of
$\{\Ghatn\}_{n=1}^{\infty}$ converges to
the same limit, i.e., all limits induce the
same probability measure on $C[0,\infty)$.  This
limit is the two-speed Brownian motion
of Definition \ref{TS.D.1}.
\end{corollary}

\begin{proof}
Proposition \ref{P4.8}
and the converse part of Proposition \ref{TS}
imply that $G^*$ in (\ref{4.72})
is a two-speed Brownian
motion.  It is apparent from Definition
\ref{TS.D.1} that a two-speed Brownian
motion with given speeds $\sigma_+^2$
and $\sigma_-^2$ is unique in law,
and hence every convergent subsequence
of $\{\Ghatn\}_{n=1}^{\infty}$ has the same 
$J_1$-weak limit.
\end{proof}

\begin{theorem}\label{T4.14}
For the full sequence $\{\Ghatn\}_{n=1}^{\infty}$,
we have
$\Ghatn\ArrowJ1 G^*$,
where $G^*$ is a two-speed Brownian
motion satisfying (\ref{4.72}).
\end{theorem}

\begin{proof}
Failure of the full sequence to converge to $G^*$
would imply the existence of a subsequence
with no sub-subsequence converging to $G^*$.
This contradicts the tightness 
of $\{\Ghatn\}_{n=1}^{\infty}$ and Corollary
\ref{C4.13}.
\end{proof}

\begin{corollary}\label{C4.15}
The convergence (\ref{4.55})
in Proposition \ref{P4.10} is convergence
of the full sequence 
$\{(\Ghatn,\Hhatn,(\Pbarn_{\times})_{\times\in\R})
\}_{n=1}^{\infty}$.
\end{corollary}

\begin{corollary}\label{C4.16}
Jointly with the convergence in Corollary \ref{C4.15},
we have the joint convergences
\begin{align*}
\Thetahatn_i
&\ArrowJ1
\Theta^*_i\circ P^{G^*}_+,\quad i=1,2,3,\\
\Thetahatn_i
&\ArrowJ1
\Theta^*_i\circ P^{G^*}_-,\quad i=6,7,8,\\
\Thetahatn_i
&\ArrowJ1
0,\quad i=4,5,9,10,
\end{align*}
where
\begin{align}
\Theta_1^*
&:=
B_{NE,X,+}\frac{\lambda_1^2}{\lambda_0+\lambda_1}\id
-B_{NE,X,-}\circ\frac{\mu_0\lambda_1}{\lambda_0+\lambda_1}
\id
+B_{E,X,+}\circ\frac{\lambda_1c}
{\lambda_0+\lambda_1}\id,\label{t1}\\
\Theta_2^*
&:=
-B_{E,X,-}\circ\frac{\mu_1c}
{\lambda_0+\lambda_1}\id
+B_{SE_+,X,+}\circ
\frac{\lambda_0\mu_1}{\lambda_0+\lambda_1}\id
-B_{SE_+,X,-}\circ\frac{\mu_1^2}{\lambda_0+\lambda_1}\id,
\label{t2}\\
\Theta_3^*
&=
B_{NE,W,+}\circ\frac{\lambda_1\lambda_2}{\lambda_0+\lambda_1}
\id
+B_{E,W,+}\circ\frac{\lambda_2c}
{\lambda_0+\lambda_1}\id
-B_{E,W,-}\circ\frac{\mu_0c}
{\lambda_0+\lambda_1}\id\nonumber\\
&\qquad
+B_{SE_+,W,+}\circ
\frac{\lambda_1\mu_1}{\lambda_0+\lambda_1}\id
-B_{SE_+,W,-}
\circ\frac{\mu_0\mu_1}{\lambda_0+\lambda_1}\id,
\label{t3}
\end{align}
\begin{align}
\Theta_6^*
&:=
-B_{SW,X,-}\circ\frac{\mu_1\mu_2}{\mu_0+\mu_1}\id
-B_{S,X,-}\circ\frac{\mu_2c}{\mu_0+\mu_1}\id
+B_{S,X,+}\circ\frac{\lambda_0c}
{\mu_0+\mu_1}\id\nonumber\\
&\qquad
-B_{SE_-,X,-}\circ
\frac{\mu_1\lambda_1}{\mu_0+\mu_1}\id
+B_{SE_-,X,+}\circ
\frac{\lambda_0\lambda_1}{\mu_0+\mu_1}\id\label{t6}\\
\Theta_7^*
&:=
B_{S,W,+}\circ\frac{\lambda_1c}
{\mu_0+\mu_1}\id
-B_{SE_-,W,-}\circ\frac{\mu_0\lambda_1}{\mu_0+\mu_1}\id
+B_{SE_-,W,+}\circ\frac{\lambda_1^2}{\mu_0+\mu_1}\id,
\label{t.7}\\
\Theta_8^*
&:=
-B_{SW,W,_-}\circ\frac{\mu_1^2}{\mu_0+\mu_1}\id
+B_{SW,W,+}\circ\frac{\lambda_0\mu_1}{\mu_0+\mu_1}\id
-B_{S,W,-}\circ\frac{\mu_1c}{\mu_0+\lambda_1}\id.
\label{t8}
\end{align}
\end{corollary}

\begin{proof}
Apply the random time change lemma
of \cite[Section~14]{Billingsley}
to the convergences in (\ref{4.41e}) and
Corollary \ref{C4.15}.
\end{proof}

\begin{remark}\label{R4.17}
{\rm The six Brownian motions in (\ref{t1})--(\ref{t8})
are independent.
For future reference, we compute their quadratic
variations, using Assumption
\ref{Assumption1} to simplify, obtaining
\begin{eqnarray*}
\langle\Theta_1^*,\Theta_1^*\rangle
=\frac{2\lambda_1}{b}\id,
&
\langle\Theta_2^*,\Theta_2^*\rangle
=\displaystyle\frac{2\mu_1}{a}\id,
&
\langle\Theta_3^*,\Theta_3^*\rangle
=\frac{2\lambda_0}{b}\id,\\
\langle\Theta_6^*,\Theta_6^*\rangle
=\frac{2\mu_0}{a}\id,
&
\langle\Theta_7^*,\Theta_7^*\rangle
=\displaystyle\frac{2\lambda_1}{b}\id,
&
\langle\Theta_8^*,\Theta_8^*\rangle
=\frac{2\mu_1}{a}\id.\\
\end{eqnarray*}
}
\end{remark}

\begin{proposition}\label{P5.4}
We have
\begin{align}
G^*
&
=(b\Theta_1^*+\Theta_2^*+\Theta_3^*)\circ P_+^{G^*}
+(\Theta_6^*+\Theta_7^*+a\Theta_8^*)\circ P_-^{G^*},
\label{4.116e}\\
\big|G^*\big|
&=
(b\Theta_1^*+\Theta_2^*
+\Theta_3^*)\circ P^{G^*}_+
-(\Theta_6^*+\Theta_7^*
+a\Theta_8^*)\circ P^{G^*}_-\nonumber\\
&\qquad
+\Gamma\big((b\Theta_1^*+\Theta_2^*
+\Theta_3^*)\circ P^{G^*}_+
-(\Theta_6^*+\Theta_7^*
+a\Theta_8^*)\circ P^{G^*}_-\big),\label{5.26}\\
\lefteqn{\sqrt{n}(a\lambda_0+b\mu_0)
\big(\Pbarn_{SE}+\Pbarn_O\big)}
\hspace{2cm}\nonumber\\
&\ArrowJ1
\Gamma\big((b\Theta_1^*+\Theta_2^*+\Theta_3^*)
\circ P^{G^*}_+
-(\Theta_6^*+\Theta_7^*+a\Theta_8^*)
\circ P^{G^*}_-\big),\label{4.93x}
\end{align}
where $\Gamma$ is the Skorohod map of
(\ref{2.22b}).
\end{proposition}

\begin{proof}
Equation (\ref{4.116e}) is the result of
taking the limit in (\ref{4.49c}),
using Corollaries \ref{C4.15} and \ref{C4.16}. 
Referring to (\ref{4.50c}), we observe that the
nondecreasing process 
$\sqrt{n}(a\lambda_0+b\mu_0)(\Pbarn_{SE}+\Pbarn_O)$
starts at zero and
is constant on intervals where $|\Ghatn|$ 
is strictly positive.
It follows that
\begin{align*}
\sqrt{n}(a\lambda_0+b\mu_0)(\Pbarn_{SE}+\Pbarn_O)
&=
\Gamma\big(\Ghatn(0)+b\Thetahatn_1+\Thetahatn_2
+\Thetahatn_3+\Thetahatn_4-\Thetahatn_5
-\Thetahatn_6-\Thetahatn_7
-a\Thetahatn_8\big).
\end{align*}
Using Assumption \ref{Assumption2},
the convergences in Corollaries \ref{C4.15}
and \ref{C4.16},
and the continuity of $\Gamma$
in the $J_1$ topology, we obtain 
(\ref{5.26}) and (\ref{4.93x}) from
the Continuous Mapping Theorem.
\end{proof}

\subsection{Convergence of $(\sWhatn,\sXhatn)$}
\label{ConvWX}

To be consistent with the notation of 
\cite[equation~(2.7)]{ALSY}, we state the final
result of this section denoting the two-speed
Brownian motion $G^*$ in (\ref{4.72}) by $B_{w,x}$,
i.e.,
$$
B_{w,x}=B\circ\big(\sigma_+^2 P_+^{B_{w,x}}
+\sigma_-^2 P^{B_{w,x}}_-\big),
$$
where $B$ is a standard Brownian motion and
$P_{\pm}^{B_{w,x}}(\theta)=\int_0^{\theta}
\ind_{\{\pm B_{w,x}(\tau)>0\}}d\tau$,
or equivalently,
$B_{w,,x}=B\circ \Theta^{-1}$,
where
$\Theta$ is given by (\ref{TS.E.3a}) and
$P_{\pm}^B(t)=\int_0^t\ind_{\{\pm B(s)>0\}}ds$.

\begin{theorem}\label{T4.16}
The limit of $(\sWhatn,\sXhatn)$
is a split Brownian motion:
$$
(\sWhatn,\sXhatn)\ArrowJ1 
\big(\max\{B_{w,x},0\},\min\{B_{w,x},0\}\big).
$$
\end{theorem}

\begin{proof}
The result follows from the convergence
$(\Ghatn,\Hhatn)\ArrowJ1 (B_{w,x},0)$,
the continuity of the transformation
(\ref{4.34d}), (\ref{4.35d}), 
and the Continuous Mapping Theorem.
\end{proof}

\begin{remark}\label{R4.20x}
{\rm To 
simplify the notation for later reference, we define
$$
\sW^*:=\max\{B_{w,x},0\}=\max\{G^*,0\},\quad
\sX^*:=\min\{B_{w,x},0\}=\min\{G^*,0\}.
$$
}
\end{remark}

\section{Bracketing queues}\label{BracketingQueues}

\setcounter{equation}{0}
\setcounter{theorem}{0}
\setcounter{figure}{0}

\subsection{Introduction}\label{BracketingIntro}

The processes $\sW^n$ and $\sX^n$
introduced at the end of Section \ref{Interior}
agree with the processes $W^n$ and $X^n$
in the limit-order book until the first
time one of the bracketing queues $V^n$
or $Y^n$ vanishes.  In this section
we study these bracketing
queues.  To postpone consideration
of the time $S^n$ when one of the bracketing
queues vanishes, we replace 
$(V^n,W^n,X^n,Y^n)$ in the analysis
by the quadruple of processes
$(\sV^n,\sW^n,\sX^n,\sY^n)$
whose dynamics are given by Figure 4.1,
regardless of the values of $\sV^n$
and $\sY^n$. We have already done this
for $(\sW^n,\sX^n)$, whose dynamics
we then described using thirty independent
Poisson processes in (\ref{3.10})--(\ref{3.12}).
For each of $\sV^n$ and $\sY^n$,
we need an additional nine independent
Poisson processes.  

\begin{remark}\label{R5.1x}
{\rm 
We have already defined $(\sWhatn,\sXhatn)$
by (\ref{DS}).
We apply the same diffusion scaling
to define
$\sVhatn(t):=\frac{1}{\sqrt{n}}\sV^n(nt)$
and $\sYhatn(t):=\frac{1}{\sqrt{n}}\sY^n(nt)$.
Our penultimate goal is to determine
the limit of the quadruple 
$(\sVhatn,\sWhatn,\sXhatn,\sYhatn)$,
the ultimate goal being the determination
of the limit of the diffusion scaled
version of $(V^n,W^n,X^n,Y^n)$
up to the stopping time $S^n$ of (\ref{3.1}).
For this we need to establish
{\em joint} convergence of $\sVhatn$,
$\sWhatn$, $\sXhatn$ and $\sYhatn$.
The convergence of $\sWhatn$ and $\sXhatn$
in Theorem \ref{T4.16} is joint, and
this convergence is joint with the other
processes appearing in Corollaries 
\ref{C4.15} and \ref{C4.16}.
The nine additional independent Poisson processes
introduced in (\ref{5.1}) below
to construct $\sVhatn$ are independent
of all the processes in Theorem \ref{T4.16}
and Corollaries \ref{C4.15} and \ref{C4.16}, as are the
the nine independent Poisson processes introduced
in (\ref{5.86}) to construct
$\sYhatn$.  Therefore, as we shall
see in the remainder of this section,
this permits us to establish joint
convergence of $\sVhatn$, $\sWhatn$, $\sXhatn$,
and $\sYhatn$.  See Remark \ref{R5.10x} in this regard.
}
\end{remark}

We provide the analysis for $\sVhat^n$
and state the analogous results for $\sYhatn$.
To set the goal for this analysis, 
we describe here the
limit of $\sVhatn$, which we denote $\sV^*$.
We shall discover that a strong
``mean reversion'' due to cancellation
results in
\be\label{5.1z}
\sV^*(t)=\kappa_L
:=\frac{\lambda_2\mu_1}{\theta_b\lambda_1}
\mbox{ if }G^*\mbox{ is on a positive excursion
at time }t.
\ee
On the other hand,
\begin{align}
\sVhat^*(t+\Lambda_{k,-})
&:=
\kappa_L+
C_{k,-}(t)+\alpha_-E_{k,-}\mbox{ if }G^*\mbox{ is on
its }k\mbox{-th negative excursion}\nonumber\\
&\qquad\mbox{with left endpoint }
\Lambda_{k,-}\mbox{ at time }\Lambda_{k,-}+t,\label{5.2z}
\end{align}
where $\{C_{k,-}\}_{k=1}^{\infty}$ 
is an independent
sequence of Brownian motions independent
of $G^*$, $E_{k,-}$ is the $k$-th negative
excursion of $G^*$, and $\alpha_-$
is defined by (\ref{5.44a}) below
To make this statement precise, we will
enumerate the negative
excursions of $G^*$.
Formulas (\ref{5.1z}) and
(\ref{5.2z}) show that $\sV^*$ will
have jumps, and indeed the convergence
of $\sVhatn$ to $\sV^*$ is in the weak-$M_1$
topology rather than the weak-$J_1$ topology
because the jumps in $\sV^*$ are not matched
by jumps in $\sV^n$.

\subsection{Governing equation for $\sV^n$}\label{Governing}

Let $N_{NE,V,-}$, $N_{SE_+,V,+}$,
$N_{SE,V,+}$, $N_{SE_-,V,+}$,
$N_{S,V,+}$, $N_{S,V,-}$, $N_{NE,V,+}$,
$N_{NE,V,-}$ and $N_{O,V,-}$
be processes independent of one another and
of the thirty Poisson processes introduced
in Section \ref{Interior}.
From Figure 4.1 we obtain the
formula (using the occupation time
processes defined by (\ref{3.10}))
\begin{align}
\sV^n(t)
&=
V^n(0)-N_{NE,V,-}\left(\int_0^t\frac{1}{\sqrt{n}}
\theta_b\big(\sV^n(s)\big)^+dP^n_{NE}(s)\right)
+N_{SE_+,V,+}\circ \lambda_2 P^n_{SE_+}(t)\nonumber\\
&\quad
+N_{SE,V,+}\circ \lambda_2P^n_{SE}(t)
+N_{SE_-,V,+}\circ \lambda_2P^n_{SE_-}(t)
+N_{S,V,+}\circ\lambda_2P^n_{S}(t)\nonumber\\
&\quad
-N_{S,V,-}\circ\mu_0P^n_S(t)
+N_{SW,V,+}\circ\lambda_1P^n_{SW}(t)
-N_{SW,V,-}\circ\mu_0P^n_{SW}(t)\nonumber\\
&\quad
-N_{O,V,-}\circ\mu_0P^n_O(t).\label{5.1}
\end{align}
The first term on the right-hand side of
(\ref{5.1}) accounts for the fact that the
rate of departures due to cancellations
is the cancellation rate $\theta_b/\sqrt{n}$
per order times the number of orders queued
at price $p_v$.
Again we scale and center these Poisson
processes, defining
$\Mhatn_{\times,V,\pm}(t):=\frac{1}{\sqrt{n}}
(N_{\times,V,\pm}(nt)-nt)$, 
and rewrite (\ref{5.1}) as
(using the fluid-scaled
occupation time processes defined by (\ref{3.14}))
\begin{align}
\sVhatn(t)
&=
\sVhatn(0)
-\Mhatn_{NE,V,-}\left(\int_0^t\theta_b
\big(\sVhatn(s)\big)^+d\Pbarn_{NE}(s)\right)
+\Mhatn_{SE_+,V,+}\circ\lambda_2\Pbarn_{SE_+}(t)
\nonumber\\
&\quad
+\Mhatn_{SE,V,+}\circ\lambda_2\Pbarn_{SE}(t)
+\Mhatn_{SE_-,V,+}\circ\lambda_2\Pbarn_{SE_-}(t)
+\Mhatn_{S,V,+}\circ\lambda_2\Pbarn_S(t)
\nonumber\\
&\quad
-\Mhatn_{S,V,-}\circ\mu_0\Pbarn_S(t)
+\Mhatn_{SW,V,+}\circ\lambda_1\Pbarn_{SW}(t)
-\Mhatn_{SW,V,-}\circ\mu_0\Pbarn_{SW}(t)
\nonumber\\
&\quad
-\Mhatn_{O,V,-}\circ\mu_0\Pbarn_O(t)
-\sqrt{n}\int_0^t\theta_b\big(\sVhatn(s)\big)^+
d\Pbarn_{NE}(s)
+\sqrt{n}\lambda_2\Pbarn_{SE_+}(t)\nonumber\\
&\quad
+\sqrt{n}\,\lambda_2\Pbarn_{SE}(t)
+\sqrt{n}\,\lambda_2\Pbarn_{SE_-}(t)
+\sqrt{n}(\lambda_2-\mu_0)\Pbarn_S(t)
\nonumber\\
&\quad
-\sqrt{n}\,c\Pbarn_{SW}(t)
-\sqrt{n}\,\mu_0\Pbarn_O(t).
\label{5.2}
\end{align}

\begin{remark}\label{R5.Brm}
{\rm
The scaled centered Poisson processes
$\Mhatn_{\times,V,\pm}$ appearing
in (\ref{5.2}) converge weakly
to Brownian motions $B_{\times,V,\pm}$.
This convergence is joint with the convergences
in (\ref{4.41e}), and all the resulting Brownian
motions are independent.
}
\end{remark}

\subsection{Diffusion-scaled occupation time
limits}\label{Occupation}

We obtained fluid-scaled occupation time
limits $\Pbarn_{\times}\ArrowJ1 \Pbar_{\times}$ 
in Proposition \ref{P4.10}.
The difficult part of showing that
$\sVhatn$ converges is to deal with
the occupation times
$\Pbarn_{\times}$ that are multiplied by
$\sqrt{n}$ in (\ref{5.2}).  For this
we need {\em diffusion-scaled occupation time}
limits, i.e., limits for the processes
$\Phatn_{\times}:=\sqrt{n}
\big(\Pbarn_{\times}-\Pbar_{\times}\big)$,
$\times\in\sR$.
We are unable to determine
the limits of all the $\Phatn_{\times}$ terms,
but we are able to determine the limits
of a set of
linear combinations of these terms
that is sufficient for our purposes.
In these linear combinations, the $\Pbar_{\times}$
terms cancel, and hence we can state the results
in terms limits of $\sqrt{n}$ times
linear combinations of $\Pbarn_{\times}$ terms

We begin
with a closer examination of the
process $\Hhatn$ of (\ref{3.18a}) and (\ref{4.34c})
and its absolute
value $|\Hhatn|$ given by (\ref{4.35c}).
These processes converge to zero (Theorem
\ref{T3.2}), and because they involve
$\sqrt{n}\,\Pbarn_{\times}$ terms, they
provide information about the limits
of these terms.  
In particular, summing (\ref{4.34c}) and (\ref{4.35c})
we obtain
\be\label{5.4}
\Phihatn_1-\Phihatn_8+\Phihatn_{10}
\ArrowJ1 0.
\ee 
Subtracting one of these equations from the
other, we also obtain
$\Phihatn_2-\Phihatn_7+\Phihatn_9
\ArrowJ1 0$.
In fact, we can separate (\ref{5.4}) into
two convergences.  Toward that end, we define
\begin{eqnarray*}
\Phi^n_{10X}:=N_{O,X,+}\circ
\lambda_1P^n_O,
&&
\Phi^n_{10W}:
=-N_{O,X,-}\circ\mu_1P^n_O,\\
\Phihatn_{10X}:=\frac{1}{\sqrt{n}}\Phi^n_{10X}(nt),
&&
\Phihatn_{10W}:=\frac{1}{\sqrt{n}}\Phi^n_{10W}(nt),
\end{eqnarray*}
so that $\Phihatn_{10}=\Phihatn_{10X}-\Phihatn_{10W}$
and 
\be\label{5.8a}
\Phihatn_1-\Phihatn_8+\Phihatn_{10}
=(\Phihatn_1+\Phihatn_{10X})
-(\Phihatn_8+\Phihatn_{10W}).
\ee

\begin{lemma}\label{L5.1}
We have\footnote{Recall that convergence to a non-random
process is joint with every other convergence.}
\be\label{5.6}
\Phihatn_1+\Phihatn_{10X}\ArrowJ1 0,\quad
\Phihatn_8+\Phihatn_{10W}\ArrowJ1 0.
\ee

\end{lemma}

\begin{proof}
We prove for every $T>0$ that
\be\label{5.10}
\max_{0\leq t\leq T}\big|\Phihatn_8(t)
+\Phihatn_{10W}\big|
\leq \big|\Hhatn(0)\big|+\max_{0\leq t\leq T}
\big|\Phihatn_1(t)-\Phihatn_8(t)+\Phihatn_{10}(t)\big|.
\ee
The right-hand
side of (\ref{5.10})
converges to zero in probability by
Theorem \ref{T3.2} and (\ref{5.4}).
This gives us the second part of (\ref{5.6}).
The first part follows from the decomposition
(\ref{5.8a}) and
another application of (\ref{5.4}).

To simplify notation, for this proof we set
\begin{align*}
\Xi_+^n
&:=
\Phi_1^n+\Phi_{10X}^n\\
&=
N_{NE,X,+}\circ\lambda_1P^n_{NE}
-N_{NE,X,-}\circ\mu_0P^n_{NE}
+N_{E,X,+}\circ\lambda_1P^n_E
+N_{O,X,+}\circ\lambda_1P^n_O,\\
\Xi_-^n
&:=
\Phi_8^n+\Phi_{10W}^n\\
&=
-N_{SW,W,-}\circ\mu_1P^n_{SW}
+N_{SW,W,+}\circ\lambda_0P^n_{SW}
-N_{S,W}\circ\mu_1 P^n_S
-N_{O,W,-}\circ\mu_1P^n_O,
\end{align*}
so
\be\label{5.8}
\Phihatn_1(t)+\Phihatn_{10X}=\frac{1}{\sqrt{n}}\Xi_+^n(nt),
\quad
\Phihatn_8(t)+\Phihatn_{10W}(t)
=\frac{1}{\sqrt{n}}\Xi_-^n(nt),\quad t\geq 0.
\ee
We first consider the
case that $(G^n(0),H^n(0))\notin NE$.
We set $\tau_0=0$ and define recursively
$$
\sigma_i:=\min\big\{t\geq \tau_{i-1}:
G^n(t)>0\big\},\quad
\tau_i:=\min\big\{t\geq\sigma_i:
G^n(t)=0\big\},\quad
i=1,2,\dots.
$$
On the intervals $[\sigma_i,\tau_i)$,
$G^n$ is on a positive
excursion, so $(G^n,H^n)$
is in $NE\cup E\cup SE_+$.
For $i\geq 2$, $G^n(\sigma_i-)=0$,
so $(G^n(\sigma_i-),H^n(\sigma_i-))\notin NE$.
Because $(G^n(0),H^n(0))\notin NE$,
we have $(G^n(\sigma_1-),H^n(\sigma_1-))
\notin NE$ as well, where we adopt
the convention that the value of a process
at time $0-$ is its value at time $0$.
Because $\Xi_+^n$ is constant until
$(G^n,H^n)$ jumps into $NE$, we have
\be\label{5.11a}
\Xi^n_+(\sigma_1-)=0.
\ee

On each of the intervals $[\sigma_i,\tau_i]$,
$(G^n,H^n)$ may have excursions
into and out of $NE$.  It must enter 
$NE$ from $E\cup O$ and exit by returning to
$E\cup O$.  Between these entrance and
exit times, including at the entrance and
exit times, $\Phi^n_2$,
$\Phi^n_7$ $\Phi^n_8$, $\Phi^n_9$
and $N_{O,W,-}\circ\mu_1P^n_O$
(part of $\Phi^n_{10}$) appearing
in (\ref{4.18c})
are constant, so the jumps in 
$\Xi^n_+=\Phi_1^n+N_{O,X,+}\circ\lambda_1P^n_O$
match the jumps in $H^n$.  Since
$H^n$ is constant (zero) on $E\cup O$, the
positive jumps of $\Xi^n_+$ 
in $[\sigma_i,\tau_i)$ cancel the negative jumps,
which implies
\be\label{5.11}
\Xi^n_+(\sigma_i-)=\Xi^n_+(\tau_i),\quad
i=1,2,\dots.
\ee
Observe that $\Xi^n_+$ is constant on each interval
$[\tau_{i-1},\sigma_i)$, $i=1,2,\dots$,
and hence
\be\label{5.12}
\Xi^n_+(\tau_{i-1})=\Xi^n_+(\sigma_i-),\quad
i=1,2\dots.
\ee
From (\ref{5.11a})--(\ref{5.12}) we have
$$
\Xi_+^n(\sigma_i-)=\Xi_+^n(\tau_i)=0,\quad
i=1,2,\dots.
$$

On the other hand, the constancy of
$P^n_{SW}$, $P^n_S$ and
$N_{O,W,-}\circ\mu_1 P_O^n$ on $[\sigma_i,\tau_i]$
and the fact that these processes do not
jump at time $\sigma_i$ implies that
$\Xi^n_-$ is constant on $[\sigma_i,\tau_i]$.
In particular,
\be\label{5.13}
\Xi^n_-(\sigma_i-)=\Xi^n_-(\sigma_i)
=\Xi^n_-(\tau_i),\quad
i=1,2,\dots.
\ee

Now consider the process
$$
\Xi_0^n(t):=\left\{\begin{array}{ll}
\Xi^n_+(\sigma_i-)+\Xi^n_-(\sigma_i-)
&\mbox{ if }t\in [\sigma_i,\tau_i)\mbox{ for some }
i\geq 1,\\
\Xi^n_+(t)+\Xi^n_-(t)&\mbox{ if }
t\in[\tau_{i-1},\sigma_i)\mbox{ for some }
i\geq 1,
\end{array}\right.
$$
which is obviously continuous at each $\sigma_i$.
From (\ref{5.11}) and (\ref{5.13}),
we see that $\Xi_0^n$ is also continuous 
at each $\tau_i$.
The constancy (at the value zero) of $\Xi^n_+$ on
the intervals $[\tau_{i-1},\sigma_i)$
and the constancy of $\Xi_-^n$
(at the value $\Xi^n_-(\sigma_i-)$)
on the intervals $[\sigma_i,\tau_i)$
implies $\Xi^n_0=\Xi^n_-$.
We conclude that for each $T>0$,
\be\label{5.16}
\max_{0\leq t\leq nT}
\big|\Xi^n_-(t)\big|
=\max_{0\leq t\leq nT}
\big|\Xi^n_0(t)\big|
\leq
\max_{0\leq t\leq nT}
\big|\Xi^n_+(t)+\Xi^n_-(t)\big|,
\ee
which, according to (\ref{5.8}), implies (\ref{5.10}).

Finally, we must consider the case
$(G^n(0),H^n(0))\in NE$.
If $(G^n,H^n)$ never exits $NE$,
then $\Xi^n_-$ is identically zero
and (\ref{5.10}) holds trivially.
If $(G^n(0),H^n(0))$ exits
$NE$, then at the time of this exit,
$\Xi_+^n$ is equal to $-H^n(0)$,
a negative quantity.  From this time forward
we can apply the preceding argument,
but because of this reduction in value
of $\Xi_+^n$, we must add
$|H^n(0)|$ to the right-hand side
of (\ref{5.16}),
thereby obtaining (\ref{5.10}) after scaling.
\end{proof}

Recall the thirty scaled centered Poisson process
$\Mhatn_{\times,*,\pm}$ defined by (\ref{3.13}),
a definition we extend to include
the nine additional Poisson processes
appearing in (\ref{5.1}).  The vector of 
thirty-nine
independent scaled centered Poisson processes
converges weakly-$J_1$ to a vector whose
thirty-nine independent components are
standard Brownian motions, which
we denote $B_{\times,*,\pm}$.
Also observe from (\ref{4.39c}) 
and (\ref{4.46c}) that
\begin{align}
\Phihatn_1+\Phihatn_{10X}
&=
\Thetahatn_1+\Mhatn_{O,X,+}\circ\lambda_1\Pbarn_O
+\sqrt{n}\big(-c\Pbarn_{NE}+\lambda_1\Pbarn_E
+\lambda_1\Pbarn_O\big),\label{5.17a}\\
\Phihatn_8+\Phihatn_{10W}
&=
\Thetahatn_8-\Mhatn_{O,X,-}\circ\mu_1\Pbarn_O
+\sqrt{n}\big(c\Pbarn_{SW}-\mu_1\Pbarn_S
-\mu_1\Pbarn_O).\label{5.18a}
\end{align}

\begin{lemma}\label{L5.2}
Jointly with the convergences in 
Theorem \ref{T4.16} and Corollaries
\ref{C4.15} and \ref{C4.16}, we
have the joint convergences
\begin{align}
\sqrt{n}\big(\lambda_1\Pbarn_E-c\Pbarn_{NE}+\lambda_1\Pbarn_O\big)
&\ArrowJ1
-\Theta_1^*\circ P^{G^*}_+,\label{5.20}\\
\sqrt{n}\big(\mu_1\Pbarn_S-c\Pbarn_{SW}+\mu_1\Pbarn_O\big)
&\ArrowJ1
\Theta^*_8\circ P^{G^*}_-.\label{5.21}
\end{align}
\end{lemma}

\begin{proof}
Equations (\ref{5.20}) and (\ref{5.21}) are a consequence of
equations (\ref{5.17a}) and (\ref{5.18a}),
Corollary \ref{C4.16}, and Lemma \ref{L5.1}.
\end{proof}

\begin{lemma}\label{L5.3}
Jointly with the convergences in Theorem \ref{T4.16},
Corollaries \ref{C4.15} and \ref{C4.16},
and Lemma \ref{L5.2}, we have the convergence
\be\label{5.22}
\Pi^n\ArrowJ1 
(\Theta_1^*+b\Theta_2^*)\circ P^{G^*}_+
-(b\Theta_7^*\circ
+\Theta_8^*)\circ P^{G^*}_-,
\ee
where
\begin{align*}
\Pi^n
&:=
\sqrt{n}\big[c\Pbarn_{NE}+(b\mu_1-\lambda_1)\Pbarn_E
-\lambda_2\Pbarn_{SE_+}-b(\lambda_0+\mu_0)\Pbarn_{SE}
-\lambda_2\Pbarn_{SE_-}\nonumber\\
&\qquad\qquad
+(\mu_0-\lambda_2)\Pbarn_{S}
+c\Pbarn_{SW}-(\lambda_1+\mu_1)\Pbarn_O\big].
\end{align*}
\end{lemma}

\begin{proof}
We pass to the limit in (\ref{4.51c}) and
(\ref{4.52c}), using Theorem \ref{T3.2},
Proposition \ref{P4.10}
and Corollary \ref{C4.16}, to obtain
\begin{align}
\lefteqn{\sqrt{n}\big[c\big(-\Pbarn_{NE}+\Pbarn_{SE_+}+\Pbarn_{SE_-}
-\Pbarn_{SW}
\big)+(\mu_1-\lambda_1)\big(\Pbarn_S-\Pbarn_E)
+(\lambda_0+\mu_0)\Pbarn_{SE}}\hspace{5.5cm}\nonumber\\
+(\lambda_1+\mu_1)\Pbarn_O\big]
&\ArrowJ1
(-\Theta_1^*-\Theta_2^*)\circ P^{G^*}_+
+(\Theta_7^*
+\Theta_8^*)\circ P^{G^*}_-,\label{5.23}\\
\lefteqn{\sqrt{n}\big[-c\big(\Pbarn_{NE}+\Pbarn_{SE_+}+\Pbarn_{SE_-}
+\Pbarn_{SW}\big)
+(\lambda_1+\mu_1)\big(\Pbarn_E+\Pbarn_S
+\Pbarn_O\big)
-(\lambda_0+\mu_0)\Pbarn_{SE}\big]}\hspace{5.5cm}\nonumber\\
&\ArrowJ1
(-\Theta_1^*+\Theta_2^*)\circ P^{G^*}_+
+(-\Theta_7^*
+\Theta_8^*)\circ P^{G^*}_-.\label{5.24}
\end{align}
We multiply (\ref{5.23}) by
$-\frac12(b+1)$, multiply (\ref{5.24}) by
$\frac12(b-1)$, and sum the resulting equations,
using Assumption \ref{Assumption1} to simplify,
to obtain (\ref{5.22}).
\end{proof}

\subsection{Boundedness in probability of $\sVhatn$}\label{Bdd}

\begin{theorem}\label{T5.5}
The sequence of c\`adl\`ag processes
$\{\sVhatn\}_{n=1}^{\infty}$
is bounded in probability on compact time
intervals (Definition \ref{D3.1a}).
\end{theorem}

The theorem follows from Lemmas \ref{L5.6}
and \ref{L5.7} below.

\begin{lemma}\label{L5.6}
The sequence of c\`adl\`ag processes
$\{\sVhatn\}_{n=1}^{\infty}$
is bounded above in probability on compact time
intervals.
\end{lemma}

\begin{proof}
We write (\ref{5.2}) as
\be\label{5.27}
\sVhatn(t)
=\sVhatn(0)+Z_1^n(t)+Z_2^n(t)+Z_3^n(t)+Z_4^n(t),
\ee
where
\begin{align*}
Z_1^n(t)
&:=
-\Mhatn_{NE,V,-}\left(\int_0^t\theta_b\big(\sVhatn(s)\big)^+
d\Pbarn_{NE}(s)\right),\nonumber\\
Z_2^n(t)
&:=
\Mhatn_{SE_+,V,+}\circ\lambda_2\Pbarn_{SE_+}(t)
+\Mhatn_{SE,V,+}\circ\lambda_2\Pbarn_{SE}(t)
+\Mhatn_{SE_-,V,+}\circ\lambda_2\Pbarn_{SE_-}(t)\nonumber\\
&\qquad
+\Mhatn_{S,V,+}\circ\lambda_2\Pbarn_S(t)
-\Mhatn_{S,V,-}\circ\mu_0\Pbarn_S(t)
+\Mhatn_{SW,V,+}\circ\lambda_1\Pbarn_{SW}(t)\nonumber\\
&\qquad
-\Mhatn_{SW,V,-}\circ\mu_0\Pbarn_{SW}(t)
-\Mhatn_{O,V,-}\circ\mu_0\Pbarn_O(t),\nonumber\\
Z_3^n(t)
&:=
-\sqrt{n}\int_0^t\theta_b\big(\sVhatn(s)\big)^+d\Pbarn_{NE}(s),
\nonumber\\
Z_4^n(t)
&:=
\sqrt{n}\,\lambda_2\Pbarn_{SE_+}(t)
+\sqrt{n}\,\lambda_2\Pbarn_{SE}(t)
+\sqrt{n}\,\lambda_2\Pbarn_{SE_-}(t)
+\sqrt{n}(\lambda_2-\mu_0)\Pbarn_S(t)\nonumber\\
&\qquad
-\sqrt{n}\,c\Pbarn_{SW}(t)
-\sqrt{n}\,\mu_0\Pbarn_O(t).
\end{align*}

By Assumption \ref{Assumption2},  $\sVhatn(0)=O(1)$.
The convergences (\ref{4.41e}), (\ref{4.56}) 
and (\ref{4.57}) and the random time change
lemma of Section 14 of \cite{Billingsley}
imply that $Z_2^n$ has 
a continuous limit,
and hence $Z_2^n=O(1)$.  We need to deal only with
$Z_1^n$, $Z_3^n$ and $Z_4^n$.

We rewrite $Z_4^n$ as
\be\label{5.28}
Z_4^n=\sqrt{n}\big[c\Pbarn_{NE}
+(b\mu_1-\lambda_1)\Pbarn_E\big]+Z_5^n,
\ee
where
\begin{align}
Z_5^n
&:=
-\sqrt{n}\big(b(\lambda_0+\mu_0)-\lambda_2\big)\Pbarn_{SE}
-\sqrt{n}(\mu_1+\lambda_1+\mu_0)\Pbarn_O\nonumber\\
&\qquad
-\sqrt{n}\big[c\Pbarn_{NE}+(b\mu_1-\lambda_1)\Pbarn_E
-\lambda_2\Pbarn_{SE_+}-b(\lambda_0+\mu_0)\Pbarn_{SE}
-\lambda_2\Pbarn_{SE_-}\nonumber\\
&\qquad\qquad+(\mu_0-\lambda_2)\Pbarn_S
+c\Pbarn_{SW}-(\lambda_1+\mu_1)\Pbarn_O\big],\label{5.29}
\end{align}
Proposition \ref{P5.4} implies that
$\sqrt{n}[\Pbarn_{SE}+\Pbarn_O]$ has a continuous limit,
and since both $\sqrt{n}\,\Pbarn_{SE}$ and $\sqrt{n}\,\Pbarn_O$
are nondecreasing, they are bounded in
probability on compact time intervals.
According to Lemma \ref{L5.3}, the last term in (\ref{5.29})
converges to a continuous process.  Therefore, $Z_5^n=O(1)$.

The boundedness of
$\sqrt{n}\,\Pbarn_O$ and (\ref{5.20}) imply that
$\sqrt{n}(\lambda_1\Pbarn_E-c\Pbarn_{NE})=O(1)$.
Thus,
\begin{align}
\sqrt{n}\big[c\Pbarn_{NE}
+(b\mu_1-\lambda_1)\Pbarn_E]
&=
\frac{cb\mu_1}{\lambda_1}\sqrt{n}\,\Pbarn_{NE}
+\frac{b\mu_1-\lambda_1}{\lambda_1}
\sqrt{n}\big(-c\Pbarn_{NE}
+\lambda_1\overline{P}^{n}_E\big)\nonumber\\
&=
\frac{cb\mu_1}{\lambda_1}\sqrt{n}\,\Pbarn_{NE}
+O(1).\label{5.30}
\end{align}
Substituting this into (\ref{5.28})
and deriving $cb=\lambda_2$
from Assumption \ref{Assumption1}, we obtain
\be\label{5.31}
Z_4^n=\frac{\lambda_2\mu_1}{\lambda_1}
\sqrt{n}\,\Pbarn_{NE}+O(1).
\ee

To conclude we use half of $Z_3^n$
to compensate the  leading term 
on the right-hand side of (\ref{5.31})
and the other half of $Z_3^n$ to control $Z_1^n$.
For this argument, we observe first that a unit intensity
Poisson process $N$ satisfies 
$\lim_{t\rightarrow\infty}N(t)/t=1$
almost surely (\cite{KaratzasShreve}, Remark 3.10, p.\ 15),
and hence $-N+\id/2$ is bounded above by
a finite random variable.  This implies that
\be\label{5.32}
Z_1^n+Z_3^n
=\frac{1}{\sqrt{n}}\left[-N_{NE,V,-}\circ(-\sqrt{n}\,Z_3^n)
-\frac12\sqrt{n}\,Z_3^n\right]+\frac12Z_3^n
\leq \frac12 Z_3^n+\frac{1}{\sqrt{n}}O(1).
\ee
It follows from (\ref{5.27}), (\ref{5.31}),
(\ref{5.32}), and the convergences of $\sVhatn(0)$
and $Z_2^n$ that
\begin{align}
\sVhatn(t)
&\leq
\frac12 Z_3^n(t)
+\frac{\lambda_2\mu_1}{\lambda_1}\sqrt{n}\,\Pbarn_{NE}(t)
+O(1)\nonumber\\
&=
\sqrt{n}\int_0^t\left(\frac{\lambda_2\mu_1}{\lambda_1}
-\frac12\theta_b\big(\sVhatn(s)\big)^+\right)
d\Pbarn_{NE}(s)+O(1).\label{5.33}
\end{align}

We show that $\sVhatn$
is bounded in probability on compact time
intervals.  For $0\leq t\leq T$,
\be\label{5.34}
\int_0^t\left(\frac{\lambda_2\mu_1}{\lambda_1}
-\frac12\theta_b\big(\sVhatn(s)\big)^+\right)
d\Pbarn_{NE}(s)\leq 0
\ee
or else
\be\label{5.35}
\int_0^t\theta_b\big(\sVhatn(s)\big)^+d\Pbarn_{NE}(s)
\leq \frac{2\lambda_2\mu_1}{\lambda_1}\Pbarn_{NE}(t)
\leq \frac{2\lambda_2\mu_1}{\lambda_1}T.
\ee
We define
$$
\tau^n(t):=\left\{\begin{array}{ll}
t&\mbox{if }(\ref{5.34})\mbox{ holds},\\
\sup\left\{s\in[0,t]:\theta_b\big(\sVhatn(s)\big)^+\leq
2\lambda_2\mu_1/\lambda_1\right\}&\mbox{if }(\ref{5.35})
\mbox{ holds}.
\end{array}\right.
$$
Note that under condition (\ref{5.35}),
$\{s\in[0,t]: \theta_b(\sVhatn(s))^+\leq 
2\lambda_2\mu_1/\lambda_1\}\neq\emptyset$.
If (\ref{5.34}) holds, then $\sVhatn(t)$ is bounded
by the $O(1)$ term in (\ref{5.33}).  On the other hand,
if (\ref{5.35}) holds, then 
\be\label{5.36}
\sVhatn(t)\leq\sVhatn\big(\tau^n(t)\big)
+\sum_{i=1}^4\big[Z_i^n(t)-Z_i^n\big(\tau^n(t)\big)\big].
\ee
We consider the right-hand side of (\ref{5.36}).
Since the jumps in $\sVhatn$ are of size $1/\sqrt{n}$,
we have
$$
\sVhatn\big(\tau^n(t)\big)
=\frac{2\lambda_2\mu_1}{\theta_b\lambda_1}
+\frac{1}{\sqrt{n}}=O(1).
$$
Because of the bound (\ref{5.35}) on the argument
of $\Mhatn_{NE,V,-}$ appearing in the formula for $Z_1^n$,
both $Z_1^n(t)$ and $Z_1^n(\tau^n(t))$ are $O(1)$.
We observed earlier that $Z_2^n$ is $O(1)$.  
It now follows from (\ref{5.31}) that
\begin{align*}
\sVhatn(t)
&\leq
Z_3^n(t)-Z_3^n\big(\tau^n(t)\big)
+Z_4^n(t)-Z_4^n\big(\tau^n(t)\big)+O(1)\\
&=
\sqrt{n}\int_{\tau^n(t)}^t
\left(\frac{\lambda_2\mu_1}{\lambda_1}
-\theta_b\big(\sVhatn(s)\big)^+\right)
d\Pbarn_{NE}(s)+O(1)\\
&\leq
-\frac{\sqrt{n}\,\lambda_2\mu_1}{\lambda_1}
\big(\Pbarn_{NE}(t)-\Pbarn_{NE}\big(\tau^n(t)\big)\big)+O(1),
\end{align*}
with the last inequality following from that fact that
$\theta_b(\sVhatn(s)\big)^+\geq 2\lambda_2\mu_1/\lambda_1$
for $s\in[\tau^n(t),t]$.
Again we have an upper bound on $\sVhatn(t)$.
The lemma is proved.
\end{proof}

\begin{lemma}\label{L5.7}
The sequence of c\`adl\`ag processes
$\{\sVhatn\}_{n=1}^{\infty}$
is bounded below in probability on compact time
intervals
\end{lemma}

\begin{proof}
We return to (\ref{5.2}) and note that because $\sVhatn$
is bounded above in probability on compact time intervals
and $d\Pbarn_{NE}\leq dt$, the sequence of processes
$\{\int_0^{\cdot}\theta_b(\sVhatn)^+d\Pbarn_{NE}\}_{n=1}^{\infty}$
is bounded in probability on compact time intervals.
Therefore $\Mhatn_{NE,V,-}\circ\int_0^{\cdot}\theta_b(\sVhatn)^+d\Pbarn_{NE}=O(1)$.
In addition, the other processes in (\ref{5.2}) involving
scaled centered Poisson processes are $O(1)$.  This 
and (\ref{5.31}) permit
us to write
\begin{align}
\sVhatn
&=
\Vhatn(0)
-\sqrt{n}\int_0^{\cdot}\theta_b(\sVhatn)^+d\Pbarn_{NE}
+Z_4^n\nonumber\\
&=
-\sqrt{n}\int_0^{\cdot}\theta_b(\sVhatn)^+\,d\Pbarn_{NE}
+\frac{\lambda_2\mu_1}{\lambda_1}\sqrt{n}\,\Pbarn_{NE}
+O(1).\label{5.37}
\end{align}
Let $t\geq 0$ be given.  Defining
$$
\rho^n(t):=0\vee\sup\big\{s\in[0,t]:\sVhatn(s)\geq 0\big\},
$$
we have
$\sVhatn\big(\rho^n(t)\big)
\geq\min\{\Vhatn(0),-1/\sqrt{n}\}$.
Because
$\sVhatn(s)<0$ for $\rho^n(t)\leq s< t$, (\ref{5.37}) implies
$$
\sVhatn(t)=\sVhatn\big(\rho^n(t)\big)
+\frac{\lambda_2\mu_1}{\lambda_1}\sqrt{n}\big[
\Pbarn_{NE}(t)-\Pbarn_{NE}\big(\rho^n(t)\big)\big]+O(1)
\geq O(1).
$$
This completes the proof.
\end{proof}

\begin{lemma}\label{L5.8}
For $n=1,2,\dots,$
\be\label{5.35y}
\sVhatn(t)=\sVhatn(0)+\sqrt{n}\,\theta_b
\int_0^t\left(\kappa_L-\big(\sVhatn(s)\big)^+
\right)d\Pbarn_{NE}(s)
+C_V^n(t),
\ee
where $\kappa_L$ is defined by (\ref{kappa}) and
$\{C_V^n\}_{n=1}^{\infty}$ is bounded in probability
on compact time intervals and has the property
that every subsequence has a sub-subsequence converging
weakly-$J_1$ to a continuous limit.
\end{lemma}

\begin{proof}
According to Proposition \ref{P5.4},
$\sqrt{n}[\Pbarn_{SE}+\Pbarn_O]$
has a continuous limit.  Because both 
$\sqrt{n}\,\Pbarn_{SE}$
and $\sqrt{n}\,\Pbarn_O$ are nondecreasing, 
each has a modulus
of continuity bounded by the modulus of continuity
of $\sqrt{n}[\Pbarn_{SE}+\Pbarn_O]$, and the 
latter converges
to zero as $n\rightarrow\infty$.  These processes
also have zero initial condition.  This implies 
tightness
in $C[0,\infty)$ of 
$\{\sqrt{n}\,\Pbarn_{SE}\}_{n=1}^{\infty}$
and of $\{\sqrt{n}\,\Pbarn_O\}_{n=1}^{\infty}$.
We observe from (\ref{5.20}) and the tightness 
in $C[0,\infty)$ of
$\sqrt{n}\{\Pbarn_O\}_{n=1}^{\infty}$ that 
$\{\sqrt{n}\big(\lambda_1\Pbarn_E
-c\Pbarn_{NE}\big)\}_{n=1}^{\infty}$
has the sub-subsequence property specified 
in the statement
of the lemma. 
The sequence $\{Z_5^n\}_{n=1}^{\infty}$
defined by (\ref{5.29}) also has this 
property because
$\{\sqrt{n}\,\Pbarn_{SE}\}_{n=1}^{\infty}$ and
$\{\sqrt{n}\,\Pbarn_O\}_{n=1}^{\infty}$ 
have the property,
and by Lemma \ref{L5.3}, 
the last term on the right-hand side
of (\ref{5.29}) has a continuous limit.  We define
$$
C_V^n:=\sqrt{n}\,\frac{b\mu_1-\lambda_1}{\lambda_1}
\big(\lambda_1\Pbarn_E-c\Pbarn_{NE}\big)
+Z_5^n.
$$
Then $Z_4^n$ in (\ref{5.28}) 
can be written as (see the first equation
in (\ref{5.30}))
$Z_4^n=\frac{\lambda_2\mu_1}{\lambda_1}\sqrt{n}\,\Pbarn_{NE}
+C_V^n$.
The lemma follows from the first equality in (\ref{5.37}).
\end{proof}

\begin{remark}\label{R5.10y}
{\rm 
In Section \ref{SubsecT6.2}
we will need the following
observations.  First of all, because
$\{\sVhatn\}_{n=1}^{\infty}$
is bounded in probability,
the $\widehat{M}^n_{NE,V,-}$ term
in (\ref{5.2}) is $O(1)$, as are
the other $\widehat{M}^n_{\times,V,\pm}$ terms.
In fact, these terms have the sub-subsequence
property specified in Lemma \ref{L5.8}.
Subtracting (\ref{5.2})
from (\ref{5.35y}), we obtain
$$
\sqrt{n}\,\theta_b\kappa_L\Pbarn_{NE}
=\sqrt{n}\big(\lambda_2\Pbarn_{SE_+}
+\lambda_2\Pbarn_{SE}+\lambda_2\Pbarn_{SE_-}
+(\lambda_2-\mu_0)\Pbarn_S
+c\Pbarn_{SW}+\mu_0\Pbarn_O\big)+C_0^n,
$$
where $C_0^n$ has the sub-subsequence property.
We saw in the proof of Lemma \ref{L5.8}
that $\sqrt{n}\,\Pbarn_{SE}$ and $\sqrt{n}\,\Pbarn_O$
have the sub-subsequence property,
so this can be simplified to
$$
\sqrt{n}\,\Pbarn_{NE}
=\sqrt{n}\left(\frac{\lambda_1}{\mu_1}\Pbarn_{SE_+}
+\frac{\lambda_1}{\mu_1}\Pbarn_{SE_-}
+\frac{\lambda_2-\mu_0}{\lambda_2\mu_1}\Pbarn_S
+\frac{\lambda_1c}{\lambda_2\mu_1}\Pbarn_{SW}\right)
+C_0^n,
$$
where $C_0^n$ is a different
process from the earlier one
but still has the sub-subsequence property.
The symmetry in Figure 4.2 permits us
to write the analogous equality
obtained by replacing $\Pbarn_{NE}$,
$\Pbarn_{SE_+}$,
$\Pbarn_{SE_-}$, $\Pbarn_S$ and $\Pbarn_{SW}$
by $\Pbarn_{SW}$, $\Pbarn_{SE_-}$, $\Pbarn_{SE_+}$,
$\Pbarn_E$ and $\Pbarn_{NE}$ respectively
and swapping $\lambda_i$ with $\mu_i$, $i=1,2,3$.
This equality is
\be\label{5.37y}
\sqrt{n}\,\Pbarn_{SW}
=\sqrt{n}\left(\frac{\mu_1}{\lambda_1}\Pbarn_{SE_-}
+\frac{\mu_1}{\lambda_1}\Pbarn_{SE_+}
+\frac{\mu_2-\lambda_0}{\mu_2\lambda_1}\Pbarn_E
+\frac{\mu_1c}{\mu_2\lambda_1}\Pbarn_{NE}\right)+
C_0^n,
\ee
where $C_0^n$ has the sub-subsequence
property specified in Lemma \ref{L5.8}. 
}
\end{remark}

\subsection{$\sVhatn$ on negative excursions of $\Ghatn$}
\label{NegExc}

\begin{remark}\label{R5.10x}
{\rm 
In Sections \ref{NegExc}--\ref{ConvVWXY}
we assume that
the Skorohod Representation Theorem has
been used to place all processes on a common
probability space so that the joint weak-$J_1$
convergence in Theorem \ref{T4.16},
Corollaries \ref{C4.15} and \ref{C4.16},
and Lemmas \ref{L5.1}, \ref{L5.2} and \ref{L5.3}
becomes convergence almost surely.
In particular, convergence of a process
to a continuous limit is convergence
uniformly on compact time intervals almost
surely.  We use this device of choosing
a convenient probability space
to show that in addition to the assumed
almost sure convergence of the processes
mentioned above, $\sVhatn$ and $\sYhatn$
converge almost surely in the $M_1$
topology on $D[0-,\infty)$
to identifiable limits $V^*$ and $Y^*$.
The result is the joint convergence
of $(\sVhatn,\sWhatn,\sXhatn,\sYhatn)$
established in Theorem \ref{T.VWXY} below.
}
\end{remark}

In this section we identify the limit
of $\sVhatn$ when $\Ghatn$ is on a negative
excursion.  We begin by establishing
convergence of negative excursion intervals of $\Ghatn$.

Recall from Theorem \ref{T4.14}
that $G^*$ is a two-speed Brownian
motion with zero initial condition
(see Assumption \ref{Assumption2}). 
Almost surely, there is no excursion
away from zero of $G^*$ that begins
at time zero.
Let $\varepsilon>0$ and a positive
integer $k$ be given, and consider the $k$-th
negative excursion
of $G^*$ whose length exceeds $\varepsilon$.  
This excursion has a left endpoint
$\Lambda$ and a right endpoint $R$.
The excursion itself is
$$
E(t)=G^*\big((t+\Lambda)\wedge R\big),
\quad t\geq 0.
$$
In particular, $E(0)=0$, $E(t)<0$
for $0<t<R-\Lambda$, and $E(t)=0$ for 
$t\geq R-\Lambda>\varepsilon$.
Consider also the $k$-th negative 
excursion
of $\Ghatn$ whose length exceeds $\varepsilon$
and that does not begin at time zero.
Denote its left endpoint $\Lambda^n$
and its right endpoint $R^n$.  The excursion
itself is
$$
E^n(t)=\Ghatn\big((t+\Lambda^n)\wedge R^n\big),
\quad t\geq 0.
$$
In particular, $\Ghatn(\Lambda^n-)=0$,
$E^n(t)<0$ for $0\leq t<R^n-\Lambda^n$,
and $E^n(t)=0$ for $t>R^n-\Lambda^n>\varepsilon$.

\begin{lemma}\label{L5.9}
We have
$\Lambda^n\rightarrow\Lambda$,
$R^n\rightarrow R$,
and $\max_{t\geq 0}\big|E^n(t)-E(t)\big|\rightarrow 0$
almost surely.
\end{lemma}

\begin{proof}
Given a path $g$ of the two-speed Brownian
motion $G^*$
and a sequence of paths $\{g^n\}_{n=1}^{\infty}$
converging to $g$ uniformly on compact time
intervals, we let $\ell_i$ and $r_i$, $i=1,2,\dots,k$,
denote in increasing
order the respective left and right endpoints
of the first $k$ negative excursions of $g$
whose length exceeds $\varepsilon$,
and we let $\ell_i^n$ and $r_i^n$ denote
in increasing order the respective left and right endpoints
of the first $k$ negative of excursions of $g^n$
whose length exceeds $\varepsilon$ and that
do not begin at time zero.  Being the path
of a two-speed Brownian motion, $g$ crosses
zero at $\ell_i$ and $r_i$.  Because of
the uniform convergence on compact
time intervals of $g_n$ to $g$,
we can find $\widehat{\ell}_i^n$ and $\widehat{r}_i^n$
such that $\widehat{\ell}_i^n\rightarrow\ell_i$
and $\widehat{r}_i^n\rightarrow r_i$ and such
that $g_n(t)<0$ for 
$t\in(\widehat{\ell}_i^n,\widehat{r}_i^n)$
and $g_n(\widehat{\ell}_i^n)=g_n(\widehat{r}_i^n)=0$
for $i=1,\dots,k$.
Therefore, for sufficiently large $n$,
$\widehat{\ell}_k^n$ and $\widehat{r}_k^n$
are the respective left and right endpoints
of at least the $k$-th negative excursion
of $g_n$ whose length exceeds $\varepsilon$
and that does not begin at time zero, which
implies that $\ell_k^n\leq\widehat{\ell}_k^n$
and $r_k^n\leq \widehat{r}_k^n$
for $n$ sufficiently large. 
We conclude that
\be\label{5.39}
\limsup_{n\rightarrow\infty}\ell_k^n
\leq\lim_{n\rightarrow\infty}
\widehat{\ell}_k^n=\ell_k,\quad
\limsup_{n\rightarrow\infty}r_k^n
\leq\lim_{n\rightarrow\infty}
\widehat{r}_k^n=r_k.
\ee
In addition, for 
$N$ sufficiently large, the sequence
$\{(\ell_i^n,r_i^n)_{i=1,2,\dots,k}\}_{n=N}^{\infty}$
is bounded.

We may now
select a subsequence of
$\{(\ell_i^n,r_i^n)_{i=1,2,\dots,k}\}_{n=N}^{\infty}$
that converges to a limit
$(\widetilde{\ell}_i, \widetilde{r}_i)_{i=1,2,\dots,k}$.
We may choose this subsequence so that
\be\label{5.40}
\widetilde{\ell}_k=\liminf_{n\rightarrow\infty}\ell_k^n,
\ee
or we may choose it so that
\be\label{5.41}
\widetilde{r}_k=\liminf_{n\rightarrow\infty}r_k^n.
\ee  
To simplify
notation, we assume convergence of the full
sequence.
Because $g_n<0$ on $(\ell_i^n,r_i^n)$,
we have $g\leq 0$ on
$[\widetilde{\ell}_i,\widetilde{r}_i]$.
Being the path of a two-speed Brownian
motion, $g$ must be strictly negative
on $(\widetilde{\ell}_i,\widetilde{r}_i)$.
Since the probability that $g$ has a negative
excursion of length exactly $\varepsilon$
is zero, and $\widetilde{r}_i-\widetilde{\ell}_i
\geq\lim_{n\rightarrow\infty}
(r_i^n-\ell_i^n)\geq\varepsilon$,
we must have
$\widetilde{r}_i-\widetilde{\ell}_i>\varepsilon$.
We conclude that $g$ has
at least $k$ negative excursions whose
length exceeds $\varepsilon$ by time
$\widetilde{r}_k$, and hence,
$\ell_k\leq\widetilde{\ell}_k$,
and $r_k\leq\widetilde{r}_k$.
When we choose the subsequence so that
(\ref{5.40}) holds, we have
\be\label{5.42}
\ell_k\leq\liminf_{n\rightarrow\infty}\ell_k^n,
\ee
and when we choose it so that (\ref{5.41}) holds, we have
\be\label{5.43}
r_k\leq \liminf_{n\rightarrow\infty}r_k^n.
\ee
Combining (\ref{5.42}) and (\ref{5.43})
with (\ref{5.39}), we conclude that
$\Lambda^n\rightarrow\Lambda$
and $R_n\rightarrow R$
almost surely.  The assertion
$\max_{t\geq 0}|E^n(t)-E(t)|\rightarrow 0$
follows from the uniform convergence of $g_n$
to $g$ on compact time intervals.
\end{proof}

In order to proceed, we need additional notation.
Recalling the independent Brownian motions (\ref{t1})--(\ref{t8}),
we define two independent Brownian motions
$$
Z_+:=b\Theta_1^*+\Theta_2^*+\Theta_3^*,\quad
Z_-:=-\Theta_6^*-\Theta_7^*-a\Theta_8^*.
$$
Observe from Proposition \ref{P5.4} that
\begin{align}
G^*
&
=Z_+\circ P_+^{G^*}
-Z_-\circ P_-^{G^*},\label{5.44y}\\
\big|G^*\big|
&=
Z_+\circ P_+^{G^*}
+Z_-\circ P_-^{G^*}
+\Gamma\big(Z_+\circ P_+^{G^*}
+Z_-\circ P_-^{G^*}\big).\nonumber
\end{align}
Using Remark \ref{R4.17} and Assumption 
\ref{Assumption1}, one can verify that
$$
\langle Z_+,Z_+\rangle=\sigma_+^2\id,\quad
\langle Z_-,Z_-\rangle=\sigma_-^2\id,
$$
where $\sigma_{\pm}$ are given by (\ref{sigmaplus})
and (\ref{sigmaminus}).  We next define constants
\be\label{5.44a}
\alpha_+:=
\frac{b\langle\Theta_1^*,\Theta_1^*\rangle
+b\langle\Theta_2^*,\Theta_2^*\rangle}
{\langle Z_+, Z_+\rangle},\quad
\alpha_-:=
\frac{b\langle\Theta_7^*,\Theta_7^*\rangle
+a\langle\Theta_8^*,\Theta_8^*\rangle}
{\langle Z_-,Z_-\rangle}
=\frac{2(\lambda_1+\mu_1)}{\sigma_-^2}
=-\frac{\rho\sigma_+}{\sigma_-}
\ee
(see (\ref{rho}))
and Brownian motions
$$
\Delta_+:=\Theta_1^*+b\Theta_2^*-\alpha_+Z_+,\quad
\Delta_-:=-b\Theta_7^*-\Theta_8^*-\alpha_-Z_-.
$$

\begin{lemma}\label{L6.2}
The Brownian motions $Z_+$, $Z_-$, $\Delta_+$, $\Delta_-$
are independent.
\end{lemma}

\begin{proof}
The claimed independence is established
by computing cross variations.
\end{proof}

\begin{proposition}\label{P5.10}
With $\Psi$ defined by
(\ref{2.24b}), we have
$G^*=\Psi(Z_+,Z_-)$
and the pair of Brownian motions
$(\Delta_+,\Delta_-)$ is independent
of $G^*$.
\end{proposition}

\begin{proof}
We invoke Lemma \ref{TS.T.7},
using (\ref{5.44y}) in place of (\ref{2.38x}),
to conclude $G^*=\Psi(Z_+,Z_-)$.
Being independent of $(Z_+,Z_-)$,
the pair $(\Delta_+,\Delta_-)$ is independent
of $\Psi(Z_+,Z_-)$.
\end{proof}

We may now rewrite (\ref{5.22}) as
\be\label{5.47}
\Pi^n\ArrowJ1
\big(\alpha_+Z_+\circ P^{G^*}_+
+\alpha_-Z_-\circ P^{G^*}_-\big)
+\big(\Delta_+\circ P_+^{G^*}
+\Delta_-\circ P_-^{G^*}\big),
\ee
and this convergence is joint with the convergences in
Corollaries \ref{C4.15} and \ref{C4.16}
and Lemma \ref{L5.2}.
Finally, using the Brownian motions
in Remark \ref{R5.Brm}, we define the Brownian motion
\begin{align*}
\Theta_0^*
&:=
B_{SE_-,V,+}\circ\frac{\lambda_1\lambda_2}
{\mu_0+\mu_1}\id
+B_{S,V,+}\circ\frac{c\lambda_2}{\mu_0+\mu_1}\id
-B_{S,V,-}\circ\frac{c\mu_0}{\mu_0+\mu_1}\id
\nonumber\\
&\qquad
+B_{SW,V,+}\circ\frac{\lambda_1\mu_1}{\mu_0+\mu_1}\id
-B_{SW,V,-}\circ\frac{\mu_0\mu_1}{\mu_0+\mu_1}\id.
\end{align*}
Because of Proposition \ref{P4.10}, 
the limit of
\begin{align*}
\Thetahatn_0
&:=
\Mhatn_{SE_-,V,+}\circ\lambda_2\Pbarn_{SE_-}
+\Mhatn_{S,V,+}\circ\lambda_2\Pbarn_S
-\Mhatn_{S,V,-}\circ\mu_0\Pbarn_S\nonumber\\
&\qquad
+\Mhatn_{SW,V,+}\circ\lambda_1\Pbarn_{SW}
-\Mhatn_{SW,V,-}\circ\mu_0\Pbarn_{SW}
\end{align*}
is $\Theta_0^*\circ P_-^{G^*}$,
i.e.,
$\Thetahatn_0\ArrowJ1\Theta^*_0\circ P_-^{G^*}$.
This convergence is joint with the convergences
in Corollaries \ref{C4.15} and \ref{C4.16},
Lemma \ref{L5.2}, and (\ref{5.47}).
The Brownian motion $\Theta_0^*$
is independent of the independent
Brownian motions $\Theta_i^*$, $i=1,\dots,8$,
in Corollary \ref{C4.16}.  Using Assumption
\ref{Assumption1}, we compute the quadratic
variation
$\langle\Theta^*_0,\Theta^*_0\rangle
=2\mu_0\id/a$.

We use the Skorohod Representation Theorem to
put all these processes on a common probability
space so the convergences become almost sure
uniformly on compact time intervals.
When $G^n$ is on a negative excursion,
$\Pbarn_{NE}$, $\Pbarn_{SE_+}$,
$\Pbarn_{SE}$ and $\Pbarn_O$
in formula (\ref{5.2}) are constant.  On
such an excursion, (\ref{5.2}) implies
\be\label{5.44}
d\sVhatn=d\Thetahatn_0
+\sqrt{n}\big[\lambda_2 d\Pbarn_{SE_-}
+(\lambda_2-\mu_0)\Pbarn_S
-c\Pbarn_{SW}\big].
\ee
The scaled centered Poisson processes
$\Mhatn_{\times,V,\pm}$ appearing
in $\Thetahatn_0$ are independent
of $\Ghatn$, and hence independent of the
beginning and ending times of the excursion
and the excursion itself.
Therefore, $\Thetahatn_0$ is independent
of these quantities, and we know its
limit is $\Theta^*_0$.
The remaining term in $d\sVhatn$,
$\sqrt{n}\big[\lambda_2 d\Pbarn_{SE_-}
+(\lambda_2-\mu_0)\Pbarn_S
-c\Pbarn_{SW}\big]$,
is more difficult.  It is not independent
of $\Ghatn$, and hence depends on the fact
that we are observing it during a negative
excursion of $\Ghatn$.  

\begin{proposition}\label{P5.12}
Let $\varepsilon>0$ and a positive integer
$k$ be given, and let $\Lambda^n$,
$R^n$, $E^n$, $\Lambda$, $R$ and $E$
be as in Lemma \ref{L5.9}.  Then
\begin{align}
\Pi^n\big((\Lambda^n+\cdot\,)\wedge R^n\big)-\Pi^n(\Lambda^n)
&\rightarrow
\alpha_-\Big[Z_-\circ
P^{G^*}_-\big((\Lambda+\cdot\,)\wedge R\big)
-Z_-\circ P^{G^*}_-(\Lambda)\Big]\nonumber\\
&\qquad
+\big[\Delta_-\circ
P^{G^*}_-\big((\Lambda+\cdot\,)\wedge R\big)
-\Delta_-\circ P^{G^*}_-(\Lambda)\big],\label{5.50}\\
\sVhatn\big((\Lambda^n+\cdot\,)\wedge R^n\big)
-\sVhatn(\Lambda^n)
&\rightarrow
\Big[\Theta^*_0\circ P^{G^*}_-
\big((\Lambda+\cdot\,)\wedge R\big)
-\Theta^*_0\circ P^{G^*}_-(\Lambda)\Big]\nonumber\\
&\qquad
-\Big[\Delta_-\circ P^{G^*}_-
\big((\Lambda+\cdot\,)\wedge R\big)
-\Delta_-\circ P^{G^*}_-(\Lambda)\Big]
+\alpha_- E
\label{5.51}
\end{align}
almost surely on compact time intervals.
\end{proposition}

\begin{proof}
The convergence (\ref{5.50}) is a consequence
of (\ref{5.47}), the Skorohod Representation
Theorem, Lemma \ref{L5.9},
the continuity of the processes,
and constancy of $P^{G^*}_+$
on $[\Lambda,R]$.
The first term on the right-hand side
of (\ref{5.51}) arises from the term $d\Thetahatn_0$\
on the right-hand side of (\ref{5.44}).
For the second and third terms on
the right-hand side of (\ref{5.51}), we observe
that on a negative excursion of $G^*$,
the occupation times $\Phatn_{NE}$,
$\Phatn_E$, $\Phatn_{SE_+}$,
$\Phatn_{SE}$, and $\Phatn_O$ are constant,
and so
\begin{align*}
\lefteqn{\sqrt{n}\Big[\lambda_2\Pbarn_{SE_-}
\big((\Lambda^n+\cdot\,)\wedge R^n\big)
+(\lambda_2-\mu_0)\Pbarn_S
\big((\Lambda^n+\cdot\,)\wedge R^n\big)
-c\Pbarn_{SW}\big((\Lambda^n+\cdot\,)\wedge R^n\big)\Big]}
\hspace{3.0in}\\
\lefteqn{-\sqrt{n}\Big[\lambda_2\Pbarn_{SE_-}(\Lambda^n)
+(\lambda_2-\mu_0)\Pbarn_S(\Lambda^n)
-c\Pbarn_{SW}(\Lambda^n)\Big]}\hspace{2.0in}\\
&=
-\Big[\Pi^n\big((\Lambda^n+\cdot\,)\wedge R^n\big)
-\Pi^n(\Lambda^n)\Big].
\end{align*}
The limit of this sequence of processes
is provided by (\ref{5.50}).
But according to (\ref{5.44y}),
$$
-\Big[Z_-\circ P^{G^*}_-\big((\Lambda+\cdot\,)\wedge R\big)
-Z_-\circ P^{G^*}_-(\Lambda)\Big]
=G^*\big((\Lambda+\cdot\,)\wedge R\big)
-G^*(\Lambda)
=E.
$$
Substitution of this into (\ref{5.50})
leads us to the right-hand side of (\ref{5.51}).
\end{proof}

\begin{remark}\label{R5.12}
{\rm 
In addition to the negative excursion $E$
of $G^*$ appearing
on the right-hand side of (\ref{5.51}), there
are the processes 
\begin{align*}
\widetilde{\Theta}_0
&:=\Theta^*_0\circ\big( P^{G^*}_-
(\Lambda)+\cdot\,\big)
-\Theta^*_0\circ P^{G^*}_-(\Lambda),\\
\widetilde{\Delta}_-
&:=
\Delta_-\circ\big( P^{G^*}_-
(\Lambda)+\cdot\,\big)
-\Delta_-\circ P^{G^*}_-(\Lambda).
\end{align*}
Because $\Theta_0^*$ is independent
of $\Theta_i^*$, $i=1,\dots,8$,
and hence independent of $G^*$,  
$\widetilde{\Theta}_0$
is the Brownian motion
$\Theta_0^*$ shifted by 
$P_-^{G^*}(\Lambda)$, a random time that is
independent of $\Theta_0^*$.
As such, $\widetilde{\Theta}_0$ 
has the same law as
$\Theta_0^*$, regardless of
the value of $P^{G^*}_-(\Lambda)$.
Because of Proposition \ref{P5.10},
the same reasoning can be applied to
$\widetilde{\Delta}_-$, which has
the same law as $\Delta_-$.
Finally, because $\Delta_-$
involves only the Brownian
motions $\Theta_6^*$, $\Theta_7^*$
and $\Theta_8^*$, it is independent
of $\Theta_0^*$, and likewise,
$\widetilde{\Delta}_-$ is independent
of $\widetilde{\Theta}_0$.  It follows that
$\widetilde{\Theta}_0-\widetilde{\Delta}_-$
is a Brownian motion with quadratic variation
$$
\langle\widetilde{\Theta}_0-\widetilde{\Delta}_-,
\widetilde{\Theta}_0-\widetilde{\Delta}_-\rangle=
\langle\widetilde{\Theta}_0,\widetilde{\Theta}_0
\rangle
+\langle\widetilde{\Delta}_-,\widetilde{\Delta}_-
\rangle=
\langle\Theta_0^*,\Theta_0^*\rangle
+\langle\Delta_-,\Delta_-\rangle.
$$
During the negative excursion $E$ of $G^*$,
we have
$dG^*=dZ_-$.
Therefore, $\alpha_- E$ contributes
quadratic variation $\alpha_-^2\langle Z_-,Z_-\rangle$
to the right-hand side of (\ref{5.51}), where
we compute this quadratic variation pathwise.
We use the independence of $\Delta_-$
and $Z_-$ and the equality 
$\langle -b\Theta_7^*-\Theta_8^*,Z_-\rangle
=\alpha_-\langle Z_-,Z_-\rangle$ below to compute
the rate of accumulation of pathwise quadratic
variation of the limit of $\sVhatn$
during a negative of excursion of $G^*$ to be
(denoting by $\mbox{}^{\prime}$ derivatives with
respect to time)
\begin{align*}
\lefteqn{\langle\Theta_0^*,\Theta_0^*\rangle'
+\langle\Delta_-,\Delta_-\rangle'
+\alpha_-^2\langle Z_-,Z_-\rangle'}\\
&=
\langle\Theta_0^*,\Theta_0^*\rangle'
+\langle -b\Theta_7^*-\Theta_8^*-\alpha_-Z_-,
-b\Theta_7^*-\Theta_8^*-\alpha_-Z_-\rangle'
+\alpha_-^2\langle Z_-,Z_-\rangle'\\
&=
\langle\Theta_0^*,\Theta_0^*\rangle'
+b^2\langle\Theta_7^*\Theta_7^*\rangle'
+\langle\Theta_8^*\Theta_8^*\rangle'
-2\alpha_-\langle -b\Theta_7^*-\Theta_8^*,Z_-\rangle'
+2\alpha_-^2\langle Z_-,Z_-\rangle'\\
&=
\frac{2\mu_0}{a}+2b\lambda_1+\frac{2\mu_1}{a}\\
&=
\sigma_+^2.
\end{align*}
The rate of accumulation of
cross variation between the limit
of $\sVhatn$ and $G^*$, computed pathwise,  is
$$
\alpha_-\langle E,E\rangle'
=\alpha_-\langle Z_-,Z_-\rangle'=2(\lambda_1+\mu_1).
$$
This results in the correlation (\ref{rho})
}
\end{remark}

\subsection{Enumerating excursions}\label{Enumerating}

Before considering $\sVhatn$
on positive excursions of $\Ghatn$,
a consideration that leads to (\ref{5.1z}),
we need to
introduce additional notation for the excursions,
both positive and negative,
of $\Ghatn$ and
$G^*$ away from zero.  
We begin with an enumeration of the
countably many positive excursions of
$G^*$.  
Let $\{\alpha_m\}_{m=1}^{\infty}$
be a sequence of positive numbers
converging down to zero.
We first list in order of left endpoints
all positive excursions whose length exceeds
$\alpha_1$.  We next list in order of left endpoints
all positive excursions whose length lies in
$[\alpha_2,\alpha_1)$.  Continuing this way, at
the $m$-th step we list in order of left endpoints
all positive excursions whose length lies
in $[\alpha_m,\alpha_{m-1})$.
From this sequence of lists, we form a single
sequence by the usual diagonalization argument.
We denote by $\{E_{k,+}\}_{k=1}^{\infty}$
this sequence of positive excursions of
$G^*$ (respectively,
by $\{E_{k,-}\}_{k=1}^{\infty}$ this sequence of negative
excursions of $G^*$), and by $\Lambda_{k,\pm}$
and $R_{k,\pm}$ the left and right endpoints of 
$E_{k,\pm}$,
respectively.  Similarly, we create two sequences
$\{E_{k,\pm}^n\}_{k=1}^{\infty}$ of excursions of
$\Ghatn$ that do not begin at zero.  We construct
this sequence so it is ordered in the same way
as the sequence $\{E_{k,\pm}\}_{k=1}^{\infty}$.
This guarantees by an argument like that of
Lemma \ref{L5.9} that
\be\label{5.54}
\Lambda_{k,\pm}^n\rightarrow\Lambda_{k,\pm},
\quad
R_{k,\pm}^n\rightarrow R_{k,\pm},\quad
\max_{t\geq 0}
\big|E_{k,\pm}^n(t)-E_{k,\pm}(t)\big|\rightarrow 0,
\quad k=1,2,\dots
\ee
almost surely; see \cite[Theorem~4.3.3]{Almost}
for a similar construction.  Because Brownian paths have
no isolated zeros, a time point cannot be both
the right endpoint and the left endpoint of
an excursion of Brownian motion.  This is also true
for two-speed Brownian motion by virtue of
(\ref{TS.E.2}) and (\ref{TS.E.3}).

Using this notation, we recap
the result of Proposition \ref{P5.12}
about the convergence of $\sVhatn$ on negative
excursions of $\Ghatn$.  For
each $k\geq 0$, we define
\begin{align}
V_{k,-}^n
&:=
\sVhatn\big((\Lambda^n_{k,-}+\cdot)
\wedge R^n_{k,-}\big)-\sVhatn(\Lambda^n_{k,-}),
\label{5.55}\\
\widetilde{C}_{k,-}
&:=
\Big[\Theta_0^*\circ \big(P^{G^*}_-(\Lambda_{k,-})
+\cdot\,\big)
-\Theta_0^*\circ P^{G^*}_-(\Lambda_{k,-})\Big]
\nonumber\\
&\qquad
-\Big[\Delta_-\circ \big(P^{G^*}_-(\Lambda_{k,-})
+\cdot\,\big)
-\Delta_-\circ P^{G^*}_{-}(\Lambda_{k,-})\Big],
\nonumber\\
C_{k,-}
&:=
\widetilde{C}_{k,-}
\big(\cdot\wedge (R_{k,-}-\Lambda_{k,-})\big).
\label{5.57}
\end{align}
Under the Skorohod representation assumption
of Proposition \ref{P5.12}, we have
\be\label{5.58}
V_{k,-}^n\rightarrow C_{k,-}
+\alpha_- E_{k,-},\quad k=1,2,\dots,
\ee 
almost surely as $n\rightarrow\infty$.
Each $C_{k,-}$ is a Brownian motion
stopped at a time independent of the
Brownian motion, and the sequence
of stopped processes
$\{C_{k,-}\}_{k=1}^{\infty}$
is independent because they are increments
of $\Theta_0^*-\Delta_-$ over non-overlapping
intervals.  Furthermore, each unstopped process
$\widetilde{C}_{k,-}$ is a Brownian
motion independent of $G^*$.
According to Remark \ref{R5.12}, the quadratic variation
of $\widetilde{C}_{k,-}$ is
\be\label{5.59}
\langle \widetilde{C}_k,\widetilde{C}_k\rangle
=\sigma_+^2\id-\alpha_-^2\langle\Delta_-,\Delta_-\rangle
=\left(\sigma_+^2-\frac{4(\lambda_1+\mu_1)^2}{\sigma_-^2}
\right)\id
=(1-\rho^2)\sigma_+^2\id.
\ee

It remains to examine $\sVhatn$ on positive
excursions of $\Ghatn$.  For each $k\geq 0$,
we define
\be\label{5.60}
V_{k,+}^n:=\sVhatn\big((\Lambda_{k,+}^n+\cdot)
\wedge R_{k,+}^n\big).
\ee
In contrast to (\ref{5.55}), 
the definition of $V_{k,+}^n$
does not subtract out the initial condition
of the excursion associated with it.
This difference between
(\ref{5.55}) and (\ref{5.60})
is consistent with the difference between (\ref{5.2z})
and (\ref{5.1z}). 

To identify the limit of $\sVhatn$,
we first determine the limit of $V_{k,+}^n$
for each $k$.  We must then combine
the limits of these fragments
of $\sVhatn$ with the limits of fragments
in (\ref{5.58}) to construct
the limit of the full process $\sVhatn$.

\subsection{$\sVhatn$ on positive excursions
of $\Ghatn$}\label{Positive}

\subsubsection{Controlling the oscillations of $\sVhatn$}

We begin with a pathwise result that
we later apply to the representation
of $\sVhatn$ obtained in Lemma \ref{L5.8}
to control the oscillations of $\sVhatn$. 
To state the result, we define, for
$x,y,z\in\R$ with $y<z$, the distance
from $x$ to the interval $[y,z]$.
For this definition, we adopt the convention
that $[z,y]:=[y,z]$.  Thus, we define
$$
\big|x-[y,z]\big|:=
\big|x-[z,y]\big|:=
\left\{\begin{array}{ll}
0&\mbox{if }x\in[y,z],\\
y-x&\mbox{if }x<y,\\
x-z&\mbox{if }x>z.
\end{array}\right.
$$
The {\em modulus of continuity} over
$[0,T]$ for $f\in D[0,\infty)$ and $\delta>0$ is
$$
w_T(f,\delta):=
\sup_{
\stackrel
{\scriptstyle{0\leq s\leq t\leq T}}
{\scriptstyle{t-s\leq\delta}}
}
\big|f(t)-f(s)\big|.
$$

\begin{proposition}\label{P8.4}
Assume $p$, $c$ and $v$ are paths in
$D[0,\infty)$ such that
$p$ is nondecreasing and continuous,
and
\be\label{8.13}
v(t)=v(0)+K\int_0^t
\big(\kappa_L-v^+(s)\big)dp(s)+c(t),
\quad t\geq 0.
\ee
Let $T>0$ and $\delta>0$ be given.
For all $0\leq t_1\leq t_2\leq t_3\leq T$
satisfying $t_3-t_1\leq\delta$, we have
\be\label{8.14}
\big|v(t_2)-\big[v(t_1),v(t_3)\big]\big|
\leq 4w_T(c,\delta).
\ee
Furthermore, if $|\kappa_L-v(t_1)|\leq\varepsilon$,
then
\be\label{8.15}
\big|\kappa_L-v(t_2)\big|\leq\varepsilon
+2w_T(c,\delta).
\ee
\end{proposition}

\begin{proof}
Let $0\leq t_1\leq t_2\leq t_3\leq T$ be given with
$t_3-t_1\leq\delta$.  We consider four cases,
according to whether $v$ is below or above
$\kappa_L$ at time $t_1$ and whether $v$
crosses the level $\kappa_L$ by time $t_3$.

\noindent
\underline{Case I}:
$\sup_{t_1\leq u<t_3}v(u)\leq\kappa_L$.

In this case the integrand in (\ref{8.13})
is nonnegative for all $u\in[t_1,t_3)$, hence
\be\label{8.16}
v(s')-v(s)\geq c(s')-c(s)\geq -w_T(c,\delta),\quad
t_1\leq s\leq s' \leq t_3.
\ee
This implies both 
$v(t_2)\geq v(t_1)-w_T(c,\delta)$ and 
$v(t_3)\geq v(t_2)-w_T(c,\delta)$,
and therefore,
$$
\big|v(t_2)-\big[v(t_1),v(t_3)\big]\big|
\leq w_T(c,\delta).
$$

If, in addition, $|\kappa_L-v(t_1)|\leq\varepsilon$,
then $v(t_1)\geq \kappa_L-\varepsilon$,
so $v(t_2)\geq\kappa_L-\varepsilon-w_T(c,\delta)$.
According to the case assumption, $v(t_2-)\leq\kappa_L$.
Since $p$ is continuous, the jumps of $v$
agree with those of $c$, and the magnitude of these
is bounded by $w_T(c,\delta)$.  Therefore,
$v(t_2)\leq \kappa_L+w_T(c,\delta)$, and
a slightly stronger version of (\ref{8.15}) holds.

\noi
\underline{Case II}: $v(t_1)<\kappa_L$
but $\sup_{t_1\leq u<t_3}v(u)>\kappa_L$.

In this case we define
$\sigma_0:=\min\big\{t\geq t_1:v(t)\geq\kappa_L\big\}$,
and note that $t_1\leq\sigma_0<t_3$.  In place
of (\ref{8.16}) we have
$$
v(s')-v(s)\geq c(s')-c(s)\geq -w_T(c,\delta),
\quad t_1\leq s\leq s'\leq\sigma_0.
$$
In particular,
$$
v(s')\geq v(t_1) -w_T(c,\delta),
\quad
t_1\leq s'\leq \sigma_0.
$$
Furthermore, $v(s')<\kappa_L$ for $t_1\leq s<\sigma_0$,
so over the time interval $[t_1,\sigma_0)$,
$v$ is confined to the interval
$[v(t_1)-w_T(c,\delta),\kappa_L]$.

For $\sigma_0\leq s\leq t_3$, we set
\be\label{8.18}
\sigma_1(s)
:=
\sup\big\{t\leq s:v(t)\geq\kappa_L\big\},\quad
\sigma_2(s)
:=
\sup\big\{t\leq s:v(t)\leq\kappa_L\big\}. 
\ee
Observe that $\sigma_0\leq\sigma_1(s)\leq s$
and $t_1\leq \sigma_2(s)\leq s$.  In fact,
either $\sigma_1(s)=s$ or $\sigma_2(s)=s$
(and both equations might hold).  There are two
subcases.

\noi
\underline{Case II.A}: $v(s)\leq\kappa_L$.

In this case the integrand in (\ref{8.13})
is nonnegative for $\sigma_1(s)\leq u\leq s$, hence
\be\label{8.20}
\kappa_L\geq v(s)
\geq v\big(\sigma_1(s)\big)+c(s)-c\big(\sigma_1(s)\big)
\geq v\big(\sigma_1(s)\big)-w_T(c,\delta)
\geq\kappa_L-2w_T(c,\delta).
\ee
The last inequality follows from the fact that
at time $\sigma_1(s)$, either $v(\sigma_1(s))=\kappa_L$,
because $v$ crossed the level $\kappa_L$
``continuously'' (including the case
when $v(s)=\kappa_L$, so that $\sigma_1(s)=s$
and $v(\sigma_1(s))=\kappa_L)$, or $v$
jumped across the level $\kappa_L$.
But the jumps in $v$ are bounded by $w_T(c,\delta)$.

\noi
\underline{Case II.B}: $v(s)\geq\kappa_L$.

In this case, the integrand in (\ref{8.13})
is nonpositive for $\sigma_2(s)\leq u\leq s$,
hence
$$
\kappa_L\leq v(s)
\leq v\big(\sigma_2(s)\big)+c(s)-c\big(\sigma_2(s)\big)
\leq v\big(\sigma_2(s)\big)+w_T(c,\delta)
\leq\kappa_L+2w_T(c,\delta).
$$
The last inequality follows by an analogue of
the argument justifying the last inequality
in (\ref{8.20}).

\vsp
In both Case II.A and Case II.B,
$v$ takes values only in
$[\kappa_L-2w_T(c,\delta),\kappa_L+2w_T(c,\delta)]$
on the interval $[\sigma_0,t_3]$.
Since $v$ is confined to
$[v(t_1)-w_T(c,\delta),\kappa_L]$
over the interval $[t_1,\sigma_0]$, it follows that
(\ref{8.14}) holds.
If, in addition, $|\kappa_L-v(t_1)|\leq\ve$,
then $\kappa_L\geq v(t_1)\geq \kappa_L-\ve$,
and it follows from the above considerations
(whether $t_2\in[t_1,\sigma_0)$ or
$t_2\in[\sigma_0,t_3)$) that 
(\ref{8.15}) holds.

\noi
\underline{Case III}: 
$\inf_{t_1\leq u< t_3}v(u)\geq\kappa_L$.

In this case the integrand in (\ref{8.13})
is nonpositive for all $u\in[t_1,t_3)$, hence
\be\label{8.23}
v(s')-v(s)\leq c(s')-c(s)\leq w_T(c,\delta),\quad
t_1\leq s\leq s'\leq t_3.
\ee
This implies both
$v(t_2)\leq v(t_1)+w_T(c,\delta)$ and 
$v(t_3)\leq v(t_2)+w_T(c,\delta)$,
and again, (\ref{8.14}) holds.  If, in addition,
$|\kappa_L-v(t_1)|\leq\ve$, then
$\kappa_L+\ve\geq v(t_1)$,
so $\kappa_L+\ve+w_T(c,\delta)\geq v(t_2)\geq \kappa_L-w_T(c,\delta)$,
and a slightly stronger version of (\ref{8.15}) holds.

\noi
\underline{Case IV:}
$v(t_1)>\kappa_L$ but $\inf_{t_1\leq u<t_3}v(u)\leq
\kappa_L$.

In this case, we define
$\sigma_0:=\min\big\{t\geq t_1:v(t)\leq\kappa_L\big\}$,
and note that $t_1\leq\sigma_0<t_3$.  In place of
(\ref{8.23}) we have
$$
v(s')-v(s)\leq c(s')-c(s)\leq w_T(c,\delta),\quad
t_1\leq s\leq s'\leq \sigma_0.
$$
In particular,
$$
v(s')\leq v(t_1)+w_T(c,\delta),\quad
t_1\leq s'\leq\sigma_0.
$$
Furthermore, $v(s')>\kappa_L$ for $t_1\leq s<\sigma_0$,
so over the time interval $[t_1,\sigma_0)$,
$v$ is confined to the interval 
$[\kappa_L,v(t_1)+w_T(c,\delta)]$.

For $\sigma_0\leq s\leq t_3$, we again define
$\sigma_1(s)$ and $\sigma_2(s)$ by
(\ref{8.18}), and observe that
in this case
$t_1\leq \sigma_1(s)\leq s$
and $\sigma_0\leq\sigma_2(s)\leq s$.
We repeat the arguments of the subcases of
Case II to see that $v$ is again confined
to the interval 
$[\kappa_L-2w_T(c,\delta),\kappa_L+2w_T(c,\delta)]$
on the interval $[\sigma_0,t_3]$, and hence
(\ref{8.14}) holds.  If in addition
$|v(t_1)-\kappa_L|\leq\ve$, then by the same
reasoning as in Case II we obtain (\ref{8.15}).
\end{proof}

\subsubsection{Convergence after the excursion begins}

\begin{proposition}\label{P5.14}
For every $k\geq 1$ and every $\varepsilon>0$
with $R_{k,+}-\Lambda_{k,+}>\varepsilon$,
we have
$\sup_{t\geq\varepsilon}|V^n_{k,+}(t)-\kappa_L|
\stackrel{\P}{\rightarrow} 0$.
\end{proposition}

\begin{proof}
We begin with Lemma \ref{L5.8},
which together with (\ref{5.60}) implies
\be\label{5.63}
V^n_{k,+}(t)=V^n_{k,+}(0)
+\sqrt{n}\,\theta_b
\int_0^t\Big(\kappa_L-\big(V_{k,+}^n(s)\big)^+\Big)
d\Pbarn_{NE}(\Lambda_{k,+}^n+s)
+C_V^n(\Lambda_{k,+}^n+t)
-C_V^n(\Lambda_{k,+}^n)
\ee
for $0\leq t\leq R_{k,+}^n-\Lambda_{k,+}^n$.
However, because $V_{k,+}^n(t)=V_{k,+}^n(R_{k,+}^n-\Lambda_{k,+}^n)$
for $t\geq R_{k,+}^n-\Lambda_{k,+}^n$,
for purposes of this proof we may and do
assume that (\ref{5.63}) is valid
for all $t\geq 0$.

We choose $\delta\in(0,1)$ and define
$$
\tau^n(t):=0\vee\sup\big\{s\in[0,t]:
\big|\kappa_L-\big(V^n_{k,+}(s)\big)^+\big|<\delta
\big\}.
$$
Henceforth we consider only $n\geq 1/\delta^2$,
so that $V_{k,n}^+$, whose jumps are of
size $1/\sqrt{n}$, cannot jump across
the interval $[-\delta,\delta]$.
We show first that
\be\label{5.64}
\tau^n\ArrowJ1\id.
\ee

For a fixed $t\geq\ve$, 
there are three cases.
Either $\tau^n(t)=t$ or 
\be\label{5.65}
\kappa_L-(V^n_{k,+})^+\leq -\delta
\mbox{ on }[\tau^n(t),t],
\ee
or
\be\label{5.66}
\kappa_L-(V^n_{k,+})^+\geq \delta
\mbox{ on }[\tau^n(t),t].
\ee
In the case of (\ref{5.65}), (\ref{5.63}) 
and the fact that $C_V^n=O(1)$ imply
$$
V_{k,+}^n(t)\leq V_{k,+}^n\big(\tau^n(t)\big)
-\sqrt{n}\,\theta_b\delta
\Big(\Pbarn_{NE}\big(\Lambda_{k,+}^n+t\big)
-\Pbarn_{NE}\big(\Lambda_{k,+}^n+\tau^n(t)\big)\Big)
+O(1),
$$
whereas in the case of (\ref{5.66}),
(\ref{5.63}) and the fact that $C_V^n=O(1)$ imply
$$
V_{k,+}^n(t)\geq V_{k,+}^n\big(\tau^n(t)\big)
+\sqrt{n}\,\theta_b\delta
\Big(\Pbarn\big(\Lambda_{k,+}^n+t\big)
-\Pbarn\big(\Lambda_{k,+}^n+\tau^n(t)\big)\Big)
+O(1).
$$
Because $V_{k,+}^n$ is $O(1)$ (Theorem \ref{T5.5}),
in either case we have
\be\label{5.67}
\Pbarn_{NE}\big(\Lambda_{k,+}^n+t\big)
-\Pbarn_{NE}\big(\Lambda_{k,+}^n+\tau^n(t)\big)
\leq\frac{1}{\sqrt{n}\,\theta_b\delta}O(1)=o(1).
\ee
Relation (\ref{5.67}) also holds if $\tau^n(t)=t$.
Consequently, 
$Q^n-Q^n\circ\tau^n\ArrowJ1 0$,
where
$$
Q^n:=\Pbarn_{NE}(\Lambda_{k,+}^n+\cdot)
-\Pbarn_{NE}(\Lambda_{k,+}^n).
$$
But Proposition \ref{P4.10}
implies $Q^n\ArrowJ1 
Q:=\lambda_1\id/(\lambda_0+\lambda_1)$,
which is strictly increasing.
In particular, for every $T\geq 0$,
\be\label{5.68}
\max_{0\leq t\leq T}
\big|Q\big(\tau^n(t)\big)-Q^n\big(\tau^n(t)\big)\big|
\leq\max_{0\leq t\leq T}
\big|Q(t)-Q^n(t)\big|\ArrowJ1 0.
\ee
In addition,
\be\label{5.69}
Q^n\circ\tau^n=Q^n-(Q^n-Q^n\circ\tau^n)
\ArrowJ1 Q.
\ee
We put (\ref{5.68}) and (\ref{5.69}) together
and use the linearity of $Q^{-1}$ to establish
(\ref{5.64}):
$$
\tau^n=Q^{-1}\circ Q\circ\tau^n
=Q^{-1}\circ(Q\circ\tau^n-Q^n\circ \tau^n)
+Q^{-1}\circ Q^n\circ \tau^n
\ArrowJ1 Q^{-1}\circ 0+Q^{-1}\circ Q=\id.
$$

To continue, we choose an arbitrary subsequence
of $\{V_{k,+}^n\}_{n=1}^{\infty}$ from which,
according to Lemma \ref{L5.8}, we may select
a sub-subsequence $\{V_{k,+}^{n_j}\}_{j=1}^{\infty}$,
along which $C_V^{n_j}$ converges weakly-$J_1$
to a continuous limit $C_V$.
Then (\ref{5.64}) implies
$$ 
C_V^{n_j}\big((\Lambda_{k,-}^{n_j}+\id)\wedge R_{k,+}^{n_j}
\big)-C_V\big((\Lambda_{k,-}^{n_j}+\tau^{n_j})
\wedge R_{k,+}^{n_j}\big)=o(1).
$$
From (\ref{5.63}), we see that
$$
V_{k,+}^{n_j}(t)=V_{k,+}^{n_j}\big(\tau^{n_j}(t)\big)
+\sqrt{n_j}\,\theta_b\int_{\tau^{n_j}(t)}^t
\Big(\kappa_L-\big(V_{k,+}^{n_j}(s)\big)^+\Big)
d\Pbar^{n_j}_{NE}(\Lambda_{k,+}^{n_j}+s)+o(1).
$$
We again consider the two cases (\ref{5.65})
and (\ref{5.66}).  However, we need to
consider subcases depending on whether $\tau^{n_j}(t)=0$
or $0<\tau^{n_j}(t)<t$ or $\tau^{n_j}=t$.

\noindent
\underline{Case I.A}: 
(\ref{5.65}) holds and $\tau^{n_j}=0$.

In this case,
\begin{align*}
\kappa_L+\delta
&\leq 
V^{n_j}_{k,+}(t)\\
&\leq
 V_{k,+}^{n_j}(0)
-\sqrt{n_j}\,\theta_b\delta
\big[\Pbar^{n_j}_{NE}(\Lambda_{k,+}^{n_j}+t)
-\Pbar^{n_j}_{NE}(\Lambda_{k,+}^{n_j})\big]
+o(1)\\
&\leq 
V^{n_j}_{k,+}(0)
-\sqrt{n_j}\,\theta_b\delta
\big[\Pbar^{n_j}_{NE}(\Lambda_{k,+}^{n_j}+\ve)
-\Pbar^{n_j}_{NE}(\Lambda_{k,+}^{n_j})\big]
+o(1).
\end{align*}
Because
$V^{n_j}_{k,+}(0)=O(1)$
and the limit of $\Pbar^{n_j}_{NE}(\varepsilon)$
is strictly positive (Proposition \ref{P4.10}),
for sufficiently large $j$ this case does not occur.
We disregard it.

\noindent
\underline{Case I.B}: (\ref{5.65}) holds
and $0<\tau^{n_j}<t$.

In this case, (\ref{5.65}) implies
\begin{align*}
\kappa_L+\delta
&\leq
V_{k,+}^{n_j}(t)\\
&\leq 
V^{n_j}_{k,+}\big(\tau^{n_j}(t)\big)
-\sqrt{n_j}\,\theta_b\delta
\big[\Pbar^{n_j}_{NE}(\Lambda_{k,+}^{n_j}+t)
-\Pbar^{n_j}_{NE}
\big(\Lambda_{k,+}^{n_j}+\tau^{n_j}(t)\big)\big]+o(1)\\
&\leq 
V^{n_j}\big(\tau^{n_j}(t)\big)
+o(1)\\
&\leq
\kappa_L+\delta+\frac{1}{\sqrt{n_j}}+o(1),
\end{align*}
the last inequality following from the fact
that $V_{k,+}^{n_j}(\tau^{n_j}(t)-)<\kappa_L+\delta$
and the jumps in $V_{k,+}^{n_j}$
are of size $1/\sqrt{n_j}$. 
 
\noindent
\underline{Case I.C}:
(\ref{5.65}) holds and $\tau^{n_j}(t)=t$.

We have $V_{k,+}^{n_j}(t-)<\kappa_L+\delta$
and $V_{k,+}^{n_j}(t)\geq \kappa_L+\delta$.
Because of jumps in $V_{k,+}^{n_j}$ are of
size $1/\sqrt{n_j}$, these inequalities imply
$$
\kappa_L+\delta\leq V_{k,+}^{n_j}(t)
\leq\kappa_L+\delta+\frac{1}{\sqrt{n_j}}.
$$

\noi
\underline{Case II}: (\ref{5.66}) holds.

Like Case I.A, we may disregard the
case that $(\ref{5.66})$ holds and $\tau^{n_j}(t)=0$.
Reversing the inequalities in Case I.B,
we can show that if
$(\ref{5.66})$
holds and $0<\tau^{n_j}<t$, then
$$
\kappa_L+\delta\geq
V_{k,+}^{n_j}(t)\geq
\kappa_L-\delta-\frac{1}{\sqrt{n_j}}+o(1).
$$
Finally, if (\ref{5.66}) holds and 
$\tau^{n_j}(t)=t$, we have
$$
\kappa_L+\delta\geq V_{k,+}^{n_j}(t)
\geq\kappa_L-\delta-\frac{1}{\sqrt{n_j}}.
$$

In every case,
$|V_{k,+}^{n_j}(t)-\kappa_L|\leq\delta+o(1)$.
Because $\delta\in(0,1)$ is arbitrary,
we conclude
$\sup_{t\geq\varepsilon}\big|V_{k,+}^{n_j}(t)-\kappa_L
\big|\stackrel{\P}{\rightarrow} 0$
almost surely.  Because every subsequence
of $\{V_{k,+}^n\}_{n=1}^{\infty}$
has a sub-subsequence with this property,
the full sequence has this property.
\end{proof}

\subsubsection{Convergence at the beginning
of the excursion}\label{Beginning}

Because a Brownian motion, and hence
a two-speed Brownian motion, has infinitely
many positive (and negative) excursions
immediately to the left of the left endpoint
of every excursion, we can use Proposition \ref{P5.14}
to prove the following lemma.

\begin{lemma}\label{L8.4a}
For $k\geq 1$, every subsequence of 
$\{\sVhatn\}_{n=1}^{\infty}$ has a sub-subsequence
$\{\sVhat^{n_j}\}_{j=1}^{\infty}$ for which
$\limsup_{n_j\rightarrow\infty}
\sup_{s\in[(\Lambda_{k,\pm}-\delta)^+,
\Lambda_{k,\pm}+\delta]}
\big|\sVhat^{n_j}(s)-\kappa_L\big|
\rightarrow0$ almost surely as
$\delta\downarrow 0$.
\end{lemma}

\begin{proof}
Fix $k\geq 1$ and let $\ve>0$ be given.
We choose $T$ so large that
$\P\{\Lambda_{k,\pm}<T\}>1-\ve/4$.  
The left endpoint $\Lambda_{k,\pm}$ of
an excursion of $G^*$ is necessarily
strictly positive, and there are infinitely
many positive excursions of $G^*$ immediately to
the left of $\Lambda_{k,\pm}$.  Let $[\lambda,\rho]$
denote the longest positive excursion interval
contained in $[(\Lambda_{k,\pm}-\ve)^+,\Lambda_{k,\pm})$.
Choose $\eta>0$ so small that
$\P\{\rho-\lambda>2\eta\}>1-\ve/4$.
There are at most $1+T/\eta$ intervals
of positive excursions
of $G^*$ that begin before time $T$ and have
length at least $\eta$.  Choose $M\geq k$
such that all of these excursion intervals
appear in the enumeration of positive
excursions of $G^*$ by step $M$.
According to Proposition \ref{P5.14},
we may choose $N_1$ so large that
$$
\P\Big\{\max_{1\leq m\leq M}
\sup_{t\geq\eta/3}\big|V_{m,+}^n(t)-\kappa_L\big|
\leq\ve\Big\}>1-\frac{\ve}{4}
$$
for all $n\geq N_1$. Because 
$\Lambda_{k,+}^n\rightarrow\Lambda_{k,+}$
and $R_{k,+}^n\rightarrow R_{k,+}$
we may  choose $N_2$ so that for $n\geq N_2$,
$$
\P\Big\{\Lambda_{m,+}^n+\frac{\eta}{3}
<\frac{\lambda+\rho}{2}<R_{m,+}^n\mbox{ for some }
m\leq M\Big\}\geq1-\frac{\ve}{4}.
$$
Thus, for $n\geq N_1\vee N_2$,
on a set with probability at least $1-\ve$ we have
$$
\left|\sVhatn\left(\frac{\lambda+\rho}{2}\right)
-\kappa_L\right|\leq\ve.
$$
We now combine Lemma \ref{L5.8} with
Proposition \ref{P8.4}, setting
$t_1=(\lambda+\rho)/2$ and 
$\delta=\eta+\Lambda_{k,+}-(\lambda+\rho)/2$
in inequality (\ref{8.15}) to conclude
\be\label{8.24g}
\big|\sVhatn(t_2)-\kappa_L\big|
\leq \ve + 2w_T(C_V^n,\delta)
\ee
for 
$(\lambda+\rho)/2\leq t_2
\leq \Lambda_{k,\pm}+\eta$.
But $(\lambda+\rho)/2\leq\Lambda_{k,\pm}-\eta$, 
so (\ref{8.24g})
holds for 
$t_2\in[\Lambda_{k,\pm}-\eta,\Lambda_{k,\pm}+\eta]$.
This implies
$$
\sup_{s\in[(\Lambda_{k,\pm}-\eta)^+,
\Lambda_{k,\pm}+\eta]}
\big|\sVhat^{n}(s)-\kappa_L\big|\wedge 1
\leq
\ve+2\big(w_T(C_V^{n},\delta)\wedge 1\big).
$$

Every subsequence of $\{C_V^n\}_{n=1}^{\infty}$
has a sub-subsequence 
$\{C_V^{n_j}\}_{j=1}^{\infty}$
converging weakly-$J_1$
to a continuous
limit, which we call $C^*_V$.  The modulus
of continuity $w_T(\cdot,\varepsilon)\wedge 1$
is continuous at continuous functions
in $D[0,\infty)$, hence
\be\label{8.24f}
\E\Big[\limsup_{n_j\rightarrow\infty}
\sup_{s\in[(\Lambda_{k,\pm}-\eta)^+,
\Lambda_{k,\pm}+\eta]}
\big|\widehat{{\cal V}}^{n_j}(s)-\kappa_L\big|
\wedge 1\Big]
\leq \ve + 2\E\left[w_T(C^*_V,\delta)\wedge 1\right].
\ee
Letting $\ve\downarrow 0$
forces $\lambda$ and $\rho$ to converge
to $\Lambda_{k,\pm}$ so $\eta$ and
$\delta$ converge to zero.
Continuity of $C_V^*$ implies then that
$w_T(C_V^*,\delta)$ converges to zero.  Thus,
letting $\ve\downarrow0$ and hence
$\eta\downarrow 0$ in (\ref{8.24f})
results in
\be\label{8.25}
\E\Big[\limsup_{n_j\rightarrow\infty}
\sup_{s\in[(\Lambda_{k,\pm}-\eta)^+,
\Lambda_{k,\pm}+\eta]}
\big|\widehat{{\cal V}}^{n_j}(s)-\kappa_L\big|
\wedge 1\Big]
\rightarrow0\mbox{ as }\eta\downarrow 0.
\ee
Relation (\ref{8.25}) implies
\be\label{5.71}
Y(\eta):=\limsup_{n_j\rightarrow\infty}
\sup_{s\in[(\Lambda_{k,\pm}-\eta)^+,
\Lambda_{k,\pm}+\eta]}
\big|\widehat{{\cal V}}^{n_j}(s)-\kappa_L\big|
\stackrel{\P}{\rightarrow}0
\ee
as $\eta\downarrow 0$, which implies
convergence almost surely along a sequence
$\eta_i\downarrow 0$.  But $Y(\eta)$ is 
monotone in $\eta$, so convergence
along a sequence implies almost sure convergence
as $\eta\downarrow 0$. 
\end{proof}

\begin{theorem}\label{T8.5}
For $k=1,2,\dots$, 
$\sup_{t\geq 0}|V_{k,+}^n(t)-\kappa_L|
\stackrel{\P}{\rightarrow}0$
as $n\rightarrow\infty$.
\end{theorem}

\begin{proof}
Fix $k\geq 1$ and let $\ve>0$ be given.
Set $T=\Lambda_{k,+}^n+2\ve$,
$t_1=\Lambda_{k,+}^n$,
$t_3=t_1+\ve$ and $\delta=\ve$.
Proposition \ref{P8.4} and Lemma \ref{L5.8} imply
\begin{align}
4w_T(C_V^n,\ve)
&\geq
\sup_{\Lambda_{k,+}^n\leq t\leq \Lambda_{k,+}^n+\ve}
\big[\sVhatn(t)-\big[\sVhatn(\Lambda_{k,+}^n),
\sVhatn_{k,+}(\Lambda_{k,+}^n+\ve)
\big]\big|\nonumber\\
&=
\sup_{0\leq s\leq\ve}\big|V_{k,+}^n(s)-
\big[V_{k,+}^n(0),V_{k,+}^n(\ve)\big]\big|.
\label{5.72}
\end{align}

According to Lemma \ref{L8.4a}, from
every subsequence of $\{V_{k,+}^n\}_{n=1}^{\infty}$
we may choose a
sub-subsequence $\{V_{k,+}^{n_j}\}_{j=1}^{\infty}$
such that $Y(\delta)$ defined by (\ref{5.71})
converges almost surely to zero 
as $\delta\downarrow 0$.
But almost sure convergence of $\Lambda_{k,+}^{n_j}$
to $\Lambda_{k,+}$ implies
$$
\limsup_{n_j\rightarrow\infty}
\big|V^{n_j}_{k,+}(0)-\kappa_L\big]
=\limsup_{n_j\rightarrow\infty}
\big|\sV^{n_j}(\Lambda_{k,+}^{n_j})-\kappa_L\big|
\leq Y(\delta)
$$
for every $\delta>0$.  In other words,
$V^{n_j}_{k,+}(0)\rightarrow\kappa_L$ almost surely.
Proposition \ref{P5.14} permits us to choose
a further sub-subsequence, also 
denoted $\{V_{k,+}^{n_j}\}_{j=1}^{\infty}$, along which
$V_{k,+}^{n_j}(\ve)\rightarrow\kappa_L$ almost surely.
It follows now from (\ref{5.72}) that
$$
\limsup_{n_j\rightarrow\infty}\sup_{0\leq s\leq\ve}
\big|V_{k,+}^{n_j}(s)-\kappa_L\big|\leq
\limsup_{{n_j}\rightarrow\infty}4w_T(C_V^{n_j},\ve).
$$
Choosing yet a further sub-subsequence, also denoted
$\{V_{k,+}^{n_j}\}_{j=1}^{\infty}$, along which $C_V^{n_j}$ has
a continuous limit $C_V^*$, we obtain from Fatou's Lemma that
$$
\limsup_{n_j\rightarrow\infty}
\E\Big[\sup_{0\leq s\leq\ve}
\big|V_{k,+}^{n_j}(s)-\kappa_L\big|\wedge 1\Big]
\leq
\E\left[4w_T(C_V^*,\ve)\wedge 1\right].
$$
Proposition \ref{P5.14} implies 
$$
\limsup_{n_j\rightarrow\infty}
\E\Big[\sup_{s\geq \varepsilon}
\big|V_{k,+}^{n_j}(s)-\kappa_L\big|\wedge 1\Big]=0.
$$
Therefore,
\begin{align*}
\lefteqn{\limsup_{n_j\rightarrow\infty}
\E\Big[\sup_{s\geq 0}\big|\sV_{k,+}^{n_j}(s)
-\kappa_L\big|\wedge 1
\Big]}\\
&\leq
\limsup_{n_j\rightarrow\infty}
\E\Big[\sup_{0\leq s\leq\ve}\big|\sV_{k,+}^{n_j}(s)
-\kappa_L\big|\wedge 1
\Big]
+\limsup_{n_j\rightarrow\infty}
\E\Big[\sup_{s\geq \ve}\big|\sV_{k,+}^{n_j}(s)
-\kappa_L\big|\wedge 1
\Big]\\
&\leq
\E\big[4w_T(C_V^*,\ve)\wedge 1\big].
\end{align*}
Letting $\ve\downarrow 0$, we obtain
$\sup_{s\geq 0}|\sV_{k,+}^{n_j}(s)
-\kappa_L|\stackrel{\P}{\rightarrow}0$.
Because every subsequence of $\{V_{k,+}^n\}_{n=1}^{\infty}$
has a sub-subsequence with this property,
the full sequence has the property.
\end{proof}

\begin{corollary}\label{C8.6a}
For $k\geq 1$, every subsequence of
$\{\sVhatn\}_{n=1}^{\infty}$ has a sub-subsequence
$\{\sVhat^{n_j}\}_{j=1}^{\infty}$ for which
$\limsup_{n_j\rightarrow\infty}
\sup_{s\in[(R_{k,+}-\delta)^+,R_{k,+}+\delta]}
\big|\sVhat^{n_j}(s)-\kappa_L\big|
\rightarrow0$ almost surely
as $\delta\downarrow 0$.
\end{corollary}

\begin{proof}
Fix $k\geq 1$ and let $\ve\in(0,1)$ be given.
We choose $T$
so large that $\P\{R_{k,+}+1<T\}>1-\ve/2$.
According to Theorem \ref{T8.5}
and the definition (\ref{5.60}) of $V_{k,+}^n$,
we may choose $N$ so large that
$$
\P\Big\{\big|\sVhatn(R_{k,+}^n)-\kappa_L\big|
\leq \ve\Big\}>1-\ve/2
$$
for $n\geq N$. 
Now combine Lemma \ref{L5.8} with Proposition
\ref{P8.4}, setting $t_1=R_{k,+}^n$
and $\delta>0$
in inequality (\ref{8.15}) to conclude
$$
\big|\sVhatn(t_2)-\kappa_L\big|
\leq\ve+2w_T(C_V^n,\delta),\quad R_{k,+}^n\leq
t_2\leq R_{k,+}^n+\delta
$$
for $n\geq N$ on a set with probability
at least $1-\ve$.  We now use
the argument from (\ref{8.24g}) to (\ref{5.71}) in the
proof of Lemma \ref{L8.4a} to conclude that
along every subsequence we can choose a sub-subsequence
such that
\be\label{8.28c}
\limsup_{n_j\rightarrow\infty}
\sup_{s\in[R_{k,+}^{n_j},R_{k,+}^{n_j}+\delta]}
\big|\sVhat^{n_j}(s)-\kappa_L\big|
\stackrel{\P}{\rightarrow}0
\ee
as $\delta\downarrow 0$.
On the other hand, Theorem \ref{T8.5}
implies
\be\label{8.29c}
\sup_{t_2\in[\Lambda_{k,+}^n,R_{k,+}^n]}
\big|\sVhatn(t_2)-\kappa_L\big|
\stackrel{\P}{\rightarrow}0
\ee
as $n\rightarrow 0$.
Putting (\ref{8.28c}) and (\ref{8.29c})
together and using the facts that
$\Lambda_{k,+}^n\rightarrow \Lambda_{k,+}
<R_{k,+}$
and $R_{k,+}^n\rightarrow R_{k,+}$,
we obtain the desired result.
\end{proof}

\begin{remark}\label{R8.7}
{\rm
If $V^*(0)=\kappa_L$, then
every subsequence of $\{\sVhatn\}_{n=1}^{\infty}$
has a sub-subsequence $\{\sVhat^{n_j}\}_{j=1}^{\infty}$
such that
$\limsup_{n_j\rightarrow\infty}
\sup_{s\in[0,\delta]}
|\sVhat^{n_j}(s)-\kappa_L|\rightarrow0$ almost surely
as $\delta\downarrow 0$.  To see this,
replace $R_{k,+}^n$ by $0$ in the proof
of Corollary \ref{C8.6a}.
}
\end{remark}

\subsection{Assembling the pieces}\label{Assembling}

To assemble the pieces whose limits
have been identified in Sections
\ref{NegExc}--\ref{Positive},
we need to characterize convergence
in the $M_1$ topology in $D[0-,\infty)$.
Recall the definition of $D[0-,\infty)$
in Section \ref{Notation}.
The following theorem is a slight extension
\cite[Theorem~12.5.1(v)]{Whitt}.   
 The details
of this extension are in \cite[Appendix~A.2]{Almost}.

\begin{theorem}\label{T8.3}
Let $x\in D[0-,\infty)$ be given and let
$x^n\in D[0,\infty)$ be embedded in $D[0-,\infty)$
by defining $x^n(0-):=x^n(0)$.  The following
are equivalent.
\begin{description}
\item[(i)]
$x^n\rightarrow x$ in $(D[0-,\infty),M_1)$.
\item[(ii)]
$x^n(0-)\rightarrow x(0-)$, for each $t\geq0$
that is a continuity point of $x$,
\be\label{8.11}
\lim_{\delta\downarrow 0}
\limsup_{n\rightarrow\infty}
\sup_{t_1,t_2\in[(t-\delta)^+,t+\delta]}
\big|x^n(t_1)-x(t_2)\big|=0,
\ee
and for each $t\geq 0$ that is a discontinuity
point of $x$,
\be\label{8.12}
\lim_{\delta\downarrow 0}
\limsup_{n\rightarrow\infty}
\sup_{(t-\delta)^+\leq t_1<t_2<t_3\leq t+\delta}
\big|x^n(t_2)-\big[x^n(t_1),x^n(t_3)\big]\big|=0.
\ee
\end{description}
\end{theorem}

We define the candidate limit of
the sequence $\{\sVhatn\}_{n=1}^{\infty}$.
First recall from Assumption \ref{Assumption2}
that $\sVhatn(0)$ has a limit $V^*(0)$.
We define
\be\label{9.3}
\sV^*(0-):=V^*(0)
\ee
and for $t\geq 0$,
\be\label{9.4}
\sV^*(t):=
\left\{\begin{array}{ll}
\kappa_L+C_{k,-}(t-\Lambda_{k,-})
-\frac{\rho\sigma_+}{\sigma_-}
 E_{k,-}(t-\Lambda_{k,-})&\mbox{if }
t\in[\Lambda_{k,-},R_{k,-})\mbox{ for some }k\geq 1,\\
\kappa_L&\mbox{if }
t\in[\Lambda_{k,+},R_{k,+})\mbox{ for some }k\geq 1,\\
\kappa_L&\mbox{otherwise},
\end{array}\right.
\ee
where $C_{k,-}$ is defined by (\ref{5.57}).

\begin{proposition}\label{P9.2}
The process $\sV^*(t)$, $t\geq 0$, defined
by (\ref{9.4}), is c\`adl\`ag.
\end{proposition}

\begin{proof}
By definition, $\sV^*$ is continuous
in the interior of each excursion interval
of $G^*$  Fix 
$t\notin\cup_{k=1}^{\infty}\big((\Lambda_{k,+},R_{k,+})
\cup(\Lambda_{k_+},R_{k,+})\big)$.
We show that $\sV^*$ is right-continuous
with a left limit at $t$.  There are three cases.

\noi
\underline{Case I}: $t=\Lambda_{k,\pm}$ for some
$k\geq 1$.

Being the left endpoint of an excursion
of $G^*$, $t$ must be strictly positive.
By definition, $\sV^*$ is right-continuous
at $t$.  We show that in fact $\sV^*$
is continuous at $t$ by showing that
\be\label{9.5}
\lim_{s\uparrow t}\sV^*(s)=\kappa_L.
\ee
Immediately to the left of $t$, there
are infinitely many positive and negative
excursions of $G^*$.  On the positive
excursions, $\sV^*=\kappa_L$, and this is also
the case when $G^*=0$.  It remains to control
$\sV^*$ on the negative excursions of $G^*$.
On each of these excursions, $\sV^*=\kappa_L$
at the left endpoint and then diffuses.
Let $\delta\in(0,t)$ be given, and
choose a subsequence $\{k_j\}_{j=1}^{\infty}$
of the positive integers so that
$\{E_{k_j,-}\}_{j=1}^{\infty}$ is an enumeration
of the countably many negative excursions
of $G^*$ that begin in the interval $[t-\delta,t)$.
Denote by $X_j$ the maximum value of
$|\sV^*-\kappa_L|$ on the $k_j$-th negative
excursion, and denote its length by 
$\ell_j:=R_{k_j,-}-\Lambda_{k_j,-}$,
so that
\be\label{Xj}
X_j=\max_{0\leq s\leq \ell_j}\left|
C_{k_j,-}(s)
-\frac{\rho\sigma_+}{\sigma_-} 
E_{k_j,-}(s)\right|,\quad j=1,2,\dots.
\ee
The quadratic variation of $C_{k_j,-}$
is
$(1-\rho^2)\sigma_+^2\id$
(see (\ref{5.59}))
and the quadratic variation of
$E_{k_j,-}$ is $\langle Z_-,Z_-\rangle=\sigma_-^2\id$
(see (\ref{rho}), (\ref{5.44a}), and Remark \ref{R5.12}).
Set
$\beta:=
\sqrt{1-\rho^2}\,\sigma_+$.
Then
$B(u):=\frac{1}{\beta\sqrt{\ell_j}}C_{k_j,-}
\big(\ell_ju\big)$ 
is a standard Brownian motion and
$E(u):=\frac{1}{\sigma_-\sqrt{\ell_j}}
E_{k_j,-}(\ell_ju)$
is an excursion of length one of a standard
Brownian motion.  We compute
\begin{align*}
\E\big[X_j^2\big]
&=
\E\Big[\max_{0\leq s\leq \ell_j}
\big(C_{k_j,-}(s)-\frac{\rho\sigma_+}{\sigma_-}
 E_{k_j,-}(s)\big)^2\Big]\\
&\leq
2\E\Big[\max_{0\leq s\leq\ell_j}C_{k_j,-}^2(s)\Big]
+2\E\Big[\max_{0\leq s\leq\ell_j}
\Big(\frac{\rho\sigma_+}{\sigma_-}
 E_{k_j,-}(s)\Big)^2\Big]\\
&\leq
2\beta^2\ell_j
\E\Big[\max_{0\leq u\leq 1}B_j^2(u)\Big]
+2\rho^2\sigma_+^2\ell_j
\E\Big[\max_{0\leq u\leq 1}E^2(u)\Big]\\
&=
K\ell_j
\end{align*}
for a constant $K$ independent of $j$.
For $\ve>0$, we have
$$
\sum_{j=1}^{\infty}\P\{X_j>\ve\}
\leq
\frac{1}{\ve^2}\sum_{j=1}^{\infty}\E\big[X_j^2\big]
\leq \frac{K}{\ve^2}\sum_{j=1}^{\infty}\ell_j
\leq \frac{Kt}{\ve ^2}<\infty.
$$
By the first Borel-Cantelli Lemma,
$\P\{X_j>\ve\mbox{ infinitely often}\}=0$.
Thus, there is an interval of the form
$(t-\nu,t)$ for some positive $\nu$
so that $|\sV^*(s)-\kappa_L|\leq\ve$
for all $s\in(t-\nu,t)$.  Since $\ve>0$
is arbitrary, (\ref{9.5}) holds.

\noi
\underline{Case II}: $t=R_{k,\pm}$ for some
$k\geq 1$.

In this case $\sV^*$ has a limit from
the left at $t$.  We must prove right continuity.
Since no right endpoint of an excursion of $G^*$ is
the left endpoint of an excursion, 
$\sV^*(t)=\kappa_L$ by definition.  Immediately
to the right of $t$, there are infinitely
many positive and negative excursions of
$G^*$.  On the positive excursions, $\sV^*=\kappa_L$.
We only need to control $\sV^*$ on the negative
excursions, and for this we use the argument
of Case I.

\noi
\underline{Case III}:
$t\notin\cup_{k=1}^{\infty}([\Lambda_{k,+},R_{k,+}]
\cup[\Lambda_{k,-},R_{k,-}])$.

In thus case, $\sV^*(t)=\kappa_L$ by definition.
There are infinitely many excursions of $G^*$
immediately to the right of $t$, and if $t>0$,
there are also infinitely many excursions
immediately to the left of $t$.
We use the argument of Case I to prove continuity
of $\sV^*$ at $t$.
\end{proof}

The proof of Theorem \ref{P9.2} and
(\ref{9.3}) and (\ref{9.4})
establish the following corollary.

\begin{corollary}\label{C9.3}
The process $\sV^*$ defined by (\ref{9.4})
is continuous on $(0,\infty)$ except on the set
$\cup_{k=1}^{\infty}\{R_{k,-}\}$.  At each
of the points in this set, $\sV^*$
is discontinuous almost surely.
The process $\sV^*$ is continuous at $0$
if and only if $V^*(0)=\kappa_L$.
\end{corollary}

\begin{theorem}\label{T9.4}
We embed the paths of $\sV^n$ into
$D[0-,\infty)$ by the device used
in Theorem \ref{T8.3}.  Then under
the Skorohod Representation assumption
of Remark \ref{R5.10x} invoked
prior to Proposition \ref{P5.12}, we
have $\sV^n\rightarrow \sV^*$ almost surely
in the $M_1$ topology on $D[0-,\infty)$.
\end{theorem}

\begin{proof}
We begin with an arbitrary subsequence
of $\{\sVhatn\}_{n=1}^{\infty}$.
From this subsequence, we can choose a sub-subsequence
$\{\sVhat^{n_j}\}_{j=1}^{\infty}$
such that $\{C^{n_j}_V\}_{n=1}^{\infty}$
in Lemma \ref{L5.8} converges weakly-$J_1$
to a continuous limit $C_V^*$,
the conclusions of Lemma \ref{L8.4a},
Corollary \ref{C8.6a}
and Remark \ref{R8.7} hold,
and the convergence in Theorem \ref{T8.5}
is almost sure.  For
$k\geq 0$, we have almost surely
\begin{align}
\lim_{\delta\downarrow 0}
\limsup_{n_j\rightarrow\infty}
\sup_{s\in[(\Lambda_{k,\pm}-\delta)^+,
\Lambda_{k,\pm}+\delta]}
\big|\sVhat^{n_j}(s)-\kappa_L\big|
&=
0,\label{9.6a}\\
\lim_{n_j\rightarrow\infty}\sup_{t\geq 0}
\big|V_{k,+}^{n_j}(t)-\kappa_L\big|
&=
0,\label{9.7a}
\end{align}
\begin{align}
\lim_{\delta\downarrow 0}
\limsup_{n_j\rightarrow\infty}
\sup_{s\in[(R_{k,+}-\delta)^+,R_{k,+}+\delta]}
\big|\sVhat^{n_j}(s)-\kappa_L\big|
&=
0,\label{9.8a}\\
\lim_{\delta\downarrow 0}
\limsup_{n_j\rightarrow\infty}
\sup_{s\in[0,\delta]}
\big|\sVhat^{n_j}(s)-\kappa_L\big|
&=
0.\label{9.9a}
\end{align}

We verify that $\{\sVhat^{n_j}\}_{j=1}^{\infty}$
and $\sV^*$
satisfy the criterion (ii) of Theorem
\ref{T8.3}.
The definition (\ref{9.3}) of $\sV^*(0-)$
is chosen to guarantee convergence at $0-$,
the first part of criterion (ii).
We next consider (\ref{8.11})
at continuity points of $\sV^*$.  In particular,
we must show that if $t$ is a continuity
point of $\sV^*$, then
\be\label{9.6}
\lim_{\delta\downarrow 0}
\limsup_{n_j\rightarrow\infty}
\sup_{t_1,t_2\in[(t-\delta)^+,t+\delta]}
\big|\sVhat^{n_j}(t_1)-\sV^*(t_2)\big|=0.
\ee
Zero is a continuity point of $\sV^*$
if and only if $V^*(0-)=\kappa_L$.
In this case,
(\ref{9.6}) holds at $t=0$
by (\ref{9.9a}).
If $t$ is strictly positive, it 
is a continuity point of $\sV^*$ if and only
if one of the following cases holds:

\vspace{-12pt}
\begin{tabbing}
$\hphantom{(iii)}$\=\\
(i)\>$t=\Lambda_{k,\pm}$ for some $k\geq 1$;\\
(ii) \>$t=R_{k,+}$ for some $k\geq 1$;\\
(iii) \>$t\in(\Lambda_{k,-},R_{k,-})$ for some $k\geq 1$\\
(iv) \>$t\in(\Lambda_{k,+},R_{k,+})$ for some $k\geq 1$.
\end{tabbing}

\noi
In case (i), (\ref{9.6}) follows from
(\ref{9.6a}), the fact that
$\sV^*(\Lambda_{k,\pm})=\kappa_L$, and the continuity
of $\sV^*$ at $\Lambda_{k,\pm}$.
For case (ii), the result follows from
(\ref{9.8a}),
the fact that $\sV^*(R_{k,_-})=\kappa_L$, 
and the continuity
of $\sV^*$ at $R_{k,\pm}$.
For case (iii), the result follows from 
the consequence of 
(\ref{5.54}) and (\ref{9.6a}) that
$\sVhatn(\Lambda_{k,-}^n)\rightarrow\kappa_L$,
relations (\ref{5.55}) and (\ref{5.58}),
and the continuity of $\sV^*$ on
$(\Lambda_{k,-},R_{k,-})$.
For case (iv), (\ref{9.6}) follows from
(\ref{5.60}) and (\ref{9.7a}).

We turn our attention to discontinuity points
of $\sV^*$, i.e., points of the form 
$t=0$ (if $V^*(0)\neq\kappa_L$) 
or $t=R_{k,-}$ for some $k\geq 1$.
We must show that (cf. (\ref{8.12}))
$$
\lim_{\delta\downarrow 0}
\limsup_{n_j\rightarrow\infty}
\sup_{(t-\delta)^+\leq t_1<t_2<t_3\leq t+\delta}
\big[\sVhat^{n_j}(t_2)
-\big[\sVhat^{n_j}(t_1),\sVhat^{n_j}(t_3)\big]
\big|=0.
$$
But according to (\ref{8.14}) of Proposition \ref{P8.4}
and Lemma \ref{L5.8},
$$
\sup_{(t-\delta)^+\leq t_1<t_2<t_3\leq t+\delta}
\big[\sVhat^{n_j}(t_2)
-\big[\sVhat^{n_j}(t_1),\sVhat^{n_j}(t_3)\big]
\big|\leq 4w_T(C_V^{n_j},\delta)
$$
for every $t\geq 0$, which yields
$$
\E\Big[\limsup_{n_j\rightarrow\infty}
\sup_{(t-\delta)^+\leq t_1<t_2<t_3\leq t+\delta}
\big[\sVhat^{n_j}(t_2)
-\big[\sVhat^{n_j}(t_1),\sVhat^{n_j}(t_3)\big]
\big|\wedge 1\Big]
\leq
\E\big[4w_T(C_V^*,\delta)\wedge 1\big].
$$
The right-hand side converges to $0$
as $\delta\downarrow 0$, hence
$$
Y(\delta):=\limsup_{n_j\rightarrow\infty}
\sup_{(t-\delta)^+\leq t_1<t_2<t_3\leq t+\delta}
\big[\sVhat^{n_j}(t_2)
-\big[\sVhat^{n_j}(t_1),\sVhat^{n_j}(t_3)\big]
\big|\stackrel{\P}{\rightarrow}0
$$
as $\delta\downarrow 0$.
Thus there is a subsequence $\delta_i\downarrow 0$
along which $Y(\delta_i)$ converges almost surely
to zero.  But $Y(\delta)$ is monotone in $\delta$,
and hence $Y(\delta)\rightarrow 0$ almost surely
as $\delta\downarrow 0$.

We have shown that every subsequence of
$\{\sVhatn\}_{n=1}^{\infty}$
has a sub-subsequence that converges
weakly-$M_1$ to $\sV^*$.  It follows that
the full sequence $\{\sVhatn\}_{n=1}^{\infty}$
converges weakly-$M_1$ to $\sV^*$.
\end{proof}

\subsection{Convergence of $\sYhatn$}\label{ConvYn}

We state without proof the convergence result for
$\sYhatn$ analogous to Theorem \ref{T9.4}
proved for $\sVhatn$.
In addition to the thirty Poisson processes
introduced in Section \ref{Interior}
to govern the interior queues
and the additional nine Poisson processes
introduced in Section \ref{Governing} to
govern $\sVhatn$, we need nine additional
Poisson processes
$N_{SW,Y,+}$, $N_{SE_-,Y,-}$, $N_{SE,Y,-}$,
$N_{SE_+,Y,-}$, $N_{E,Y,-}$, $N_{E,Y,+}$, $N_{NE,Y,-}$,
$N_{NE,Y,+}$ and $N_O,Y,+$.
These forty-eight Poisson processes
are taken to be independent.  In terms of the
last nine, from Figure 4.1 we obtain the
formula (cf. (\ref{5.1}))
\begin{align}
\sY^n(t)
&=
Y^n(0)+N_{SW,Y,+}\left(\int_0^t\frac{1}{\sqrt{n}}
\theta_s\big(\sY^n(s)\big)^-\right)dP_{SW}^n(s)
-N_{SE_-,Y,-}\circ\mu_2P^n_{SE_-}(t)\nonumber\\
&\qquad
-N_{SE,Y,-}\circ\mu_2P^n_{SE}(t)
-N_{SE_+,Y,-}\circ\mu_2P^n_{SE_+}(t)
-N_{E,Y,-}\circ\mu_2P^n_E(t)\nonumber\\
&\qquad
+N_{E,Y,+}\circ\lambda_0P^n_E(t)
-N_{NE,Y,-}\circ\mu_1P^n_{NE}(t)
+N_{NE,Y,+}\circ\lambda_0P^n_{NE}(t)\nonumber\\
&\qquad
+N_{O,Y,+}\circ\lambda_0P^n_O(t).\label{5.86}
\end{align}
We define $\sYhatn(t):=\frac{1}{\sqrt{n}}\sY^n(nt)$ and
(cf.\ (\ref{9.4}))
\be\label{5.87}
\sY^*(t):=
\left\{\begin{array}{ll}
\kappa_R+C_{k,+}(t-\Lambda_{k,+})
-\frac{\rho\sigma_-}{\sigma_+}E_{k,+}(t-\Lambda_{k,+})
&\mbox{if }t\in[\Lambda_{k,+},R_{k,+})
\mbox{ for some }k\geq 1,\\
\kappa_R&\mbox{if }t\in[\Lambda_{k,-},R_{k,-})
\mbox{ for some }k\geq 1,\\
\kappa_R&\mbox{otherwise },
\end{array}\right.
\ee
where $\kappa_R$ is defined by (\ref{kappa})
and $\{C_{k,+}\}_{k=1}^{\infty}$ is an independent
sequence of processes independent of $G^*$
and $\{C_{k,-}\}_{k=1}^{\infty}$ in (\ref{9.4}).
Furthermore, each $C_{k,+}$ is a Brownian stopped
at time $R_{k,+}-\Lambda_{k,+}$,
and the quadratic variation of each $C_{k,+}$
is $(1-\rho^2)\sigma_-^2\id$.
The process
$\sY^*$ is c\`adl\`ag (cf.\ Proposition \ref{P9.2}).
We have the following analogue to
Theorem \ref{T9.4}.  

\begin{theorem}\label{T5.24}
We embed the paths of $\sYhatn$ into $D[0-,\infty)$
by the device used in Theorem \ref{T8.3}.
Then under the Skorohod Representation assumption
of Remark \ref{R5.10x} and a Skorohod
Representation analogous to the one invoked
prior to Proposition \ref{P5.12}, we have
$\sYhatn\rightarrow\sY^*$ almost surely
in the $M_1$ topology on $D[0-,\infty)$.
\end{theorem}

\subsection{Convergence of 
$(\sVhatn,\sWhatn,\sXhatn,\sYhatn)$}
\label{ConvVWXY}
By use of the Skorohod Representation Theorem,
we have established the existence of a probability
space on which $(\sWhatn, \sXhatn)$
converges almost surely in the $J_1$-topology
to the split Brownian motion $(\sW^*,\sX^*)$
of Remark \ref{R4.20x}
(see Theorem \ref{T4.16}
and Remark \ref{R5.10x}), $\sVhatn$
converges almost surely in the $M_1$-topology
to $\sV^*$ of (\ref{9.4}) (see Theorem \ref{T9.4}),
and $\sYhatn$ converges almost surely
in the $M_1$-topology to $\sY^*$
of (\ref{5.87}) (see Theorem \ref{T5.24}).
In particular, if 
$f:D[0-,\infty)\times D[0,\infty)\times D[0,\infty)
\times D[0-,\infty)\rightarrow \R$ is bounded
and continuous in the 
$M_1\times J_1\times J_1\times M_1$ topology, then
\be\label{5.88}
\lim_{n\rightarrow\infty}
\E f\big(\sVhatn,\sWhatn,\sXhatn,\sYhatn)
=\E f\big(\sV^*,\sW^*,\sX^*,\sY^*).
\ee
But the laws of $(\sVhatn,\sWhatn,\sXhatn,\sYhatn)$
and $(\sV^*,\sW^*,\sX^*,\sY^*)$
do not depend on the probability space
on which they are constructed.  Thus, we may
dispense with the use of the Skorohod
Representation Theorem and summarize the conclusion
of Sections \ref{Interior queues} and 
\ref{BracketingQueues} by the following theorem,
an equivalent way of stating (\ref{5.88}).

\begin{theorem}\label{T.VWXY}
We have 
$D_{M_1}[0-,\infty)\times D_{J_1}[0,\infty)
\times D_{J_1}[0,\infty)\times D_{M_1}[0,\infty)$-weak
convergence of $(\sVhatn,\sWhatn,\sXhatn,\sYhatn)$
to $(\sV^*,\sW^*,\sX^*,\sY^*)$.
\end{theorem}

\section{When bracketing queues vanish}\label{Vanish}

\setcounter{equation}{0}
\setcounter{theorem}{0}

\subsection{Introduction}
Consider the sequence of four-dimensional processes
$\{(V^n,W^n,X^n,Y^n)\}_{n=1}^{\infty}$
with initial conditions satisfying
Assumption \ref{Assumption2} and governed
by Figure 4.1, or equivalently,
agreeing with the sequence of four-dimensional processes
$\{(\sV^n,\sW^n,\sX^n,\sY^n)\}_{n=1}^{\infty}$
given by (\ref{3.10})--(\ref{3.12}), (\ref{5.1}),
and (\ref{5.86}) up to and including at the
stopping time $S^n$ of (\ref{3.1}).

We have scaled these processes to obtain
$$
\big(\sVhatn(t),\sWhatn(t),\sXhatn(t),\sYhatn(t)\big)
:=\left(\frac{1}{\sqrt{n}}\sV^n(nt),
\frac{1}{\sqrt{n}}\sW^n(nt),
\frac{1}{\sqrt{n}}\sX^n(nt),
\frac{1}{\sqrt{n}}\sY^n(nt)\right).
$$
Consistent with this, we scale the stopping times $S^n$
to obtain
\be\label{Shatn}
\Shatn:=\inf\big\{t\geq 0:
\sVhatn(t)=0\mbox{ or }\sYhatn(t)=0\big\}
=\frac{1}{n}S^n.
\ee
For the limiting four-dimensional processes
$(\sV^*,\sW^*,\sX^*,\sY^*)$ we define
the stopping time
\be\label{Sstar}
S^*:=\inf\big\{t\geq 0:\sV^*(t)=0\mbox{ or }
\sY^*(t)=0\big\}.
\ee

The first goal of this section is to prove
the following theorem. 

\begin{theorem}\label{T6.1}
We have 
$D_{M_1}[0-,\infty)\times D_{J_1}[0,\infty)
\times D_{J_1}[0,\infty)\times D_{M_1}[0-,\infty)$-weak
convergence of the stopped process
$(\sVhatn(\cdot\wedge \Shatn),\sWhatn(\cdot\wedge \Shatn),
\sXhatn(\cdot\wedge \Shatn),\sYhatn(\cdot\wedge \Shatn))$
to the stopped process $(\sV^*(\cdot\wedge S^*),\sW^*(\cdot\wedge S^*),
\sX^*(\cdot\wedge S^*),\sY^*(\cdot\wedge S^*))$.
\end{theorem}

Let us further define
\be\label{6.2x}
S^*_v:=
\inf\big\{t\geq 0:\sV^*(t)=0\big\},\quad
S^*_y:=
\inf\big\{t\geq 0:\sY^*(t)=0\big\},
\ee
so that $S^*=S^*_v\wedge S^*_y$.
Observe that $\P\{S^*_v=S^*_y\}=0$.
The second task of this section is to
expand the window to
consider the queue length process
$U^n$ immediately to the left of $V^n$
and the queue length process $Z^n$ immediately
to the right of $Y^n$.  
We prove the following.

\begin{theorem}\label{T6.2}
We have
\begin{align}
\lefteqn{\left(
\frac{1}{\sqrt{n}}U^n(S^n),
\frac{1}{\sqrt{n}}V^n(S^n),
\frac{1}{\sqrt{n}}W^n(S^n),
\frac{1}{\sqrt{n}}X^n(S^n),
\frac{1}{\sqrt{n}}Y^n(S^n),
\frac{1}{\sqrt{n}}Z^n(S^n)\right)}\qquad\nonumber\\
&\Longrightarrow
(\kappa_L,0,0,\sX^*(S^*),\kappa_R,0)\ind_{\{S^*_v<S^*_y\}}
+(0,\kappa_L,\sW^*(S^*),0,0,\kappa_R)
\ind_{\{S^*_y<S^*_v\}},\label{6.3y}
\end{align}
and on the set $\{S^*_v<S^*_y\}$,
we have $\sX^*(S^*)<0$, whereas on the set
$\{S^*_y<S^*_v\}$, we have $\sW^*(S^*)>0$.
The convergence is weak convergence of
probability measures on $\R^6$.
\end{theorem}

Recall that $\kappa_L$ defined by (\ref{kappa})
is positive and $\kappa_R$ defined by
(\ref{kappa}) is negative.
In the case that $S^*_v<S^*_y$,
we say that there is a {\em price decrease}
from the initial condition of Assumption \ref{Assumption2},
and we continue after time $S^n$
by designating $U^n$ and $X^n$
the bracketing queues and $V^n$ and $W^n$
the interior queues.  In particular,
Assumption \ref{Assumption2} is satisfied
with
$(0,\frac{1}{\sqrt{n}}U^n(S^n),\frac{1}{\sqrt{n}}V^n(S^n),
\frac{1}{\sqrt{n}}W^n(S^n),\frac{1}{\sqrt{n}}X^n(S^n),
\frac{1}{\sqrt{n}}Y^n(S^n))$
replacing
$(\frac{1}{\sqrt{n}}U^n(0),
\frac{1}{\sqrt{n}}V^n(0),\frac{1}{\sqrt{n}}W^n(0),
\frac{1}{\sqrt{n}}X^n(0),\frac{1}{\sqrt{n}}Y^n(0),
\frac{1}{\sqrt{n}}Z^n(0))$
in that assumption.
On the other hand, if $S^*_y<S^*_v$,
there is a {\em price increase}, and we
continue after time $S^n$ by designating
$W^n$ and $Z^n$ the bracketing queues and
$X^n$ and $Y^n$ the interior queues.
In this case, Assumption \ref{Assumption2}
is satisfied with the
$(\frac{1}{\sqrt{n}}U^n(0),
\frac{1}{\sqrt{n}}V^n(0),\frac{1}{\sqrt{n}}W^n(0),
\frac{1}{\sqrt{n}}X^n(0),\frac{1}{\sqrt{n}}Y^n(0),
\frac{1}{\sqrt{n}}Z^n(0))$
replaced by 
$(\frac{1}{\sqrt{n}}V^n(S^n),
\frac{1}{\sqrt{n}}W^n(S^n),\frac{1}{\sqrt{n}}X^n(S^n),
\frac{1}{\sqrt{n}}Y^n(S^n),\frac{1}{\sqrt{n}}Z^n(S^n),
0)$.
In either case, the analysis of the previous
sections applies, where we replace the initial
time zero in Assumption \ref{Assumption2}
by $S^n$.  By this device of restarting
at price changes, we iteratively construct
the limiting processes for all time.

\begin{remark}
{\rm 
We do not address convergence of the full process
$\frac{1}{\sqrt{n}}U^n$, but rather only
its behavior on the interior of positive
and negative excursions of $G^n$.
In fact, Lemma \ref{L6.7} shows that
$\frac{1}{\sqrt{n}}U^n(n\id)$ converges to
zero almost surely on positive excursions of $G^n$
and Lemma \ref{L6.8} proves convergence to
$\kappa_L$ on negative excursions.  Such a
sequence of processes
cannot have a c\`adl\`ag limit.
Analogous statements can be made about
$\frac{1}{\sqrt{n}}Z^n(n\id)$.
}
\end{remark}

The remainder of this section is devoted
to the proofs of Theorem \ref{T6.1}
and \ref{T6.2}.  Of course, to consider
Theorem \ref{T6.2}, we must first define
the processes $U^n$ and $Z^n$.  We define
$U^n$ in Section \ref{SubsecT6.2} below
and appeal to analogy for the definition
of $Z^n$.

\subsection{Proof of Theorem \ref{T6.1}}\label{SubsecT6.1}

Once again, throughout this section and without further
mention, we use the Skorohod Representation
Theorem to put all processes on a common probability
space so that the convergence in Theorem \ref{T.VWXY}
is almost sure.  We then prove Theorems \ref{T6.1}
and \ref{T6.2} by establishing almost sure convergence,
or in some cases, convergence in probability.

\begin{lemma}\label{L6.3}
For every $\delta>0$,
\be\label{6.3}
\inf_{0\leq t\leq(S^*-\delta)^+}\sV^*(t)>0,\quad
\sup_{0\leq t\leq(S^*-\delta)^+}\sY^*(t)<0,
\ee
and for all sufficiently large $n$,
\be\label{6.4}
\inf_{0\leq t\leq(S^*-\delta)^+}\sVhatn(t)>0,\quad
\sup_{0\leq t\leq(S^*-\delta)^+}\sYhatn(t)<0.
\ee
\end{lemma}

\begin{proof}
We prove the first inequalities in (\ref{6.3})
and (\ref{6.4}).  The second
are analogous.

Except on negative excursions of $G^*$,
$\sV^*=\kappa_L>0$.  
On each negative excursion of $G^*$, 
$\sV^*$ and $G^*$ behave like
non-perfectly correlated Brownian motions
because $C_{k,-}$ in (\ref{9.4}) is independent
of the excursion $E_{k,-}$ of $G^*$.  There is thus zero
probability that a negative excursion of $G^*$
ends at the time when the left limit of $\sV^*$ is
zero.  In other words, on every negative excursion
interval
of $G^*$ intersected with $[0,(S^*-\delta)^+]$, $\sV^*$ is bounded
away from zero.  This is also the case for every
finite set of negative excursion intervals of $G^*$.
If the number of excursions of $G^*$ beginning
before time $(S^*-\delta)^+$ is not finite,
choose a subsequence $\{k_j\}_{j=1}^{\infty}$
of positive integers such that $\{E_{k_j,-}\}_{j=1}^{\infty}$
is an enumeration of all the negative
excursions of $G^*$ beginning by time $S^*-\delta$.
Define $X_j$ as in (\ref{Xj}) in the proof of
Proposition \ref{P9.2}.  Following that proof, 
we conclude that
$\P\{X_j>\kappa_L/2\mbox{ infinitely often}\}=0$.
This shows that $\sV^*$ is bounded away from
zero on $[0,(S^*-\delta)^+]$.

Convergence in the $M_1$ topology is uniform convergence
of function graphs (see the proof of Lemma \ref{L6.5}
for more detail) on compact time intervals.  We have shown
that the graph of $\sV^*$ is bounded away from zero
on $[0,(S^*-\delta)^+]$, and so must the graph of
$\sVhatn$ for sufficiently large $n$ also be bounded
away from zero.  
\end{proof}

We define
$\Shatn_v:=\inf\{t\geq 0: \sVhatn(t)=0\big\}$ and
$\Shatn_y:=\inf\{t\geq 0: \sYhatn(t)=0\big\}$,
so that $\Shatn=\Shatn_v\wedge \Shatn_y$.  
Because $\sVhatn$ and $\sYhatn$ are driven
by independent Poisson processes, there
is zero probability they arrive simultaneously
at zero, i.e.,
$\P\{\Shatn_v=\Shatn_y\}=0$.

\begin{lemma}\label{L6.4}
Almost surely as $n\rightarrow\infty$,
$\Shatn_v\rightarrow S^*_v$, 
$\Shatn_y\rightarrow S^*_y$, and
$\Shatn\rightarrow S^*$.

\end{lemma}

\begin{proof}
In order for $\sV^*$ defined by (\ref{9.4})
to reach zero, $G^*$ must be on a negative excursion
and hence $\sW^*$ must be negative (Remark \ref{R4.20x}).
 According to (\ref{9.4}) (or see
Corollary \ref{C9.3}) and the fact that
there is zero probability a negative excursion
of $G^*$ ends at time $S^*_v$, $\sV^*$ is continuous
in a neighborhood of $S^*_v$. 
Having Brownian
motion paths,
$\sV^*$ takes negative values
in $(S^*_v,S^*_v+\delta)$ for every $\delta>0$.  According to
\cite[Theorem~12.5.1(v)]{Whitt}, convergence of $\sVhatn$
to $\sV^*$ in the $M_1$-topology implies local
uniform convergence of $\sVhatn$ to $\sV^*$
at $S^*_v$, and thus for every $\delta>0$
and $n$ sufficiently large,
$\sVhatn$ takes negative
values in the time interval $(S^*_v,S^*_v+\delta)$. 
Lemma \ref{L6.3} implies
that for sufficiently large $n$, $\sVhatn$ is
strictly positive on the time interval $[0,S^*-\delta]$.
It follows that for every $\delta>0$,
$S^*_v-\delta<\Shatn<S^*_v+\delta$ 
for all sufficiently large $n$.
This establishes the convergence 
$\Shatn_v\rightarrow S^*_v$.

The proof that $\Shatn_y\rightarrow S^*_y$ is analogous.
The first two convergences in the lemma
and the equations $\Shatn=\Shatn_v\wedge \Shatn_y$,
$S^*=S^*_v\wedge S^*_y$ establish the third one.
\end{proof}

\begin{lemma}\label{L6.5}
Let $\{x^n\}_{n=1}^{\infty}$ be a sequence of functions
in $D[0-,\infty)$ converging in the $M_1$ topology
to a function $x^*\in D[0-,\infty)$.  Let $s^n$ be a sequence
of nonnegative numbers converging to $s^*$.
If $s^*$ is a continuity point of $x^*$, then
$x^n(\cdot\wedge s^n)$ converges to $x^*(\cdot\wedge s^*)$
in the $M_1$ topology.
\end{lemma}
\begin{proof}
The graph of $x^*$ is defined to be
$$
\Gamma_{x^*}:=\big\{(z,t)\in\R\times[0,\infty):
z\in[x^*(t-)\wedge x^*(t),x^*(t-)\vee x^*(t)]\big\}.
$$
A parametrization of the graph of $x^*$ is 
a pair of continuous functions $(u^*,r^*)$ such that 
$r^*(0)=0$, $u^*(0)=x^*(0-)$,
and the pair $(u^*,r^*)$ is a bijection of
$[0,\infty)$ onto $\Gamma_x$.  In particular, $r^*$
is nondecreasing and strictly increasing
at continuity points of $x^*$.
The graph of $x^n$ and a parametrization of the graph
of $x^n$ are defined in the same way.
Convergence of $x^n$ to $x$ in the $M_1$ topology
is equivalent to the existence, for each $n$,
of a parametrization
$(u^n,r^n)$ of the graph of $x^n$, and the existence
of a parametrization
$(u^*,r^*)$ of the graph of $x^*$, such that
\be\label{6.6}
\lim_{n\rightarrow \infty}
\max_{0\leq t\leq T}\big(\big|u^n(t)-u^*(t)\big|
+\big|r^n(t)-r^*(t)\big|\big)=0
\ee
for every $T\in(0,\infty)$.
Given such parametrizations, define
$t^n:=\max\big\{t\geq 0: r^n(t)=s^n\big\}$.
Because $r^*$ is strictly increasing at continuity
points of $x^*$, there is a unique $t^*$
such that $r^*(t^*)=s^*$.
Then $(r^n(\cdot\wedge t^n)+(\cdot-t^n)^+,u^n(\cdot\wedge t^n))$
is a parametrization of the graph of $x^n(\cdot\wedge s^n)$
and $(r^*(\cdot\wedge t^*)+(\cdot-t^*)^+,u^*(\cdot\wedge t^*))$
is a parametrization of the graph of $x^*(\cdot\wedge s^*)$.
It remains to show that
\begin{align}
\lim_{n\rightarrow\infty}
\max_{0\leq t\leq T}
\big|u^n(t\wedge t^n)-u^*(t\wedge t^*)\big|
&=
0,\label{6.7}\\
\lim_{n\rightarrow\infty}
\max_{0\leq t\leq T}
\big|r^n(t\wedge t^n)+(t-t^n)^+
-r^*(t\wedge t^*)-(t-t^*)^+\big|
&=
0,\label{6.8}
\end{align}
for every $T\in(0,\infty)$.  In fact,
the sequence $\{t^n\}_{n=1}^{\infty}$ is bounded
because $\{r(t^n)\}_{n=1}^{\infty}=\{s^n\}_{n=1}^{\infty}$ is bounded,
$r^*(t)\rightarrow\infty$ as $t\rightarrow\infty$,
and $r^n$ converges to $r^*$ uniformly on
compact time intervals.  Therefore, we may omit
the operator $\max_{0\leq t\leq T}$ in (\ref{6.7}) and
(\ref{6.8}).

Because
$r^*$ is strictly 
increasing at $t^*$, for every $\delta>0$,
there exists $\varepsilon>0$ such that
\be\label{6.9}
r^*(t)\leq r^*(t^*)-\varepsilon\mbox{ for }t\leq t^*-\delta,
\quad
r^*(t)\geq r^*(t^*)+\varepsilon\mbox{ for }t\geq t^*+\delta.
\ee
But $r(t^n)=s^n$ and $r(t^*)=s^*$, so
\be\label{6.10}
\big|r^*(t^n)-r(t^*)\big|
\leq
\big|r^*(t^n)-r^n(t^n)\big|
+\big|s^n-s^*\big|.
\ee
By choosing $n$ sufficiently
large, we can use (\ref{6.6}) to make
the first term on the right-hand side of
(\ref{6.10}) less than $\varepsilon/2$, and we can also make
the second term less than $\varepsilon/2$.
For such $n$, (\ref{6.9}) implies
$t^*-\delta<t^n<t^*+\delta$.  In other words,
$t^n\rightarrow t^*$ as $n\rightarrow\infty$.
This implies
$$
\big|u^n(t\wedge t^n)-u^*(t\wedge t^*)\big|
\leq\big|u^n(t\wedge t^n)-u^*(t\wedge t^n)\big|
+\big|u^*(t\wedge t^n)-u^*(t\wedge t^*)\big|,
$$
and both terms on the right-hand side have limit
zero, the first by (\ref{6.6}) and the second
by continuity of $u^*$.  This establishes
(\ref{6.7}).  The proof of (\ref{6.8}) is similar.
\end{proof}

\noi
{\em Proof of Theorem \ref{T6.1}.}
We use Lemmas \ref{L6.4} and \ref{L6.5} and
the fact that $G^*$ is almost surely non-zero
at $S^*$, and hence $S^*$ is a continuity point of
$\sV^*$ and $\sY^*$.
The processes $\sW^*$ and $\sX^*$ are continuous.
$\hfill\Box$

\subsection{Proof of Theorem \ref{T6.2}}\label{SubsecT6.2}

Theorem \ref{T6.2} requires that we broaden
the window to consider the processes $U^n$
to the left of $V^n$ and $Z^n$ to the right of $Y^n$.
We focus on $U^n$. The discussion for $Z^n$
would be analogous.  For this purpose,
we modify Figure 4.1 to show the forces
acting on $U^n$ in each of the eight possible
configurations of $(W^n,X^n)$ when $V^n$ is positive
and $Y^n$ (not shown in Figure 6.1) is negative.  

\begin{figure}[h]
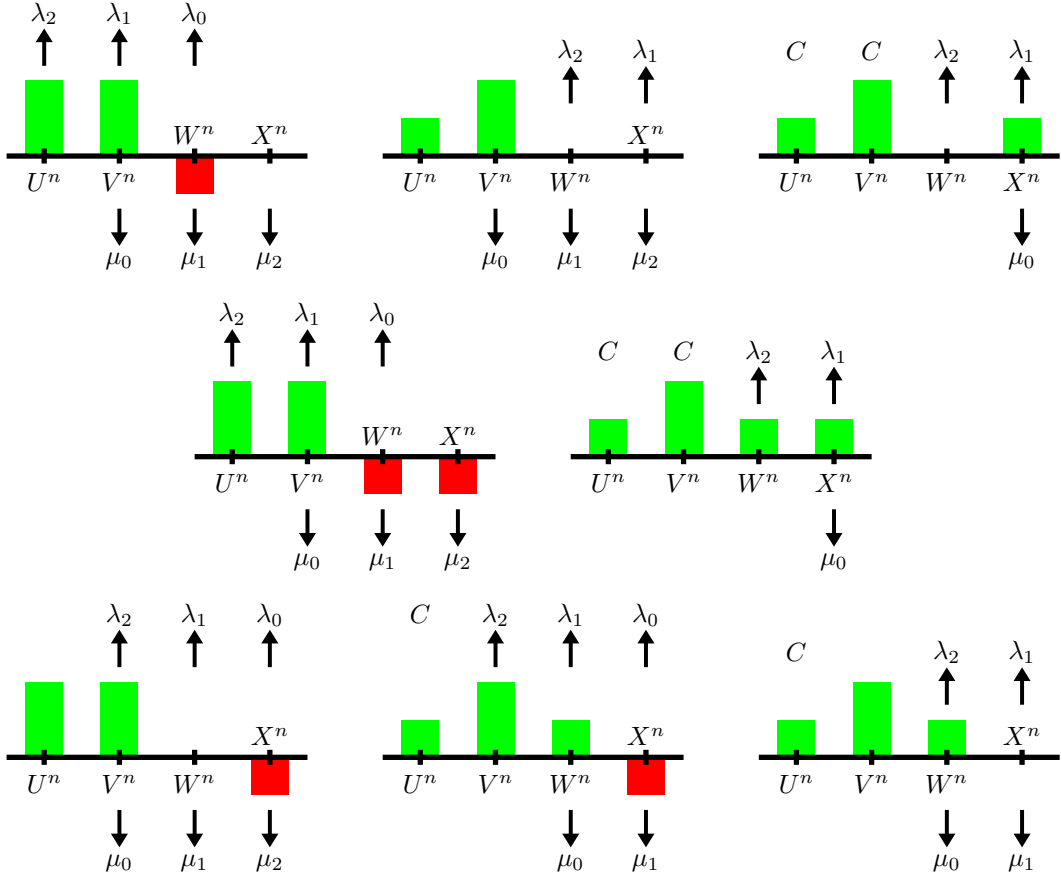
\label{F6.1}
\begin{pgfpicture}{0cm}{-0.2cm}{14cm}{11cm}
%
%
\begin{pgfscope}
\color{green}
\pgfmoveto{\pgfxy(0.25,1.00)}
\pgflineto{\pgfxy(0.25,2.00)}
\pgflineto{\pgfxy(0.75,2.00)}
\pgflineto{\pgfxy(0.75,1.00)}
\pgffill
\pgfmoveto{\pgfxy(1.25,1.00)}
\pgflineto{\pgfxy(1.25,2.00)}
\pgflineto{\pgfxy(1.75,2.00)}
\pgflineto{\pgfxy(1.75,1.00)}
\pgffill
\pgfclosestroke
\pgfmoveto{\pgfxy(5.25,1.00)}
\pgflineto{\pgfxy(5.25,1.50)}
\pgflineto{\pgfxy(5.75,1.50)}
\pgflineto{\pgfxy(5.75,1.00)}
\pgffill
\pgfmoveto{\pgfxy(6.25,1.00)}
\pgflineto{\pgfxy(6.25,2.00)}
\pgflineto{\pgfxy(6.75,2.00)}
\pgflineto{\pgfxy(6.75,1.00)}
\pgffill
\pgfclosestroke
\pgfmoveto{\pgfxy(10.25,1.00)}
\pgflineto{\pgfxy(10.25,1.50)}
\pgflineto{\pgfxy(10.75,1.50)}
\pgflineto{\pgfxy(10.75,1.00)}
\pgffill
\pgfmoveto{\pgfxy(11.25,1.00)}
\pgflineto{\pgfxy(11.25,2.00)}
\pgflineto{\pgfxy(11.75,2.00)}
\pgflineto{\pgfxy(11.75,1.00)}
\pgffill
\pgfclosestroke
\pgfmoveto{\pgfxy(7.25,1.00)}
\pgflineto{\pgfxy(7.25,1.50)}
\pgflineto{\pgfxy(7.75,1.50)}
\pgflineto{\pgfxy(7.75,1.00)}
\pgffill
\pgfclosestroke
\pgfmoveto{\pgfxy(12.25,1.00)}
\pgflineto{\pgfxy(12.25,1.50)}
\pgflineto{\pgfxy(12.75,1.50)}
\pgflineto{\pgfxy(12.75,1.00)}
\pgffill
\pgfclosestroke
\end{pgfscope}
\begin{pgfscope}
\color{red}
\pgfmoveto{\pgfxy(3.25,1.00)}
\pgflineto{\pgfxy(3.25,0.50)}
\pgflineto{\pgfxy(3.75,0.50)}
\pgflineto{\pgfxy(3.75,1.00)}
\pgffill
\pgfclosestroke
\pgfmoveto{\pgfxy(8.25,1.00)}
\pgflineto{\pgfxy(8.25,0.50)}
\pgflineto{\pgfxy(8.75,0.50)}
\pgflineto{\pgfxy(8.75,1.00)}
\pgffill
\pgfclosestroke
\end{pgfscope}
\begin{pgfscope}
\pgfsetlinewidth{1.5pt}
\pgfsetarrowsend{Triangle[scale=0.75pt]}
\pgfxyline(1.5,2.2)(1.5,2.7)
\pgfxyline(2.5,2.2)(2.5,2.7)
\pgfxyline(3.5,2.2)(3.5,2.7)
\pgfxyline(1.5,0.3)(1.5,-0.2)
\pgfxyline(2.5,0.3)(2.5,-0.2)
\pgfxyline(3.5,0.3)(3.5,-0.2)
\pgfxyline(6.5,2.2)(6.5,2.7)
\pgfxyline(7.5,2.2)(7.5,2.7)
\pgfxyline(8.5,2.2)(8.5,2.7)
\pgfxyline(7.5,0.3)(7.5,-0.2)
\pgfxyline(8.5,0.3)(8.5,-0.2)
\pgfxyline(12.5,1.7)(12.5,2.2)
\pgfxyline(13.5,1.7)(13.5,2.2)
\pgfxyline(12.5,0.3)(12.5,-0.2)
\pgfxyline(13.5,0.3)(13.5,-0.2)
\end{pgfscope}
\begin{pgfscope}
\pgfsetlinewidth{2pt}
\pgfxyline(0,1)(4,1)
\pgfxyline(5,1)(9,1)
\pgfxyline(10,1)(14,1)
\pgfxyline(0.5,0.9)(0.5,1.1)
\pgfxyline(1.5,0.9)(1.5,1.1)
\pgfxyline(2.5,0.9)(2.5,1.1)
\pgfxyline(3.5,0.9)(3.5,1.1)
\pgfxyline(5.5,0.9)(5.5,1.1)
\pgfxyline(6.5,0.9)(6.5,1.1)
\pgfxyline(7.5,0.9)(7.5,1.1)
\pgfxyline(8.5,0.9)(8.5,1.1)
\pgfxyline(10.5,0.9)(10.5,1.1)
\pgfxyline(11.5,0.9)(11.5,1.1)
\pgfxyline(12.5,0.9)(12.5,1.1)
\pgfxyline(13.5,0.9)(13.5,1.1)
\end{pgfscope}
\pgfputat{\pgfxy(1.5,2.9)}%
{\pgfbox[center,center]{$\lambda_2$}}
\pgfputat{\pgfxy(2.5,2.9)}%
{\pgfbox[center,center]{$\lambda_1$}}
\pgfputat{\pgfxy(3.5,2.9)}%
{\pgfbox[center,center]{$\lambda_0$}}
\pgfputat{\pgfxy(6.5,2.9)}%
{\pgfbox[center,center]{$\lambda_2$}}
\pgfputat{\pgfxy(7.5,2.9)}%
{\pgfbox[center,center]{$\lambda_1$}}
\pgfputat{\pgfxy(8.5,2.9)}%
{\pgfbox[center,center]{$\lambda_0$}}
\pgfputat{\pgfxy(12.5,2.4)}%
{\pgfbox[center,center]{$\lambda_2$}}
\pgfputat{\pgfxy(13.5,2.4)}%
{\pgfbox[center,center]{$\lambda_1$}}
\pgfputat{\pgfxy(1.5,-0.4)}%
{\pgfbox[center,center]{$\mu_0$}}
\pgfputat{\pgfxy(2.5,-0.4)}%
{\pgfbox[center,center]{$\mu_1$}}
\pgfputat{\pgfxy(3.5,-0.4)}%
{\pgfbox[center,center]{$\mu_2$}}
\pgfputat{\pgfxy(7.5,-0.4)}%
{\pgfbox[center,center]{$\mu_0$}}
\pgfputat{\pgfxy(8.5,-0.4)}%
{\pgfbox[center,center]{$\mu_1$}}
\pgfputat{\pgfxy(12.5,-0.4)}%
{\pgfbox[center,center]{$\mu_0$}}
\pgfputat{\pgfxy(13.5,-0.4)}%
{\pgfbox[center,center]{$\mu_1$}}
\pgfputat{\pgfxy(0.5,0.65)}%
{\pgfbox[center,center]{$U^n$}}
\pgfputat{\pgfxy(1.5,0.65)}%
{\pgfbox[center,center]{$V^n$}}
\pgfputat{\pgfxy(2.5,0.65)}%
{\pgfbox[center,center]{$W^n$}}
\pgfputat{\pgfxy(3.5,1.28)}%
{\pgfbox[center,center]{$X^n$}}
\pgfputat{\pgfxy(5.5,0.65)}%
{\pgfbox[center,center]{$U^n$}}
\pgfputat{\pgfxy(6.5,0.65)}%
{\pgfbox[center,center]{$V^n$}}
\pgfputat{\pgfxy(7.5,0.65)}%
{\pgfbox[center,center]{$W^n$}}
\pgfputat{\pgfxy(8.5,1.28)}%
{\pgfbox[center,center]{$X^n$}}
\pgfputat{\pgfxy(10.5,0.65)}%
{\pgfbox[center,center]{$U^n$}}
\pgfputat{\pgfxy(11.5,0.65)}%
{\pgfbox[center,center]{$V^n$}}
\pgfputat{\pgfxy(12.5,0.65)}%
{\pgfbox[center,center]{$W^n$}}
\pgfputat{\pgfxy(13.5,1.28)}%
{\pgfbox[center,center]{$X^n$}}
%
\begin{pgfscope}
\color{green}
\pgfmoveto{\pgfxy(2.75,5.00)}
\pgflineto{\pgfxy(2.75,6.00)}
\pgflineto{\pgfxy(3.25,6.00)}
\pgflineto{\pgfxy(3.25,5.00)}
\pgffill
\pgfclosestroke
\pgfmoveto{\pgfxy(3.75,5.00)}
\pgflineto{\pgfxy(3.75,6.00)}
\pgflineto{\pgfxy(4.25,6.00)}
\pgflineto{\pgfxy(4.25,5.00)}
\pgffill
\pgfclosestroke
\pgfmoveto{\pgfxy(7.75,5.00)}
\pgflineto{\pgfxy(7.75,5.50)}
\pgflineto{\pgfxy(8.25,5.50)}
\pgflineto{\pgfxy(8.25,5.00)}
\pgffill
\pgfclosestroke
\pgfmoveto{\pgfxy(8.75,5.00)}
\pgflineto{\pgfxy(8.75,6.00)}
\pgflineto{\pgfxy(9.25,6.00)}
\pgflineto{\pgfxy(9.25,5.00)}
\pgffill
\pgfclosestroke
\end{pgfscope}
\begin{pgfscope}
\color{green}
\pgfmoveto{\pgfxy(9.75,5.00)}
\pgflineto{\pgfxy(9.75,5.50)}
\pgflineto{\pgfxy(10.25,5.50)}
\pgflineto{\pgfxy(10.25,5.00)}
\pgffill
\pgfclosestroke
\pgfmoveto{\pgfxy(10.75,5.00)}
\pgflineto{\pgfxy(10.75,5.50)}
\pgflineto{\pgfxy(11.25,5.50)}
\pgflineto{\pgfxy(11.25,5.00)}
\pgffill
\pgfclosestroke
\end{pgfscope}
\begin{pgfscope}
\color{red}
\pgfmoveto{\pgfxy(4.75,5.00)}
\pgflineto{\pgfxy(4.75,4.50)}
\pgflineto{\pgfxy(5.25,4.50)}
\pgflineto{\pgfxy(5.25,5.00)}
\pgffill
\pgfclosestroke
\pgfmoveto{\pgfxy(5.75,5.00)}
\pgflineto{\pgfxy(5.75,4.50)}
\pgflineto{\pgfxy(6.25,4.50)}
\pgflineto{\pgfxy(6.25,5.00)}
\pgffill
\pgfclosestroke
\end{pgfscope}
\begin{pgfscope}
\pgfsetlinewidth{2pt}
\pgfxyline(2.5,5)(6.5,5)
\pgfxyline(7.5,5)(11.5,5)
\pgfxyline(3.0,4.9)(3.0,5.1)
\pgfxyline(4.0,4.9)(4.0,5.1)
\pgfxyline(5.0,4.9)(5.0,5.1)
\pgfxyline(6.0,4.9)(6.0,5.1)
\pgfxyline(8.0,4.9)(8.0,5.1)
\pgfxyline(9.0,4.9)(9.0,5.1)
\pgfxyline(10.0,4.9)(10.0,5.1)
\pgfxyline(11.0,4.9)(11.0,5.1)
\end{pgfscope}
\begin{pgfscope}
\pgfsetlinewidth{1.5pt}
\pgfsetarrowsend{Triangle[scale=0.75pt]}
\pgfxyline(3.0,6.2)(3.0,6.7)
\pgfxyline(4.0,6.2)(4.0,6.7)
\pgfxyline(5.0,6.2)(5.0,6.7)
\pgfxyline(4.0,4.3)(4.0,3.8)
\pgfxyline(5.0,4.3)(5.0,3.8)
\pgfxyline(6.0,4.3)(6.0,3.8)
\pgfxyline(10.0,5.7)(10.0,6.2)
\pgfxyline(11.0,5.7)(11.0,6.2)
\pgfxyline(11.0,4.3)(11.0,3.8)
\end{pgfscope}
\pgfputat{\pgfxy(3.0,6.9)}%
{\pgfbox[center,center]{$\lambda_2$}}
\pgfputat{\pgfxy(4.0,6.9)}%
{\pgfbox[center,center]{$\lambda_1$}}
\pgfputat{\pgfxy(5.0,6.9)}%
{\pgfbox[center,center]{$\lambda_0$}}
\pgfputat{\pgfxy(10.0,6.4)}%
{\pgfbox[center,center]{$\lambda_2$}}
\pgfputat{\pgfxy(11.0,6.4)}%
{\pgfbox[center,center]{$\lambda_1$}}
\pgfputat{\pgfxy(4.0,3.6)}%
{\pgfbox[center,center]{$\mu_0$}}
\pgfputat{\pgfxy(5.0,3.6)}%
{\pgfbox[center,center]{$\mu_1$}}
\pgfputat{\pgfxy(6.0,3.6)}%
{\pgfbox[center,center]{$\mu_2$}}
\pgfputat{\pgfxy(11.0,3.6)}%
{\pgfbox[center,center]{$\mu_0$}}
\pgfputat{\pgfxy(3.0,4.65)}%
{\pgfbox[center,center]{$U^n$}}
\pgfputat{\pgfxy(4.0,4.65)}%
{\pgfbox[center,center]{$V^n$}}
\pgfputat{\pgfxy(5.0,5.28)}%
{\pgfbox[center,center]{$W^n$}}
\pgfputat{\pgfxy(6.0,5.28)}%
{\pgfbox[center,center]{$X^n$}}
\pgfputat{\pgfxy(8.0,4.65)}%
{\pgfbox[center,center]{$U^n$}}
\pgfputat{\pgfxy(9.0,4.65)}%
{\pgfbox[center,center]{$V^n$}}
\pgfputat{\pgfxy(10.0,4.65)}%
{\pgfbox[center,center]{$W^n$}}
\pgfputat{\pgfxy(11.0,4.65)}%
{\pgfbox[center,center]{$X^n$}}
\begin{pgfscope}
\color{green}
\pgfmoveto{\pgfxy(0.25,9.00)}
\pgflineto{\pgfxy(0.25,10.00)}
\pgflineto{\pgfxy(0.75,10.00)}
\pgflineto{\pgfxy(0.75,9.00)}
\pgffill
\pgfclosestroke
\pgfmoveto{\pgfxy(1.25,9.00)}
\pgflineto{\pgfxy(1.25,10.00)}
\pgflineto{\pgfxy(1.75,10.00)}
\pgflineto{\pgfxy(1.75,9.00)}
\pgffill
\pgfclosestroke
\pgfmoveto{\pgfxy(5.25,9.00)}
\pgflineto{\pgfxy(5.25,9.50)}
\pgflineto{\pgfxy(5.75,9.50)}
\pgflineto{\pgfxy(5.75,9.00)}
\pgffill
\pgfclosestroke
\pgfmoveto{\pgfxy(6.25,9.00)}
\pgflineto{\pgfxy(6.25,10.00)}
\pgflineto{\pgfxy(6.75,10.00)}
\pgflineto{\pgfxy(6.75,9.00)}
\pgffill
\pgfclosestroke
\pgfmoveto{\pgfxy(10.25,9.00)}
\pgflineto{\pgfxy(10.25,9.50)}
\pgflineto{\pgfxy(10.75,9.50)}
\pgflineto{\pgfxy(10.75,9.00)}
\pgffill
\pgfclosestroke
\pgfmoveto{\pgfxy(11.25,9.00)}
\pgflineto{\pgfxy(11.25,10.00)}
\pgflineto{\pgfxy(11.75,10.00)}
\pgflineto{\pgfxy(11.75,9.00)}
\pgffill
\pgfclosestroke
\end{pgfscope}
\begin{pgfscope}
\color{green}
\pgfmoveto{\pgfxy(13.25,9.00)}
\pgflineto{\pgfxy(13.25,9.50)}
\pgflineto{\pgfxy(13.75,9.50)}
\pgflineto{\pgfxy(13.75,9.00)}
\pgffill
\pgfclosestroke
\end{pgfscope}
\begin{pgfscope}
\color{red}
\pgfmoveto{\pgfxy(2.25,9.00)}
\pgflineto{\pgfxy(2.25,8.50)}
\pgflineto{\pgfxy(2.75,8.50)}
\pgflineto{\pgfxy(2.75,9.00)}
\pgffill
\pgfclosestroke
\end{pgfscope}
\begin{pgfscope}
\pgfsetlinewidth{1.5pt}
\pgfsetarrowsend{Triangle[scale=0.75pt]}
\pgfxyline(0.5,10.2)(0.5,10.7)
\pgfxyline(1.5,10.2)(1.5,10.7)
\pgfxyline(2.5,10.2)(2.5,10.7)
\pgfxyline(1.5,8.3)(1.5,7.8)
\pgfxyline(2.5,8.3)(2.5,7.8)
\pgfxyline(3.5,8.3)(3.5,7.8)
\pgfxyline(7.5,9.7)(7.5,10.2)
\pgfxyline(8.5,9.7)(8.5,10.2)
\pgfxyline(6.5,8.3)(6.5,7.8)
\pgfxyline(7.5,8.3)(7.5,7.8)
\pgfxyline(8.5,8.3)(8.5,7.8)
\pgfxyline(12.5,9.7)(12.5,10.2)
\pgfxyline(13.5,9.7)(13.5,10.2)
\pgfxyline(13.5,8.3)(13.5,7.8)
\end{pgfscope}
\begin{pgfscope}
\pgfsetlinewidth{2pt}
\pgfxyline(0,9)(4,9)
\pgfxyline(5,9)(9,9)
\pgfxyline(10,9)(14,9)
\pgfxyline(0.5,8.9)(0.5,9.1)
\pgfxyline(1.5,8.9)(1.5,9.1)
\pgfxyline(2.5,8.9)(2.5,9.1)
\pgfxyline(3.5,8.9)(3.5,9.1)
\pgfxyline(5.5,8.9)(5.5,9.1)
\pgfxyline(6.5,8.9)(6.5,9.1)
\pgfxyline(7.5,8.9)(7.5,9.1)
\pgfxyline(8.5,8.9)(8.5,9.1)
\pgfxyline(10.5,8.9)(10.5,9.1)
\pgfxyline(11.5,8.9)(11.5,9.1)
\pgfxyline(12.5,8.9)(12.5,9.1)
\pgfxyline(13.5,8.9)(13.5,9.1)
\end{pgfscope}
\pgfputat{\pgfxy(0.5,10.9)}%
{\pgfbox[center,center]{$\lambda_2$}}
\pgfputat{\pgfxy(1.5,10.9)}%
{\pgfbox[center,center]{$\lambda_1$}}
\pgfputat{\pgfxy(2.5,10.9)}%
{\pgfbox[center,center]{$\lambda_0$}}
\pgfputat{\pgfxy(7.5,10.4)}%
{\pgfbox[center,center]{$\lambda_2$}}
\pgfputat{\pgfxy(8.5,10.4)}%
{\pgfbox[center,center]{$\lambda_1$}}
\pgfputat{\pgfxy(12.5,10.4)}%
{\pgfbox[center,center]{$\lambda_2$}}
\pgfputat{\pgfxy(13.5,10.4)}%
{\pgfbox[center,center]{$\lambda_1$}}
\pgfputat{\pgfxy(1.5,7.6)}%
{\pgfbox[center,center]{$\mu_0$}}
\pgfputat{\pgfxy(2.5,7.6)}%
{\pgfbox[center,center]{$\mu_1$}}
\pgfputat{\pgfxy(3.5,7.6)}%
{\pgfbox[center,center]{$\mu_2$}}
\pgfputat{\pgfxy(6.5,7.6)}%
{\pgfbox[center,center]{$\mu_0$}}
\pgfputat{\pgfxy(7.5,7.6)}%
{\pgfbox[center,center]{$\mu_1$}}
\pgfputat{\pgfxy(8.5,7.6)}%
{\pgfbox[center,center]{$\mu_2$}}
\pgfputat{\pgfxy(13.5,7.6)}%
{\pgfbox[center,center]{$\mu_0$}}
\pgfputat{\pgfxy(0.5,8.65)}%
{\pgfbox[center,center]{$U^n$}}
\pgfputat{\pgfxy(1.5,8.65)}%
{\pgfbox[center,center]{$V^n$}}
\pgfputat{\pgfxy(2.5,9.28)}%
{\pgfbox[center,center]{$W^n$}}
\pgfputat{\pgfxy(3.5,9.28)}%
{\pgfbox[center,center]{$X^n$}}
\pgfputat{\pgfxy(5.5,8.65)}%
{\pgfbox[center,center]{$U^n$}}
\pgfputat{\pgfxy(6.5,8.65)}%
{\pgfbox[center,center]{$V^n$}}
\pgfputat{\pgfxy(7.5,8.65)}%
{\pgfbox[center,center]{$W^n$}}
\pgfputat{\pgfxy(8.5,9.28)}%
{\pgfbox[center,center]{$X^n$}}
\pgfputat{\pgfxy(10.5,8.65)}%
{\pgfbox[center,center]{$U^n$}}
\pgfputat{\pgfxy(11.5,8.65)}%
{\pgfbox[center,center]{$V^n$}}
\pgfputat{\pgfxy(12.5,8.65)}%
{\pgfbox[center,center]{$W^n$}}
\pgfputat{\pgfxy(13.5,8.65)}%
{\pgfbox[center,center]{$X^n$}}
\pgfputat{\pgfxy(10.5,10.4)}%
{\pgfbox[center,center]{$C$}}
\pgfputat{\pgfxy(11.5,10.4)}%
{\pgfbox[center,center]{$C$}}
\pgfputat{\pgfxy(9.0,6.4)}%
{\pgfbox[center,center]{$C$}}
\pgfputat{\pgfxy(8.0,6.4)}%
{\pgfbox[center,center]{$C$}}
\pgfputat{\pgfxy(5.5,2.9)}%
{\pgfbox[center,center]{$C$}}
\pgfputat{\pgfxy(10.5,2.4)}%
{\pgfbox[center,center]{$C$}}
\end{pgfpicture}
\caption{Forces acting on $U^n$}
\end{figure}

Recall the processes $\sV^n$, $\sW^n$,
$\sX^n$ and $\sY^n$ introduced
in Sections \ref{Interior}
and \ref{BracketingIntro}
whose dynamics are given by Figure 
4.1 regardless of the values
of $\sV^n$ and $\sY^n$.
In the same way, we define $\sU^n$.
To precisely define $\sU^n$, we
introduce six
more intensity-one Poisson processes,
$N_{NE,U,-}$, $N_{E,U,-}$,
$N_{SE_+,U,-}$, $N_{SE,U,-}$, $N_{SE_-,U,-}$, 
$N_{SW,U,+}$,
independent
of the Poisson processes used to
define $\sV^n$, $\sW^n$,
$\sX^n$, and $\sY^n$.  In terms of these,
\begin{align}
\sU^n(t)
&=
U^n(0)-N_{NE,U,-}\left(\int_0^t\frac{1}{\sqrt{n}}
\theta_b\big(\sU^n(s)\big)^+dP_{NE}^n(s)\right)
\nonumber\\
&\quad
-N_{E,U,-}\left(\int_0^t\frac{1}{\sqrt{n}}
\theta_b\big(\sU^n(s)\big)^+dP_{E}^n(s)\right)
-N_{SE_+,U,-}\left(\int_0^t\frac{1}{\sqrt{n}}
\theta_b\big(\sU^n(s)\big)^+dP_{SE_+}^n(s)\right)\nonumber\\
&\quad
-N_{SE,U,-}\left(\int_0^t\frac{1}{\sqrt{n}}
\theta_b\big(\sU^n(s)\big)^+dP_{SE}^n(s)\right)\nonumber\\
&\quad
-N_{SE_-,U,-}\left(\int_0^t\frac{1}{\sqrt{n}}
\theta_b\big(\sU^n(s)\big)^+dP_{SE_-}^n(s)\right)
+N_{SW,U,+}\circ\lambda_2P_{SW}^n(t).\label{6.11y}
\end{align}
Scaling and centering these Poisson processes
by defining $\widehat{M}_{\times,U,\pm}^n(t)
=\frac{1}{\sqrt{n}}(N_{\times,U,\pm}(nt)-nt)$,
we may write the diffusion-scaled process
$\sUhatn(t)=\frac{1}{\sqrt{n}}\sU^n(nt)$ as
(cf.\ (\ref{5.1}) and (\ref{5.2}) for a similar
calculation)
\begin{align}
\sUhatn(t)
&=
\sUhatn(0)-
\widehat{M}_{NE,U,-}^n
\left(\int_0^t\theta_b\big(\sUhatn(s)\big)^+
d\Pbarn_{NE}(s)\right)\nonumber\\
&\quad
-\widehat{M}_{E,U,-}\left(\int_0^t
\theta_b\big(\sUhatn(s)\big)^+
d\Pbarn_E(s)\right)
-\widehat{M}_{SE_+,U,-}
\left(\int_0^t\theta_b\big(\sUhatn(s)\big)^+
d\Pbarn_{SE_+}(s)\right)\nonumber\\
&\quad
-\widehat{M}_{SE,U,-}\left(\int_0^t\theta_b
\big(\sUhatn(s)\big)^+d\Pbarn_{SE}(s)\right)
-\widehat{M}_{SE_-,U,-}
\left(\int_0^t\theta_b\big(\sUhatn(s)\big)^+
d\Pbarn_{SE_-}(s)\right)\nonumber\\
&\quad
-\sqrt{n}\int_0^t\theta_b\big(\sUhatn(s)\big)^+
\big(d\Pbarn_{NE}(s)+d\Pbarn_E(s)
+d\Pbarn_{SE_+}(s)+d\Pbarn_{SE}(s)
+d\Pbarn_{SE_-}(s)\big)\nonumber\\
&\quad
+\widehat{M}_{SW,U,+}\circ\lambda_2\Pbarn_{SW}(t)
+\sqrt{n}\,\lambda_2\Pbarn_{SW}(t).\label{6.11}
\end{align}

\begin{lemma}\label{L6.6}
The c\`adl\`ag processes $\sUhatn$ are
nonnegative, and the sequence
$\{\sUhatn\}_{n=1}^{\infty}$ is bounded
in probability on compact time intervals.
\end{lemma}
\begin{proof}
By Assumption \ref{Assumption2},
$\sU^n(0)\geq 0$, and since the decreases
in $\sU^n$ are all due to cancellation
(see (\ref{6.11y})),
$\sU^n$ can never become negative.
It remains to show that $\sUhatn$ is
bounded above in probability on compact time
intervals.

To simplify the notation slightly, we define
the auxiliary process
\begin{align}
\sA^n(t)
&:=
U^n(0)-N_{NE,U,-}\left(\int_0^t\frac{1}{\sqrt{n}}
\theta_b\sA(s)dP^n_{NE}(s)\right)
-N_{E,U,-}\left(\int_0^t\frac{1}{\sqrt{n}}
\theta_b\sA^n(s)dP^n_E(s)\right)\nonumber\\
&\quad
-N_{SE_+,U,-}\left(\int_0^t\frac{1}{\sqrt{n}}
\theta_b\sA^n(s)dP^n_{SE_+}(s)\right)
-N_{SE_-,U,-}\left(\int_0^t\frac{1}{\sqrt{n}}
\theta_b\sA^n(s)dP^n_{SE_-}(s)\right)\nonumber\\
&\quad
+N_{SW,U,+}\circ\lambda_2P^n_{SW}(t)\label{6.13z}
\end{align}
and its diffusion-scaled version
\begin{align}
\sAhatn(t)
&=
\sUhatn(0)-\widehat{M}_{NE,U,-}^n
\left(\int_0^t\theta_b\sAhatn(s)d\Pbarn_{NE}(s)\right)
-\widehat{M}_{E,U,-}^n
\left(\int_0^t\theta_b\sAhatn(s)d\Pbarn_{E}(s)\right)
\nonumber\\
&\quad
-\widehat{M}_{SE_+,U,-}^n
\left(\int_0^t\theta_b\sAhatn(s)d\Pbarn_{SE_+}(s)\right)
-\widehat{M}_{SE_-,U,-}^n
\left(\int_0^t\theta_b\sAhatn(s)d\Pbarn_{SE_-}(s)\right)
\nonumber\\
&\quad
-\sqrt{n}\int_0^t\theta_b\sAhatn(s)
\big(d\Pbarn_{NE}(s)+d\Pbarn_E(s)
+d\Pbarn_{SE_+}(s)+d\Pbarn_{SE_-}(s)\big)\nonumber\\
&\quad
+\widehat{M}_{SW,U,+}^n\circ\lambda_2\Pbarn_{SW}(t)
+\sqrt{n}\,\lambda_2\Pbarn_{SW}(t).\label{6.14z}
\end{align}
The process $\sA^n$ is the same as the process
$\sU^n$ except that $\sA^n$ suffers no
cancellations in the region $SE$, and hence
stochastically
dominates $\sU^n$.  Because $\sU^n$ is nonnegative,
so is $\sA^n$. Thus we may write $\sA^n(s)$ instead of
$(\sA^n(s))^+$ in (\ref{6.13z}) and $\sAhatn(s)$ instead of
$(\sAhatn(s))^+$ in (\ref{6.14z}).
It remains to show that $\sAhatn$
is bounded above in probability on compact
time intervals.

We rewrite (\ref{5.37y}) as
$$
\sqrt{n}\,\lambda_2\Pbarn_{SW}
=\sqrt{n}(\alpha_{SE_-}\Pbarn_{SE_-}
+\alpha_{SE_+}\Pbarn_{SE_+}
+\alpha_E\Pbarn_E
+\alpha_{NE}\Pbarn_{NE})+O(1)
$$
for appropriate constants $\alpha_{\times}$
and substitute this into (\ref{6.14z}),
observing that $\widehat{M}^n_{SW,U,+}\circ
\lambda_2\Pbarn_{SW}=O(1)$, to obtain
\begin{align}
\sAhatn(t)
&=
\sUhatn(0)-\widehat{M}_{NE,U,-}^n
\left(\int_0^t\theta_b\sAhatn(s)d\Pbarn_{NE}(s)\right)
-\widehat{M}_{E,U,-}^n
\left(\int_0^t\theta_b\sAhatn(s)d\Pbarn_{E}(s)\right)
\nonumber\\
&\quad
-\widehat{M}_{SE_+,U,-}^n
\left(\int_0^t\theta_b\sAhatn(s)d\Pbarn_{SE_+}(s)\right)
-\widehat{M}_{SE_-,U,-}^n
\left(\int_0^t\theta_b\sAhatn(s)d\Pbarn_{SE_-}(s)\right)
\nonumber\\
&\quad
+\sqrt{n}\int_0^t\big(\alpha_{NE}-\theta_b\sAhatn(s)\big)
d\Pbarn_{NE}(s)
+\sqrt{n}\int_0^t\big(\alpha_E-\theta_b\sAhatn(s)\big)
d\Pbarn_E(s)\nonumber\\
&\quad
+\sqrt{n}\int_0^t\big(\alpha_{SE_+}-\theta_b\sAhatn(s)\big)
d\Pbarn_{SE_+}(s)
+\sqrt{n}\int_0^t\big(\alpha_{SE_-}-\theta_b\sAhatn(s)\big)
d\Pbarn_{SE_-}(s)\nonumber\\
&\quad
+O(1).\label{6.16z}
\end{align}
There is a finite random variable $M^*$ that upper
bounds
\begin{align*}
&
-\widehat{M}_{NE,U,-}^n\left(\int_0^t\theta_b
\sAhatn(s)d\Pbarn_{NE}(s)\right)
-\frac12\int_0^t\theta_b\sAhatn(s)d\Pbarn_{NE}(s)\\
&\quad
-\widehat{M}_{E,U,-}^n\left(\int_0^t\theta_b
\sAhatn(s)d\Pbarn_{E}(s)\right)
-\frac12\int_0^t\theta_b\sAhatn(s)d\Pbarn_{E}(s)\\
&\quad
-\widehat{M}_{SE_+,U,-}^n\left(\int_0^t\theta_b
\sAhatn(s)d\Pbarn_{SE_+}(s)\right)
-\frac12\int_0^t\theta_b\sAhatn(s)d\Pbarn_{SE_+}(s)\\
&\quad
-\widehat{M}_{SE_-,U,-}^n\left(\int_0^t\theta_b
\sAhatn(s)d\Pbarn_{SE_-}(s)\right)
-\frac12\int_0^t\theta_b\sAhatn(s)d\Pbarn_{SE_-}(s)
\end{align*}
(cf.\ proof of Lemma \ref{L5.6}
just before (\ref{5.32})).
Define $\alpha:=\max\{\alpha_{SE_-},\alpha_{SE_+},
\alpha_{E},\alpha_{NE}\big\}$. From (\ref{6.16z})
we see that
\begin{align}
\sAhatn(t)
&\leq 
\sUhatn(0)+M^*
+\sqrt{n}\int_0^t\left(\alpha-\frac12\theta_b\sAhatn(s)
\right)d\Pbarn_{NE}(s)\nonumber\\
&\quad
+\sqrt{n}\int_0^t\left(\alpha-\frac12\theta_b\sAhatn(s)
\right)d\Pbarn_{E}(s)
+\sqrt{n}\int_0^t\left(\alpha-\frac12\theta_b\sAhatn(s)
\right)d\Pbarn_{SE_+}(s)\nonumber\\
&\quad
+\sqrt{n}\int_0^t\left(\alpha-\frac12\theta_b\sAhatn(s)
\right)d\Pbarn_{SE_-}(s)+O(1).\label{6.17z}
\end{align}

Fix $T>0$.  For $0\leq t\leq T$, we have either
\be\label{6.18z}
\int_0^t\left(\alpha-\frac12\theta_b\sAhatn(s)\right)
\big(d\Pbarn_{NE}(s)+d\Pbarn_{E}(s)
+d\Pbarn_{SE_+}(s)+d\Pbarn_{SE_-}(s)\big)
\leq 0,
\ee
or else
\begin{align}
\lefteqn{
\int_0^t\theta_b\sAhatn(s)
\big(d\Pbarn_{NE}(s)+d\Pbarn_E(s)
+d\Pbarn_{SE_+}+d\Pbarn_{SE_-}(s)\big)}
\qquad\qquad\qquad\nonumber\\
&< 
2\alpha
\big(\Pbarn_{NE}(t)+\Pbarn_E(t)
+\Pbarn_{SE_+}(t)+\Pbarn_{SE_-}(t)\big)\nonumber\\
&\leq 
2\alpha T.\label{6.19z}
\end{align}
If (\ref{6.18z}) holds,
(\ref{6.17z}) implies 
$\sAhatn(t)\leq \sUhatn(0)+M^*$, and we
have the desired upper bound.
If (\ref{6.19z}) holds, the set
$\{s\in[0,t]:\theta_b\sAhatn(s)\leq 2\alpha\}$
is nonempty, and we define $\tau^n(t)$
to be its supremum.  
Because the arguments of
the centered Poisson processes 
$\widehat{M}_{\times,U,-}$ in (\ref{6.16z}) are
bounded by $2\alpha T$ and both the infimum
and the supremum over $s\in[0,2\alpha T]$ of
$$
-\widehat{M}_{NE,U,-}(s)-\widehat{M}_{E,U,-}(s)
-\widehat{M}_{SE_+,U,-}-\widehat{M}_{SE_-,U,-}(s)
$$
are finite, (\ref{6.16z}) implies
\begin{align}
\sAhatn(t)
&=
\sAhatn\big(\tau^n(t)\big)
+\sqrt{n}\int_{\tau^n(t)}^t\big(\alpha_{NE}-\theta_b
\sAhatn(s)\big)d\Pbarn_{NE}(s)\nonumber\\
&\quad
+\sqrt{n}\int_{\tau^n(t)}^t\big(\alpha_{E}-\theta_b
\sAhatn(s)\big)d\Pbarn_{E}(s)
+\sqrt{n}\int_{\tau^n(t)}^t\big(\alpha_{SE_+}-\theta_b
\sAhatn(s)\big)d\Pbarn_{SE_+}(s)\nonumber\\
&\quad
+\sqrt{n}\int_{\tau^n(t)}^t\big(\alpha_{SE_-}-\theta_b
\sAhatn(s)\big)d\Pbarn_{SE_-}(s)+O(1).\label{6.20z}
\end{align}
But for
$s\in(\tau^n(t),t]$, we have $\theta_b\sAhatn(s)>2\alpha$,
and the integrands in (\ref{6.20z}) are negative.
This implies
$$
\sAhatn(t)\leq \sAhatn\big(\tau^n(t)\big)+O(1)
\leq \frac{2\alpha}{\theta_b}+\frac{1}{\sqrt{n}}
+O(1)=O(1).
$$

\vspace{-24pt}
\end{proof}

Using the excursion notation of Section
\ref{Enumerating}, we define
$U_{k,+}^n:=\sUhatn\big((\Lambda_{k,+}^n+\cdot)
\wedge R_{k,+}^n\big)$,
which is the process $\sUhatn$ on the $k$-th
positive excursion of $G^n$.

\begin{lemma}\label{L6.7}
For every $\varepsilon>0$, we have
$\sup_{t\geq\varepsilon}U_{k,+}^n(t)
\rightarrow 0$ almost surely.
\end{lemma}
\begin{proof}
When $\Ghatn$ is on a positive excursion,
$(\sWhatn,\sXhatn)\in NE\cup E\cup SE_+$
(see (\ref{3.17a})).  In this case, 
(\ref{6.11y}) shows $U_{k,+}^n$ is nonincreasing,
so it suffices to show that 
$U_{k,+}^n(\varepsilon)\rightarrow 0$.
But (\ref{6.11}) implies that
\begin{align*}
U_{k,+}^n(\varepsilon)
&=
U_{k,+}^n(0)
-\sqrt{n}\int_{\Lambda_{k,+}^n}^{\Lambda_{k,+}^n
+\varepsilon}
\theta_b\big(\sUhatn(s)\big)^+
\big(d\Pbarn_{NE}(s)+d\Pbarn_E(s)
+d\Pbarn_{SE_+}(s)\big)+O(1)\\
&=
-\sqrt{n}\int_{\Lambda_{k,+}^n}^{\Lambda_{k,+}^n
+\varepsilon}
\theta_b\big(\sUhatn(s)\big)^+ds+O(1)\\
&\leq -\sqrt{n}\theta_b\varepsilon
U_{k,+}^n(\varepsilon)+O(1).
\end{align*}
Because the left-hand side of this relation
is nonnegative, 
$U_{k,+}^n(\varepsilon)\rightarrow 0$
as $n\rightarrow\infty$.
\end{proof}

Similarly, we define
$U_{k,-}^n:=\sUhatn\big((\Lambda_{k,-}^n+\cdot)
\wedge R_{k,-}^n\big)$,
which is the process $\sUhatn$
on the $k$-th negative excursion of $G^n$.

\begin{lemma}\label{L6.8}
For every $\varepsilon>0$, we have
$\sup_{t\geq \varepsilon}|U_{k,-}^n(t)-\kappa_L|
\stackrel{\P}{\rightarrow} 0$ almost surely.
\end{lemma}
\begin{proof}
When $\Ghatn$ is on a negative excursion,
then $(\sWhatn,\sXhatn)\in SE_-\cup S\cup SW$ and
for $t\geq 0$,
(\ref{6.11}) becomes
\begin{align}
\sUhatn\big((\Lambda_{k,-}^n+t)\wedge R_{k,-}^n\big)
&=
\sUhatn(\Lambda_{k,-}^n)
+C_1^n(t)
-\sqrt{n}\int_{\Lambda_{k,-}^n}^{(\Lambda_{k,-}^n+t)
\wedge R_{k,-}^n}
\theta_b\big(\sUhatn(s)\big)^+
d\Pbarn_{SE_-}(s)\nonumber\\
&\quad
+\sqrt{n}\,\lambda_2\Big(\Pbarn_{SW}
\big((\Lambda_{k,-}^n+t)\wedge R_{k,-}^n\big)
-\Pbarn_{SW}(\Lambda_{k,-}^n)\Big),\label{6.13}
\end{align}
where
\begin{align*}
C_1^n(t)
&:=
-\left[\widehat{M}_{SE_-,U,-}
\left(\int_0^{(\Lambda_{k,-}^n+t)\wedge R_{k,-}^n}
\theta_b\big(\sUhatn(s)\big)^+d\Pbarn_{SE_-}(s)\right)
\right.\nonumber\\
&\hspace{1in}\left.
-\widehat{M}_{SE_-,U,-}
\left(\int_0^{\Lambda_{k,-}^n}
\theta_b\big(\sUhatn(s)\big)^+d\Pbarn_{SE_-}(s)\right)
\right]\nonumber\\
&\quad
+\left[\widehat{M}_{SW,U,+}\circ
\lambda_2\Pbarn_{SW}
\big((\Lambda_{k,-}^n+t)\wedge R_{k,-}^n\big)
-\widehat{M}_{SW,U,+}\circ\lambda_2\Pbarn_{SW}
(\Lambda_{k,-}^n)\right].\nonumber\\
\end{align*}
We rewrite (\ref{6.13}) as
\begin{align*}
U_{k,-}^n(t)
&=
U_{k,-}^n(0)
+C_1^n(t)+C_2^n(t)\nonumber\\
&\quad
+\sqrt{n}\,\theta_b
\int_0^{t\wedge (R_{k,-}^n-\Lambda_{k,-}^n)}
\big(\kappa_L-(U_{k,-}^n(u))^+\big)
d\Pbarn_{SE_-}(\Lambda_{k,-}^n+u),
\end{align*}
where
\begin{align*}
C_2^n(t)
&:=
\sqrt{n}\,\lambda_2
\left(
\Pbarn_{SW}\big((\Lambda_{k,-}^n+t)\wedge R_{k,-}^n\big)
-\Pbarn_{SW}(\Lambda_{k,-}^n)\right)\\
&\quad
-\sqrt{n}\,\frac{\lambda_2\mu_1}{\lambda_1}
\left(\Pbarn_{SE_-}
\big((\Lambda_{k,-}+t)\wedge R_{k,-}^n\big)
-\Pbarn_{SE_-}(\Lambda_{k,-}^n)
\right).
\end{align*}
Lemma \ref{L6.6} and the convergence
of $\widehat{M}_{SE_-,U,-}$ and
$\widehat{M}_{SW,U,+}$ to (continuous)
Brownian motions implies that $\{C_1^n\}_{n=1}$
has the subsequence/sub-subsequence 
continuity property specified
in Lemma \ref{L5.8}.
Equation (\ref{5.37y}) of Remark \ref{R5.10y}
shows that $C_2^n$ also has the sequence/sub-subsequence
continuity property specified in Lemma \ref{L5.8}
(recall that $\Pbarn_{SE_+}$, $\Pbarn_E$
and $\Pbarn_{NE}$ are constant on
negative excursions of $\Ghatn$).
We conclude that $U_{k,-}^n$ has
the properties set forth in Lemma \ref{L5.8}.
We can now follow the proof of Proposition
\ref{P5.14} to conclude.
\end{proof}

\vspace{6pt}
\noi
{\em Proof of Theorem \ref{T6.2}.}
By construction, the left-hand side of (\ref{6.3y})
is 
$$
\big(\sUhatn(\Shatn),\sVhatn(\Shatn),\sWhatn(\Shatn),
\sXhatn(\Shatn),\sYhatn(\Shatn),\sZhatn(\Shatn)\big).
$$
Consider the case $S^*_v<S^*_y$,
so that $\sV^*(S^*)=0$.  
Because $\sV^*$ is the constant $\kappa_L$
off negative excursions of $G^*$, $G^*$ must be
on a negative excursion at time $S^*$.
In particular, $G^*$ is on a negative
excursion in a neighborhood of $S^*$.
Because $\Shatn\rightarrow S^*$,
Lemma \ref{L6.8} implies
$\sUhatn(\Shatn)\stackrel{\P}{\rightarrow}\kappa_L$.  
Because $\sV^*$ is continuous at $S^*$,
$$
\sVhatn(\Shatn)=\big(\sVhatn(\Shatn)-\sV^*(\Shatn)\big)+
\big(\sV^*(\Shatn)-\sV^*(S^*)\big)\rightarrow 0
$$
almost surely.
Remark \ref{R4.20x} implies
$$
(\sW^*(S^*),\sX^*(S^*)\big)
=\big(\max\big\{G^*(S^*),0\big\},\min\big\{G^*(S^*),0\big\}
\big)
=\big(0,G^*(S^*)\big).
$$
This plus continuity of $\sW^*$ and $\sX^*$ imply
that $\sWhatn(\Shatn)\rightarrow 0$ almost surely
and $\sXhatn(\Shatn)$ has the almost sure
negative limit $\sX^*(S^*)$.
Theorem \ref{T5.24} and continuity
of $\sY^*$ at $S^*$ imply
$\sYhatn(S^n)\rightarrow \sY^*(S^*)=\kappa_R$.
Finally, Lemma \ref{L6.7} that $\sUhatn\rightarrow 0$
almost surely on the interior of positive
excursions of $G^*$ shows by analogy that 
$\Zhatn\rightarrow0$ almost surely on the interior of
negative excursions of $G^*$.
This establishes the convergence (\ref{6.3y})
on the set $\{S^*_v<S^*_y\}$.  The argument
for the complementary set $\{S^*_y<S^*_v\}$
is analogous.
$\hfill\Box$

\section{Statistics of the limit}\label{Statistics}
\setcounter{equation}{0}
\setcounter{theorem}{0}

This section provides two calculations
for the limiting system.
One calculation concerns the {\em time to renewal}
$S^*$ of (\ref{Sstar}). 
The
process $(\frac{1}{\sqrt{n}}V^n(nt),
\frac{1}{\sqrt{n}}W^n(nt),\frac{1}{\sqrt{n}}X^n(nt),
\frac{1}{\sqrt{n}}Y^n(nt))_{t\geq 0}$ begins with
$V^n(0)$ strictly positive and
$Y^n(0)$ strictly negative.  This process,
stopped the first time either $V^n$ or
$Y^n$ reaches zero, has limit
$(\sV^*,\sW^*,\sX^*,\sY^*)$
stopped at $S^*$, the first time 
either
$\sV^*$ or $\sY^*$ vanishes.
 According to
Assumption \ref{Assumption2},
the initial condition for the limiting
system is
$(V^*(0),0,0,Y^*(0))$, where
$V^*(0)$ is strictly positive and $Y^*(0)$
is strictly negative.
If $\sV^*(S^*)=0$, the price
has shifted downward one tick,
and if $\sY^*(S^*)=0$, the price has moved
upward.  Theorem \ref{T6.2} says that
after this price shift, we are in a new
initial state of the same form
$(+,0,0,-)$ as $(V^*(0),0,0,Y^*(0))$
but shifted left or right one tick.
We say there has been a
{\em renewal}.  In Section \ref{BetweenRenewals}
we compute the characteristic function
of the time to renewal and 
the probabilities for downward and upward
price shifts.

Section \ref{ToRenewal} deals with a related
problem.  At Lebesgue-almost-every time after the initial
time but before the first renewal, the system
$(\sV^*,\sW^*,\sX^*,\sY^*)$ satisfies
$\sV^*>0$, $\sY^*<0$, and either $\sW^*>0$, $\sX^*=0$
or else $\sW^*=0$, $\sX^*<0$.  
We consider the latter case, i.e.,
at time $t_0$, $\sV^*(t_0)=v_1>0$, $\sW^*(t_0)=0$
and $\sX^*(t_0)=x_1<0$.  In Theorems 
\ref{T7.1} and \ref{T7.2}
we provide
the probability $\sV^*$ reaches zero before
$\sX^*$ and, conditioned on this event, 
the probability density function
of the first passage time of $\sV^*$ to zero.
Similarly, we report the probability that
$\sX^*$ reaches zero before $\sV^*$,
and conditioned on this event, the probability
density function of the first passage time
of $\sX^*$ to zero.

\subsection{Time to renewal}\label{ToRenewal}

\begin{lemma}\label{L7.1a}
Let $t_0>0$, $v_1>0$
and $x_1<0$ be given.  Assume $G^*$ is on a negative
excursion $E$ at time $t_0$ and define
\be\label{7.12}
\tau_{\sV^*}^{t_0}=\inf\big\{t\geq t_0:\sV^*(t)=0\big\},
\quad
\tau_{\sX^*}^{t_0}=\inf\big\{t\geq t_0 :\sX^*(t)=0\big\}.
\ee
There exists a pair of Brownian motions 
$(\Dhat,\Ehat)$ with
covariance matrix\footnote{Here $\mbox{}^{\prime}$
denotes time derivative.}
\be\label{cov}
\left[\begin{array}{cc}
\langle \Dhat,\Dhat\rangle'&
\langle \Dhat,\Ehat\rangle'\\
\langle\Dhat,\Ehat\rangle'&
\langle\Ehat,\Ehat\rangle'
\end{array}\right]
=\left[\begin{array}{cc}
\sigma_+^2&\rho\sigma_+\sigma_-\\
\rho\sigma_+\sigma_-&\sigma_-^2
\end{array}\right],
\ee
where
$\sigma_{\pm}$ is defined by
(\ref{sigmaplus}), (\ref{sigmaminus})
and $\rho\in(-1,0)$
is defined by (\ref{rho}),
such that, with
$\tau_{\Dhat}:=\inf\{t\geq 0:\Dhat(t)=0\}$,
$\tau_{\Ehat}:=\inf\{t\geq 0:\Ehat(t)=0\}$, 
\begin{align}
\lefteqn{\P\big\{\tau_{\sV^*}^{t_0}
<\tau_{\sX^*}^{t_0}\big|
\sV^*(t_0)=v_1,\sX^*(t_0)=x_1\big\}}
\hspace{2cm}\nonumber\\
&=
\P\big\{\tau_{\Dhat}<\tau_{\Ehat}\big|
\Dhat(0)=v_1,\Ehat(0)=-x_1\big\}
\label{7.13d}\\
\lefteqn{\P\{\tau_{\sV^*}^{t_0}\in t_0+
ds\big|\tau_{\sV^*}^{t_0}<\tau_{\sX^*}^{t_0},
\sV^*(t_0)=v_1,\sX^*(t_0)=x_1\big\}},
\hspace{2cm}\nonumber\\
&=
\P\big\{\tau_{\Dhat}\in ds\big|
\tau_{\Dhat}<\tau_{\Ehat},\Dhat(0)=v_1,\Ehat(0)=-x_1
\big\},\quad s\geq 0,\label{7.14d}\\
\lefteqn{
\P\big\{\tau_{\sX^*}^{t_0}<\tau_{\sV^*}^{t_0}
\big|\sV^*(t_0)=v_1,
\sX^*(t_0)=x_1\big\}}\hspace{2cm}\nonumber\\
&=
\P\big\{\tau_{\Ehat}<\tau_{\Dhat}\big|
\Dhat(0)=v_1,\Ehat(0)=-x_1\big\},\label{7.15}\\
\lefteqn{\P\big\{\tau_{\sX^*}^{t_0}\in t_0+dt\big|
\tau_{\sX^*}^{t_0}<\tau_{\sV^*}^{t_0},
\sV^*(t_0)=v_1,\sX^*(t_0)=x_1
\big\}}\hspace{2cm}\nonumber\\
&=
\P\big\{\tau_{\Ehat}\in dt\big|
\tau_{\Ehat}<\tau_{\Dhat},\Dhat(0)=v_1,
\Ehat(0)=-x_1\big\},\quad t\geq 0.\label{7.16}
\end{align}
\end{lemma}

\begin{proof}
We assume that $G^*$ is on a negative
excursion $E$ at time $t_0$.  Let
$\Lambda$ and $R$ denote the respective
left and right endpoints of
this excursion.

We condition on $G^*(t_0)=x_1$.
According to Corollary \ref{C4.11}, on 
its negative excursions, $G^*$ is a Brownian
motion with speed $\sigma_-^2$.  
In particular, there exists
a Brownian motion $\Ehat$ 
with speed $\sigma_-^2$ and initial
condition $\Ehat(0)=-x_1$ such that
$$
\Ehat(t)=-G^*(t_0+t)=-\sX^*(t_0+t), 
\quad 0\leq t\leq \tau_{\Ehat}.
$$ 

On the other hand, 
according to (\ref{9.4}),
\be\label{7.3}
\sV^*(t)=\kappa_L+C(t-\Lambda)
-\frac{\rho\sigma_+}{\sigma_-}E(t-\Lambda),
\quad \Lambda\leq t\leq R,
\ee
where $C$ is a Brownian motion
$\widetilde{C}$ stopped at $R-\Lambda$,
independent of $G^*$, and 
with speed (see (\ref{5.59}))
$(1-\rho^2)\sigma_+^2$.
We condition on $\sV^*(t_0)=v_1$.
This leads us to define the Brownian motion
$$
\Dhat(t):=\kappa_L+\Ctilde(t_0-\Lambda+t)
+\frac{\rho\sigma_+}{\sigma_-}\Ehat(t),
\quad t\geq 0,
$$
where $\Ctilde$ and $\Ehat$ are independent,
$\Ehat(0)=-E(t_0-\Lambda)$,
and $\Dhat(0)=\sV^*(t_0)=v_1$.
With this construction, we have 
$\tau_{\sV^*}^{t_0}=t_0+\tau_{\Dhat}$
on $\{\tau_{\sV^*}^{t_0}<\tau_{\sX^*}^{t_0}\}
=\{\tau_{\Dhat}<\tau_{\Ehat}\}$,
$\tau_{\sX^*}^{t_0}=R=t_0+\tau_{\Ehat}$
on $\{\tau_{\sX^*}^{t_0}<\tau_{\sV^*}^{t_0}\}
=\{\tau_{\Ehat}<\tau_{\Dhat}\}$, and
$$
\big(\sV^*(t_0+t),\sX^*(t_0+t))
=\big(\Dhat(t),-\Ehat(t)\big),\quad
0\leq t\leq \tau_{\Dhat}\wedge\tau_{\Ehat}.
$$
It is  now straightforward to verify
the covariance matrix formula (\ref{cov})
and the conditioning formulas
(\ref{7.13d})-- (\ref{7.16}).
\end{proof}

\begin{theorem}\label{T7.1}
Let $t_0>0$, $v_1>0$ and $x_1<0$ be given.
Assume $G^*$ is on a negative excursion
at time $t_0$ and define $\tau_{\sV^*}^{t_0}$
and $\tau_{\sX^*}^{t_0}$ by (\ref{7.12}).
Further define
\be\label{7.17a}
\alpha
:=
\arctan\left(-\frac{\sqrt{1-\rho^2}}{\rho}\right)
\in(0,\pi/2).\quad
\ee
Then
\begin{align}
\P\big\{\tau_{\sV^*}^{t_0}
<\tau_{\sX^*}^{t_0}\big|\sV^*(t_0)=v_1,
\sX^*(t_0)=x_1\big\}
&=
\frac{\theta_0}{\alpha},\label{7.17}\\
\P\big\{\tau_{\sX^*}^{t_0}
<\tau_{\sV^*}^{t_0}\big|\sV^*(t_0)=v_1,
\sX^*(t_0)=x_1\big\}
&=
\frac{\alpha-\theta_0}{\alpha}.\label{7.18}
\end{align}
\end{theorem}

\begin{proof}
From (\ref{7.13d}) and (\ref{7.15}), we see
that we must calculate the probabilities
that the pair $(\Dhat,\Ehat)$ of correlated, 
zero-drift Brownian motions exits the first
quadrant on the vertical axis (equation
(\ref{7.17})) or the horizontal axis
(equation (\ref{7.18})). 

Iyengar \cite{Iyengar} has computed
the passage time probabilities we need,
and certain formulas in \cite{Iyengar}
have been corrected by Metzler \cite{Metzler}.
We follow the notation in \cite{Metzler}, which
defines independent Brownian motions
$Z_1(t):=
\frac{1}{\sigma_+\sigma_-\sqrt{1-\rho^2}}
\big(\sigma_-\Dhat(t)-\rho\sigma_+\Ehat(t)\big)$,
$Z_2(t):=
\frac{1}{\sigma_-}\Ehat(t)$,
and their radial component
$R(t)=\sqrt{Z_1^2(t)+Z_2^2(t)}$ so that
\be\label{7.21c}
R^2(0)=\frac{1}{1-\rho^2}\left(\frac{v_1^2}{\sigma_+^2}
+\frac{2\rho v_1x_1}{\sigma_+\sigma_-}
+\frac{x_1^2}{\sigma_-^2}\right),\quad
(Z_1(0),Z_2(0))
=(R(0)\cos \theta_0,R(0)\sin\theta_0).
\ee
According to (2.5) in \cite{Metzler},
\begin{align}
\lefteqn{\P\big\{R(\tau_{\Dhat}\wedge\tau_{\Ehat})\in dr,
\tau_{\Dhat}<\tau_{\Ehat}\big|\Dhat(0)=v_1,
\Ehat(t)=-x_1\big\}}\hspace{1in}\nonumber\\
&=
\frac{dr}{\alpha R(0)}\cdot
\frac{(r/R(0))^{(\pi/\alpha)-1}\sin(\pi\theta_0/\alpha)}
{\sin^2(\pi\theta_0/\alpha)+
\big[(r/R(0))^{\pi/\alpha}+\cos(\pi\theta_0/\alpha)
\big]^2},\quad r>0.\label{7.24}
\end{align}
With the change of variable
$y=
((r/R(0))^{\pi/\alpha}+\cos(\pi\theta_0/\alpha))/
\sin(\pi\theta_0/\alpha)$,
integration yields
$$
\P\big\{\tau_{\Dhat}<\tau_{\Ehat}\big|
\Dhat(0)=v_1,\Ehat(0)=-x_1\big\}=
\frac12-\frac{1}{\pi}\arctan(\cot(\pi\theta_0/\alpha)).
$$
Let $x=\pi\theta_0/\alpha$ and 
$y=\arctan(\cot x)$.  Then
$\sin y/\cos y=\tan y =\cot x =\cos x/\sin x$,
or equivalently,
$\cos(x+y)=\cos x\cos y-\sin x \sin y=0$.
Hence, $x+y$ is an odd multiple of $\pi/2$.
But $0<x<\pi$ and $-\pi/2<y<\pi/2$,
which implies $x+y=\pi/2$.  This shows that
$1/2-y/\pi=\theta_0/\alpha$, i.e.,
\be\label{7.15e}
\P\big\{\tau_{\Dhat}<\tau_{\Ehat}\big|
\Dhat(0)=v_1,\Ehat(0)=-x_1\big\}
=\frac{\theta_0}{\alpha}.
\ee
Equation (\ref{7.17}) now follows from (\ref{7.13d}).
Equation (\ref{7.18}) follows from (\ref{7.17})
or from a similar analysis based on (2.4)
in \cite{Metzler}.
\end{proof}

Metzler's \cite{Metzler} formulas
(3.2) and (3.3) for the joint density
$f(s,t;v_1,x_1)dsdt
:=\P\big\{\tau_{\Dhat}\in ds,\tau_{\Ehat}\in dt
\big|\Dhat(0)=v_1,\Ehat(0)=-x_1\big\}$
are
\begin{align}
\lefteqn{f(s,t;v_1,x_1)}\nonumber\\
&=
\frac{\pi\sin\alpha}{2\alpha^2(t-s)
\sqrt{s(t-s\cos^2\alpha)}}
\exp\left(-\frac{R^2(0)}{2s}\cdot
\frac{t-s\cos 2\alpha}{(t-s)+(t-s\cos 2\alpha)}\right)
\nonumber\\
&\quad
\times\sum_{n=1}^{\infty}n
\sin\left(\frac{n\pi(\alpha-\theta_0)}{\alpha}\right)
I_{n\pi/(2\alpha)}
\left(\frac{R^2(0)}{2s}\cdot
\frac{t-s}{(t-s)+(t-s\cos 2\alpha)}\right)dsdt,
\quad s<t,
\label{7.27}
\end{align}
\begin{align}
\lefteqn{f(s,t;v_1,x_1)}\nonumber\\
&=
\frac{\pi\sin\alpha}{2\alpha^2(s-t)
\sqrt{t(s-t\cos^2\alpha)}}
\exp\left(-\frac{R^2(0)}{2t}\cdot
\frac{s-t\cos 2\alpha}{(s-t)+(s-t\cos 2\alpha)}\right)
\nonumber\\
&\quad
\times\sum_{n=1}^{\infty}n
\sin\left(\frac{n\pi\theta_0}{\alpha}\right)
I_{n\pi/(2\alpha)}
\left(\frac{R^2(0)}{2t}\cdot
\frac{s-t}{(s-t)+(s-t\cos 2\alpha)}\right)dsdt,
\quad s>t,\nonumber
\end{align}
where 
\be\label{7.25a}
I_{\nu}(z):=\left(\frac{z}{2}\right)^{\nu}
\sum_{k=0}^{\infty}\left(\frac{z^2}{4}\right)^k
\frac{1}{k!\Gamma(\nu+k+1)}
\ee 
is the modified Bessel
function of the first kind of order $\nu$.
It follows from (\ref{7.15e}) that
\be\label{7.29}
\P\big\{\tau_{\Dhat}\in ds\big|
\tau_{\Dhat}<\tau_{\Ehat},
\Dhat(0)=v_1,\Ehat(0)=-x_1
\big\}
=\left(\frac{\alpha}{\theta_0}
\int_s^{\infty}f(s,t)dt\right)ds,\quad s>0.
\ee
Similarly,
\be\label{7.30}
\P\big\{\tau_{\Ehat}\in dt\big|
\tau_{\Ehat}<\tau_{\Dhat},
\Dhat(0)=v_1,\Ehat(0)=-x_1
\big\}
=\left(\frac{\alpha}{\alpha-\theta_0}
\int_t^{\infty}f(s,t)ds\right)dt,\quad t>0.
\ee

\begin{theorem}\label{T7.2}
Let $t_0>0$, $v_1>0$ and $x_1<0$ be given.  
Assume $G^*$ is on a negative excursion
at time $t_0$ and define
$\tau_{\sV^*}^{t_0}$ and $\tau_{\sX^*}^{t_0}$
by (\ref{7.12}).  Then
\begin{align}
\P\big\{\tau_{\sV^*}^{t_0}\in t_0+ds\big|
\tau_{\sV^*}^{t_0}<\tau_{\sX^*}^{t_0},
\sV^*(t_0)=v_1,\sX^*(t_0)=x_1
\big\}
&=
\left(\frac{\alpha}{\theta_0}\int_s^{\infty}f(s,t)dt
\right)ds,\,\,
s>0,\label{7.31}\\
\P\big\{\tau_{\sX^*}^{t_0}\in t_0+dt\big|
\tau_{\sX^*}^{t_0}<\tau_{\sV^*}^{t_0},
\sV^*(t_0)=v_1,\sX^*(t_0)=x_1
\big\}
&=
\left(\frac{\alpha}{\alpha-\theta_0}
\int_t^{\infty}f(s,t)ds\right)dt,\,\, t>0,\label{7.32}
\end{align}
where $\alpha$ and $\theta_0$ are defined
by (\ref{7.17a}).
\end{theorem}

\begin{proof}
Equation (\ref{7.31}) follows from (\ref{7.14d})
and (\ref{7.29}); (\ref{7.32}) follows
from (\ref{7.16}) and (\ref{7.30}).
\end{proof}

\subsection{Time between renewals}\label{BetweenRenewals}
According to Assumption \ref{Assumption2},
the initial condition for the limiting
systems is 
\be\label{7.30a}
\big(\sV^*(0),\sW^*(0),\sX^*(0),\sY^*(0)\big)
=\big(V^*(0),0,0,Y^*(0)\big),
\ee
where $V^*(0)>0$ and $Y^*(0)<0$.
The interior queues $(\sW^*,\sX^*)$
are the split Brownian motion described
by Remark \ref{R4.20x}, where the two-speed
Brownian motion $G^*$ used in that remark
is from Corollary \ref{C4.11}.
These processes are defined for all time
by the dynamics described in Remark \ref{R4.20x}
and Corollary \ref{C4.11}.  Likewise,
the bracketing processes $\sV^*$
and $\sY^*$ are defined for all time by
(\ref{9.4}) and (\ref{5.87}).
From the initial time until one of
the bracketing queues 
$\sV^*$ and $\sX^*$ vanishes,
the interior queues are the
$J_1$-limit of the scaled limit-order book
pair of queues $(\frac{1}{\sqrt{n}}W^n(nt),
\frac{1}{\sqrt{n}}X^n(nt))$
and the bracketing queues are the $M_1$-limit
of the scaled limit-order book pair of queues
$(\frac{1}{\sqrt{n}}V^n(nt),\frac{1}{\sqrt{n}}Y^n(nt))$.

The initial condition (\ref{7.30a}) corresponds
to $G^*(0)=0$.  Immediately after time
zero, $G^*$ takes the value zero infinitely
many times.  Therefore, the
state in which $\sV^*>0$, $\sW^*=0$,
$\sX^*=0$ and $\sY^*<0$ will occur
repeatedly.  Eventually on an excursion
of $G^*$ away from the origin, one of the bracketing
queues will vanish. When $G^*$ is on a negative
excursion, $\sV^*$ can vanish, and when $G^*$
is on a positive excursion, $\sY^*$ can vanish.
When this happens, we say there is a renewal.
The state at the time of renewal is
described in Theorem \ref{T6.2}.

We have constructed
$G^*$ and $(\sV^*,\sW^*,\sX^*,\sY^*)$
so that they continue on unaffected by
the renewal.  After the renewal, their
behavior is no longer
relevant to the system we want to study.
However, their insensitivity to the renewal
permits us to study them using the vehicle
of Poisson random measures.

To study the time to renewal, we focus
on the case that $G^*$ goes on a negative
excursion and $\sV^*$ vanishes before
$G^*$ returns to zero.  The other case is analogous.
Each time $G^*$, or equivalently, $\sX^*$,
goes on a negative excursion away from zero,
$\sV^*$ begins at $\kappa_L>0$
and has a chance to reach zero.  If $\sV^*$
fails to reach zero before the end of
the excursion of $G^*$, it is reset to $\kappa_L$.
We begin
by deriving a formula for the first passage
time of $\sV^*$ to zero conditioned on
$\sX^*$ being on an excursion of length
$\ell>0$.  To choose such
an excursion, we set a positive
threshold and consider in chronological order
the iid sequence
of negative excursions of $G^*$
whose lengths exceed
the threshold.  
Let $E$ denote a generic excursion
chosen from this sequence.
The threshold is irrelevant because
$$
\P\big\{E\in C\big|\lambda(E)=\ell\big\}
=\P^{\sigma^2_-\ell}
\big\{e\in\sE_-:e\circ (\sigma_-^2\id)\in C\big\},
\quad C\in\sB(\sE_-),
$$
regardless
of the threshold, provided that it is less
than $\ell$.
Let $\Lambda$ denote the left-endpoint
of the generic excursion $E$.
The first passage time after $\Lambda$
of $\sV^*$ to zero is
\be\label{tauE}
\tau_{\sV^*}^E:=\inf\{s\geq 0:
\sV^*(\Lambda+s)\big\}.
\ee

\begin{lemma}\label{L7.3}
For $\ell>0$ and $0<s<\ell$, we have
\begin{align}
\hspace{-3pt}p_{\sV^*}(s,\ell)ds
&:=
\P\big\{\tau_{\sV^*}^E\in ds\big|\lambda(E)=\ell\big\}
\nonumber\\
&=
\frac{\sqrt{2\pi(1-\rho^2)\ell^3}\,\pi^2\sigma_+
\sin\alpha}{2\kappa_L\alpha^3(\ell-s)
\sqrt{s(\ell-s\cos^2\alpha)}}
\exp\left(-\frac{\kappa_L^2}
{2\sigma_+^2(1-\rho^2)s}
\cdot\frac{\ell-s\cos 2\alpha}{(\ell-s)
+(\ell-s\cos 2\alpha)}\right)\nonumber\\
&\quad
\times\sum_{n=1}^{\infty}(-1)^{n-1}n^2
I_{n\pi/(2\alpha)}
\left(\frac{\kappa_L^2}{2\sigma_+^2(1-\rho^2)s}
\cdot\frac{\ell-s}{(\ell-s)+\ell-s\cos 2\alpha}
\right)ds.\label{7.32a}
\end{align}
\end{lemma}

\begin{proof}
Let $\varepsilon\in(0,\ell)$
and $t_1\in(0,\varepsilon)$ be given and define
$$
\tau_{\sV^*}^{E,t_1}:=\inf\big\{s\geq t_1:
\sV^*(\Lambda+s)=0\}
$$
On the event $\{\tau_{\sV^*}^E>t_1\}$, we have
$\tau_{\sV^*}^E=\tau_{\sV^*}^{E,t_1}$,
$\sV^*(\Lambda+t_1)>0$ and $\sX^*(\Lambda+t_1)<0$.
Let $\delta\in(0,(\ell-\varepsilon)/3)$  and
$s\in(\varepsilon+\delta,\ell-\delta)$ be given.
Then
\begin{align}
\lefteqn{\Big|
\P\big\{\tau_{\sV^*}^E\in ds\big|\lambda(E)=\ell\big\}
}\qquad
\nonumber\\
\lefteqn{-\P\big\{
\tau_{\sV^*}^E>t_1,
\tau_{\sV^*}^{E,t_1}\in ds,
\kappa_L/2\leq\sV^*(\Lambda+t_1)\leq 2\kappa_L,
-1\leq\sX^*(\Lambda+t_1)< 0
\big|\lambda(E)=\ell\big\}\Big|}\quad\nonumber\\
&\leq
\P\big\{\sV^*(\Lambda+t_1)<\kappa_L/2
\big|\lambda(E)=\ell\big\}
+\P\big\{\sV^*(\Lambda+t_1)> 2\kappa_L
\big|\lambda(E)=\ell\big\}\hspace{2cm}\nonumber\\
&\qquad
+\P\big\{\sX^*(\Lambda+t_1)< -1\big|\lambda(E)
=\ell\big\}.\label{7.35c}
\end{align}
Because $\sV^*(\Lambda)=\kappa_L>0$
(see (\ref{7.3}))
and $\sX^*(\Lambda)=0$, the limits
of the three terms on the right-hand side
of (\ref{7.35c}) are zero as
$t_1\downarrow 0$
Therefore, 
\begin{align}
\lefteqn{\P\big\{\tau_{\sV^*}^E\in ds
\big|\lambda(E)=\ell\big\}}\nonumber\\
&=
\lim_{t_1\downarrow 0}
\P\big\{\tau_{\sV^*}^E>t_1,
\tau_{\sV^*}^{E,t_1}\in ds,
\kappa_L/2\leq\sV^*(\Lambda+t_1)\leq 2\kappa_L,
-1\leq\sX^*(\Lambda+t_1)< 0
\big|\lambda(E)=\ell\big\}\nonumber\\
&=
\lim_{t_1\downarrow 0}
\P\big\{
\tau_{\sV^*}^{E,t_1}\in ds,
\kappa_L/2\leq\sV^*(\Lambda+t_1)\leq 2\kappa_L,
-1\leq\sX^*(\Lambda+t_1)< 0
\big|\lambda(E)=\ell\big\}.\label{7.33b}
\end{align}

Still conditioning on $\lambda(E)=\ell$,
using Proposition \ref{PA.1} in Appendix
\ref{appendix} and the representation (\ref{7.3})
of $\sV^*$, we construct a pair
of Brownian motions $(\Dhat,\Ehat)$ with
covariance (\ref{cov}) such that
$$
\big(\sV^*(\Lambda+t_1+t),-\sX^*(\Lambda+t_1+t)\big)
=\big(\Dhat(t),\Ehat(t)\big),\quad
0\leq t\leq \ell-t_1.
$$
With $\tau_{\Dhat}:=\inf\{t\geq 0:\Dhat(t)=0\}$,
$\tau_{\Ehat}:=\inf\{t\geq 0:\Ehat(t)=0\}$,
and $\Ctilde$ as in the proof of Lemma \ref{L7.1a}, 
we may write
\begin{align}
\lefteqn{\P\big\{
\tau_{\sV^*}^{E,t_1}\in ds,
\kappa_L/2\leq\sV^*(\Lambda+t_1)\leq 2\kappa_L,
-1\leq \sX^*(\Lambda+t_1)<0
\big|\lambda(E)=\ell\big\}}\nonumber\\
&=
\int_{x_1=-1}^0\int_{v_1=\kappa_L/2}^{2\kappa_L}
\P\big\{\tau_{\sV^*}^{E,t_1}\in ds\big|
\sV^*(\Lambda+t_1)=v_1,
\sX^*(\Lambda+t_1)=x_1,\lambda(E)=\ell\big\}
\nonumber\\
&\qquad\times
\P\big\{\sV^*(\Lambda+t_1)\in dv_1,
\sX^*(\Lambda+t_1)\in dx_1
\big|\lambda(E)=\ell\big\}\nonumber\\
&=
\int_{x_1=-1}^0\int_{v_1=\kappa_L/2}^{2\kappa_L}
\P\big\{\tau_{\Dhat}\in ds-t_1\big|
\Dhat(0)=v_1,\Ehat(0)=-x_1,
\tau_{\Ehat}=\ell-t_1\big\}
\nonumber\\
&\qquad\times
\P\big\{\Ctilde(t_1)\in dv_1-\kappa_L
+\rho\sigma_+x_1/\sigma_-,
E(t_1)\in dx_1
\big|\lambda(E)=\ell\big\}.\label{7.34b}
\end{align}
Because of (\ref{7.33b}), it suffices to
show that the limit
of the right-hand side of (\ref{7.34b})
as $t_1\downarrow 0$ agrees
with the right-hand side of (\ref{7.32a}).
Because $\varepsilon$ and $\delta$
can be made arbitrarily small, we
will have obtained (\ref{7.32a}) for
$0<s<\ell$.

We consider the first factor in
the integrand on the right-hand side of 
(\ref{7.34b}).
The joint probability density function
of $(\tau_{\Dhat},\tau_{\Ehat})$ on 
$\{\tau_{\Dhat}<\tau_{\Ehat}\}$
conditioned on $\Dhat(0)=v_1$ and
$\Ehat(0)=-x_1$ is given by (\ref{7.27}).
The marginal density of
$\tau_{\Ehat}$ with the same conditioning
is the density of the first passage time
of $\Ehat$ from $-x_1$ to $0$.  This is
$$
\P\big\{\tau_{\Ehat}\in dt\big|
\Ehat(0)=-x_1\big\}
=\frac{|x_1|}{\sigma_-\sqrt{2\pi t^3}}
\exp\left(-\frac{x_1^2}{2\sigma_-^2t}\right)dt.
$$
Finally, $\sin(n\pi(\alpha-\theta_0)/\alpha)
=(-1)^{n-1}\sin(n\pi\theta_0/\alpha)$.
Therefore, with $R(0)$ given by (\ref{7.24}),
we have
\begin{align}
\lefteqn{\varphi(s,v_1,x_1;t_1)ds}\nonumber\\
&:=
\P\big\{\tau_{\Dhat}\in ds-t_1\big|
\Dhat(0)=v_1,
\Ehat(0)=-x_1,\tau_{\Ehat}=\ell-t_1\big\}
\hspace{5cm}\nonumber\\
&=
\frac{f(s-t_1,\ell-t_1;v_1,x_1)dsd\ell}
{\P\{\tau_{\Ehat}\in d\ell-t_1|
\Ehat(0)=-x_1\}}\nonumber\\
&=
\frac{\sigma_-\sqrt{2\pi(\ell-t_1)^3}\,
\pi\sin\alpha}{2\alpha^2(\ell-s)
\sqrt{(s-t_1)(\ell-t_1-(s-t_1)\cos^2\alpha)}}
\nonumber\\
&\quad\times
\exp\left(\frac{x_1^2}{2\sigma_-^2(\ell-t_1)}
-\frac{R^2(0)}{2(s-t_1)}
\cdot\frac{\ell-t_1-(s-t_1)\cos 2\alpha}
{(\ell-s)+(\ell-t_1-(s-t_1)\cos 2\alpha)}\right)
\label{7.35}\\
&\quad\times
\sum_{n=1}^{\infty}\frac{n}{x_1}
(-1)^n\sin\left(\frac{n\pi\theta_0)}{\alpha}\right)
I_{n\pi/(2\alpha)}\left(\frac{R^2(0)}{2(s-t_1)}
\cdot\frac{\ell-s}{(\ell-s)+(\ell-t_1
-(s-t_1)\cos 2\alpha)}\right)ds.\nonumber
\end{align}

To compute the limit of (\ref{7.34b}), 
we need a bound on
$\varphi(s,v_1,x_1;t_1)$.
The conditions on $s$ and $t_1$ guarantee
the $\ell-s$, $s-t_1$, $\ell-t_1-(s-t_1)\cos^2\alpha$,
$\ell-t_1$, and $(\ell-s)
+(\ell-t_1-(s-t_1)\cos 2\alpha)$
are bounded away from zero.
The integrals on the right-hand side of (\ref{7.34b})
are over bounded intervals, and for $v_1$
and $x_1$ in these intervals,
$R(0)$ given by (\ref{7.21c}) is bounded.
In particular
\be\label{7.39c}
\sup_{
\footnotesize{
\begin{array}{c}
0<t_1<\varepsilon\\
\varepsilon+\delta<s<\ell-\delta\\
\kappa_L/2\leq v_1\leq 2\kappa_L
\end{array}}}
\varphi(s,v_1,x_1;t_1)
\leq
K_1\sum_{n=1}^{\infty}
\frac{n}{|x_1|}
\left|\sin\left(\frac{n\pi\theta_0)}{\alpha}
\right)\right|
I_{n\pi/(2\alpha)}(K_1),
\ee
for $-1\leq x_1<0$ and
for some finite constant $K_1$
that does not depend on $x_1\in[-1,0)$.

Observe that
$|\sin(n\pi\theta_0/\alpha)|\leq n\pi\theta_0/\alpha$
and $\arctan(x)\leq x$ for $x\geq 0$.
From (\ref{7.17a}) we have
\be\label{7.41}
\theta_0:=
\arctan\left(\frac{\sigma_+\sqrt{1-\rho^2}\,|x_1|}
{\sigma_-v_1+\sigma_+\rho x_1}\right)
\leq \frac{\sigma_+\sqrt{1-\rho^2}\,|x_1|}
{\sigma_-v_1+\sigma_+\rho x_1}
\leq\frac{\sigma_+|x_1|}{\sigma_-v_1}
\ee
because both $\rho$ and $x_1$ are negative.
Set $\nu=\pi/(2\alpha)$.  We have
\be\label{7.42}
K_1\sum_{n=1}^{\infty}
\frac{n}{|x_1|}
\left|\sin\left(\frac{n\pi\theta_0)}{\alpha}
\right)\right|I_{n\pi/(2\alpha)}(K_1)
\leq
2\nu K_1\frac{\sigma_+}{\sigma_-v_1}
\sum_{n=1}^{\infty}n^2I_{n\nu}(K_1).
-1\leq x_1<0.
\ee 
We use the ratio test to show that the sum
on the right-hand side of (\ref{7.42}) is finite.
Because
$\lim_{z\rightarrow\infty}\Gamma(z)/\Gamma(z+\nu)=0$,
for all large $n$ and all $k\geq 0$, we have
$\left(\frac{n+1}{n}\right)^2
\frac{(K_1/2)^\nu\Gamma(n\nu+k+1)}
{\Gamma\big((n+1)\nu+k+1\big)}
\leq\frac12$.
For such $n$,
\begin{align*}
\lefteqn{(n+1)^2I_{(n+1)\nu}(K_1)}\nonumber\\
&=
(n+1)^2
\left(\frac{K_1}{2}\right)^{(n+1)\nu}
\sum_{k=0}^{\infty}\left(\frac{K_1^2}{4}\right)^k
\frac{1}{k!\Gamma\big((n+1)\nu+k+1)}\nonumber\\
&=
n^2
\left(\frac{K_1}{2}\right)^{n\nu}
\sum_{k=0}^{\infty}\left(\frac{K_1^2}{4}\right)^k
\frac{1}{k!\Gamma(n\nu+k+1)}
\cdot\left(\frac{n+1}{n}\right)^2
\frac{(K_1/2)^{\nu}\Gamma(n\nu+k+1)}{\Gamma\big((n+1)
\nu+k+1)}\nonumber\\
&\leq
\frac12n^2I_{n\nu}(K_1).
\end{align*}
The conclusion we draw from (\ref{7.39c}),
(\ref{7.42}), and the convergence of the sum
on the right-hand side of (\ref{7.42}) is that
$$
\sup\big\{\varphi(s,v_1,x_1;t_1):
0<t_1<\varepsilon, \varepsilon+\delta<\ell-\delta,
\kappa_L/2\leq v_1\leq 2\kappa_L,
-1\leq x_1<0\big\}<\infty
$$

We rewrite (\ref{7.34b}) as
\begin{align}
\lefteqn{\P\big\{\tau_{\sV^*}^{E,t_1}\in ds,
\kappa_L/2\leq\sV^*(\Lambda+t_1)\leq 2\kappa_L,
-1\leq\sX^*(\Lambda+t_1)< 0\big|\lambda(E)=\ell\}}
\hspace{2cm}\nonumber\\
&=
\E\big[\varphi\big(s,\kappa_L+\Ctilde(t_1)
-\rho\sigma_+E(t_1)/\sigma_-,E(t_1);t_1\big)
\nonumber\\
&\qquad\times
\ind_{\{\kappa_L/2\leq\kappa_L+\Ctilde(t_1)
-\rho\sigma_+E(t_1)/\sigma_-\leq 2\kappa_L,
-1\leq E(t_1)< 0\}}\big|\Lambda(E)=\ell\big].\label{7.46}
\end{align}
As $t_1\downarrow 0$,
$\Ctilde(t_1)\rightarrow 0$ and
$E(t_1)\rightarrow 0$ with $E(t_1)<0$
for $t_1\in(0,\varepsilon)$, so
the indicator function in (\ref{7.46})
converges to $1$.  Furthermore, $\varphi$
is bounded, which implies
\begin{align}
\lefteqn{\lim_{t_1\downarrow 0}
\P\big\{\tau_{\sV^*}^{E,t_1}\in ds,
\kappa_L/2\leq\sV^*(\Lambda+t_1)\leq 2\kappa_L,
-1\leq\sX^*(\Lambda+t_1)< 0\big|\lambda(E)=\ell\}}
\hspace{6cm}
\nonumber\\
&=
\lim_{
\footnotesize{\begin{array}{cc}
v_1\rightarrow\kappa_L\\
x_1<0,x_1\rightarrow 0\\
t_1\downarrow 0\\
\end{array}}}
\varphi(s,v_1,x_1;t_1)ds.\label{7.47}
\end{align}

It remains to compute the right-hand side
of (\ref{7.47}).  The only issue
with the right-hand side of (\ref{7.35})
is the indeterminate form
$\frac{1}{x_1}\sin(n\pi\theta_0/\alpha)$.
(Recall from (\ref{7.41}) 
the dependence
of $\theta_0$ on $v_1$ and $x_1$.)
According to L'H\^opital's Rule,
\be\label{7.48}
\lim_{
\footnotesize{
\begin{array}{cc}
v_1\rightarrow \kappa_L\\
x_1<0,x_1\rightarrow 0
\end{array}
}}
\frac{1}{x_1}
\sin\left(\frac{n\pi\theta_0}{\alpha}\right)
=-\frac{n\pi\sigma_+\sqrt{1-\rho^2}}
{\alpha\kappa_L\sigma_-}.
\ee
Because of the domination (\ref{7.42}), we may take
the limit (\ref{7.48}) inside the infinite
sum in (\ref{7.35}).  Putting (\ref{7.33b}) and
(\ref{7.47}) together, we obtain (\ref{7.32a}).
\end{proof}

\begin{remark}\label{R7.5}
{\rm 
Each time $G^*$, or equivalently, $\sW^*$,
goes on a positive excursion away from zero,
$\sY^*$ begins at $\kappa_R<0$
and has a chance to reach zero.
If $\sY^*$ fails to reach zero before
the end of the excursion of $G^*$, it is
reset to $\kappa_R$.
Let $E$ be such an excursion, let $\Lambda$
be its left endpoint, and define
$$
\tau_{\sY^*}^E:=\inf\{s\geq 0:
\sY^*(\Lambda+s)=0\}.
$$
Completely analogously to the proof of Lemma \ref{L7.3},
we can compute the probability density
function of $\sY^*$ conditioned on
$\lambda(E)=\ell$.  Indeed, we can replace
$(\sV^*,\sX^*)$ in the proof of Lemma \ref{L7.3}
by $(-\sY^*,-\sW^*)$ and use (\ref{5.87}) in
place of (\ref{9.4}) to arrive
at the formula
\begin{align*}
\lefteqn{p_{\sY^*}(s,\ell))}\nonumber\\
&:=
\P\big\{\tau_{\sY^*}^E\in ds\big|\lambda(E)=\ell\big\}
\nonumber\\
&=
\frac{\sqrt{2\pi(1-\rho^2)\ell^3}\,\pi^2\sigma_-
\sin\alpha}{2|\kappa_R|\alpha^3(\ell-s)
\sqrt{s(\ell-s\cos^2\alpha)}}
\exp\left(-\frac{\kappa_R^2}
{2\sigma_-^2(1-\rho^2)s}
\cdot\frac{\ell-s\cos 2\alpha}{(\ell-s)
+(\ell-s\cos 2\alpha)}\right)\nonumber\\
&\quad
\times\sum_{n=1}^{\infty}(-1)^{n-1}n^2
I_{n\pi/(2\alpha)}
\left(\frac{\kappa_R^2}{2\sigma_-^2(1-\rho^2)s}
\cdot\frac{\ell-s}{(\ell-s)+\ell-s\cos 2\alpha}
\right),\quad0<s<\ell.
\end{align*}  
}
\end{remark}

\begin{remark}\label{R7.6}
{\rm
In the context of Lemma \ref{L7.3}
and Remark \ref{R7.5}, we define
\begin{align*}
p_{\sV^*}(\ell)
&
:=\P\big\{\tau_{\sV^*}^E<\ell\big|
\lambda(E)=\ell\}
=\int_0^{\ell}p_{\sV^*}(s,\ell)ds,\\
p_{\sY^*}(\ell)
&:=
\P\big\{\tau_{\sY^*}^E<\ell\big|
\lambda(E)=\ell\}
=\int_0^{\ell}p_{\sY^*}(s,\ell)ds.
\end{align*}
Conditioned on $G^*$ being on a negative
(respectively, positive) excursion of
length $\ell$, these are the probabilities
$\sV^*$ (respectively, $\sY^*$) reaches zero
before the excursion ends.
}
\end{remark}

\begin{lemma}\label{L7.8}
The numbers
\be\label{7.43b}
\lambda_-:=
\frac{1}{\sigma_-}\int_0^{\infty}
\frac{p_{\sV^*}(\ell)d\ell}{\sqrt{2\pi\ell^3}},
\quad
\lambda_{+}:=
\frac{1}{\sigma_+}\int_0^{\infty}
\frac{p_{\sY^*}(\ell)d\ell}{\sqrt{2\pi\ell^3}}
\ee
are positive and finite.
\end{lemma}
\begin{proof}
It is obvious that $\lambda_+$
and $\lambda_-$ are positive.
We show that $\lambda_-<\infty$;
the proof that $\lambda_+<\infty$ is analogous.
Because $p_{\sV^*}(\ell)$ is a probability
and hence less than or equal to $1$,
$\int_1^{\infty}\frac{p_{\sV^*}(\ell)d\ell}
{\sqrt{2\pi\ell^3}}<\infty$.
We show
\be\label{7.44e}
\int_0^1\frac{p_{\sV^*}(\ell)d\ell}
{\sqrt{2\pi\ell^3}}<\infty.
\ee
Let $E$ be a negative excursion of $G^*$
of length $\ell$ and recall the representation
(\ref{7.3}) of $\sV^*$ during the excursion
$E$.  The probability that $\sV^*$ reaches
zero during this excursion is
\begin{align*}
p_{\sV^*}(\ell)
&=
\P\left\{\left.\min_{0\leq t\leq\ell}
\left(\kappa_L+\Ctilde(t)-\frac{\rho\sigma_+}{\sigma_-}
E(t-\Lambda)\right)\leq 0\right|\lambda(E)=\ell
\right\}\nonumber\\
&\leq
\P\big\{\min_{0\leq t\leq \ell}\Ctilde(t)\leq
-\kappa_L\big\}\nonumber\\
&\leq
2\P\big\{\Ctilde(\ell)\leq -\kappa_L\big\}
\end{align*}
by the reflection principle.
According to 
\cite[Problem~9.22,~p.~122~with~solution~on~p.~125]{KaratzasShreve},
$$
\P\big\{\Ctilde(\ell)\leq -\kappa_L\big\}
=
\frac{1}{\sqrt{2\pi}}
\int_{-\infty}^{-\kappa_L/\sqrt{(1-\rho^2)\sigma_+^2
\ell}}
e^{-u^2/2}du
\leq
\frac{1}{\kappa_L}
\sqrt{\frac{(1-\rho^2)\sigma_+^2\ell}{2\pi}}
e^{-\kappa_L^2/(2(1-\rho^2)\sigma_+^2\ell)}.
$$
It is now clear that the integral
(\ref{7.44e}) is convergent.
\end{proof}

We begin from time zero and 
consider $\sS_v^*$, $\sS_y^*$
and $\sS^*=\sS_v^*\wedge\sS_y^*$
of (\ref{Sstar}) and (\ref{6.2x}).
We have a downward price shift
at the time of the first renewal
if and only if $\sS_v^*<\sS_y^*$.

\begin{theorem}\label{T7.7}
We have
$
\P\big\{S_v^*<S_y^*\big\}
=\frac{\lambda_-}{\lambda_++\lambda_-}$
and
$\P\big\{S_v^*>S_y^*\big\}
=\frac{\lambda_+}{\lambda_++\lambda_-}.
$
\end{theorem}

\begin{proof}
The process $G^*$ is a two-speed Brownian
motion with the excursion representation
of Proposition \ref{TS.C.10}:
$$
G^*(\theta)=\int_{(0,L^Z(\theta)]}
\int_{\sE}e\big(\theta-A^Z(s-)\big)N^Z(ds\,de),
$$
where
$A^Z(s)=\int_{(0,s]}\lambda(e)N^Z(du\,de)$,
$L^Z(\theta)=\inf\{s\geq 0:A^Z(s)>\theta\}$,
and $N^Z$ is given by (\ref{N1.15}) and (\ref{N1.6}).
Consider the independent Poisson
random measures $N^Z_{\pm}$ for the
positive and negative excursions of $G^*$  given by
(\ref{N1.15}).  
We create new Poisson random
measures for the lengths of excursions
counted by $N^Z_{\pm}$ by defining
\begin{align*}
N_{\pm}^*\big((s,t]\times D\big)
&=
N_{\pm}^Z\big((s,t]
\times\{e\in\sE_{\pm}
:\lambda(e)\in D\}\big),
\quad 0\leq s\leq t,\,\,D\in\sB(0,\infty).
\end{align*}
According to (\ref{N1.15}) and (\ref{nBpm}),
the characteristic measures of $N_{\pm}^*$
are given by
\begin{align*}
n_{\pm}^*(D)
&=
n_{\pm}^B\big\{e\in\sE_{\pm}:
\lambda(\psi_{\pm}e)\in D\big\}
=
n_{\pm}^B\big\{e\in\sE_\pm:
\lambda(e)\in\sigma_{\pm}^2D\big\}\\
&=
\int_{\sigma_{\pm}^2D}\frac{d\ell}{\sqrt{2\pi\ell^3}}
=
\frac{1}{\sigma_{\pm}}\int_D
\frac{d\ell}{\sqrt{2\pi\ell^3}},\quad
D\in\sB(0,\infty).
\end{align*}
Because $\sV^*$ and $\sY^*$ are independent
of $G^*$, and hence independent
of $N_{\pm}^*$, we can decompose
$N_{\pm}^*$ into the independent Poisson random measures,
defined for $D\in\sB(0,\infty)$, by
\begin{align*}
N_-^0\big((s,t]\times D\big)
&:=
N^Z_-\big((s,t]\times
\{e\in\sE_-\colon\lambda(e)\in D\mbox{ and }
\sV^*\mbox{ reaches zero during }e\}\big),\\
N_-^{\times}\big((s,t]\times D\big)
&:=
N^Z_-\big((s,t]\times
\{e\in\sE_-\colon\lambda(e)\in D\mbox{ and }
\sV^*\mbox{ does not reach zero during }e\}\big),\\
N_+^0\big((s,t]\times D\big)
&:=
N^Z_+\big((s,t]\times
\{e\in\sE_+\colon\lambda(e)\in D\mbox{ and }
\sY^*\mbox{ reaches zero during }e\}\big),\\
N_+^{\times}\big((s,t]\times D\big)
&:=
N^Z_+\big((s,t]\times
\{e\in\sE_+\colon\lambda(e)\in D\mbox{ and }
\sY^*\mbox{ does not reach zero during }e\}\big).
\end{align*}
The respective characteristic measures 
of these Poisson random measures are
\begin{align*}
n_{-}^0(D)
&:=
\frac{1}{\sigma_-}
\int_D\frac{p_{\sV^*}(\ell)d\ell}{\sqrt{2\pi\ell^3}},
\qquad
n_-^{\times}(D)
:=
\frac{1}{\sigma_-}\int_D
\frac{(1-p_{\sV^*}(\ell))d\ell}
{\sqrt{2\pi\ell^3}},\\
n_{+}^0(D)
&:=
\frac{1}{\sigma_+}
\int_D\frac{p_{\sY^*}(\ell)d\ell}{\sqrt{2\pi\ell^3}},
\qquad
n_+^{\times}(D)
:=
\frac{1}{\sigma_+}\int_D
\frac{(1-p_{\sY^*}(\ell))d\ell}
{\sqrt{2\pi\ell^3}},\quad D\in\sB(0,\infty).
\end{align*}

The Poisson processes
$N_{\pm}^0((0,t]\times(0,\infty))$,
$t\geq 0$, are independent with intensities
$\lambda_{\pm}$ given by (\ref{7.43b}).
We define
$\sL_{\pm}=\inf\big\{t\geq 0:
N_{\pm}^0((0,t]\times(0,\infty))=1\big\}$.
Starting from time zero,
$\sV^*$ vanishes before $\sY^*$
if and only if
$\sL_-<\sL_+$.  These are independent exponential
random variables with means $1/\lambda_{\pm}$, i.e.,
$
\P\{\sL_-\in dt_1,\sL_+\in dt_2\}
=\lambda_-\lambda_+e^{-\lambda_-t_1-\lambda_+t_2},
\quad t_1>0,t_2>0.
$
Straightforward calculation shows
$\P\{\sL_-<\sL_+\}
=\lambda_-/(\lambda_++\lambda_-)$.
\end{proof}

\begin{theorem}\label{T7.8}
We have the characteristic function formulas,
defined for $\alpha\in\R$,
\begin{align}
\lefteqn{\E\big[e^{i\alpha\sS^*}\big|\sS_v^*<\sS_y^*\big]}
\nonumber\\
&=
\frac{\frac{\lambda_-+\lambda_+}{\lambda_-\sigma_-}
\int_{\ell=0}^{\infty}\frac{1}{\sqrt{2\pi\ell^3}}
\int_{s=0}^{\ell}e^{i\alpha s}p_{\sV^*}(s,\ell)dsd\ell}
{\frac{1}{\sigma_-}\int_0^{\infty}
e^{i\alpha\ell}
\frac{p_{\sV^*}(\ell)}{\sqrt{2\pi\ell^3}}d\ell
+\frac{1}{\sigma_+}\int_0^{\infty}
e^{i\alpha\ell}
\frac{p_{\sY^*}(\ell)}{\sqrt{2\pi\ell^3}}d\ell
+\left(\frac{1}{\sigma_-}+\frac{1}{\sigma_+}\right)
\sqrt{|\alpha|}\big(1-\mbox{sign}(\alpha)i\big)},
\label{cf1}\\
\lefteqn{\E\big[e^{i\alpha\sS^*}\big|\sS_v^*>\sS_y^*\big]}
\nonumber\\
&=
\frac{\frac{\lambda_-+\lambda_+}{\lambda_+\sigma_+}
\int_{\ell=0}^{\infty}\frac{1}{\sqrt{2\pi\ell^3}}
\int_{s=0}^{\ell}e^{i\alpha s}p_{\sY^*}(s,\ell)dsd\ell}
{\frac{1}{\sigma_-}\int_0^{\infty}
e^{i\alpha\ell}
\frac{p_{\sV^*}(\ell)}{\sqrt{2\pi\ell^3}}d\ell
+\frac{1}{\sigma_+}\int_0^{\infty}
e^{i\alpha\ell}
\frac{p_{\sY^*}(\ell)}{\sqrt{2\pi\ell^3}}d\ell
+\left(\frac{1}{\sigma_-}+\frac{1}{\sigma_+}\right)
\sqrt{|\alpha|}\big(1-\mbox{sign}(\alpha)i\big)},
\label{cf2}
\end{align}
\begin{align}
\lefteqn{\E\big[e^{i\alpha\sS^*}\big]}\nonumber\\
&=
\frac{\frac{1}{\sigma_-}
\int_{\ell=0}^{\infty}\frac{1}{\sqrt{2\pi\ell^3}}
\int_{s=0}^{\ell}e^{i\alpha s}p_{\sV^*}(s,\ell)dsd\ell
+\frac{1}{\sigma_+}
\int_{\ell=0}^{\infty}\frac{1}{\sqrt{2\pi\ell^3}}
\int_{s=0}^{\ell}e^{i\alpha s}p_{\sY^*}(s,\ell)dsd\ell}
{\frac{1}{\sigma_-}\int_0^{\infty}
e^{i\alpha\ell}
\frac{p_{\sV^*}(\ell)}{\sqrt{2\pi\ell^3}}d\ell
+\frac{1}{\sigma_+}\int_0^{\infty}
e^{i\alpha\ell}
\frac{p_{\sY^*}(\ell)}{\sqrt{2\pi\ell^3}}d\ell
+\left(\frac{1}{\sigma_-}+\frac{1}{\sigma_+}\right)
\sqrt{|\alpha|}\big(1-\mbox{sign}(\alpha)i\big)}.
\label{cf3}
\end{align}
\end{theorem}
\begin{proof}
There are three parts to $\sS^*$, which
are (1) the sum of the lengths of the negative
excursions of $G^*$ that conclude before $\sS^*$,
(2) the sum of the lengths of the positive
excursions of $G^*$ that conclude before $\sS^*$,
and (3) the time elapsed on the excursion
of $G^*$ that is in progress at time $\sS^*$. 
To capture time spent on positive and negative
excursions in which zero is not reached
by the appropriate bracketing process, we define
$$
H_{\pm}(s):=
\int_{u\in (0,s)}\int_{\ell\in(0,\infty)}
\ell N_{\pm}^{\times}(du\,d\ell).
$$
To capture the time elapsed on the excursion
of $G^*$ that is in progress at time $\sS^*$, we
define
\begin{align*}
R_{\sV^*}
&=
\mbox{Time elapsed
on the excursion beginning at local
time }\sL_-\mbox{ before }\sV^*\mbox{ reaches zero},\\
R_{\sY^*}
&=
\mbox{Time elapsed
on the excursion beginning at local
time }\sL_+\mbox{ before }\sY^*\mbox{ reaches zero},
\end{align*}
Then
$$
\sS^*=
\left\{\begin{array}{ll}
H_-(\sL_-)+H_+(\sL_-)+R_{\sV^*}&\mbox{if }
\sL_-<\sL_+,\\
H_-(\sL_+)+H_+(\sL_+)+R_{\sY^*}&\mbox{if }
\sL_+<\sL_-.
\end{array}\right.
$$

The random variable $R_{\sV^*}$
is independent
of $N_{\pm}^{\times}$ and consequently
independent of $H_{\pm}$.  It is
also independent of $\sL_{\pm}$.
To see this, note from the finiteness of
$\lambda_-=n_-^0(0,\infty)$ that
the Poisson random measure $N_-^0$ charges
only finitely many excursions in finite time,
and these excursions form 
an iid sequence with
$\P\{\lambda(E)\in d\ell\}=n_-^0(d\ell)/\lambda_-$.
In particular, this is the distribution
of the length of the first excursion 
in this sequence, which is the excursion beginning
at local time $\sL_-$.  This distribution
does not depend on $\sL_-$ nor $\sL_+$.
Let $E$ denote this excursion.  
Then $R_{\sV^*}$ is $\tau_{\sV^*}^E$
of (\ref{tauE}), but the choice of $E$
dictates that $\tau_{\sV^*}^E<\lambda(E)$.
The distribution of $R_{\sV^*}$ is the distribution
of $\tau_{\sV^*}^E$ under this condition, i.e.,
$$
\P\big\{R_{\sV^*}\in ds\big|\lambda(E)=\ell\big\}
=\frac{p_{\sV^*}(s,\ell)}{p_{\sV^*}(\ell)}ds,
\quad 0<s<\ell,
$$
which again does not depend on
$\sL_-$ nor $\sL_+$.  Therefore,
$$
\P\big\{R_{\sV^*}\in ds\big\}
=\frac{1}{\lambda_-}\int_{\ell=s}^{\infty}
\frac{p_{\sV^*}(s,\ell)}{p_{\sV^*}(\ell)}
n_-^0(d\ell)ds
=\frac{1}{\lambda_-\sigma_-}\int_{\ell=s}^{\infty}
\frac{p_{\sV^*}(s,\ell)}{\sqrt{2\pi\ell^3}}
d\ell ds,\quad s>0,
$$
which does not depend on $\sL_-$ nor $\sL_+$.

We begin the computation of (\ref{cf1}) with
the observation
\begin{align}
\E\big[e^{i\alpha \sS^*}|\sS_v^*<\sS_y^*\big]
&=
\E\big[e^{i\alpha(H_-(\sL_-)+H_+(\sL_-)+R_{\sV^*}
)}
\big|\sL_-<\sL_+\big]\nonumber\\
&=
\E\big[e^{i\alpha R_{\sV^*}}\big]\cdot
\E\big[e^{i\alpha(H_-(\sL_-)+H_+(\sL_-))}
\big|\sL_-<\sL_+\big],\label{7.55}
\end{align}
and
\be\label{7.56}
\E\big[e^{i\alpha R_{\sV^*}}\big]=
\frac{1}{\lambda_-\sigma_-}
\int_{s=0}^{\infty}e^{i\alpha s}\int_{\ell=s}^{\infty}
\frac{p_{\sV^*}(s,\ell)}{\sqrt{2\pi\ell^3}}
d\ell ds
\frac{1}{\lambda_-\sigma_-}
\int_{\ell=0}^{\infty}\frac{1}{\sqrt{2\pi\ell^3}}
\int_{s=0}^{\ell}e^{i\alpha s}p_{\sV^*}(s,\ell)dsd\ell.
\ee
The last term on the right-hand side of (\ref{7.55}) is
\begin{align}
\lefteqn{\E\big[e^{i\alpha(H_-(\sL_-)+H_+(\sL_-))}
\big|\sL_-<\sL_+\big]}\nonumber\\
&=
\frac{\lambda_-+\lambda_+}{\lambda_-}
\int_{t_1=0}^{\infty}\int_{t_2=t_1}^{\infty}
\E\big[e^{i\alpha(H_-(t_1)+H_+(t_2))}\big]
\lambda_-\lambda_+e^{-\lambda_-t_1-\lambda_+t_2}
dt_2 dt_1\nonumber\\
&=
(\lambda_-+\lambda_+)
\int_0^{\infty}\E\big[e^{i\alpha H_-(t)}\big]
\E\big[e^{i\alpha H_+(t)}\big]
e^{-(\lambda_1+\lambda_2)t}dt.\label{7.58}
\end{align}
According to the L\'evy-Hin\v{c}in formula,
\begin{align}
\E\big[e^{i\alpha H_-(t)}\big]
&=
\exp\left[-t\int_0^{\infty}
(1-e^{i\alpha\ell})n_-^{\times}(d\ell)
\right]
=
\exp\left[-t\int_0^{\infty}
(1-e^{i\alpha \ell})\frac{(1-p_{\sV^*}(\ell))}
{\sigma_-\sqrt{2\pi\ell^3}}d\ell\right],
\label{7.59}\\
\E\big[e^{i\alpha H_+(t)}\big]
&=
\exp\left[-t\int_0^{\infty}
(1-e^{i\alpha\ell})n_+^{\times}(d\ell)
\right]
=
\exp\left[-t\int_0^{\infty}
(1-e^{i\alpha\ell})
\frac{(1-p_{\sY^*}(\ell))}{\sigma_+
\sqrt{2\pi\ell^3}}d\ell\right].\label{7.60}
\end{align}
Substitution of (\ref{7.59}) and (\ref{7.60})
into (\ref{7.58}) results in
\begin{align*}
\lefteqn{\E\big[e^{i\alpha(H_-(\sL_-)+H_+(\sL_-))}
\big|\sL_-<\sL_+\big]}\nonumber\\
&=
\frac{\lambda_-+\lambda_+}
{\lambda_-+\lambda_+
+\int_0^{\infty}(1-e^{i\alpha\ell})
\frac{1-p_{\sV^*}(\ell)}{\sigma_-\sqrt{2\pi\ell^3}}
d\ell
+\int_0^{\infty}(1-e^{i\alpha\ell})
\frac{1-p_{\sY^*}(\ell)}{\sigma_+\sqrt{2\pi\ell^3}}
d\ell}.
\end{align*}
Using the formula
\be\label{7.62}
\int_0^{\infty}(1-e^{i\alpha\ell})
\frac{d\ell}{\sqrt{2\pi\ell^3}}
=\sqrt{|\alpha|}\big(1-\mbox{sign}(\alpha)i\big),\quad
\alpha\in\R,
\ee
proved in Appendix \ref{AppendixC}
and (\ref{7.43b}), we can rewrite this as
\begin{align}
\lefteqn{\E\big[e^{i\alpha(H_-(\sL_-)+H_+(\sL_-))}
\big|\sL_-<\sL_+\big]}\nonumber\\
&=
\frac{\lambda_-+\lambda_+}
{\frac{1}{\sigma_-}\int_0^{\infty}
e^{i\alpha\ell}
\frac{p_{\sV^*}(\ell)}{\sqrt{2\pi\ell^3}}d\ell
+\frac{1}{\sigma_+}\int_0^{\infty}
e^{i\alpha\ell}
\frac{p_{\sY^*}(\ell)}{\sqrt{2\pi\ell^3}}d\ell
+\left(\frac{1}{\sigma_-}+\frac{1}{\sigma_+}\right)
\sqrt{|\alpha|}\big(1-\mbox{sign}(\alpha)i\big)}.
\label{7.63}
\end{align}
Putting (\ref{7.55}), (\ref{7.56})
and (\ref{7.63}) together, we obtain (\ref{cf1}).

An analogous argument establishes (\ref{cf2}).
Curiously,
$$
\E\big[e^{i\alpha(H_-(\sL_-)+H_+(\sL_-))}
\big|\sL_-<\sL_+\big]
=\E\big[e^{i\alpha (H_-(\sL_+)+H_+(\sL_+))}
\big|\sL_->\sL_+\big].
$$
Equation (\ref{cf3}) follows from
Theorem \ref{T7.7} and the identity
$$
\E\big[e^{i\alpha \sS^*}\big]
=\E\big[e^{i\alpha\sS^*}\big|\sS_v^*<\sS_y^*\big]
\P\big\{S_v^*<S_y^*\big\}
+\E\big[e^{i\alpha \sS^*}\big|\sS_v^*>\sS_y^*\big]
\P\big\{\sS_v^*>\sS_y^*\big\}.
$$
\end{proof}

\appendix\section{Two-speed Brownian motion}\label{TwoSpeedBrMot}
\numberwithin{equation}{section}

In this section we make precise the notion
of two-speed Brownian motion alluded to
in Section \ref{SecIntro}, and we develop 
properties of this process needed for
subsequent sections.

\subsection{Definition and basic properties}\label{Definition}

\setcounter{theorem}{0}
\setcounter{equation}{0}
\setcounter{figure}{0}

\begin{definition}\label{TS.D.1}
{\rm Let $B$ be a standard Brownian motion, and
define
$P^B_{\pm}(t):=\int_0^t\ind_{\{\pm B(s)>0\}}ds$,
$t\geq 0$,
to be the occupation times of the positive
and negative half-lines.  Let
$\sigma_{\pm}$ be positive numbers,
and define
\be\label{TS.E.3a}
\Theta:=\frac{1}{\sigma_+^2}P^B_+
+\frac{1}{\sigma_-^2}P^B_-.
\ee
We call the process
\be\label{TS.E.1}
Z=B\circ \Theta^{-1}
\ee
a {\em two-speed Brownian motion}
with speed $\sigma_+^2$
when positive and speed
$\sigma_-^2$ when negative.}
\end{definition}

\begin{remark}\label{TS.R.1}
{\rm Two-speed Brownian motion is closely
related to {\em skew Brownian motion}, which
is defined as follows.
Begin with a reflected Brownian motion,
all of whose excursions away from zero are
positive, and then independently for each
excursion, with probability
$1-\alpha$, where $0<\alpha<1$, flip the excursion
so that it becomes negative.
With probability $\alpha$, let the excursion
remain positive.  It\^o \& McKean 
\cite[Section~4.2,~Problem~1]{ItoMcKean}
introduce this construction and
call the resulting process {\em skew
Brownian motion}.  Skew Brownian motion
was further developed by Walsh \cite{Walsh}
and Harrison \& Shepp \cite{HarrisonShepp},
and has since found application in models
of diffusion through semi-permeable membranes
and limits of queueing systems.

Using its scale function and speed measure,
Harrison \& Shepp \cite{HarrisonShepp}
provide a different construction
of skew Brownian motion as
$X_\alpha=r_{\alpha}\big(B\circ \Theta_{\alpha}^{-1}\big)$,
where
$$
r_{\alpha}(x)=
\left\{\begin{array}{ll}
x/(1-\alpha)&\mbox{if }x\geq 0,\\
x/\alpha&\mbox{if }x<0,
\end{array}\right.
$$
and 
$\Theta_{\alpha}(t)=\frac{1}{(1-\alpha)^2}P^B_+(t)
+\frac{1}{\alpha^2} P^B_-(t)$.
Setting $\alpha=\sigma_-/(\sigma_++\sigma_-)$,
we have
$\Theta_{\alpha}=(\sigma_++\sigma_-)^2\Theta$
and hence
$Z(\theta)=r_{\alpha}^{-1}
(X_\alpha((\sigma_++\sigma_-)^2\theta))$.
This may be rewritten as
\be\label{3.4h}
Z(\theta)=\frac{1}{\sigma_++\sigma_-}
\varphi\big(X_{\alpha}\big((\sigma_++\sigma_-
\big)^2\theta)\big),\quad \theta\geq 0,
\ee
where $\varphi:\R\rightarrow\R$ is defined by
\be\label{f}
\varphi(x):=\left\{\begin{array}{ll}
\sigma_+x&\mbox{if }x\geq 0,\\
\sigma_-x&\mbox{if }x<0.
\end{array}\right.
\ee

Using the Poisson point process excursion
representation of Brownian motion, we
see in Theorem \ref{T3.11} below that
two-speed Brownian motion 
$Z$ has the same law as the function
$\varphi$ applied to skew Brownian motion
constructed by excursion flipping as
in \cite{ItoMcKean}.
When $\sigma_++\sigma_-=1$,
we also have from (\ref{3.4h})
that $Z=\varphi(X_{\alpha})$.
Inverting these two formulas for $Z$,
we have an alternative proof that $X_{\alpha}$
has the law of the skew Brownian motion
in \cite{ItoMcKean}.
}
\end{remark}

\begin{proposition}\label{TS}
Let $Z$ be a two-speed Brownian motion as in
Definition \ref{TS.D.1}.  Then
\be\label{TS.E.2}
Z=B\circ\big(\sigma_+^2P^Z_++\sigma_-^2P^Z_-),
\ee
where
$P^Z_{\pm}(\theta)
=\int_0^\theta\ind_{\{\pm Z(\tau)>0\}}d\tau$.
Moreover,
\be\label{TS.E.3}
P^Z_++P^Z_-=\id.
\ee
Conversely, suppose $B$ is a standard
Brownian motion and $Z$ is a process satisfying
(\ref{TS.E.2}) and (\ref{TS.E.3}).  Then
$Z$ is a two-speed Brownian motion.
\end{proposition}

\begin{proof}
The paths of $\Theta$ 
of (\ref{TS.E.3a}) are strictly increasing,
Lipschitz continuous, and
$$
\Theta'(t)=\frac{1}{\sigma_+^2}\ind_{\{B(t)>0\}}
+\frac{1}{\sigma_-^2}\ind_{\{B(t)<0\}},\quad
t\geq 0.
$$
We denote by $T$ the inverse of $\Theta$.
Like $\Theta$, $T$ is strictly
increasing, Lipschitz continuous, and
\be\label{dAinv}
T'(\theta)
=
\frac{1}{\Theta'(T(\theta))}
=
\sigma_+^2\ind_{\{B\circ T(\theta)>0\}}
+\sigma_-^2\ind_{\{B\circ T(\theta)<0\}}
=
\sigma_+^2\ind_{\{Z(\theta)>0\}}
+\sigma_-^2\ind_{\{Z(\theta)<0\}}.
\ee
Integrating this equation, we obtain
\be\label{TS.E.4}
T=\sigma_+^2P^Z_++\sigma_-^2P^Z_-.
\ee
We rewrite (\ref{TS.E.1}) as
$Z=B\circ T$ and use (\ref{TS.E.4})
to convert this to (\ref{TS.E.2}).  Finally,
$$
\{\theta\geq 0:Z(\theta)=0\}
=\Theta\big(\{t\geq 0:B(t)=0\}\big)
$$
has Lebesgue measure zero because
$\{t\geq 0:B(t)=0\}$ has Lebesgue measure
zero and $\Theta$ is Lipschitz.
Equation (\ref{TS.E.3}) follows.

For the converse, we define $T$ by
(\ref{TS.E.4}), which is strictly
increasing by (\ref{TS.E.3}).  
Let $\Theta$ denote the inverse of $T$.  According to
(\ref{TS.E.2}), $Z\circ \Theta=B$.
Therefore,
$$
\frac{d\Theta(t)}{dt}=\frac{1}{T'(\Theta(t))}
=\frac{1}{\sigma_+^2}\ind_{\{Z\circ \Theta(t)>0\}}
+\frac{1}{\sigma_-^2}\ind_{\{Z\circ \Theta(t)<0\}}
=\frac{1}{\sigma_+^2}\ind_{\{B(t)>0\}}
+\frac{1}{\sigma_-^2}\ind_{\{B(t)<0\}}.
$$
Integration yields (\ref{TS.E.3a}), implying
$Z\circ\left(\frac{1}{\sigma_+^2}P^B_+
+\frac{1}{\sigma_-^2}P^B_-\right)=B$,
and (\ref{TS.E.1}) follows.  
\end{proof}

\subsection{Pathwise mappings}\label{Pp}

To develop the pathwise properties of two-speed
Brownian motion, we need to investigate
some maps on continuous paths.
Let $C_0[0,\infty)$ be the set of continuous
functions from $[0,\infty)$ to $\R$
with zero initial condition.  We equip $C_0[0,\infty)$
with the metric
$$
d(x,y)=\sum_{n=1}^{\infty}\frac{1}{2^n}
\left(1\wedge\sup_{0\leq \theta\leq n}
\big|x(\theta)-y(\theta)\big|\right),\,\,
x,y\in C_0[0,\infty).
$$
Convergence under this metric is uniform
on compact sets.
Let $\sB$ be the Borel $\sigma$-algebra
generated by this topology, and let $\sB\otimes\sB$
be the product $\sigma$-algebra on
$C_0[0,\infty)\times C_0[0,\infty)$.
We define the $\sB\otimes\sB$-measurable set
\be\label{2.22}
\sD=\big\{(z_+,z_-)\in C_0[0,\infty)\times C_0[0,\infty):
\liminf_{\theta\rightarrow \infty}z_+(\theta)
=\liminf_{\theta\rightarrow \infty}z_-(\theta)
=-\infty\big\}.
\ee

We recall the {\em Skorohod map} $\Gamma$
from $D[0,\infty)$ to itself defined by
\be\label{2.22b}
\Gamma(z)(\theta):=0\vee\max_{0\leq\nu\leq\theta}
\big({-z}(\nu)\big),\,\,\theta\geq 0,
\ee
for $z\in D[0,\infty)$.
For $z\in D[0,\infty)$, $\Gamma(z)$
is the unique nondecreasing function
starting at $0\vee(-z(0))$ for which $z+\Gamma(z)$
is nonnegative and $\Gamma(z)$ is constant on
intervals where $z+\Gamma(z)$ is strictly
positive.
We define $\Phi_{\pm}$ on $\sD$ by
\begin{align}
\Phi_+(z_+,z_-)(\theta)
&:=
\max\big\{\nu\in[0,\theta]:
\Gamma(z_+)(\nu)
=\Gamma(z_-)(\theta-\nu)\big\},\label{2.24}\\
\Phi_-(z_+,z_-)(\theta)
&:=
\min\big\{\nu\in[0,\theta]:\Gamma(z_-)(\nu)
=\Gamma(z_+)(\theta-\nu)\big\},\label{2.25}
\end{align}
for $\theta\geq 0$.
We show in the following Lemma \ref{TS.L.3} that
$\Phi_{\pm}$ maps $\sD$ into $C_0[0,\infty)$.
Finally, we define $\Psi:\sD\rightarrow C_0[0,\infty)$
by
\be\label{2.24b}
\Psi(z_+,z_-):=z_+\circ\Phi_+(z_+,z_-)
-z_-\circ\Phi_-(z_+,z_-).
\ee
The measurability of these functions
is proved in Yu \cite{Yu}, Appendix A.

To simplify notation, for the remainder
of this section we fix $(z_+,z_-)\in\sD$ and
denote
\be\label{2.26}
\lambda_{\pm}=\Gamma(z_{\pm}),\quad
p_{\pm}=\Phi_{\pm}(z_+,z_-),\quad
z=z_+\circ p_+-z_-\circ p_-.
\ee
Because $(z_+,z_-)\in\sD$, we have immediately that
\be\label{2.23}
\lim_{\theta\rightarrow\infty}\lambda_+(\theta)
=\lim_{\theta\rightarrow\infty}\lambda_-(\theta)
=\infty.
\ee
We rewrite (\ref{2.24}) and (\ref{2.25}) as
\begin{align}
p_+(\theta)
&=
\max\big\{\nu\in[0,\theta]:
\lambda_+(\nu)=\lambda_-(\theta-\nu)\big\},
\label{2.26a}\\
p_-(\theta)
&=
\min\big\{\nu\in[0,\theta]:
\lambda_-(\nu)=\lambda_+(\theta-\nu)\big\}.\nonumber
\end{align}

\begin{lemma}\label{TS.L.3}
The functions $p_{\pm}$ 
of (\ref{2.26}) are continuous, nondecreasing,
and satisfy
\begin{align}
p_{\pm}(0)
&=
0,\label{2.27}\\
\lambda_+\circ p_+
&=
\lambda_-\circ p_-,\label{2.28}\\
p_++p_-
&=
\id.\label{2.29}
\end{align}
\end{lemma}

\begin{proof}
Equation (\ref{2.27}) follows immediately
from the fact that $\lambda_{\pm}(0)=0$.

Because $\lambda_{\pm}$ is nondecreasing,
continuous, and $\lambda_{\pm}(0)=0$,
for every $\theta\geq 0$ there exists
$\nu_1\in[0,\theta]$ such that
$\lambda_+(\nu_1)=\lambda_-(\theta-\nu_1)$
and there exists $\nu_2\in[0,\theta]$ (in fact, we
can take $\nu_2=\theta-\nu_1$) such that
$\lambda_-(\nu_2)=\lambda_+(\theta-\nu_2)$.
Therefore $p_{\pm}$ takes values
in $[0,\theta]$. It is apparent that
the maximum $\nu_1$ for which
$\lambda_+(\nu_1)=\lambda_-(\theta-\nu_1)$
corresponds to the minimum $\nu_2=\theta-\nu_1$
for which $\lambda_-(\nu_2)=\lambda_+(\theta-\nu_2)$.
Therefore, (\ref{2.29}) holds.

By construction, 
$\lambda_+(p_+(t))=\lambda_-(t-p_+(t))$,
and (\ref{2.29}) gives (\ref{2.28}).

To see that $p_+$ is nondecreasing, let
$0\leq \theta_1<\theta_2$ be given.
By construction we have $\lambda_+(p_+(\theta_1))
=\lambda_-(\theta_1-p_+(\theta_1))$.
If, in addition,
$\lambda_+(p_+(\theta_1))
=\lambda_-(\theta_2-p_+(\theta_1))$,
then because $p_+(\theta_2)$ is the maximum of all
numbers satisfying
 $\lambda_+(\nu)=\lambda_-(\theta_2-\nu)$,
we have $p_+(\theta_2)\geq p_+(\theta_1)$.
On the other hand, if 
$\lambda_+(p_+(\theta_1))
<\lambda_-(\theta_2-p_+(\theta_1))$,
then $\lambda_+(p_+(\theta_2))
=\lambda_-(\theta_2-p_+(\theta_2))$
implies $p_+(\theta_2)>p_-(\theta_1)$.

We use a similar argument to show that
$p_-$ is nondecreasing.  Again let
$0\leq\theta_1<\theta_2$ be given.
Then $\lambda_-(p_-(\theta_2))
=\lambda_+(\theta_2-p_-(\theta_2))$.
If $\lambda_-(p_-(\theta_2))=
\lambda_+(\theta_1-p_-(\theta_2))$,
then $p_-(\theta_2)\geq p_-(\theta_1)$.
On the other hand, if $\lambda_-(p_-(\theta_2))
>\lambda_+(\theta_1-p_-(\theta_2))$,
then $p_-(\theta_2)>p_-(\theta_1)$.

We now establish right continuity.
Suppose $\theta_n\downarrow\theta$.  Then
$\lambda_+(p_+(\theta_n))
=\lambda_-(\theta_n-p_+(\theta_n))$.
Letting $n\rightarrow\infty$, we obtain
$\lambda_+(\lim_{n\rightarrow\infty}p_+(\theta_n))
=\lambda_-(\theta-\lim_{n\rightarrow\infty}p_+(\theta_n))$.
The definition of $p_+$ implies
$p_+(\theta)\geq\lim_{n\rightarrow\infty}p_+(\theta_n)$.
The reverse inequality holds because
$p_+$ is nondecreasing, and hence
$p_+$ is right continuous.
The right continuity of $p_-$ follows
from (\ref{2.29}). 

Finally, we prove left continuity.
Suppose $\theta_n\uparrow\theta$.  Then
$\lambda_-(p_-(\theta_n))
=\lambda_+(\theta_n-p_-(\theta_n))$.
Letting $n\rightarrow\infty$, we obtain
$\lambda_-(\lim_{n\rightarrow\infty}p_-(\theta_n))
=\lambda_+(\theta-\lim_{n\rightarrow\infty}
p_-(\theta_n))$.  This implies
$p_-(\theta)\leq\lim_{n\rightarrow\infty}p_-(\theta_n)$. 
The non-decrease of $p_-$ implies the reverse
inequality.  The left-continuity of $p_+$
follows from (\ref{2.29}).
\end{proof}

\begin{lemma}\label{TS.L.5}
The functions of (\ref{2.26}) satisfy
\begin{align}
&
|z|
=
z_+\circ p_++z_-\circ p_-
+\Gamma(z_+\circ p_++z_-\circ p_-),\label{2.32}\\
&
\Gamma(z_+\circ p_++z_-\circ p_-)
=
\lambda_+\circ p_++\lambda_-\circ p_-
=2\lambda_{\pm}\circ p_{\pm},\label{2.33}\\
&
\int_0^\theta\ind_{\{\pm z(\nu)>0\}}d\nu
\leq
p_{\pm}(\theta)
\leq
\int_0^{\theta}\ind_{\{\pm z(\nu)\geq 0\}}d\nu,
\,\,\theta\geq 0,\label{2.34}\\
&
(z_++\lambda_+)\circ p_+
\cdot 
(z_-+\lambda_-)\circ p_-
=0.\label{2.35}
\end{align}
\end{lemma}

\begin{proof}
We first prove (\ref{2.35}).
Assume for some $\theta$ that
$(z_++\lambda_+)\circ p_+(\theta)>0$.
We must show that under this assumption,
\be\label{2.36}
(z_-+\lambda_-)\circ p_-(\theta)=0.
\ee
Define
$$
a=\max\big\{\nu\in[0,\theta]:
(z_++\lambda_+)\circ p_+(\nu)=0\big\},\,\,
b=\inf\big\{\nu\in[\theta,\infty):
(z_++\lambda_+)\circ p_+(\nu)=0\big\}.
$$
We have $a\in[0,\theta)$ and
$(z_++\lambda_+)\circ p_+(a)=0$.
We have $b\in(\theta,\infty]$,
and because $p_+$ is nondecreasing,
$p_+(b)$ is defined in $[0,\infty]$.
On the interval or half-line $(a,b)$,
$(z_++\lambda_+)\circ p_+$
is strictly positive.  According
to the properties of the Skorohod map,
$\lambda_+\circ p_+$ is, consequently, constant
and equal to $\lambda_+(p_+(a))$
on $(a,b)$.  Note that $p_+(a)<p_+(\theta)
\leq p_+(b)$.

We show that
$p_+$ is linear on $[a,b)$, i.e.,
\be\label{2.41}
p_+(s)=p_+(a)+s-a,\,\,\forall s\in[a,b).
\ee
For $u\in(p_+(a),p_+(b))$ we have
$\lambda_+(u)=\lambda_+(p_+(a))=\lambda_-(a-p_+(a))$.
We must also have
$\lambda_-(a-p_+(a))>\lambda_-(a-u)$,
or else $u$ would satisfy the equation
$\lambda_+(u)=\lambda_-(a-u)$, a contradiction
to the definition of $p_+(a)$.  We conclude that
\be\label{2.37}
\lambda_+\big(p_+(a)\big)
=\lambda_-\big(a-p_+(a)\big)
>\lambda_-(a-u),\,\forall u\in\big(p_+(a),p_+(b)\big).
\ee

Now consider $s\in[a,b)$.  Because
$p_+(a)\leq p_+(s)\leq p_+(b)$ and
$\lambda_+$ is constant on the interval
or half-line $[p_+(a),p_+(b))$, we have
\begin{align}
p_+(s)
&=
\max\big\{\nu\in[0,s]:
\lambda_+(\nu)=\lambda_-(s-\nu)\big\}\nonumber\\
&=
\sup\big\{\nu\in[p_+(a),s\wedge p_+(b)]:
\lambda_+(\nu)=\lambda_-(s-\nu)\big\}\nonumber\\
&=
\sup\big\{\nu\in[p_+(a),s\wedge p_+(b)]:
\lambda_+(p_+(a))=\lambda_-(s-\nu)\big\}.
\label{2.38} 
\end{align}
Relation (\ref{2.37}) shows that if $\nu$
were not constrained from above, the supremum
in (\ref{2.38}) would be attained when
$s-\nu=a-p_+(a)$, i.e., at
$\nu=p_+(a)+s-a$.  However, because of
the constraint, the supremum is attained
instead at $\nu=(p_+(a)+s-a)\wedge p_+(b)$, i.e.,
\be\label{2.39}
p_+(s)=\big(p_+(a)+s-a)\wedge p_+(b),\,\,
\forall s\in[a,b).
\ee

We wish to remove the term $\wedge p_+(b)$
in (\ref{2.39}), thereby obtaining
(\ref{2.41}).  If $p_+(b)=\infty$,
this is trivial.  If $p_+(b)<\infty$
and $b=\infty$, when we choose a sequence
$b_n$ approaching $\infty$, and we
have $\lim_{n\rightarrow\infty}p_+(b_n)=p_+(b)<\infty$.
But (\ref{2.26a})
implies $\lambda_+(p_+(b_n))=\lambda_-(b_n-p_+(b_n))$, 
and the left-hand side converges to
$\lambda_+(p_+(b))<\infty$, whereas
the second equation in (\ref{2.23}) implies
the right-hand side converges to $\infty$.
Because of this contradiction, we conclude
that if $p_+(b)<\infty$, then also $b<\infty$.
We continue under the assumption
that $b<\infty$ and $p_+(b)<\infty$.
To show that 
\be\label{2.40}
p_+(b)\geq p_+(a)+b-a,
\ee
and consequently the term
$\wedge p_+(b)$ in (\ref{2.39}) may be removed,
we assume this is not the case and choose
$s\in(p_+(b)-p_+(a)+a,b)$, a nonempty subset of $(a,b)$.  
From (\ref{2.39}) we have
$p_+(s)=p_+(b)$, hence
$(z_++\lambda_+)\circ p_+(s)
=(z_++\lambda_+)\circ p_+(b)=0$.
This contradicts the definition of $b$,
and (\ref{2.40}), hence (\ref{2.41}), are proved.

For $\ve\in(0,\theta-a)$, we have from
(\ref{2.41}) that 
\be\label{2.42}
p_+(\theta-\ve)=p_+(\theta)-\ve,
\ee
and from the definition of $p_+$ that
$\lambda_+(p_+(\theta))=\lambda_-(\theta-p_+(\theta))$.
We must also have
\be\label{2.43}
\lambda_+(p_+(\theta))>
\lambda_-(\theta-\ve-p_+(\theta)),
\ee
or else $u=p_+(\theta)$ would satisfy
the equation $\lambda_+(u)=\lambda_-(\theta-\ve-u)$,
implying $p_+(\theta-\ve)\geq p_+(\theta)$,
a contradiction to (\ref{2.42}).
We use (\ref{2.28}) and (\ref{2.29}) to
rewrite (\ref{2.43}) as
$\lambda_-(p_-(\theta))>\lambda_-(p_-(\theta)-\ve)$
and conclude that $\lambda_-$
is not constant on any open interval
containing $p_-(\theta)$.  
It follows from the Skorohod map
property that $\lambda_-$
is constant on intervals where
$z_-+\lambda_-$ is positive that $z_-+\lambda_-$
cannot be positive at $p_-(\theta)$, i.e.,
(\ref{2.36}) holds.

We now turn our attention to (\ref{2.32}).
Because of (\ref{2.28}) and the third
equation in (\ref{2.26}), we can write
$z$ as the difference of the nonnegative
functions $(z_++\lambda_+)\circ p_+$
and $(z_-+\lambda_-)\circ p_-$.
Not both $(z_++\lambda_+)\circ p_+$
and $(z_-+\lambda_-)\circ p_-$
can be positive, and this implies
\be\label{2.44}
|z|=(z_++\lambda_+)\circ p_+
+(z_-+\lambda_-)\circ p_-
=z_+\circ p_++z_-\circ p_-
+2\lambda_{\pm}\circ p_{\pm}.
\ee
On intervals where $z$ is strictly positive,
$(z_++\lambda_+)\circ p_+$ is strictly
positive, and a Skorohod map property
implies $\lambda_+\circ p_+$ is constant.
Similarly, $\lambda_-\circ p_-$
is constant on intervals where $z$
is strictly negative.  Therefore,
$2\lambda_{\pm}\circ p_{\pm}$
is a nondecreasing
process starting at zero that is constant
on intervals where $|z|$ is positive
and for which the right-hand side of 
(\ref{2.44}) is nonnegative.
This implies (\ref{2.33}) and also (\ref{2.32}).

It remains to prove (\ref{2.34}).
Being open, the set
$$
\{\theta>0:z(\theta)>0\big\}=
\big\{\theta>0:(z_++\lambda_+)\circ p_+(\theta)>0
\big\}=\bigcup_{i\in I}(a_i,b_i)
$$
is, as indicated, the union of disjoint open
intervals, where the index set $I$ is either
finite or countably infinite and one of these
intervals may be an open half-line.  
Relation (\ref{2.41}) implies
$$
p_+(\theta)=p_+(a_i)+\int_{a_i}^{\theta}
\ind_{\{z(s)>0\}}ds,\,\,\forall \theta\in[a_i,b_i),
\,\,i\in I.
$$
Because $p_+$ is nondecreasing,
$$
p_+(\theta)\geq
\sum_{i\in I}\int_{a_i\wedge\theta}^{b_i\wedge\theta}
\ind_{\{z(s)>0\}}ds
=\int_0^\theta\ind_{\{z(s)>0\}}\,ds,\,\,\theta\geq 0.
$$
A symmetric argument shows that
$p_-(\theta)\geq \int_0^{\theta}\ind_{\{z(s)<0\}}ds$,
$\theta \geq 0$.
Therefore,
$$
p_{\pm}(\theta)=\theta-p_{\mp}(\theta)
\leq\int_0^{\theta}\big(1-\ind_{\{\mp z(s)>0\}}\big)ds
=\int_0^{\theta}\ind_{\{\pm z(s)\geq 0\}}ds,
\,\,\theta\geq 0,
$$
and (\ref{2.34}) is established.
\end{proof}

\subsection{Decomposition}\label{Disintegration}

We apply the mappings $\Phi_{\pm}$
and $\Psi$ of (\ref{2.24})--(\ref{2.24b})
to decompose two-speed Brownian motion
into independent standard Brownian motions.
Let $Z$ be a two-speed Brownian motion
as in Definition \ref{TS.D.1} and let
$\Theta$ be given by (\ref{TS.E.3a}). 
For $\theta\geq 0$, $T(\theta):=\Theta^{-1}(\theta)$
is a bounded stopping time for $\{\sF(t)\}_{t\geq 0}$,
the filtration generated by $B$.
By the Optional Sampling Theorem,
$Z=B\circ T$ is a martingale relative to the
filtration $\{\sF(T(\theta))\}_{\theta\geq 0}$.

We define
\begin{align}
M_{\pm}(\theta)
&:=
\pm\int_0^{\theta}\ind_{\{\pm Z(\tau)>0\}}dZ(\tau),
\label{2.45}\\
P^Z_{\pm}(\theta)
&:=
\int_0^{\theta}\ind_{\pm\{Z(\tau)>0\}}\,d\tau,
\label{2.46}\\
(P^Z_{\pm})^{-1}(t)
&:=
\min\big\{\theta\geq 0: P^Z_{\pm}(\theta)>t\big\},
\nonumber\\
Z_{\pm}
&:=
M_{\pm}\circ (P^Z_{\pm})^{-1},\label{2.48}\\
L^Z_{\pm}
&=
\Gamma(Z_{\pm}).\label{2.49}
\end{align}
Because $Z$ spends Lebesgue-measure zero time at the origin
(see (\ref{TS.E.3})),
$$
Z=M_+-M_-=Z_+\circ P^Z_+-Z_-\circ P^Z_-.
$$

\begin{lemma}\label{TS.L.6}
The processes $Z_+$ and $Z_-$ are independent
Brownian motions (relative to their own
filtrations) with variances $\sigma_+^2$
and $\sigma_-^2$, i.e., there exist two
independent standard Brownian motions $B_+$
and $B_-$ such that
$Z_{\pm}(\theta)=B_{\pm}( \sigma_{\pm}^2\theta)$,
$\theta\geq 0$.
\end{lemma}

\begin{proof}
We observe from (\ref{TS.E.1}) that
$\langle Z\rangle=\Theta^{-1}=T$, and 
(\ref{2.45}) and (\ref{dAinv})
imply
$$
\langle M_{\pm}\rangle(\theta)
=\int_0^{\theta}\ind_{\{\pm Z(\tau)>0\}}
d\langle Z\rangle(\tau)
=\int_0^{\theta}\ind_{\{\pm Z(\tau)>0\}}
\frac{dT(\tau)}{d\tau}\,d\tau
=\sigma_{\pm}^2P^Z_{\pm}(\theta),\,\,\theta\geq 0.
$$
Therefore,
$$
\langle M_{\pm}\rangle^{-1}(t)
:=
\inf\big\{\theta\geq 0: \langle M_{\pm}
\rangle(\theta)>t\big\}
=
\inf\big\{\theta\geq 0: P^Z_{\pm}(\theta)>t/\sigma_{\pm}^2
\big\}
=
(P^Z_{\pm})^{-1}(t/\sigma_{\pm}^2),
$$
which approaches $\infty$ as $t\rightarrow\infty$,
because $P^Z_{\pm}$ has bounded growth.
We define 
$$
B_{\pm}(t):=Z_{\pm}(t/\sigma_{\pm}^2)
=M_{\pm}\circ (P^Z_{\pm})^{-1}(t/\sigma_{\pm}^2)
=M_{\pm}\circ \langle M_{\pm}\rangle^{-1}(t),\,\,t\geq 0,
$$
We also observe from (\ref{2.45})
that $\langle M_+,M_-\rangle=0$.
Knight's Theorem
\cite[Theorem~3.4.13]{KaratzasShreve} 
implies that
$B_+$ and $B_-$ are independent Brownian motions.
\end{proof}

\subsection{Reconstruction}
\label{Reconstruction}

For the next lemma, it is not necessary
for $Z_{\pm}$ to be defined by (\ref{2.48}),
only that there are independent Brownian 
motions $Z_{\pm}$ related to a process
$Z$ by the formula
\be\label{2.38x}
Z=Z_+\circ P_+^Z-Z_-\circ P_-^Z.
\ee

\begin{lemma}\label{TS.T.7}
Let $Z_+$ and $Z_-$ be independent
Brownian motions with variances
$\sigma_+^2$ and $\sigma_-^2$, related to a process
$Z$ via (\ref{2.38x}), where
$P^Z_{\pm}$
and $L^Z_{\pm}$ are defined by 
(\ref{2.46}) and (\ref{2.49}).  Then
\begin{align}
L^Z_+\circ P^Z_+
&=
L^Z_-\circ P^Z_-,\label{2.54}\\
P^Z_{\pm}
&=
\Phi_{\pm}(Z_+,Z_-),\label{2.55}\\
Z
&=
\Psi(Z_+,Z_-),\label{2.56}
\end{align}
where $\Phi_{\pm}$ and $\Psi$ are defined
by (\ref{2.24})--(\ref{2.24b}).
\end{lemma}

\begin{proof}
We first verify (\ref{2.54}).
Define
$$
U_{\pm}(\theta):=
\max\big\{\nu\in[0,\theta]:
-Z_{\pm}(\nu)=L_{\pm}^Z(\theta)\big\},\,\,\theta \geq 0,
$$
and note from the definition (\ref{2.22b})
of the Skorohod map that this maximum
is attained.  Then 
\be\label{2.57}
-Z_{\pm}\circ U_{\pm}=L_{\pm}^Z.
\ee
Define also
$\Delta_{\pm}(t)
:=\max\big\{u\in[0,t]:
P^Z_{\pm}(u)=U_{\pm}\circ P^Z_{\pm}(t)\big\}$,
$t\geq 0$,
so that 
$P^Z_{\pm}\circ\Delta_{\pm}=U_{\pm}\circ P^Z_{\pm}$,
and hence, in light of (\ref{2.57}),
\be\label{2.58}
Z_{\pm}\circ P^Z_{\pm}\circ\Delta_{\pm}
=Z_{\pm}\circ U_{\pm}\circ P^Z_{\pm}
=-L^Z_{\pm}\circ P^Z_{\pm}.
\ee
Let $t\geq 0$ be given.  At time $\Delta_+(t)$,
either $Z$ is on a negative excursion or
else $Z(\Delta_+(t))\geq 0$.  In the latter case,
we use (\ref{2.38x}) and (\ref{2.58}) to write
\begin{align*}
0
&\leq 
Z\big(\Delta_+(t)\big)\\
&=
Z_+\circ P^Z_+\circ\Delta_+(t)
-Z_-\circ P^Z_-\circ\Delta_+(t)\\
&=
-L^Z_+\circ P^Z_+(t)-(Z_-+L^Z_-)\circ P^Z_-\circ
\Delta_+(t)+L^Z_-\circ P^Z_-\circ \Delta_+(t)\\
&\leq
-L^Z_+\circ P^Z_+(t)+L^Z_-\circ P^Z_-\circ\Delta_+(t)\\
&\leq
-L^Z_+\circ P^Z_+(t)+L^Z_-\circ P^Z_-(t).
\end{align*}
In the event that $Z$ is on a negative excursion
at time $\Delta_+(t)$ that began at time
$\ell(t)<\Delta_+(t)$, we have
$P^Z_+(\ell(t))=P^Z_+(\Delta_+(t))$, and
we use (\ref{2.38x}) and (\ref{2.58}) to write
\begin{align*}
0
&=
Z(\ell(t))\\
&=
Z_+\circ P^Z_+\big(\ell(t)\big)
-Z_-\circ P^Z_-\big(\ell(t)\big)\\
&=
Z_+\circ P^Z_+\circ\Delta_+(t)
-Z_-\circ P^Z_-\big(\ell(t)\big)\\
&=
-L^Z_+\circ P^Z_+(t)
-(Z_-+L^Z_-)\circ P^Z_-\big(\ell(t)\big)
+L^Z_-\circ P^Z_-\big(\ell(t)\big)\\
&\leq
-L^Z_+\circ P^Z_+(t)+L^Z_-\circ P^Z_-\big(\ell(t)\big)\\
&\leq
-L^Z_+\circ P^Z_+(t)+L^Z_-\circ P^Z_-(t).
\end{align*}
We conclude that
$L^Z_+\circ P^Z_+\leq L^Z_-\circ P^Z_-$.
Reversing the roles of $Z_+$ and $Z_-$
in this argument, we obtain the opposite
inequality.  Equation (\ref{2.54}) is established.

We next prove (\ref{2.55}).  Define
$Q_{\pm}=\Phi_{\pm}(Z_+,Z_-)$, so that
by (\ref{2.24}), (\ref{2.25}), and (\ref{2.49}),
$$
L^Z_+\big(Q_+(t)\big)=L^Z_-\big(t-Q_+(t)\big),\quad
L^Z_-\big(Q_-(t)\big)=L^Z_+\big(t-Q_-(t)\big),
\quad t\geq 0.
$$
Because of (\ref{TS.E.3}), (\ref{2.54}) implies
$$
L^Z_+\big(P^Z_+(t)\big)=L^Z_-\big(t-P^Z_+(t)\big),\quad
L^Z_-\big(P^Z_-(t)\big)=L_+\big(t-P^Z_-(t)\big),\quad
t\geq 0.
$$
The definition of $\Phi_+$ implies
that $P^Z_{+}(t)\leq Q_{+}(t)$.
Because $L^Z_+(u)$ is increasing
in $u$ and $L^Z_-(t-u)$
is decreasing in $u$,
we have for $u\in[P^Z_+(t),Q_+(t)]$ that
\begin{align*}
L^Z_+\big(P_+(t)\big)
&\leq 
L^Z_+(u)
\leq 
L^Z_+\big(Q_+(t)\big)
=
L^Z_-\big(t-Q_+(t)\big)\\
&\leq 
L^Z_-(t-u)
\leq 
L^Z_-\big(t-P^Z_+(t)\big)
=
L^Z_+\big(P^Z_+(t)\big),
\end{align*}
and hence
$L_+^Z(u)$ and $L^Z_-(t-u)$ are equal
to $L^Z_+(P_+(t))$
for $u\in[P^Z_+(t),Q_+(t)]$.
We recall from (\ref{2.22b})
that $L^Z_{\pm}$ is the maximum-to-date
of the 
Brownian motion $-Z_{\pm}$.
If $P^Z_+(t)<Q_+(t)$,
then both $L^Z_+$ and $L^Z_-$ have 
``flat spots'' at level $L^Z_+(P_+(t))$.
These ``flat spots'' correspond to jumps
in 
$(L^Z_+)^{-1}(\theta)
:=\min\big\{u\geq 0:L^Z_+(u)>\theta\big\}$ 
and 
$(L^Z_-)^{-1}(\theta)
:=\min\big\{u\geq 0:L^Z_-(u)>\theta\big\}$
at $\theta:=L^Z_+(P_+(t))$, i.e.,
$(L^Z_{\pm})^{-1}(\theta-)<(L^Z_{\pm})^{-1}(\theta)$.
But  the distribution of jumps
of $(L^Z_+)^{-1}$ and $(L^Z_-)^{-1}$
are non-atomic, these processes are independent
because $Z_+$ and $Z_-$ are independent,
and consequently the probability that they
have simultaneous jumps is zero.  We conclude
that $P^Z_+(t)=Q_+(t)=\Phi_+(Z_+,Z_-)$ almost surely.
This conclusion holds for each fixed $t$, and hence
for countably many $t$, but
both $P^Z_+$ and $Q_+$ are continuous, so
their entire paths must agree almost surely.
The proof that $P^Z_-=Q_-=\Phi_-(Z_+,Z_-)$
is analogous.
Equation (\ref{2.56}) is now just a restatement
of (\ref{2.38x}); see (\ref{2.24b}).
\end{proof}

We conclude this section with a representation
of two-speed Brownian motion.

\begin{theorem}\label{TS.T.8}
Let $Z_+$ and $Z_-$ be independent
Brownian motions with variances
$\sigma_+^2$ and $\sigma_-^2$.  Then
$Z:=\Psi(Z_+,Z_-)$ is a two-speed Brownian
motion.  In particular, $Z$ satisfies
(\ref{2.38x}), where $P^Z_{\pm}$ are
defined by (\ref{2.46}).
\end{theorem}

\begin{proof}
Let $\W^{\sigma_{\pm}^2}$ denote Wiener
measure on $C[0,\infty)$ for Brownian
motion with variance $\sigma_{\pm}^2$,
and let $\W^{\sigma_+^2}\otimes\W^{\sigma_-^2}$
denote the product measure induced
by $(Z_+,Z_-)$.  The set $\sD$ of (\ref{2.22})
has measure one under
$\W^{\sigma_+^2}\otimes\W^{\sigma_-^2}$,
and hence $\Psi(Z_+,Z_-)$ is defined almost surely.
To see that $Z:=\Psi(Z_+,Z_-)$ is a two-speed Brownian
motion, it suffices to show it induces
the same measure on $C[0,\infty)$ as a
two-speed Brownian motion.  This follows
from Lemmas \ref{TS.L.6} and \ref{TS.T.7},
which show that a two-speed
Brownian motion can be written as $\Psi$
of independent Brownian motions with
variances $\sigma_+^2$ and $\sigma_-^2$,
and hence the distribution of a two-speed
Brownian motion is
$(\W^{\sigma_+^2}\otimes\W^{\sigma_-^2})\circ \Psi^{-1}$.
Thus, $Z$ is a two-speed Brownian motion.

The last part of the proof of Lemma
\ref{TS.T.7} uses the fact that $Z$
is a two-speed Brownian motion to justify
(\ref{TS.E.3}), but does not use
(\ref{2.38x}).  That argument can
be used here to show that
$\Phi_{\pm}(Z_+,Z_-)=P^Z_{\pm}$.
Equation (\ref{2.38x}) now follows from 
the definition (\ref{2.24b}) of $\Psi$.
\end{proof}

\subsection{Excursion representation}\label{Excursions}
To conclude the study of two-speed
Brownian motion, we recall the Brownian
excursion theory
of L\'evy \cite{Levy} and It\^o
\cite{Ito}
as presented in Ikeda \& Watanabe
\cite[pp.~113--129]{IkedaWatanabe}.
Let $\sE_+$ (respectively $\sE_-$) be the set of
continuous functions from $[0,\infty)$ to
$[0,\infty)$ (respectively $(-\infty,0])$
with the property that for every 
$e\in\sE:=\sE_+\cup\sE_-$, we have
$e_0=0$,
$$
\lambda(e):=\min\{t>0:e(t)=0\}
$$
is finite, and $e(t)=0$ for all $t\geq \lambda(t)$.
We call $e\in\sE$ an {\em excursion} and call
$\lambda(e)$ the {\em length of the excursion}.
For $\ell>0$, let $\sE^{\ell}_{\pm}$ denote the set
of positive, respectively negative, excursions
of length $\ell$.
We define $\sB(\sE_{\pm})$ and $\sB(\sE)$
to be the $\sigma$-algebras
on $\sE_{\pm}$ and $\sE$ 
generated by the finite-dimensional
cylinder sets and $\sB(\sE_{\pm}^{\ell})$
to be the trace $\sigma$-algebra on
$\sE_{\pm}^{\ell}$.  We construct finite
measures on $\sE_{\pm}^{\ell}$ as follows.
(We need the following detailed
formulas in order to observe consequences
of scaling in Section \ref{Excursions2}.)
Define
\begin{align}
K_{\pm}(t,x)
&=
\sqrt{\frac{2}{\pi t^3}}
|x|\exp\left(-\frac{x^2}{2t}\right)\!,\,\,
t>0,x\geq 0\mbox{ or }x\leq 0,\label{K}\\
p_0(t,x,y)
&=
\frac{1}{\sqrt{2\pi t}}
\left[\exp\left(-\frac{(x-y)^2}{2t}\right)
-\exp\left(-\frac{(x+y)^2}{2t}\right)\right]\!,
\,\,t>0,x,y\in\R,\label{p}\\
h^{\ell}_{\pm}(s,a;t,b)
&=
\left\{\begin{array}{ll}
\displaystyle\frac{K_{\pm}(\ell-t,b)}{K_{\pm}(\ell-s,a)}
p_0(t-s,a,b),\,\,0<s<t<\ell,a,b>0
\mbox{ or }a,b<0,\\
\sqrt{\displaystyle\frac{\pi}{2}\ell^3}\,
K_{\pm}(t,b)K_{\pm}(\ell-t,b),\,\,
s=0,0<t<\ell,a=0,b>0\mbox{ or }b<0,
\end{array}\right.\label{h}
\end{align}
where the ``second cases'' such as 
``or $x\leq 0$'' correspond to the
$-$ subscript.
Let $\ell>0$ and $0<t_1<t_2<\cdots<t_n<\ell$ be given.
On $(\sE_{\pm}^{\ell},\sB(\sE_{\pm}^{\ell}))$ we define
the probability 
measures $\P_{\pm}^{\ell}$
(see Appendix \ref{AppendixB}) by specifying that
\begin{align}
\lefteqn{\P_{\pm}^{\ell}\big\{e(t_1)\in dx_1,
e(t_2)\in dx_2,\dots,
e(t_n)\in dx_n\big\}}\nonumber\\
&=
h^{\ell}_{\pm}(0,0;t_1,x_1)
h^{\ell}_{\pm}(t_1,x_1;t_2,x_2)\cdots
h^{\ell}_{\pm}(t_{n-1},x_{n-1};
t_n,x_n)dx_1dx_2\cdots dx_n.\label{Pl}
\end{align}
Finally, we define $\sigma$-finite measures
$n^B_{\pm}$ on $(\sE_{\pm},\sB(\sE_{\pm}))$ by
\be\label{nBpm}
n^B_{\pm}(C)=\int_0^{\infty}
\P_{\pm}^{\ell}\big(C\cap\sE^{\ell}_{\pm}\big)
\frac{d\ell}{\sqrt{2\pi \ell^3}},\quad
C\in\sB(\sE_{\pm}).
\ee
These are the characteristic measures for the
positive and negative excursions of Brownian motion.
Associated with each of these characteristic measures
there is a Poisson point process with counting
measure $N_{\pm}^B$, i.e., a Poisson random
measure on $(0,\infty)\times\sE_{\pm}$ with
the property that
$$
\E N_{\pm}^B\big((s,t]\times C\big)
=(t-s)n^B_{\pm}(C),\,\,0\leq s\leq t,\,
C\in\sB(\sE_{\pm}).
$$

We have constructed separate characteristic measures
and corresponding Poisson random measures
on the spaces of positive
and negative excursions, but
we can combine them to obtain a single characteristic
measure $n^B$ and a corresponding Poisson
random measure $N^B$ on the space of all excursions,
defined 
for $0\leq s\leq t$ and $C\in\sB(\sE)$ by
\begin{align*}
n^B(C)
&=n_+(C\cap\sE_+)+n_-(C\cap\sE_-),\\
N^B((s,t]\times C)
&=
N_+^B\big((s,t]\times (C\cap\sE_+)\big)
+N_-^B((s,t]\times (C\cap\sE_-)\big).
\end{align*}
At this stage, the time axis is in units of
Brownian local time, and we map this local
time into chronological time by defining
$$
A^B(s)=\int_{(0,s]}\int_{\sE}\lambda(e)N^B(du\,de),
\,\,s\geq 0.
$$
This is the sum of the lengths of all excursions
of the Brownian motion constructed below that begin
by the time the local time at zero of the Brownian
motion reaches $s$.  The process $A^B$
is right-continuous and pure jump,
and it is strictly increasing because 
$N^B((s,t]\times\sE)$ is strictly positive
when $0\leq s<t$.  
We invert this process to move from chronological
time to local time, defining
$$
L^B(t)=\inf\{s\geq 0:A^B(s)>t\},\quad t\geq 0.
$$

The paths of the local time $L^B$ are continuous,
nondecreasing, and have flat spots corresponding
to the jumps of $A^B$, which correspond to
excursions.  While we have
$L^B\circ A^B=\id$, because of the flat spots
in $L^B$, $A^B\circ L^B\neq \id$.
Given a time $t$, we have
either that $A^B(L^B(t)-)<A^B( L^B(t))$,
in which case $N^B(\{L^B(t)\}\times \sE)=1$,
i.e, there is an excursion at local time $L^B(t)$
that begins at chronological time
$A^B(L^B(t)-)$,
or else $A^B(L^B(t)-)=A^B(L^B(t))$, in which
case $N^B(\{L^B(t)\}\times \sE)=0$,
i.e, there is no excursion beginning at local time $L^B(t)$
and hence no excursion beginning at chronological time
$A^B(L^B(t)-)=A^B(L^B(t))=t$.
We define
\be\label{ExcBM}
B(t)=\left\{\begin{array}{ll}
e\big(t-A^B(L^B(t)-)\big)&\mbox{if }
A^B\big(L^B(t)-\big)<A^B\big(L^B(t)\big),\\
0&\mbox{if }A^B\big(L^B(t)-\big)=A^B\big(L^B(t)\big),
\end{array}\right.
\ee
where in the first line of this definition
$e$ is the excursion that begins at local
time $L^B(t)$.
Then $B$ is a standard Brownian motion.
Equation (\ref{ExcBM}) can be written more succinctly
as
\begin{align}
B(t)
&=
\int_{(0,L^B(t)]}\int_{\sE}e\big(t-A^B(s-)\big) 
N^B(ds\,de)\nonumber\\
&=
\int_{(0,L^B(t)]}\int_{\sE_+}e\big(t-A^B(s-)\big) 
N_+^B(ds\,de)\nonumber\\
&\qquad
+\int_{(0,L^B(t)]}\int_{\sE_-}e\big(t-A^B(s-)\big) 
N_-^B(ds\,de),\,\,t\geq 0,\label{ExcBM'}
\end{align}
because excursions are zero after they end
and so at most one excursion appears
in the integrals.
The occupation times by $B$ of
the positive and negative half-lines
satisfy
\be\label{2.55a}
A^B_{\pm}(s):=
\int_{(0,s]}\int_{\sE_{\pm}}
\lambda(e)N_{\pm}^B(du\,de)
=P^B_{\pm}\big(A^B(s)\big),\quad s\geq 0.
\ee

We return to the two-speed Brownian
motion of Definition \ref{TS.D.1}.
Recall the strictly increasing,
Lipschitz-continuous time change
$\Theta=\frac{1}{\sigma_+^2}\P_+^B
+\frac{1}{\sigma_-^2}P^B_-$
of (\ref{TS.E.3a}) with inverse $T$ of (\ref{TS.E.4}).
We have the following excursion
representation.

\begin{lemma}\label{TS.P.2}
Let $Z$ be a two-speed Brownian motion
as in Definition \ref{TS.D.1}.  For
$\theta\geq 0$,
\begin{align}
Z(\theta)
&=
\int_{(0,L^Z(\theta)]}
\int_{\sE_+}\!\!e\big(\sigma_+^2(\theta-A^Z(s-)\big)
N_+^B(ds\,de)\nonumber\\
&\qquad
+\int_{(0,L^Z(\theta)]}
\int_{\sE_-}\!\!e\big(\sigma_-^2(\theta-A^Z(s-)\big)
N_-^B(ds\,de),\label{2.17a}
\end{align}
where $L^Z:=L^B\circ T$ and
\be\label{2.18a}
A^Z(s):=
\inf\big\{\theta\geq 0: L^Z(\theta)>s\big\}.
\ee
\end{lemma}

\begin{proof}
Proposition \ref{TS} implies $Z=B\circ T$,
and substitution into (\ref{ExcBM'}) yields
\begin{align}
Z(\theta)
&=
B\big(T(\theta)\big)\nonumber\\
&=
\int_{(0,L^B\circ T(\theta)]}\int_{\sE_+}
e\big(T(\theta)-A^B(s-)\big)
N_+^B(ds\,de)\nonumber\\
&\qquad
+\int_{(0,L^B\circ T(\theta)]}\int_{\sE_-}
e\big(T(\theta)-A^B(s-)\big)
N_-^B(ds\,de).\label{2.20a}
\end{align}
To rewrite this, we first prove that
\be\label{2.21a}
A^B=T\circ A^Z,
\ee
or equivalently,
$\Theta\circ A^B=A^Z$.
Observe that
$L^Z\circ\Theta\circ A^B
=L^B\circ T\circ \Theta\circ A^B
=L^B\circ A^B=\id$,
which implies from (\ref{2.18a}) that
$\Theta\circ A^B \leq A^Z$.  However,
for $s'>s$,
$L^Z\circ\Theta\circ A^B(s')
=s'>s$,
which implies 
$\Theta\circ A^B(s')\geq A^Z(s)$
for every $s'>s$,
and because $A^B$ is right-continuous
and $\Theta$ is continuous,
$\Theta\circ A^B(s)\geq A^Z(s)$.
Equation (\ref{2.21a}) is established.

In (\ref{2.20a}), $T(\theta)$
and $A^B(s-)=T(A^Z(s-))$
are on the same excursion, either positive
or negative, and $T$ grows at constant
rate $\sigma_{\pm}^2$ on positive (negative)
excursions.  Therefore,
$T(\theta)-A^B(s-)
=\sigma_{\pm}^2\big(\theta-A^Z(s-)\big)$,
$\sigma_+^2$ or $\sigma_-^2$ appearing
depending whether the excursion is positive
or negative.
Substitution into (\ref{2.20a}) results
in (\ref{2.17a}).
\end{proof}

Let us define functions
$\psi_{\pm}:\sE_{\pm}\rightarrow\sE_{\pm}$ by
$(\psi_{\pm}e)(\theta)=e(\sigma_{\pm}^2\theta)$,
$\theta\geq 0$.
We define new Poisson random measures
$N_{\pm}^Z$ on $(0,\infty)\times\sE_{\pm}$ by
\be\label{N1.15}
N_{\pm}^Z\big((s,t]\times C\big)
=N_{\pm}^B\big((s,t]\times \psi_{\pm}^{-1}(C)\big),\,\,
0\leq s\leq t,\,\,C\in\sB(\sE_{\pm}),
\ee
and rewrite (\ref{2.17a}) as
$$
Z(\theta)
=\int_{(0,L^Z(\theta)]}\int_{\sE_+}
e\big(\theta-A^Z(s-)\big)N_+^Z(ds\,de)
+\int_{(0,L^Z(\theta)]}\int_{\sE_-}
e\big(\theta-A^Z(s-)\big)N_-^Z(ds\,de).
$$
We combine these two Poisson random measures
to obtain $N^Z$ defined by
\be\label{N1.6}
N^Z\big((s,t]\times C\big)
=N_+^Z\big((s,t]\times(C\cap\sE_+)\big)
+N_-^Z\big((s,t]\times(C\cap\sE_-)\big),\,\,
0\leq s\leq t,\,\,C\in\sB(\sE).
\ee
Because $\lambda(\psi_{\pm}(e))=
\lambda(e)/\sigma_{\pm}^2$, we have
$$
\int_{(0,\theta]}\int_{\sE_{\pm}}
\lambda(e)N_{\pm}^Z(ds\,de)=
\frac{1}{\sigma_{\pm}^2}
\int_{(0,\theta]}\int_{\sE_{\pm}}
\lambda(e)N_{\pm}^B(ds\,de)
=\frac{1}{\sigma_{\pm}^2}
A_{\pm}^B(\theta)
=\frac{1}{\sigma_{\pm}^2}(P_{\pm}^B\circ A^B)(\theta),
$$
where we have used (\ref{2.55a}) in the last step.
Therefore,
$$
\int_{(0,\theta]}\int_{\sE}
\lambda(e)N^Z(ds\,de)
=\left(\frac{1}{\sigma_+^2}P^B_+
+\frac{1}{\sigma_-^2}P^B_-\right)
\circ A^B(\theta)=\Theta\circ A^B(\theta)=A^Z(\theta).
$$
We have proved the following proposition, 
which represents
two-speed Brownian motion completely
analogously to the representation
(\ref{ExcBM'}) for standard Brownian motion.

\begin{proposition}\label{TS.C.10}
The two-speed Brownian motion $Z$ has
the excursion representation
$$
Z(\theta)=\int_{(0,L^Z(\theta)]}\int_{\sE}
e\big(\theta-A^Z(s-)\big)N^Z(ds\,de),
$$
where 
$A^Z(s)=\int_{(0,s]}\lambda(e)N^Z(du\,de)$
and $L^Z(\theta)=\inf\{s\geq 0:A^Z(s)>\theta\}$.
\end{proposition}

\subsection{Skew Brownian motion}\label{Excursions2}
To relate the two-variance Brownian
motion $Z$ of Definition \ref{TS.D.1}
and Proposition \ref{TS.C.10} to the flipped excursion
definition of skew Brownian motion
given by It\^o \& McKean \cite{ItoMcKean}
and discussed in Remark \ref{TS.R.1},
we proceed through an analysis of
characteristic measures.
The characteristic measures for the 
positive and negative excursions
of two-speed Brownian motion are given,
with $0\leq s<t$, by the equation
$$
n^Z_{\pm}(C)=
\frac{1}{t-s}\E N^Z_{\pm}\big((s,t]\times C\big)=
\frac{1}{t-s}\E N^B_{\pm}\big((s,t]
\times \psi_{\pm}^{-1}(C)\big)=
n_{\pm}^B\big(\psi_{\pm}^{-1}(C)\big),
\quad C\in\sB(\sE_\pm).
$$
Let us consider a set in $\sB(\sE_{\pm})$ of the form
\be\label{N1.21}
C=\big\{f\in\sE_{\pm}:
\big(f(t_1),\dots,f(t_n)\big)\in
A_1\times\dots\times A_n,\lambda(f)>t_n\big\},
\ee
where $0<t_1<\dots<t_n$ and
$A_1,\dots,A_n$ are sets in $\sB[0,\infty)$,
respectively, $\sB(-\infty,0]$.
Then
$$
\psi_{\pm}^{-1}(C)
=\big\{e\in\sE_{\pm}:
\big(e(\sigma_{\pm}^2t_1),\dots,e(\sigma_{\pm}^2t_n)\big)
\in A_1\times\dots\times A_n,\lambda(e)>
\sigma_{\pm}^2t_n\big\}.
$$
From (\ref{K})--(\ref{h}), we have for
$\sigma>0$, $0<s<t<\ell$,
and $a,b>0$ or $a,b<0$, that
\be\label{N1.22}
h^{\ell}_{\pm}(0,0;\sigma^2t,b)
=\frac{1}{\sigma}h^{\ell/\sigma^2}_{\pm}
\left(0,0;t,\frac{b}{\sigma}\right),\quad
h^{\ell}_{\pm}(\sigma^2s,a;\sigma^2t,b)
=\frac{1}{\sigma}
h^{\ell/\sigma^2}_{\pm}\left(s,\frac{a}{\sigma};
t,\frac{b}{\sigma}\right).
\ee
With $C$ given by (\ref{N1.21}), we now use (\ref{Pl}),
(\ref{nBpm}), and (\ref{N1.22}) to compute
\begin{align}
n_{\pm}^Z(C)
&=
n_{\pm}^B(\psi_{\pm}^{-1}C)\nonumber\\
&=
\int_{\sigma_{\pm}^2t_n}^{\infty}
\P^{\ell}_{\pm}\big\{e\in\sE_{\pm}:
\big(e(\sigma_{\pm}^2t_1),\dots,
e(\sigma_{\pm}^2t_n))\in A_1\times\dots A_n\big\}
\frac{d\ell}{\sqrt{2\pi\ell^3}}\nonumber\\
&=
\int_{\sigma_{\pm}^2t_n}^{\infty}
\int_{A_n}\dots\int_{A_2}\int_{A_1}
h_{\pm}^{\ell}(0,0;\sigma_\pm^2t_1,x_1)
h_{\pm}^{\ell}(\sigma_{\pm}^2t_1,x_1;\sigma_{\pm}^2t_2,x_2)
\nonumber\\
&\qquad\qquad\cdots
h_{\pm}^{\ell}(\sigma_{\pm}^2t_{n-1},x_{n-1};
\sigma_{\pm}^2t_n,x_n)dx_1dx_2\dots dx_n
\frac{d\ell}{\sqrt{2\pi\ell^3}}\nonumber\\
&=
\int_{\sigma_{\pm}^2t_n}
\int_{A_n}\dots\int_{A_2}\int_{A_1}\frac{1}{\sigma_{\pm}^n}
h_{\pm}^{\ell/\sigma_\pm^2}
\left(0,0;t_1,\frac{x_1}{\sigma_{\pm}}\right)
h_{\pm}^{\ell/\sigma_\pm^2}
\left(t_1,\frac{x_1}{\sigma_{\pm}};t_2,\frac{x_2}{\sigma_\pm}
\right)\nonumber\\
&\qquad\qquad
\cdots h_{\pm}^{\ell/\sigma_\pm^2}
\left(t_{n-1},\frac{x_{n-1}}{\sigma_{\pm}};
t_n,\frac{x_n}{\sigma_\pm}\right)dx_1dx_2\dots dx_n
\frac{d\ell}{\sqrt{2\pi\ell^3}}\nonumber\\
&=
\frac{1}{\sigma_\pm}\int_{t_n}^{\infty}
\int_{A_n/\sigma_\pm}\dots\int_{A_2/\sigma_{\pm}}
\int_{A_1/\sigma_{\pm}}
h_{\pm}^{T}(0,0;t_1,y_1)
h_{\pm}^{T}(t_1,y_1;t_2,y_2)\nonumber\\
&\qquad\qquad
\cdots h_{\pm}^{T}(t_{n-1},y_{n-1};t_n,y_n)
dy_1dy_2\dots dy_n\frac{dT}{\sqrt{2\pi T^3}}\nonumber\\
&=
\frac{1}{\sigma_{\pm}}n_{\pm}^B\left(\frac{1}{\sigma_\pm}C\right).
\label{N1.27}
\end{align} 
Equation (\ref{N1.27}) holds for every $C$ of the form
(\ref{N1.21}), and since this collection of sets $C$
is closed under pairwise intersection and generates
$\sB(\sE_\pm)$, equation (\ref{N1.27}) holds for every
$C\in\sB(\sE_\pm)$.  We extend the
definition (\ref{f}), defining
$\varphi:\sE_{\pm}\rightarrow\sE_{\pm}$ by
\be\label{Extendedvarphi}
\varphi(e):=\left\{\begin{array}{ll}
\sigma_+e&\mbox{if }e\in\sE_+,\\
\sigma_-e&\mbox{if }e\in\sE_-.
\end{array}\right.
\ee
We may now rewrite (\ref{N1.27})
as
\be\label{nZ}
n_{\pm}^Z(C)=\frac{1}{\sigma_{\pm}}
n_{\pm}^B\big(\varphi^{-1}(C)\big),\quad 
C\in\sB(\sE_{\pm}),
\ee
where $\varphi$ is defined by (\ref{Extendedvarphi}).

We next construct the characteristic measure
of skew Brownian motion as defined by \cite{ItoMcKean}.
The characteristic measure
of reflected Brownian motion $|B|$ is given by
$$
n^{|B|}(C)=n_+^B(C)+n_-^B(-C)
=2n_+^B(C)=2n_-^B(-C),\quad C\in\sB(\sE_+).
$$
Consider the skew Brownian motion
$X$ that takes the excursions
of $|B|$ and assigns them positive or negative
signs independently with probabilities
$\sigma_-/(\sigma_++\sigma_-)$ and 
$\sigma_+/(\sigma_++\sigma_-)$,
respectively.  A characteristic measure of $X$ is thus
$$
n^X(C)=
\left\{\begin{array}{ll}
\displaystyle
\frac{\sigma_-}{\sigma_++\sigma_-}n^{|B|}(C)
=\frac{2\sigma_-}{\sigma_++\sigma_-}n_+^B(C)&\mbox{if }
C\in\sB(\sE_+),\\
\displaystyle
\frac{\sigma_+}{\sigma_++\sigma_-}n^{|B|}(-C)
=\frac{2\sigma_+}{\sigma_++\sigma_-}
n_-(C)&\mbox{if }
C\in\sB(\sE_-).
\end{array}\right.
$$
Let $N^X$ be the associated Poisson random
measure, i.e., $N^X$ is a Poisson random
measure on $(0,\infty)\times\sE$ such that
$\E N^X((s,t]\times C)=(t-s)n^X(C)$
for $0\leq s<t$ and $C\in\sB(\sE)$. Then
$$
X(t)=\int_{(0,L^X(t)]}\int_{\sE}
e\big(t-A^X(u-)\big)N^X(du\,de),
$$
where
$A^X(s)=\int_{(0,s]}\int_{\sE}\lambda(e)
N^X(dv\,de)$
and $L^X(t)=\inf\{s\geq 0:A^X(s)>t\}$.

\begin{theorem}\label{T3.11}
Two-speed Brownian motion $Z$ of
Definition \ref{TS.D.1} has the same
law as $\varphi(X)$.
\end{theorem}
\begin{proof}
Define
$\gamma=\frac{1}{2\sigma_+}+\frac{1}{2\sigma_-}
=\frac{\sigma_-+\sigma_+}{2\sigma_+\sigma_-}$,
so that
$\gamma\frac{2\sigma_{\mp}}{\sigma_++\sigma_-}
=\frac{1}{\sigma_\pm}$.
Define a Poisson random measure $N^Y$ by
$N^Y((s,t]\times C):=N^X((\gamma s,\gamma t]\times C)$
for $0\leq s <t$ and $C\in\sE$.
The associated characteristic measure is
\be\label{cmX}
n^Y(C)=\frac{1}{t-s}\E N^Y\big((s,t]\times C\big)
=\gamma n^X(C)=\left\{\begin{array}{ll}
\displaystyle
\frac{1}{\sigma_+}n_+^B(C)&\mbox{if }C\in\sB(\sE_+),\\
\displaystyle
\frac{1}{\sigma_-}n_-^B(C)&\mbox{if }C\in\sB(\sE_-).
\end{array}\right.
\ee
We define
$$
A^Y(s):=\int_{(0,s]}\int_{\sE}\lambda(e)
N^Y(du\,de)
=\int_{(0,\gamma s]}\int_{\sE}\lambda(e)
N^X(dv\,de)=A^X(\gamma s),
$$
so
$$
L^Y(t):=\inf\big\{s\geq 0: A^Y(s)>t\big\}
=\frac{1}{\gamma}\inf\big\{u\geq 0:A^X(u)>t\big\}
=\frac{1}{\gamma}L^X(t).
$$
Then
\begin{align*}
Y(t)
&:=
\int_{(0,L^Y(t)]}\int_{\sE}
e\big(t-A^Y(s-)\big)N^Y(ds\,de)\\
&=
\int_{(0,\gamma L^Y(t)]}\int_{\sE}
e\left(t-A^Y\left(\frac{u}{\gamma}-\right)\right)N^X(du\,de)\\
&=
\int_{(0,L^X(t)]}\int_{\sE}
e\big(t-A^X(u-)\big)N^X(du\,de)\\
&=
X(t),\quad t\geq 0.
\end{align*}
In other words, another characteristic
measure for the excursions of $X$ is
given by (\ref{cmX})\footnote{Characteristic
measure of excursions are determined only
up to a multiplicative constant}.

Finally, observe 
from (\ref{nZ}) and (\ref{cmX})
that a characteristic measure of $Z$
is $\varphi^{-1}(n^Y)$.
In other words, $Z$ is obtained from $X$
by scaling (in space, not time)
the positive excursions of $X$
by $\sigma_+$ and the negative excursions
by $\sigma_-$. In particular, $Z$ is equal in law to
$\varphi(X)$.
\end{proof}

\section{Brownian excursion and absorbed
Brownian motion}\label{appendix}
\numberwithin{equation}{section}
\setcounter{equation}{0}
\setcounter{theorem}{0}

In this appendix we establish the relationship
between absorbed Brownian motion
and Brownian excursions used in Section
\ref{ToRenewal}.  Let $W$ be a standard Brownian
motion with initial condition $W(0)=x_0>0$.
Define $\tau_0:=\inf\{t\geq 0: W(t)=0\}$
to be the first passage time of this Brownian
motion to zero, and define
$W_0(t)=W(t\wedge\tau_0)$, $t\geq 0$, to be the Brownian
motion $W$ absorbed at $\tau_0$.  According
to the reflection principle,
for $t_1>0$, $x_1>0$,
\be\label{A.1}
\P\big\{\tau_0\in d\ell\big|\tau_0>t_1,W_0(t_1)=x_1\big\}
=\frac12K_+(\ell-t_1,x_1)d\ell,\quad
\ell>t_1,
\ee
where we have used the notation (\ref{K}).
(The condition $\tau_0>t_1$ on the left-hand
side of (\ref{A.1}) is redundant, but is included
here to maintain parallelism in formulas
such as (\ref{A.7}) below). 

Now let $t_1>0$ be given,
and let $\{E_{k,+}\}_{k=1}^{\infty}$ be an
enumeration of the positive excursions
away from zero of a standard Brownian whose
lengths $\lambda(E_{k,+})$ exceed
$t_1$.
This is a sequence of independent,
identically distributed excursion,
and because the characteristic measure
for positive excursions of Brownian motion
is given by (\ref{nBpm}), the law of these
excursions is
\be\label{A.2}
\mu:=\left.\int_{t_1}^{\infty}
\P^{\ell}_+\frac{d\ell}{\sqrt{2\pi\ell^3}}\right/
\int_{t_1}^{\infty}\frac{d\ell}{\sqrt{2\pi\ell^3}}
=\frac12\int_{t_1}^{\infty}\P_+^{\ell}
\sqrt{\frac{t_1}{\ell^3}}d\ell,
\ee
which is a probability measure on
$(\sE_+,\sB(\sE_+))$.
In particular,
\be\label{A.3}
\mu\big\{\lambda(E_{k,+})\in d\ell\big|
\lambda(E_{k,+})>t_1\big\}
=\frac12\sqrt{\frac{t_1}{\ell^3}}d\ell,\quad \ell>t_1.
\ee
(The condition on the left-hand side
of (\ref{A.3}) is redundant, but is included
here and elsewhere
to emphasize that we are considering
only excursions whose length exceeds $t_1$.)

\begin{proposition}\label{PA.1}
Let $x_1$ be positive.
Conditioned on $\lambda(E_{k,+})>t_1$
and $E_{k,+}(t_1)=x_1$, the law of
$(E_{k,+})_{t\geq t_1}$ agrees with the law of
$(W_0)_{t\geq t_1}$ conditioned on $\tau_0>t_1$
and $W_0(t_1)=x_1$.
\end{proposition}

\begin{proof}
Define
$$
f(t,x):=\sqrt{\frac{2}{\pi t}}
\int_x^{\infty}\exp\left(-\frac{x^2}{2t}\right)dx
=\sqrt{\frac{2}{\pi}}
\int_{x/\sqrt{t}}^{\infty}\exp\left(-\frac{y^2}{2}\right)dy,
\quad t>0, x>0.
$$
Then
$f(0,x):=\lim_{t\downarrow 0}f(t,x)=0$,
$f(\infty,x):=\lim_{t\rightarrow \infty}f(t,x)=1$,
and $\frac{\partial }{\partial t}f(t,x)
=\frac12 K_+(t,x)$,
where $K_+(t,x)$ is given by (\ref{K}).
According to (\ref{A.2}) and (\ref{Pl}),
\begin{align}
\lefteqn{\mu\big\{e\in\sE_+:\lambda(e)\in d\ell,
e(t_1)\in dx_1\big|\lambda(e)>t_1\}}
\hspace{1in}\nonumber\\
&=
\sqrt{\frac{\pi}{2}\ell^3}K_+(t_1,x_1)
K_+(\ell-t_1,x_1)\frac12\sqrt{\frac{t_1}{\ell^3}}d\ell
dx_1\nonumber\\
&=
\sqrt{\frac{\pi t_1}{2}}
K_+(t_1,x_1)\frac12K_+(\ell-t_1,x_1)d\ell dx_1,
\quad \ell>t_1,x_1>0.
\label{A.5}
\end{align}
Hence,
\begin{align}
\mu\big\{e\in\sE_+:e(t_1)\in dx_1\big|\lambda(e)>t_1\big\}
&=
\sqrt{\frac{\pi t_1}{2}}K_+(t_1,x_1)
f(\ell-t_1,x_1)\Big|_{\ell=t_1}^{\ell=\infty}\nonumber\\
&=
\sqrt{\frac{\pi t_1}{2}}K_+(t_1,x_1)dx_1,
\quad x_1>0.\label{A.6}
\end{align}
Dividing (\ref{A.5}) by (\ref{A.6}), we obtain
\begin{align}
\mu\big\{e\in\sE_+:\lambda(e)\in d\ell
\big|\lambda(e)>t_1,e(t_1)=x_1\big\}
&=
\frac12K_+(\ell-t_1,x_1)d\ell\nonumber\\
&=
\P\big\{\tau_0\in d\ell\big|\tau_0>t_1,W_0(t_1)=x_1\big\}.
\label{A.7}
\end{align}

We next show that for $\ell>t_1$
and $t_2,\dots,t_n$ satisfying
$t_1<t_2<\dots<t_n<\ell$, we have
\begin{align}
\lefteqn{\mu\big\{e\in\sE_+:
e(t_2)\in dx_2,\dots,e(t_n)\in dx_n\big|
e(t_1)=x_1,\lambda(e)=\ell\big\}}\hspace{1in}
\nonumber\\
&=
\P\big\{W_0(t_2)\in dx_2,\dots, W_0(t_n)\in dx_n
\big|W_0(t_1)=x_1,\tau_0=\ell\big\}.\label{A.8}
\end{align}
We can then multiply the respective
left- and right-hand sides of (\ref{A.7})
by (\ref{A.8}) to complete the proof.
We begin with the left-hand side of (\ref{A.8}).
It is clear from (\ref{A.2}) and (\ref{Pl}) that
for all $\ell>t_n$,
\begin{align}
\lefteqn{
\mu\big\{e\in\sE_+:e(t_2)\in dx_2,\dots
e(t_n)\in dx_n\big|e(t_1)=x_1,\lambda(e)=\ell\big\}}
\nonumber\\
&=
\P^{\ell}\big\{e(t_2)\in dx_2,\dots,e(t_n)\in dx_n
\big|e(t_1)=x_1\big\}\nonumber\\
&=
h_+^{\ell}(t_1,x_1;t_2,x_2)\cdots
h_+^{\ell}(t_{n-1},x_{n-1};t_n,x_n)dx_n\cdots dx_1
\nonumber\\
&=
\frac{K_+(\ell-t_n,x_n)}{K_+(\ell-t_1,x_1)}
p_0(t_2-t_1,x_1,x_2)\cdots p_0(t_n-t_{n-1},
x_{n-1},x_n)dx_n\cdots dx_1.\label{A.9}
\end{align}

We turn to the right-hand side of
(\ref{A.8}).
According to the reflection principle,
the transition density for $W_0$ is
$p_0$ given by (\ref{p}) for $x>0$
and $y>0$.  This is a defective probability
density function, with
\be\label{A.10}
\P\big\{W_0(t+s)=0\big|W_0(s)=x\big\}=
1-\int_0^{\infty}p_0(t,x,y)dy=
f(t,x).
\ee
Although defective, this transition density
has the semigroup property
\be\label{A.11}
\int_0^{\infty}p_0(s,x,y)p_0(t,y,z)dy
=p_0(s+t,x,z),\quad s>0,t>0,x>0,z>0,
\ee
as can be verified by direct computation
(see Appendix B of \cite{Yu}).
For $t_n<t_{n+1}<t_{n+2}$, we have
\begin{align}
\lefteqn{\P\big\{W_0(t_2)\in dx_2,\dots,
W_0(t_n)\in dx_n, t_{n+1}<\tau_0\leq t_{n+2}\big|
W_0(t_1)=x_1\big\}}\hspace{0.5in}\nonumber\\
&=
\P\big\{W_0(t_2)\in dx_2,\dots,W_0(t_n)\in dx_n,
W_0(t_{n+1})>0,W_0(t_{n+2})=0\big|W_0(t_1)=x_1\big\}
\nonumber\\
&=
p_0(t_2-t_1,x_1,x_2)\cdots p_0(t_n-t_{n-1},x_{n-1},x_n)
\int_0^{\infty}p_0(t_{n+1}-t_n,x_n,x_{n+1})
\nonumber\\
&\qquad\times
\left(1-\int_0^{\infty}p_0(t_{n+2}-t_{n+1},x_{n+1},
x_{n+2})dx_{n+2}\right)dx_{n+1}dx_n\dots dx_2.
\label{A.12}
\end{align}
We deal first with the integrals,
using the semigroup property (\ref{A.11})
and the second equation in (\ref{A.10}), to obtain
\begin{align}
\lefteqn{\int_0^{\infty}
p_0(t_{n+1}-t_n,x_n,x_{n+1})
\left(1-\int_0^{\infty}p_0(t_{n+2}-t_{n+1},x_{n+1},
x_{n+2})dx_{n+2}\right)dx_{n+1}}
\nonumber\\
&=
\int_0^{\infty}p_0(t_{n+1}-t_n,x_n,x_{n+1})dx_{n+1}
-\int_0^{\infty}p_0(t_{n+2}-t_n,x_n,x_{n+2})dx_{n+2}
\nonumber\\
&=
f(t_{n+2}-t_n,x_n)-f(t_{n+1}-t_n,x_n)\nonumber\\
&=\int_{t_{n+1}}^{t_{n+2}}\frac{\partial}{\partial \ell}
f(\ell-t_n,x_n)d\ell\nonumber\\
&=
\frac12
\int_{t_{n+1}}^{t_{n+2}}K_+(\ell-t_n,x_n)d\ell\nonumber\\
&=
\int_{t_{n+1}}^{t_{n+2}}\frac{K_+(\ell-t_n,x_n)}
{K_+(\ell-t_1,x_1)}
\P\big\{\tau_0\in d\ell\big|W_0(t_1)=x_1\big\}.
\label{A.13}
\end{align}
Putting (\ref{A.12}) and (\ref{A.13})
together, we have
\begin{align*}
\lefteqn{
\P\big\{W_0(t_2)\in dx_2,\dots,W_0(t_n)\in dx_n,
t_{n+1}<\tau_0\leq t_{n+1}\big|W_0(t_1)=x_1\big\}}
\nonumber\\
&=
p_0(t_2-t_1,x_1,x_2)\cdots p_0(t_n-t_{n-1},x_{n-1},x_n)
\nonumber\\
&\hspace{1in}
\times\int_{t_{n+1}}^{t_{n+2}}
\frac{K_+(\ell-t_n,x_n)}{K_+(\ell-t_1,x_1)}
\P\big\{\tau_0\in d\ell\big|W_0(t_1)=x_1\big\}
dx_n\dots dx_2,
\end{align*}
which implies that
\begin{align}
\lefteqn{
\P\big\{W_0(t_2)\in dx_2,\dots,W_0(t_n)\in dx_n
\big|W_0(t_1)=x_1,\tau_0=\ell\big\}}\nonumber\\
&=
p_0(t_2-t_1,x_1,x_2)\cdots p_0(t_n-t_{n-1},x_{n-1},x_n)
\frac{K_+(\ell-t_n,x_n)}{K_+(\ell-t_1,x_1)}
dx_n\dots dx_2.\label{A.14}
\end{align}
This equation holds for all $\ell\in(t_{n+1},t_{n+2})$.
But $t_{n+1}$ can be chosen arbitrarily close to
$t_n$ and $t_{n+2}$ can be chosen arbitrarily
large.  Therefore, (\ref{A.14}) holds for all
$\ell>t_n$.
The right-hand side of (\ref{A.14})
agrees with the right-hand side
of (\ref{A.9}), which establishes
(\ref{A.8}).
\end{proof}

\section{$\P_{\pm}^{\ell}$ is a probability measure}\label{AppendixB}
\setcounter{equation}{0}
\setcounter{theorem}{0}

For this calculation, we consider only
$\P_+^{\ell}$ and show that it is a probability
measure on $(\sE_+^{\ell},\sB(\E_+^{\ell}))$.
To simplify the notation, we suppress the subscript $+$.
We need to verify that
\begin{align}
\lefteqn{
\int_0^{\infty}\cdots\int_0^{\infty}\int_0^{\infty}
\P^{\ell}\big\{e(t_1)\in dx_1,e(t_2)\in dx_2,\dots,
e(t_n)\in dx_n\big\}dx_n\cdots dx_2dx_1}\nonumber\\
&=
\int_0^{\infty}\cdots\int_0^{\infty}\int_0^{\infty}
h^{\ell}(0,0;t_1,x_1)h^{\ell}(t_1,x_1;t_2,x_2)
\cdots h^{\ell}(t_{n-1},x_{n-1};t_n,x_n)
dx_n\cdots dx_2 dx_1\nonumber\\
&=
1.\label{B.1}
\end{align}

\begin{lemma}\label{LB.1}
For $0<s<t$ and $a>0$,
$\int_0^{\infty}h^{\ell}(s,a;t,b)db=1$.
\end{lemma}
\begin{proof}
According to (\ref{K})--(\ref{h}),
\begin{align*}
\int_0^{\infty}h^{\ell}(s,a;t,b)db
&=
\int_0^{\infty}\frac{K(\ell-t,b)}{K(\ell-s,a)}
p_0(t-s,a,b)db\\
&=
\int_0^{\infty}\frac{b}{a}\sqrt{\frac{(\ell-s)^3}{(\ell-t)^3}}
\exp\left(-\frac{b^2}{2(\ell-t)}
+\frac{a^2}{2(\ell-s)}\right)
p_0(t-s,a,b)db\\
&=
-I_++I_-,
\end{align*}
where
\begin{align*}
I_{\pm}
&=
\int_0^{\infty}\frac{b}{a}\sqrt{\frac{(\ell-s)^3}{(\ell-t)^3}}
\exp\left(-\frac{b^2}{2(\ell-t)}
+\frac{a^2}{2(\ell-s)}\right)
\frac{1}{\sqrt{2\pi(t-s)}}
\exp\left(-\frac{(a\pm b)^2}{2(t-s)}\right)db\\
&=
\frac{1}{\sqrt{2\pi(t-s)}}\int_0^{\infty}
\frac{b}{a}
\sqrt{\frac{(\ell-s)^3}{(\ell-t)^3}}
\exp\left(-\frac{1}{2(t-s)}
\left(\sqrt{\frac{\ell-s}{\ell-t}}b
\pm\sqrt{\frac{\ell-t}{\ell-s}}a\right)^2\right)db.
\end{align*}
We make the change of variable
$
y=\sqrt{\frac{\ell-s}{\ell-t}}b\pm
\sqrt{\frac{\ell-t}{\ell-s}}a,
$
and rewrite these integrals as
\begin{align*}
I_{\pm}
&=
\frac{1}{\sqrt{2\pi(t-s)}}
\int_{\pm\sqrt{\frac{\ell-t}{\ell-s}}a}^{\infty}
\left(\frac{y}{a}\mp\sqrt{\frac{\ell-t}{\ell-s}}\right)
\sqrt{\frac{\ell-s}{\ell-t}}
\exp\left(-\frac{y^2}{2(t-s)}\right)dy\\
&=
-\frac{1}{a}\sqrt{\frac{t-s}{2\pi}}\left.
\exp\left(-\frac{y^2}{2(t-s)}\right)
\right|_{y=\sqrt{\frac{\ell-t}{\ell-s}}a}^{\infty}
\mp\frac{1}{\sqrt{2\pi(t-s)}}
\int_{\pm\sqrt{\frac{\ell-t}{\ell-s}}a}^{\infty}
\exp\left(-\frac{y^2}{2(t-s)}\right)dy\\
&=
\frac{1}{a}\sqrt{\frac{t-s}{2\pi}}
\exp\left(-\frac{(\ell-t)a^2}{(t-s)(\ell-s)}\right)
\mp\frac{1}{\sqrt{2\pi(t-s)}}
\int_{\pm\sqrt{\frac{\ell-t}{\ell-s}}a}^{\infty}
\exp\left(-\frac{y^2}{2(t-s)}\right)dy.
\end{align*}
It follows that
\begin{align*}
-I_++I_-
&=
\frac{1}{\sqrt{2\pi(t-s)}}
\int_{\sqrt{\frac{\ell-t}{\ell-s}}a}^{\infty}
\exp\left(-\frac{y^2}{2(t-s)}\right)dy
+\frac{1}{\sqrt{2\pi(t-s)}}
\int_{-\sqrt{\frac{\ell-t}{\ell-s}}a}^{\infty}
\exp\left(-\frac{y^2}{2(t-s)}\right)dy\\
&=
\frac{1}{\sqrt{2\pi(t-s)}}
\int_{-\infty}^{\infty}
\exp\left(-\frac{y^2}{2(t-s)}\right)dy\\
&=
1.
\end{align*}

\vspace{-24pt}
\end{proof}

Repeated applications of Lemma \ref{LB.1}
establish that
\begin{align*}
\lefteqn{\int_0^{\infty}\cdots\int_0^{\infty}\int_0^{\infty}
h^{\ell}(0,0;t_1,x_1)h^{\ell}(t_1,x_1;t_2,x_2)
\cdots h^{\ell}(t_{n-1},x_{n-1};t_n,x_n)
dx_n\cdots dx_2 dx_1}\hspace{3.5in}\\
&=
\int_0^{\infty}h^{\ell}(0,0;t_1,x_1)dx_1.
\end{align*}
The next lemma completes the proof of (\ref{B.1}).

\begin{lemma}\label{LB.2}
For $t>0$,
$$
\int_0^th^{\ell}(0,0;t,b)db=1.
$$
\end{lemma}

\begin{proof}
From (\ref{K}) and (\ref{h}), we have
\begin{align*}
\int_0^{\infty}h^{\ell}(0,0;t,b)
&=
\sqrt{\frac{\pi}{2}\ell^3}\int_0^{\infty}
K(t,b)K(\ell-t,b)db\\
&=
\sqrt{\frac{\pi}{2}\ell^3}
\int_0^{\infty}
\sqrt{\frac{2}{\pi t^3}}\sqrt{\frac{2}{\pi(\ell-t)^3}}
b^2\exp\left(-\frac{b^2}{2t}-\frac{b^2}{2(\ell-t)}\right)db\\
&=
\frac{\ell}{t(\ell-t)}\cdot 2\sqrt{\frac{\ell}{2\pi t(\ell-t)}}
\int_0^{\infty}b^2\exp\left(-\frac{\ell b^2}{t(\ell-t)}\right)
db\\
&
=1.
\end{align*}

\vspace{-24pt}
\end{proof}

\section{Proof of (\ref{7.62})}\label{AppendixC}
\numberwithin{equation}{section}
\setcounter{equation}{0}
\setcounter{theorem}{0}

Let $B$ be a standard Brownian motion.
For $b>0$ define
$S_b:=\inf\big\{t\geq 0: B(t)>b\big\}$.
According to \cite[p.~411]{KaratzasShreve}, 
there exists a Poisson
random measure $\nu$
on $(0,\infty)^2$ such that
$S_b=\int_0^{\infty}\ell\nu((0,b]\times d\ell)$
and the characteristic measure of
$\nu$ is $\mu(d\ell)=d\ell/\sqrt{2\pi\ell^3}$.
According to the L\'evy-Hin\v{c}in formula,
for $\alpha\in\R$,
\be\label{C.1}
\E\big[e^{i\alpha S_b}\big]
=\exp\left[-b\int_0^{\infty}(1-e^{i\alpha\ell})
\mu(d\ell)\right],\quad b>0.
\ee
Define $\theta=\frac12\sqrt{|\alpha|}
(1-\mbox{sign}(\alpha)i)$.
Then
$M(t):=e^{2\theta B(t)-2\theta^2t}
=e^{2\theta B(t)+i\alpha t}$
is a martingale and $M(t\wedge T_b)$ is bounded.
Therefore,
$$
1=\lim_{t\rightarrow \infty}\E M(t\wedge S_b)
=\E\big[\lim_{t\rightarrow\infty} M(t\wedge S_b)\big]
=\E\big[e^{2\theta b+i\alpha S_b}\big].
$$
This implies
\be\label{C.2}
\E\big[e^{i\alpha S_b}\big]=e^{-2\theta b}
=\exp\big(-b\sqrt{|\alpha|}(1-\mbox{sign}(\alpha)i
)\big),\quad
b>0.
\ee
Comparison of (\ref{C.1}) and (\ref{C.2})
establishes (\ref{7.62}).

\end{document}